\documentclass{article}
\usepackage[utf8]{inputenc}
\usepackage[T1]{fontenc}
\usepackage{amsmath,amssymb}
\usepackage{graphicx}
\usepackage{amsthm, enumerate}
\theoremstyle{remark}
\usepackage[colorlinks=true]{hyperref}
\usepackage{latexsym}
\usepackage{xypic}
\usepackage{pxfonts,kpfonts}
\usepackage{mathrsfs}
\usepackage{fullpage}
\usepackage{wasysym}
\usepackage{enumitem}
\usepackage[affil-it]{authblk}
\usepackage{bbold}
\usepackage[all]{xy}
\usepackage[toc,page]{appendix}

\theoremstyle{definition}
\newtheorem{defi}{Definition}[section]
\newtheorem{const}[defi]{Construction}
\newtheorem{expl}[defi]{Example}

\newtheorem{rmk}[defi]{Remark}
\newtheorem*{conv}{Convention}
\newtheorem*{plan}{Organization of the paper}
\newtheorem*{merci}{Acknowlegments}
\newtheorem{notat}[defi]{Notation}

\theoremstyle{plain}

\newtheorem{pro}[defi]{Proposition}
\newtheorem{thm}[defi]{Theorem}
\newtheorem*{thmA}{Theorem A}
\newtheorem*{thmB}{Theorem B}
\newtheorem*{thmC}{Theorem C}
\newtheorem*{thmD}{Theorem D}
\newtheorem*{thmE}{Theorem E}
\newtheorem{cor}[defi]{Corollary}
\newtheorem{conj}[defi]{Conjecture}
\newtheorem{lmm}[defi]{Lemma}

\theoremstyle{remark}

\newtheorem*{sproof}{Sketch of proof}

\renewcommand{\Top}{{\mathrm{Top}}}

\newcommand{\hocolim}{\operatorname{hocolim}}

\title{Delooping of high-dimensional spaces of string links}
\author{J. Ducoulombier \thanks{The author was supported by the grant ERC-2015-StG 678156 GRAPHCPX \\ \textbf{2010 Mathematics Suject Classification} Primary:18D50, 55P48, 57R45, Secondary: 55S37, 57Q35, 57Q45, 57R40. }}
\date{}

\begin{document}

\maketitle \vspace{-30pt}

\abstract{We study a connection between a  multivariable version of the Goodwillie-Weiss' calculus of functors and derived mapping spaces of $k$-fold bimodules over a family of  operads extending the results obtained in \cite{Ducoulombier18}. As our main application, under the assumption $d_{i}+3\leq n$ for all $i\in \{1,\ldots,k\}$,  we produce explicit deloopings of high-dimensional spaces of string links modulo immersions from $\sqcup_{i}\mathbb{R}^{d_{i}}$ to $\mathbb{R}^{n}$ and their polynomial approximations.}

\tableofcontents \vspace{-10pt}

\section*{Introduction}

In the present work, we give explicit deloopings of high-dimensional generalizations of spaces of string links modulo immersions and their polynomial approximations in the sense of Goodwillie-Weiss. Let $n$ be the dimension of the ambient space. For $k\geq 1$ and fixed integers $d_{1}, \cdots, d_{k} \geq 1$, a high-dimensional string link of $k$-strands of dimensions $d_{1},\ldots,d_{k}$, respectively, is a smooth embedding 
$$
f:\displaystyle\coprod_{1\leq i\leq k}\mathbb{R}^{d_{i}}\longrightarrow \mathbb{R}^{n}
$$
that coincides outside a compact set with a fixed smooth embedding $\iota$ affine on each components. The space of high-dimensional string links with $k$ strands, endowed with the weak $\mathcal{C}^{\infty}$-topology, is denoted by $Emb_{c}(\,\coprod_{i=1}^{k}\mathbb{R}^{d_{i}}\,;\, \mathbb{R}^{n})$. Similarly, we consider the space of smooth immersions $Imm_{c}(\,\coprod_{i=1}^{k}\mathbb{R}^{d_{i}}\,;\, \mathbb{R}^{n})$ which coincide with $\iota$ outside a compact set and we focus our attention on the homotopy fiber over $\iota$ denoted by
\begin{equation}\label{M5}
\mathcal{L}(d_{1},\ldots,d_{k}\,;\,n) \coloneqq hofib\left(\,
Emb_{c}\left(\,\underset{1\leq i\leq k}{\coprod}\mathbb{R}^{d_{i}}\,;\, \mathbb{R}^{n}\,\right)\longrightarrow 
Imm_{c}\left(\,\underset{1\leq i\leq k}{\coprod}\mathbb{R}^{d_{i}}\,;\, \mathbb{R}^{n}\,\right)\,
\right).\vspace{10pt}
\end{equation}

For $k=1$ and $d+3\leq n$, the homotopy fiber (\ref{M5}), also called the high-dimensional space of long knots modulo immersions in the literature, has been intensively studied  (see \cite{Budney07,Dwyer12,Sakai14,Sinha09,Turchin04}) and is well understood. In particular, Fresse, Turchin and Willwacher in \cite{FresseTWbig} express the rational homotopy type of $\mathcal{L}(d\,;\,n)$ in terms of graph complexes using the following description of the homotopy fiber as a $(d+1)$-iterated loop space:
\begin{equation}\label{M2}
\mathcal{L}(d\,;\,n)\,\, \simeq \,\,\Omega^{d}Bimod_{\mathcal{C}_{d}}^{h}(\mathcal{C}_{d}\,;\,\mathcal{C}_{n})\,\, \simeq \,\,\Omega^{d+1}Operad^{h}(\mathcal{C}_{d}\,;\,\mathcal{C}_{n}).\footnote{Similar results have been obtained by Boavida-Weiss \cite{Weiss15} in the context of framed embeddings using configuration categories. Another delooping theorem has also been obtained by Sakai \cite{Sakai14} in terms of topological Stiefel manifold. Furthermore, Dwyer and Hess are working on an alternative proof using the Boardman-Vogt tensor product in the category of bimodules \cite{Dwyer14} while Batanin and DeLeger are also working on a proof based on an interpretation of bimodules over an operad in terms of algebras over a polynomial monad \cite{BataninDL17}.\vspace{-25pt}}
\end{equation}
The weak homotopy equivalences have been proved by the author in collaboration with Turchin \cite{Ducoulombier17,Ducoulombier18} using the category $Bimod_{O}$ of bimodules over an operad $O$. If we think about an operad as a monoid in the category of sequences with the plethysm, then a bimodule is a sequence $M=\{M(n),\,n\geq 0\}$ which is a left and right module over $O$. For instance, each operad is a bimodule over itself and each morphism of operads $\eta:O_{1}\rightarrow O_{2}$ is a morphism of bimodules over $O_{1}$. In the identifications (\ref{M2}), $\mathcal{C}_{d}$ is the $d$-dimensional little cubes operad  which is an equivalent version of the well known little disks operads encoding the structure of iterated loop spaces. The categories of operads and bimodules are equipped with model category structures and the symbol $"h"$ refers to the derived mapping space.\\{}\\ \vspace{-7pt}

For $k>1$, the homotopy fiber (\ref{M5}) is more mysterious. It is already known that the space $\mathcal{L}(d_{1},\ldots,d_{k}\,;\,n)$ has the homotopy type of a $d$-iterated loop space with $d=min\{ d_{1},\ldots, d_{k}\}$. Unfortunately, the technique developed to prove $(\ref{M2})$ doesn't generalized easily to the case $k>1$. Nevertheless, Tsopm\'en\'e and Turchin in \cite{Turchin15} have been able to get some information about the rational homotopy and rational homology of $\mathcal{L}(d_{1},\ldots,d_{k}\,;\,n)$ using graph complex homology under the condition $2d_{i}+2\leq n$ for all $i\in \{1,\ldots,k\}$. Without explicit description of the iterated loop space, we don't know how to improve this codimension condition. In the present paper, we show that the homotopy fiber (\ref{M5}) admits a description in terms of $d$-iterated loop space  similar to (\ref{M2}) under the condition $d_{i}+3\leq n$ for all $i\in \{1,\ldots,k\}$.

\newpage

For this purpose, we introduce the category of $k$-fold bimodules over $\vec{O}$ where $\vec{O}$ denotes a family of operadic maps $f_{i}:O_{i}\rightarrow O$, with $i\in \{1,\ldots,k\}$. This category is denoted by $Bimod_{\vec{O}}$. Roughly speaking, an object in $Bimod_{\vec{O}}$ is a family of spaces 
$$
M(n_{1},\ldots,n_{k}),\hspace{15pt} \text{ with } n_{i}\in \mathbb{N}\cup \{+\} \text{ and } (n_{1},\ldots,n_{k})\neq (+,\ldots,+),
$$
which is a bimodule over the operad $O_{i}$ along the $i$-th variable. The left module structure admits some constraints encoded by the operad $O$. For the family of operadic maps $f_{i}:\mathcal{C}_{d_{i}}\rightarrow \mathcal{C}_{n}$ and $n>d_{i}$, the two following families of spaces are examples of $k$-fold bimodules:
$$
\mathbb{O}^{+}(n_{1},\ldots,n_{k})=\underset{\substack{1\leq i \leq k \\ n_{i}\neq +}}{\prod} \mathcal{C}_{d_{i}}(n_{i}) \hspace{15pt}\text{and}\hspace{15pt} \mathcal{R}^{k}_{n}(n_{1},\ldots,n_{k}) = \mathcal{R}_{n}\left( \underset{\substack{1\leq i \leq k \\ n_{i}\neq +}}{\textstyle\sum} n_{i} \right),
$$
where $\mathcal{R}_{n}$ is an equivalent version of the little cubes operad called the $n$-dimensional little rectangles operad. The aim of this paper is to study model category structures on $Bimod_{\vec{O}}$ and to prove the following statement:

\begin{thmA}[Theorem \ref{R3}]
For $\vec{O}=\{f_{i}:\mathcal{C}_{d_{i}}\rightarrow\mathcal{C}_{n}\}$ with $d_{i}+3\leq n$ for all $i\in \{1,\ldots,k\}$, one has 
\begin{equation*}
\mathcal{L}(d_{1},\ldots,d_{k}\,;\,n)\,\,\simeq\,\, \Omega^{d}\mathbb{F}'_{\vec{O}}(\mathcal{R}_{n}^{k}),
\end{equation*}
where $\mathbb{F}'_{\vec{O}}(\mathcal{R}_{n}^{k})$ is defined as a limit of mapping spaces of $k$-fold bimodules. In particular, if $d_{1}=\cdots =d_{k}=d$, then 
\begin{equation*}
\mathcal{L}(d,\ldots,d\,;\,n)\,\,\simeq\,\, \Omega^{d}Bimod_{\vec{O}}^{h}(\mathbb{O}^{+}\,;\,\mathcal{R}_{n}^{k}).
\end{equation*} 
\end{thmA}

The above theorem is proved in two steps. We introduce the notion of $k$-fold infinitesimal bimodule over a family of operadic maps. The first step consists in linking derived mapping spaces of $k$-fold infinitesimal bimodules with embedding spaces whose source manifold has $k$ components. For this purpose, we use a multivariable version of the Goodwillie-Weiss' calculus of functors developped in  \cite{Weiss99.2,Munson14}. The second step connects derived mapping spaces of $k$-fold bimodules and $k$-fold infinitesimal bimodules. The rest of the introduction gives more detailed account of these two steps.

\subsubsection*{First step: Multivariable manifold calculus theory}

This step relies on the papers \cite{Arone14, Munson14}. The theory of manifold calculus of functors developed by Weiss \cite{Weiss99} and Goodwillie-Weiss \cite{Weiss99.2} provides an approximation of contravariant functors $F:\mathcal{O}(M)\rightarrow Top$, where $M$ is a smooth manifold and $\mathcal{O}(M)$ is the poset of its open subsets. They build a tower 
$$
\xymatrix{
& F \ar[dl] \ar[d] \ar[dr] \ar[drr] & & & \\
T_{0}F & T_{1}F\ar[l] & T_{2}F \ar[l] & T_{3}F \ar[l] & \cdots \ar[l]
}
$$ 
which converges to $F$ under some conditions on the functor. By converging, we mean that the induced natural transformation $F\rightarrow T_{\infty}F$, where $T_{\infty}F$ is the limit of the tower, is a pointwise weak homotopy equivalence. The $r$-th polynomial approximation $T_{r}F$ has the advantage to be easier to understand than the initial functor. For instance, Arone and Turchin \cite[Theorem 5.9]{Arone14} prove that $T_{r}F$ can be identified with a derived mapping space of infinitesimal bimodules (or just right modules depending on the context) over the little cubes operad (see Definition \ref{M3}).\\{}\\ \vspace{-7pt}

Thereafter, this theory has been extended to manifold $M$ with many components by Munson and Volic \cite{Munson12} in order to study the usual space of string links (i.e. in the particular case $d_{1}=\cdots = d_{k}=1$). From the homotopy fiber (\ref{M5}) viewed as a contravariant functor 
$$
\vec{\mathcal{U}}=\big(\, \mathcal{U}_{1},\ldots,\mathcal{U}_{k}\,\big)\longmapsto  \mathcal{L}(\vec{U})=hofib\left(\,
Emb_{c}\left(\,\vec{\mathcal{U}}\,;\, \mathbb{R}^{n}\,\right)\longrightarrow 
Imm_{c}\left(\,\vec{\mathcal{U}}\,;\, \mathbb{R}^{n}\,\right)\,
\right)
$$ 
where $\mathcal{U}_{i}$ ranges over a certain category of open subsets of $\mathbb{R}^{d_{i}}$, this theory provides a multivariable polynomial approximation $ T_{\vec{r}}\,\mathcal{L}(-)$, for any $\vec{r}=(r_{1},\ldots,r_{k})\in \mathbb{N}^{k}$. The $\vec{r}$-th approximation has the property to be homotopically characterized by open subsets for which $\mathcal{U}_{i}$ is diffeomorphic to at most $r_{i}$ disjoint open cubes and one anti-cube (see Section \ref{M9}). Similarly to \cite{Arone14} for high-dimensional space of long knots modulo immersions, each polynomial approximation can be identified with derived mapping space of truncated $k$-fold infinitesimal bimodules.\\{}\\ \vspace{-7pt}

The category of $k$-fold infinitesimal bimodules over $\vec{O}$, denoted by $Ibimod_{\vec{O}}$, has for objects families of spaces $N=\{N(n_{1},\ldots,n_{k}), \,n_{i}\in \mathbb{N}\}$ which are right modules over $O_{i}$ along the $i$-th variable. Such an object has also left operations with constraints encoded by $O$ that differ from the left operations of $k$-fold bimodules. For the family of operadic maps $f_{i}:\mathcal{C}_{d_{i}}\rightarrow \mathcal{C}_{n}$ and $n>d_{i}$, the two following families of spaces are examples of $k$-fold infinitesimal bimodules:
$$
\mathbb{O}(n_{1},\ldots,n_{k})=\mathcal{C}_{d_{1}}(n_{1})\times \cdots \times \mathcal{C}_{d_{k}}(n_{k}) \hspace{15pt}\text{and}\hspace{15pt} \mathcal{R}^{k}_{n}(n_{1},\ldots,n_{k}) = \mathcal{R}_{n}(n_{1}+\cdots + n_{k}).
$$
By $\vec{r}$-truncated objects, with $\vec{r}=(r_{1},\ldots,r_{k})\in \mathbb{N}^{k}$, we mean families $N=\{N(n_{1},\ldots,n_{k}), \, 1\leq n_{i}\leq  r_{i}\}$ with the same kind of structure. By abuse of notation, we denote by $T_{\vec{r}}\,Ibimod_{\vec{O}}$ the category of $\vec{r}$-truncated $k$-fold infinitesimal bimodules and we prove the following result:  \\{}\\ \vspace{-13pt}

\begin{thmB}[Theorem \ref{K2}]
For $\vec{O}=\{f_{i}:\mathcal{C}_{d_{i}}\rightarrow \mathcal{C}_{n}\}$, one has weak homotopy equivalences
$$
\begin{array}{rll}\vspace{10pt}
T_{\vec{r}}\,\mathcal{L}(d_{1},\cdots,d_{k}\,;\,n)   & \hspace{-8pt}\simeq  T_{\vec{r}}\,Ibimod_{\vec{O}}^{h}(\mathbb{O}\,;\,\mathcal{R}_{n}^{k}),    & \text{for } \vec{r}\in \mathbb{N}^{k},\\ 
\mathcal{L}(d_{1},\cdots,d_{k}\,;\,n)   & \hspace{-8pt} \simeq T_{\vec{\infty}}\,\mathcal{L}(d_{1},\cdots,d_{k}\,;\,n)  \simeq  Ibimod_{\vec{O}}^{h}(\mathbb{O}\,;\,\mathcal{R}_{n}^{k}), & \text{if } d_{i}+3\leq n \text{ for all } i\in \{1,\ldots,k\}.
\end{array} 
$$
\end{thmB}\vspace{3pt}

\subsubsection*{Second step: Relations between the categories $Ibimod_{\vec{O}}$ and $Bimod_{\vec{O}}$}

In this step, we fix a family of operadic maps $\vec{O}=\{f_{i}:O_{i}\rightarrow O\}$. Similarly to the previous step with the little cubes operads, we can consider the $k$-fold infinitesimal bimodule $\mathbb{O}$ and the $k$-fold bimodule $\mathbb{O}^{+}$ defined as follows:\vspace{7pt}
\begin{equation*}
\mathbb{O}(n_{1},\ldots,n_{k})=O_{1}(n_{1})\times \cdots \times O_{k}(n_{k}) \hspace{15pt}\text{and}\hspace{15pt} \mathbb{O}^{+}(n_{1},\ldots,n_{k})=\underset{\substack{1\leq i\leq k\\ n_{i}\neq +}}{\prod} O_{i}(n_{i}).\vspace{4pt}
\end{equation*}
Given a morphism of $k$-fold bimodules $\eta:\mathbb{O}^{+}\rightarrow M$, we prove that $M$ inherits a $k$-fold infinitesimal bimodule structure. We also introduce the following two subspaces where $M_{i}=\{M_{i}(n)=M(+,\ldots,+,n,+,\ldots,n)\}$ is a bimodule over $O_{i}$ and $T_{\vec{r}}\,Bimod_{\vec{O}}$ is the category of $\vec{r}$-truncated $k$-fold bimodules: 
$$
\begin{array}{rcc}\vspace{10pt}
\mathbb{F}_{\vec{O}}(M) & \subset & Map_{\ast}\left( 
\textstyle\sum \big( \underset{i \in A}{\textstyle\prod} O_{i}(2)\big)\,;\, Bimod_{\vec{O}}^{h}(\mathbb{O}^{+}\,;\, M)\right) \times \underset{1\leq i\leq k}{\displaystyle \prod} Map_{\ast}\left( 
\textstyle\sum O_{i}(2)\,;\, Bimod_{O_{i}}^{h}(O_{i}\,;\, M_{i})\right), \\ 
T_{\vec{r}}\,\mathbb{F}_{\vec{O}}(M) & \subset & Map_{\ast}\left( 
\textstyle\sum \big( \underset{i \in A}{\textstyle\prod} O_{i}(2)\big)\,;\, T_{\vec{r}}\,Bimod_{\vec{O}}^{h}(\mathbb{O}^{+}\,;\, M)\right) \times \underset{1\leq i\leq k}{\displaystyle \prod} Map_{\ast}\left( 
\textstyle\sum O_{i}(2)\,;\, T_{r_{i}}\,Bimod_{O_{i}}^{h}(O_{i}\,;\, M_{i})\right).
\end{array} \vspace{5pt}
$$  
The symbol $\Sigma$ refers to the suspension and $Map_{\ast}$ to morphisms of pointed spaces. The above spaces have explicit descriptions in terms of limits of  diagrams (see Section \ref{P1}). In particular, if $O_{1}=\cdots =O_{k}=O$, then the limits can be simplified and the above spaces take the following form: 
$$
\begin{array}{rcl}\vspace{10pt}
\mathbb{F}_{\vec{O}}(M) & = & Map_{\ast}\left( 
\textstyle\sum O(2)\,;\, Bimod_{\vec{O}}^{h}(\mathbb{O}^{+}\,;\, M)\right), \\ 
T_{\vec{r}}\,\mathbb{F}_{\vec{O}}(M) & = & Map_{\ast}\left( 
\textstyle\sum O(2)\,;\, T_{\vec{r}}\,Bimod_{\vec{O}}^{h}(\mathbb{O}^{+}\,;\, M)\right).
\end{array} 
$$

\newpage

By using well chosen cofibrant resolutions of $\mathbb{O}$ and $\mathbb{O}^{+}$ in the categories of $k$-fold infinitesimal bimodules and $k$-fold bimodules, respectively, we are able to build explicitly the two following maps only if the operads $O_{1},\ldots,O_{k}$ are $2$-reduced (i.e. $O_{i}(0)=O_{i}(1)=\ast$): 
\begin{equation}\label{N0}
\begin{array}{rcl}\vspace{7pt}
\gamma:\mathbb{F}_{\vec{O}}(M) & \longrightarrow  & Ibimod_{\vec{O}}^{h}(\mathbb{O}\,;\,M), \\ 
\gamma_{\vec{r}}:T_{\vec{r}}\,\mathbb{F}_{\vec{O}}(M) & \longrightarrow  & T_{\vec{r}}\,Ibimod_{\vec{O}}^{h}(\mathbb{O}\,;\,M).
\end{array} 
\end{equation}

Unfortunately, the maps (\ref{N0}) are not necessarily weak homotopy equivalences. In \cite{Ducoulombier18}, in the particular case $k=1$, we give with Turchin conditions for (\ref{N0}) to be weak homotopy equivalences. An operad satisfying this condition is said to be \textit{coherent}. The property of being coherent is expressed in terms of sequence of morphisms $\{\rho_{i}\}_{i\geq 1}$ of certain homotopy colimits that must be weak homotopy equivalences (see Definition \ref{N1}). An operad is $r$-coherent, with $r\in \mathbb{N}$, if the morphisms $\{\rho_{i}\}_{1\leq i\leq r}$ are weak homotopy equivalences.  \\{}\\ \vspace{-15pt}

\begin{thmC}[Theorem \ref{D3}] Let $\vec{j}=(j_{1},\ldots,j_{k})\in \mathbb{N}^{k}$ and  $\vec{r}\leq \vec{j}$. If, for all $i\in \{1,\ldots, k\}$, the operad $O_{i}$ is well pointed, $\Sigma$-cofibrant and $j_{i}$-coherent, then the map $\gamma_{\vec{r}}$ is a weak homotopy equivalence under some technical conditions on $M$. In particular, if the operads $O_{1},\ldots,O_{k}$ are coherent, then the map $\gamma$ is a weak homotopy equivalence.\vspace{6pt}
\end{thmC}

In \cite{Ducoulombier18}, we show that the little cubes operad is coherent. Unfortunately, the little cubes operad is only weakly $2$-reduced (i.e. $\mathcal{C}_{d_{i}}(0)=\ast$ and $\mathcal{C}_{d_{i}}(1)\simeq \ast$). By using properties on the model category structures associated to the categories of $k$-fold bimodules and $k$-fold infinitesimal bimodules, we can extend the previous statement to the family of little cubes operads. For this purpose, we use the Fulton-MacPherson operad which is $2$-reduced and connected by a zig-zag of weak homotopy equivalences to the little cubes operad. \\{}\\ \vspace{-15pt}

\begin{thmD}[Theorem \ref{R1}]
Let $\vec{O}=\{f_{i}:\mathcal{C}_{d_{1}}\rightarrow \mathcal{C}_{n}\}$ and let $\eta:\mathbb{O}^{+}\rightarrow M$ be a map of $k$-fold bimodules. Under some technical conditions on $M$, one has the following weak homotopy equivalences: 
\begin{equation*}
\begin{array}{rcl}\vspace{10pt}
\mathbb{F}_{\vec{O}}(M) & \simeq & Ibimod_{\vec{O}}^{h}(\mathbb{O}\,;\,M), \\ 
T_{\vec{r}}\,\mathbb{F}_{\vec{O}}(M) & \simeq & T_{\vec{r}}\,Ibimod_{\vec{O}}^{h}(\mathbb{O}\,;\,M).
\end{array} 
\end{equation*}
In particular, if $d_{1}=\cdots = d_{k}=d$, then one has
\begin{equation*}
\begin{array}{rcl}\vspace{10pt}
Map_{\ast}\left( 
\textstyle\sum \mathcal{C}_{d}(2)\,;\, Bimod_{\vec{O}}^{h}(\mathbb{O}^{+}\,;\, M)\right) & \simeq & Ibimod_{\vec{O}}^{h}(\mathbb{O}\,;\,M), \\ 
Map_{\ast}\left( 
\textstyle\sum \mathcal{C}_{d}(2)\,;\, T_{\vec{r}}\,Bimod_{\vec{O}}^{h}(\mathbb{O}^{+}\,;\, M)\right) & \simeq & T_{\vec{r}}\,Ibimod_{\vec{O}}^{h}(\mathbb{O}\,;\,M).
\end{array}
\end{equation*}\vspace{3pt}
\end{thmD}

The last step consists in identifying the iterated loop spaces.  Since the suspension of $\mathcal{C}_{d}(2)$ is homotopically equivalent to the sphere of dimension $d$, the space $\sum(\prod_{i}\mathcal{C}_{d_{i}}(2))$ can be factorized through $S^{d}$. Finally, we show that \vspace{5pt}
$$
\mathbb{F}_{\vec{O}}(M) \simeq  \Omega^{d}\mathbb{F}'_{\vec{O}}(M)\hspace{15pt}\text{and}\hspace{15pt} T_{\vec{r}}\,\mathbb{F}_{\vec{O}}(M)\simeq \Omega^{d}\big(\,T_{\vec{r}}\,\mathbb{F}'_{\vec{O}}(M)\,\big),\vspace{5pt}
$$
where $\mathbb{F}'_{\vec{O}}(M)$ and $T_{\vec{r}}\,\mathbb{F}'_{\vec{O}}(M)$ are subspaces 
$$
\mathbb{F}'_{\vec{O}}(M) \subset Map_{\ast}\left( 
 S^{0} \,;\, Bimod_{\vec{O}}^{h}(\mathbb{O}^{+}\,;\, M)\right)\times \underset{1\leq i \leq k}{\displaystyle \prod}  Map_{\ast}\left( 
 S^{d_{i}-d} \,;\, Bimod_{O_{i}}^{h}(O_{i}\,;\, M_{i})\right), \vspace{3pt}
$$

$$
T_{\vec{r}}\,\mathbb{F}'_{\vec{O}}(M)   \subset Map_{\ast}\left( 
S^{0}\,;\, T_{\vec{r}}\,Bimod_{\vec{O}}^{h}(\mathbb{O}^{+}\,;\, M)\right) \times \underset{1\leq i \leq k}{\displaystyle \prod} Map_{\ast}\left( 
 S^{d_{i}-d}\,;\, T_{r_{i}}\,Bimod_{O_{i}}^{h}(O_{i}\,;\, M_{i})\right).\vspace{5pt}
$$
with explicit descriptions in terms of limits of mapping spaces of $k$-fold bimodules.

\begin{thmE}[Theorem \ref{L6}]
Let $\vec{O}=\{f_{i}:\mathcal{C}_{d_{i}}\rightarrow \mathcal{C}_{n}\}$ and let $\eta:\mathbb{O}^{+}\rightarrow M$ be a map of $k$-fold bimodules. Under some technical conditions on $M$, one has the following weak homotopy equivalences: 
\begin{equation*}
\begin{array}{rcl}\vspace{10pt}
\Omega^{d}\big(\mathbb{F}'_{\vec{O}}(M)\big) & \simeq & Ibimod_{\vec{O}}^{h}(\mathbb{O}\,;\,M), \\ 
\Omega^{d}\big( T_{\vec{r}}\,\mathbb{F}'_{\vec{O}}(M)\big) & \simeq & T_{\vec{r}}\,Ibimod_{\vec{O}}^{h}(\mathbb{O}\,;\,M).
\end{array} 
\end{equation*}
In particular, if $d_{1}=\cdots = d_{k}=d$, then one has
\begin{equation*}
\begin{array}{rcl}\vspace{10pt}
\Omega^{d}\big( Bimod_{\vec{O}}^{h}(\mathbb{O}^{+}\,;\, M)\big) & \simeq & Ibimod_{\vec{O}}^{h}(\mathbb{O}\,;\,M), \\ 
\Omega^{d} \big( T_{\vec{r}}\,Bimod_{\vec{O}}^{h}(\mathbb{O}^{+}\,;\, M)\big) & \simeq & T_{\vec{r}}\,Ibimod_{\vec{O}}^{h}(\mathbb{O}\,;\,M).
\end{array}
\end{equation*}\vspace{3pt}
\end{thmE}

 As a consequence of this result, we are able to give an explicit description of $\mathcal{L}(d_{1},\ldots,d_{k}\,;\,n)$ and their polynomial approximations. Finally, we produce another application of the above theorem in order to get a description of the iterated loop spaces associated to polynomial approximations of high-dimensional spaces of string links modulo immersions with singularities (see Section \ref{N5}). In the latter case, we don't know if the polynomial approximation converges. So, we only obtain results on the polynomial approximations.\vspace{7pt}

\begin{plan}
In Section \ref{N2}, we review some preliminaries about topological operads, bimodules and infinitesimal bimodules over an operad. Little cubes operads and non-overlapping little cubes bimodules are introduced.  \vspace{3pt}

Section \ref{B5} is devoted to the category of $k$-fold infinitesimal bimodules over a family of topological operads relative to another operad. We build the free $k$-fold infinitesimal bimodule functor and we introduce two different model category structures: Reedy and projective. In the last subsection, we give an explicit and functorial way to build  cofibrant replacements in both model category structures. \vspace{3pt}

The main goal of Section \ref{L3} is to prove \textbf{Theorem B} (alias Theorem \ref{K2}). We recall some basics about multivariable Goodwillie-Weiss' calculus of functors and we adapt this theory for manifolds which are not necessarily compact. We prove that the category of $k$-fold infinitesimal bimodules over a family of operads is equivalent to a category of diagrams in spaces. We use this description to prove that polynomial approximations of good context-free functors produce $k$-fold infinitesimal bimodules. \vspace{3pt}

Section \ref{N3} is devoted to the category of $k$-fold bimodules over a family of topological operads relative to another operad. Similarly to Section \ref{B5}, we build the free $k$-fold bimodule functor and we introduce two different model category structures: Reedy and projective. In the last subsection, we give an explicit and functorial way to build  cofibrant replacements in both model category structures. In particular, we show that the resolution associated to the $k$-fold bimodule $\mathbb{O}^{+}$ has some specific properties used to prove the main theorems in Section \ref{D5}.\vspace{3pt}

In Section \ref{D5} we prove \textbf{Theorem C}  (alias Theorem \ref{D3}). For this purpose, we introduce the notion of coherent operad and we use explicit cofibrant replacements of $\mathbb{O}$ and $\mathbb{O}^{+}$. Unfortunately, the resolution introduced in Section \ref{B5} is not well adapted to define the maps (\ref{N0}). To solve this problem, we consider an alternative resolution of $\mathbb{O}$ obtained as a semi-direct product of the resolution of $\mathbb{O}^{+}$ as a $k$-fold bimodule and another $k$-fold sequence $\mathcal{I}$. \vspace{3pt}

In the last section \ref{N4} we prove \textbf{Theorem D} (alias Theorem \ref{J3}). We show how to get iterated loop spaces from a map $\eta:\mathbb{O}^{+}\rightarrow M$ of the $k$-fold bimodules. In particular we prove \textbf{Theorem A} (alias Theorem \ref{L6}). Then we give another application to the space $\vec{u}$-immersions modulo immersions (see Definition \ref{N5}).\vspace{7pt}
\end{plan}

\begin{conv}
By a space we mean a compactly generated Hausdorff space and by abuse of notation we denote by $Top$ this category (see e.g. \cite[section 2.4]{Hovey99}). If $X$, $Y$ and $Z$ are spaces, then $Top(X;Y)$ is equipped with the compact-open topology in order to have a homeomorphism $Top(X;Top(Y;Z))\cong Top(X\times Y;Z)$. We also refer the reader to \cite[Appendix A]{Ducoulombier18} for a more detailed study and more references. \vspace{3pt}

The category $Top$ is endowed with a cofibrantly generated monoidal model category structure with weak homotopy equivalences and Serre fibrations. In the paper the categories considered are enriched over $Top$. By convention, if $\mathcal{C}$ is a model category enriched over $Top$, then the derived mapping space $\mathcal{C}^{h}(A;B)$ is the space $\mathcal{C}(A^{c};B^{f})$ with $A^{c}$ a cofibrant replacement of $A$ and $B^{f}$ a fibrant replacement of $B$.\vspace{-10pt}
\end{conv}

\newpage

\section{Background and convention} \label{N2}
In the following, we recall the terminology related to the notion of topological operad and "infinitesimal" bimodule. Let $\Sigma$ be the category of finite sets and isomorphisms between them. By a $\Sigma$-sequence we mean a covariant functor from $\Sigma$ to the category of topological spaces. We will write $M(n)$ for $M(\{1,\ldots,n\})$ and $M(0)$ for $M(\emptyset)$. In practice, a $\Sigma$-sequence is given by a family of topological spaces $M(0),$ $M(1),\ldots$ together with  actions of the symmetric groups: for each permutation $\sigma \in \Sigma_{n}$, there is a continuous map
\begin{equation}\label{A0}
\begin{array}{rcl}
\sigma^{\ast}:M(n) & \longrightarrow & M(n); \\ 
 x & \longmapsto & x\cdot \sigma,
\end{array} 
\end{equation}
satisfying the relation $(x\cdot\sigma)\cdot \tau=x\cdot(\sigma\tau)$ with $\tau\in \Sigma_{n}$. We denote by $Seq$ the category of $\Sigma$-sequences. \vspace{7pt}

Given an integer $r\geq 1$, we also consider the category of $r$-truncated sequences $T_{r}Seq$. Let $\Sigma_{r}$ be the category of finite sets of cardinality smaller than $r$ and isomorphisms between them. An $r$-truncated sequence is a functor from $\Sigma_{r}$ to the category of topological spaces. In practice, an $r$-truncated sequence is given by a family of topological spaces $M(0),\ldots, M(r)$ together with an action of the symmetric group $\Sigma_{n}$ for each $n\leq r$. A (possibly truncated) sequence is said to be \textit{pointed} if there is a distinguished element $\ast_{1}\in M(1)$ called \textit{unit}. One has an obvious functor 
$$
T_{r}(-):Seq \longrightarrow T_{r}Seq.\vspace{7pt}
$$

\begin{defi}\textbf{Topological operad}

\noindent An \textit{operad} is a pointed $\Sigma$-sequence $O$ together with operations called \textit{operadic compositions}
\begin{equation}\label{A1}
\circ_{a}:O(A)\times O(B)\longrightarrow O(A\cup_{a}B),\hspace{15pt} \text{with }a\in A, 
\end{equation}
where $A\cup_{a}B=(A\setminus\{a\})\coprod B$. These operations satisfy compatibility relations with the symmetric group action as well as associativity and unit axioms \cite{Arone14}. A map between two operads should respect the operadic compositions.  We denote by $Operad$ the categories of operads. An operad $O$ is said to be \textit{reduced} (resp. $2$\textit{-reduced}) if $O(0)=\ast$ (resp. $O(0)=O(1)=\ast$).\vspace{7pt}

Given an integer $r\geq 1$,we also consider the category of $r$-truncated operads $T_{r}Operad$. The objects are pointed $r$-truncated sequences endowed with operadic compositions (\ref{A1}) for $|A\cup_{a}B|\leq r$ and $|A|\leq r$. One has an obvious functor
$$
T_{r}(-):Operad \longrightarrow T_{r}Operad.\vspace{7pt}
$$

In practice, an operad is determined by a family of topological spaces $O(0),$ $O(1),\ldots$ together with actions of the symmetric groups and operadic compositions of the form 
$$
\circ_{i}:O(n)\times O(m)\longrightarrow O(n+m-1),\hspace{15pt} \text{with } i\in \{1,\ldots,n\},
$$
satisfying some relations. In the rest of the paper, we often switch between the two definitions in order to simplify the notation.\vspace{7pt}
\end{defi}

\begin{expl}\textbf{The overlapping little rectangles operad $\mathcal{R}^{\infty}_{d}$}\label{E0}

\noindent A $d$-dimensional little rectangle is a continuous map $r:[0\,,\,1]^{d}\rightarrow [0\,,\,1]^{d}$ arising from an embedding preserving the direction of the axes. In other words, a little rectangle $r$ is an application of the form $$r(t_{1},\ldots,t_{d})=(a_{1}t_{1}+b_{1},\ldots, a_{d}t_{d}+b_{d})$$ for some real constant $a_{i}$ and $b_{i}$, with $a_{i}>0$. The operad $\mathcal{R}^{\infty}_{d}$ is the sequence $\{\mathcal{R}^{\infty}_{d}(n)\}$ whose $n$-th component is given by $n$ little rectangles. The unit is the identity map whereas $\sigma\in \Sigma_{n}$ permutes the parameters as follows:\vspace{4pt}
$$
\sigma^{\ast}:\mathcal{R}^{\infty}_{d}(n)\longrightarrow \mathcal{R}^{\infty}_{d}(n)\,\,\,;\,\,\, <r_{1},\ldots,r_{n}>  \longmapsto  <r_{\sigma(1)},\ldots,r_{\sigma(n)}>.\vspace{4pt}
$$ 
The operadic composition $\circ_{i}$ is given by the formula\vspace{3pt}
$$
\begin{array}{clcl}\vspace{3pt}
\circ_{i}: & \mathcal{R}^{\infty}_{d}(n)\times \mathcal{R}^{\infty}_{d}(m) & \longrightarrow & \mathcal{R}^{\infty}_{d}(n+m-1); \\ 
 &  <r_{1},\ldots,r_{n}>\,;\, <r'_{1},\ldots,r'_{m}> & \longmapsto & <r_{1},\ldots,r_{i-1}, r_{i}\circ r'_{1},\ldots, r_{i}\circ r'_{m},r_{i+1},\ldots,r_{n}>.
\end{array} \vspace{3pt}
$$
By convention $\mathcal{R}^{\infty}_{d}(0)$ is the one point topological space and the operadic composition $\circ_{i}$ with this point consists in removing the $i$-th little rectangle. 
\end{expl}\vspace{-20pt}

\newpage

\begin{expl}\textbf{The little rectangles operad $\mathcal{R}_{d}$ and the little cubes operad $\mathcal{C}_{d}$}

\noindent The $d$-dimensional little rectangles operad is a sub-operad of $\mathcal{R}^{\infty}_{d}$ whose $n$-th component is the configuration of $n$ little rectangles with disjoint interiors. In other words, $\mathcal{R}_{d}(n)$ is the subspace of $\mathcal{R}_{d}^{\infty}(n)$ formed by configurations $<r_{1},\ldots,r_{n}> $ satisfying the relation 
$$
r_{i}(]0\,,\,1[)\,\cap\, r_{j}(]0\,,\,1[)=\emptyset,\hspace{15pt}\forall i\neq j.\vspace{2pt}
$$

The $d$-dimensional little cubes operad $\mathcal{C}_{d}$ is a sub-operad of $\mathcal{R}_{d}$ whose $n$-th component is the configuration of $n$ little rectangles with disjoint interiors $<r_{1},\ldots,r_{n}>$ for which the rectangles $r_{i}$ are of the form 
$$
r_{i}(t_{1},\ldots,t_{d})=(a^{i}t_{1}+b_{1}^{i},\ldots, a^{i}t_{d}+b_{d}^{i}),\hspace{15pt} \text{with } a^{i}>0. \vspace{7pt}  
$$

In both cases, the operadic compositions and the action of the symmetric group arise from the overlapping little rectangles operad $\mathcal{R}_{d}^{\infty}$. Furthermore, these two operads are obviously equivalent to the well known little discs operad.

\begin{figure}[!h]
\begin{center}
\includegraphics[scale=0.27]{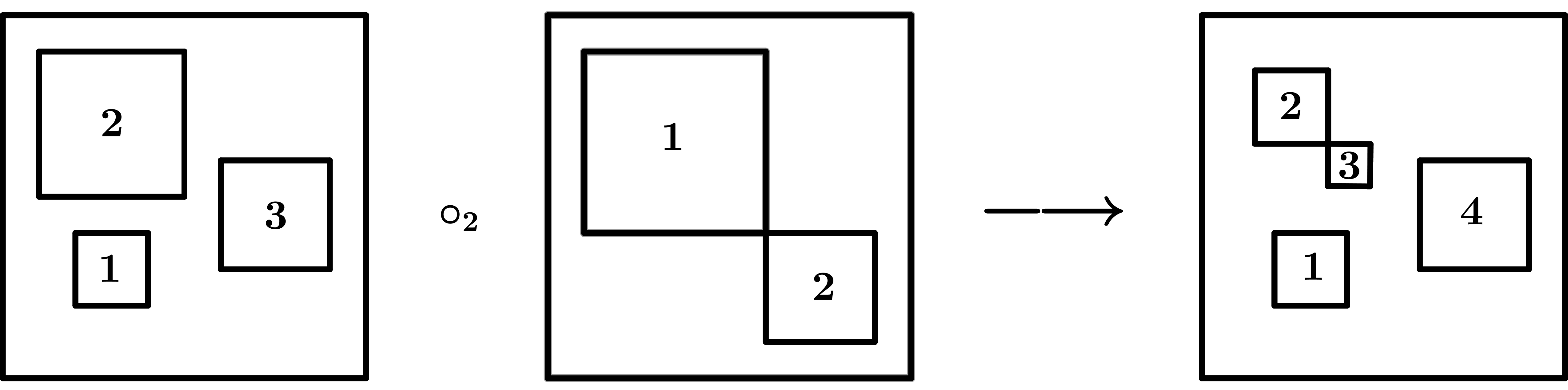}
\caption{Illustration of the operadic composition $\circ_{2}:\mathcal{C}_{2}(3)\times \mathcal{C}_{2}(2)\rightarrow \mathcal{C}_{2}(4)$.}\vspace{-10pt}
\end{center}
\end{figure}

For $d\leq d'$, there are two maps of operads $\iota:\mathcal{R}_{d}\rightarrow\mathcal{R}_{d'}$ and $\iota':\mathcal{C}_{d}\rightarrow\mathcal{C}_{d'}$. The map $\iota$ sends a little rectangle $r$ of dimension $d$ to the rectangle of dimension $d'$ given by the formula 
$$
\iota(r)(t_{1},\ldots,t_{d'})=(r(t_{1},\ldots,t_{d}),t_{d+1},\ldots, t_{d'}).$$
The map $\iota'$ sends a little rectangle of dimension $d$ of the form $r(t_{1},\ldots,t_{d})=(a t_{1}+b_{1},\ldots,at_{d}+b_{d})$ to the little rectangle of dimension $d'$ given by  the formula
$$
\iota'(r)(t_{1},\ldots,t_{d'})=\big( at_{1}+b_{1},\ldots, at_{d}+b_{d},at_{d+1}+ (1-a)/2,\ldots, at_{d'}+ (1-a)/2\big).\vspace{-15pt}
$$

\begin{figure}[!h]
\hspace{-32pt}\includegraphics[scale=0.55]{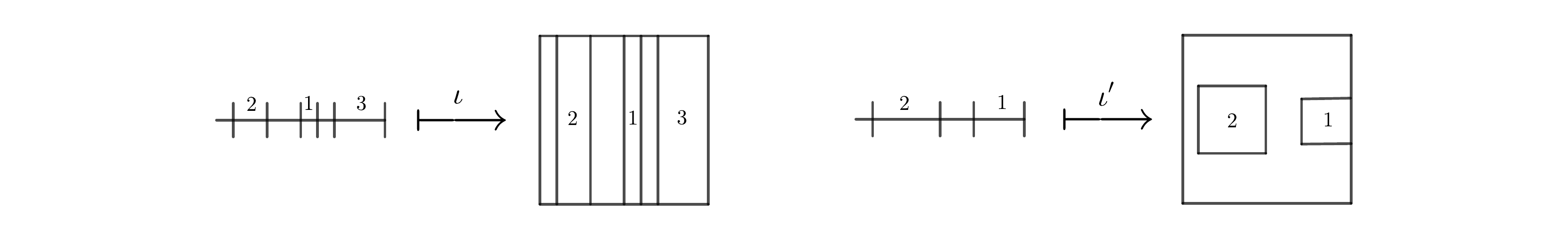}\vspace{-10pt}
\caption{Illustration of the maps $\iota:\mathcal{R}_{1}(3)\rightarrow\mathcal{R}_{2}(3)$ and $\iota':\mathcal{C}_{1}(2)\rightarrow\mathcal{C}_{2}(2)$.}
\end{figure}
\end{expl}

\begin{defi}\label{B9}\textbf{Bimodule over an operad}

\noindent A bimodule over $O$, also called $O$-\textit{bimodule}, is given by a sequence $M\in Seq$ together with operations 
\begin{equation}\label{A2}
\begin{array}{llr}\vspace{7pt}
\gamma_{r}: & M(A)\times \underset{a\in A}{\prod} O(B_{a})\longrightarrow  M(\underset{a\in A}{\coprod} B_{a}), & \text{right operations},\\ \vspace{7pt}
\gamma_{l}: & O(A)\times \underset{a\in A}{\prod} M(B_{a})\longrightarrow  M(\underset{a\in A}{\coprod} B_{a}),& \text{left operations},\vspace{-5pt}
\end{array}
\end{equation}
satisfying compatibility with the symmetric group action, associativity and unit axioms \cite{Arone14}. In particular, there is a continuous map $\gamma_{\emptyset}:O(\emptyset)\rightarrow M(\emptyset)$ in arity $0$. A map between $O$-bimodules should respect the operations. We denote by $Bimod_{O}$ the category of $O$-bimodules. Thanks to the unit in $O(1)$, the right operations $\gamma_{r}$ can equivalently be defined as a family of continuous maps
$$
\circ^{a}:M(A)\times O(B)\longrightarrow M(A\cup_{a}B),\hspace{15pt}\text{with }a\in A.\vspace{-20pt}
$$

\newpage

Given an integer $r\geq 0$, we also consider the category of $r$-truncated bimodules $T_{r}Bimod_{O}$. An object is an $r$-truncated sequence endowed with left and right operations (\ref{A2}) under the conditions $|A|\leq r$ and $\sum |B_{a}|\leq r$ for $\gamma_{r}$ and the condition $\sum |B_{a}|\leq r$ for $\gamma_{l}$. One has an obvious functor 
$$
T_{r}(-):Bimod_{O}\longrightarrow T_{r}Bimod_{O}.\vspace{2pt}
$$

In practice, a bimodule over $O$ is determined by a family of topological spaces $M(0),$ $M(1),\ldots$ together with actions of the symmetric groups and operations of the form 
\begin{equation*}
\begin{array}{llr}\vspace{7pt}
\circ^{i}: & M(n)\times O(m)\longrightarrow  M(m+n-1), & \text{right operations},\\ \vspace{7pt}
\gamma_{l}: & O(n)\times M(m_{1})\times \cdots \times M(m_{n}) \longrightarrow  M(m_{1}+\cdots + m_{n}),& \text{left operations}.\vspace{-5pt}
\end{array}
\end{equation*}
\end{defi}

\begin{expl}\textbf{The $m$-overlapping little rectangles bimodule $\mathcal{R}_{d}^{(m)}$}

\noindent The $d$-dimensional $m$-overlapping little rectangles bimodule has been introduced by Dobrinskaya and Turchin in \cite{Turchin14}. Its $n$-th component is the subspace of $\mathcal{R}^{\infty}_{d}(n)$ formed by configurations $<r_{1},\ldots,r_{n}> $ satisfying the relation \vspace{2pt}
$$
\forall S\subset \{1,\ldots, n\},\,\, |S|=m,\hspace{15pt}  \underset{s\in S}{\textstyle\bigcap}\,\,r_{s}(]0\,,\,1[)=\emptyset.
$$
In particular, $\mathcal{R}_{d}^{(2)}$ coincides with the little rectangles operad. For $m>2$, the $\Sigma$-sequence $\mathcal{R}_{d}^{(m)}=\{\mathcal{R}_{d}^{(m)}(n)\}$ is not an operad since the operadic composition introduced in Definition \ref{E0} doesn't necessarily preserve the number of intersections between the little rectangles. Nevertheless, the $m$-overlapping little rectangles inherits a bimodule structure over $\mathcal{C}_{d}$ (or over $\mathcal{R}_{d}$) from the operadic structure of $\mathcal{R}_{d}^{\infty}$.\vspace{-5pt}

\begin{figure}[!h]
\begin{center}
\includegraphics[scale=0.32]{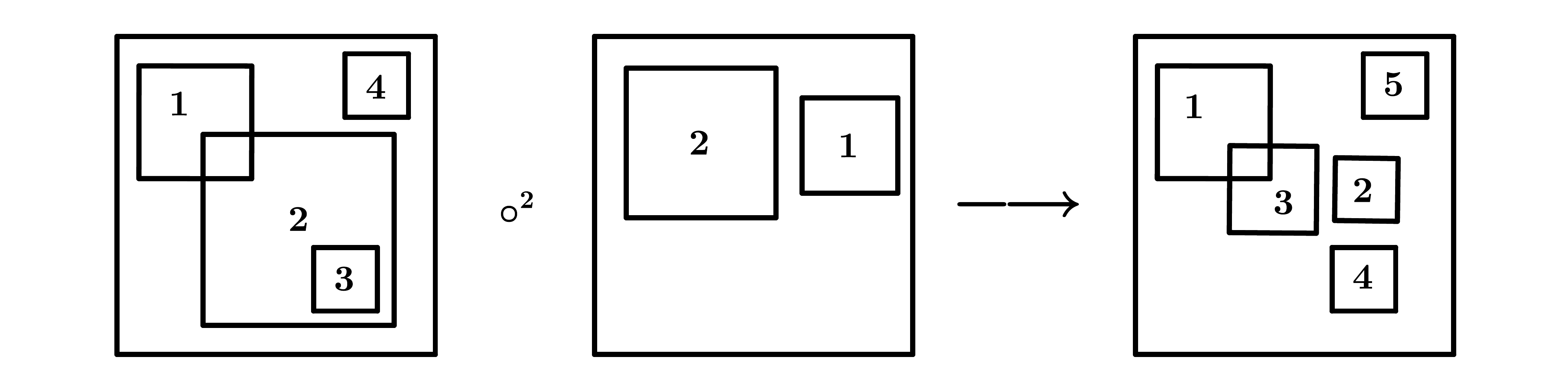}\vspace{-5pt}
\caption{Illustration of the operation $\circ^{2}:\mathcal{R}_{2}^{(3)}(4)\times\mathcal{C}_{2}(2)\rightarrow\mathcal{R}_{2}^{(3)}(5)$.}\vspace{-10pt}
\end{center}
\end{figure}
\end{expl}

\begin{defi}\label{M3}\textbf{Infinitesimal bimodule over an operad}

\noindent An infinitesimal bimodule over $O$, or \textit{$O$-Ibimodule}, is a  sequence $N\in Seq$ endowed with operations 
\begin{equation}\label{C3}
\begin{array}{llr}\vspace{7pt}
\circ_{a}:O(A)\times N(B) 
\rightarrow N(A\cup_{a}B) & \text{for } a\in A,  & \text{infinitesimal left operations,}\\ 
\circ^{b}:N(B)\times O(A)
\rightarrow N(B\cup_{b}A) & \text{for } b\in B, & \text{infinitesimal right operations,}
\end{array} 
\end{equation}
satisfying unit, associativity, commutativity and compatibility with the symmetric group axioms \cite{Arone14}. A map between infinitesimal bimodules should respect the operations. We denote by $Ibimod_{O}$ the category of infinitesimal bimodules over $O$. \vspace{7pt}

Given an integer $r\geq 1$, we also consider the category of $r$-truncated bimodules $T_{r}Ibimod_{O}$. An object is an $r$-truncated sequence endowed with infinitesimal left and right operations (\ref{C3}) under the conditions $|B\cup_{b}A|\leq r$ and $|B|\leq r$. One has an obvious functor 
$$
T_{r}(-):Ibimod_{O}\longrightarrow T_{r}Ibimod_{O}.
$$

In practice, an infinitesimal bimodule over $O$ is determined by a family of topological spaces $N(0),$ $N(1),\ldots$ together with  actions of the symmetric groups and operations of the form 
\begin{equation*}
\begin{array}{llr}\vspace{7pt}
\circ_{i}:O(n)\times N(m) 
\rightarrow N(n+m-1) & \text{for } 1\leq i\leq n,  & \text{infinitesimal left operations,}\\ 
\circ^{i}:N(m)\times O(n)
\rightarrow N(n+m-1) & \text{for } 1\leq i\leq n, & \text{infinitesimal right operations,}
\end{array} 
\end{equation*}
\end{defi}

\begin{expl}
Any operad $O$ is obviously an infinitesimal bimodule over itself. Unfortunately, a bimodule over $O$ is not necessarily an infinitesimal bimodule. Nevertheless, if there exists a map of $O$-bimodules $\eta:O\rightarrow M$, then $M$ inherits an infinitesimal bimodule structure over $O$. Since the right operations and the right infinitesimal operations are the same, we just need to define the left infinitesimal operations
$$
\begin{array}{rcl}\vspace{4pt}
\circ_{i}:O(n)\times M(m) & \longrightarrow & M(n+m-1); \\ 
(\,x\,;\,y) & \longmapsto & \gamma_{l}(\,x\,;\,\underset{i-1}{\underbrace{\eta(\ast_{1}),\ldots,\eta(\ast_{1})}},y,\underset{n-i}{\underbrace{\eta(\ast_{1}),\ldots,\eta(\ast_{1})}}\,).
\end{array} 
$$ 
\end{expl}

\section{The model category of $k$-fold infinitesimal bimodules}\label{B5}

As shown in \cite{Arone14}, the classical notion of infinitesimal bimodule is convenient in order to understand spaces of embeddings since its $r$-th polynomial approximation (see Definition \ref{J5}) is related to derived mapping spaces of $r$-truncated infinitesimal bimodules over the little cubes operad. In the present work we deal with a multivariable version of the polynomial approximation and, for this reason, we need a generalization of the classical notion of infinitesimal bimodule. This section is devoted to introduce the category of $k$-fold infinitesimal bimodules over a family of reduced operads. In particular, we build model category structures as well as functorial cofibrant replacements. 

\subsection{The category of "truncated" $k$-fold infinitesimal bimodules}\label{B1}

Let $\Sigma^{\times k}$ be the category of families of $k$ finite sets $(A_{1},\ldots, A_{k})$ and families of isomorplisms between them. A $k$-fold sequence is a covariant functor $N$ from $\Sigma^{\times k}$ to the category of topological spaces. We will write $N(n_{1},\ldots,n_{k})$ for the space $N(\{1,\ldots,n_{1}\},\ldots,\{1,\ldots,n_{k}\})$. In practice, a $k$-fold sequence is given by a family of spaces $N(n_{1},\ldots,n_{k})$, with $(n_{1},\ldots,n_{k})\in \mathbb{N}^{k}$, together with actions of the symmetric groups: for each element $\sigma\in \Sigma_{n_{1}}\times\cdots\times \Sigma_{n_{k}}$, there is a map 
$$
\begin{array}{rcl}
\sigma^{\ast}:N(n_{1},\ldots,n_{k}) & \longrightarrow & N(n_{1},\ldots,n_{k}); \\ 
 x & \longmapsto & x\cdot \sigma,
\end{array} 
$$
satisfying some relations. We denote by $Seq_{k}$ the category of $k$-fold sequences. \vspace{7pt}

Given an element $\vec{r}=(r_{1},\ldots,r_{k})\in \mathbb{N}^{k}$, we also consider the sub-category $T_{\vec{r}}\,\Sigma^{\times k}$ whose objects are families of finite sets $(A_{1},\ldots,A_{k})$ with $|A_{i}|\leq r_{i}$. An $\vec{r}$-truncated $k$-fold sequence is a covariant functor from $T_{\vec{r}}\,\Sigma^{\times k}$ to spaces.  We denote by  $T_{\vec{r}}\,Seq_{k}$ the category of $\vec{r}$-truncated $k$-fold sequences. Furthermore, there exists an obvious functor
$$
T_{\vec{r}}\,(-):Seq_{k}\longrightarrow T_{\vec{r}}\,Seq_{k}.\vspace{7pt}
$$

\begin{defi}\textbf{The $k$-fold sequence $\vec{O}$}

\noindent Let $O_{1},\cdots,O_{k}$ be a family of reduced operads and let $X$ be topological monoid equipped with a family of maps of topological monoids $f_{i}:O_{i}(1)\rightarrow X$. We say that $O_{1},\cdots,O_{k}$ is a family of reduced operads relative to $X$. We introduce the $k$-fold sequence\vspace{3pt} 
\begin{equation}\label{J7}
\vec{O}(A_{1},\ldots,A_{k})=O_{1}(A_{1}^{\ast})\underset{X}{\times} \cdots \underset{X}{\times} O_{k}(A_{k}^{\ast}),\vspace{3pt}
\end{equation}
with $A_{i}^{\ast}=A_{i}\sqcup\{\ast\}$ the pointed set marked by $\ast$. The product  over $X$ is obtained using the composite map $O_{i}(A_{i}^{\ast})\rightarrow O_{i}(\ast)=O_{i}(1)\rightarrow X$ where the first map consists in composing all the inputs other than the marked one with the unique point in $O_{i}(0)$. The $k$-fold sequence (\ref{J7}) inherits an algebraic structure in the sense that one has an associative operation\vspace{3pt}
\begin{equation}\label{B6}
\begin{array}{cccc}\vspace{4pt}
\tilde{\mu}: & \vec{O}(A_{1},\ldots,A_{k})\times \vec{O}(B_{1},\ldots,B_{k}) & \longrightarrow & \vec{O}(A_{1}\sqcup B_{1},\ldots,A_{k}\sqcup B_{k}); \\ 
 & (x_{1},\ldots,x_{k})\,;\,(y_{1},\ldots,y_{k})  & \longmapsto & (x_{1}\circ_{\ast}y_{1},\ldots,x_{k}\circ_{\ast}y_{k}).
\end{array} \vspace{3pt}
\end{equation}
We will write $\vec{O}(n_{1},\ldots,n_{k})$ for the space $O_{1}(\{1,\ldots,n_{1}\})\times_{X}\cdots\times_{X}O_{k}(\{1,\ldots,n_{k}\})$, with $n_{i}\geq 1$, where the set $\{1,\ldots,n_{i}\}$ is assumed to be pointed by $1$. In that case, the associative operation comes from the operadic composition $\circ_{1}$ on each variable. 
\end{defi}

\begin{defi}\textbf{The category of $k$-fold infinitesimal bimodules over $\vec{O}$}

\noindent A $k$-fold infinitesimal bimodule over $\vec{O}$, or just $\vec{O}$-\textit{Ibimodule}, is a $k$-fold sequence $N$ together with operations called $k$-fold infinitesimal right operations and $k$-fold infinitesimal left operations, respectively,\vspace{2pt}
\begin{equation}\label{A3}
\begin{array}{ll}\vspace{7pt}
\circ_{i}^{a}:N(A_{1},\ldots,A_{k})\times O_{i}(B) 
\longrightarrow N(A_{1},\ldots,A_{i}\cup_{a}B,\ldots, A_{k}), & \text{for } 1\leq i\leq k \text{ and } a\in A_{i},\\ 
\mu:\vec{O}(B_{1},\ldots,B_{k})\times N(A_{1},\ldots,A_{k})
\longrightarrow N(A_{1}\sqcup B_{1},\ldots, A_{k}\sqcup B_{k}), & 
\end{array} \vspace{2pt}
\end{equation}
satisfying compatibility relations with the symmetric group,  associativity and unit axioms (see Appendix \ref{W0}). A map between $k$-fold infinitesimal bimodules should respect the operations. The category of $k$-fold infinitesimal bimodules is denoted by $Ibimod_{\vec{O}}$.\vspace{5pt}

Given an element $\vec{r}=(r_{1},\ldots,r_{k})\in \mathbb{N}^{k}$, we also consider the category of $\vec{r}$-truncated $k$-fold infinitesimal bimodules over $\vec{O}$, denoted by $T_{\vec{r}}\,Ibimod_{\vec{O}}$. An object is an $\vec{r}$-truncated $k$-fold sequence together with operations of the form (\ref{A3}) with $|A_{i}|\leq r_{i}$ and $|A_{i}\sqcup B_{i}|\leq r_{i}$. One has an obvious functor
$$
T_{\vec{r}}\,(-):Ibimod_{\vec{O}}\longrightarrow T_{\vec{r}}\,Ibimod_{\vec{O}}
$$

In practice, a $k$-fold infinitesimal bimodule over $\vec{O}$ is determined by a family of topological spaces $N(n_{1},\ldots,n_{k})$, with $(n_{1},\ldots,n_{k})\in \mathbb{N}^{k}$, together with an action of the symmetric group $\Sigma_{n_{1}}\times \cdots \times \Sigma_{n_{k}}$ and operations of the form
$$
\begin{array}{ll}\vspace{7pt}
\circ_{i}^{j}:N(n_{1},\ldots,n_{k})\times O_{i}(m)\longrightarrow N(n_{1},\ldots, n_{i}+m-1,\ldots n_{k}), & \text{with } i\leq k \text{ and } j\leq n_{i}, \\ 
\mu:\vec{O}(n_{1},\ldots,n_{k})\times N(m_{1},\ldots,m_{k})\longrightarrow N(n_{1}+m_{1}-1,\ldots,n_{k}+m_{k}-1). & 
\end{array}\vspace{2pt} 
$$
For the rest of the paper, we use also the following notation:
$$
\begin{array}{ll}\vspace{7pt}
x\circ_{i}^{j}y=\circ_{i}^{j}(x\,;\,y) & \text{for } x\in N(n_{1},\ldots,n_{k}) \,\text{ and } y\in O_{i}(m).
\end{array} 
$$

Sometimes, we use the terminology of family of operads $O_{1},\ldots,O_{k}$ relative to another operad $O$ in the context of $k$-fold infinitesimal bimodules. In that case, it means that the family operads $O_{1},\ldots,O_{k}$  is relative to the topological monoid $O(1)$ arising from the arity $1$ of the operad. Indeed, any topological monoid can be seen as an operad concentrated in arity $1$ and, conversely, the arity $1$ of any operad produces a topological monoid.\vspace{10pt}
\end{defi}

\begin{pro}\label{B3}
The category of $1$-fold infinitesimal bimodules over a reduced operad $O$ is equivalent to the usual category of infinitesimal bimodules over $O$.
\end{pro}

\begin{proof}
The aim of this proposition is to show that the notion of $k$-fold infinitesimal bimodule is already a generalization of the usual notion of infinitesimal bimodule. Let $N$ be an infinitesimal bimodule over $O$. So, $N$ is obviously a right module over $O$ and the $1$-fold infinitesimal left operation is given by:\vspace{4pt}
$$
\begin{array}{cccc}
\mu: & O(n)\times N(m) & \longrightarrow & N(n+m-1); \\ 
 & (x\,;\,y) & \longmapsto & x\circ_{1}y.
\end{array} \vspace{4pt}
$$
Conversely, if $N$ is a $1$-fold infinitesimal bimodule, then $N$ is still a right module over $O$. The left infinitesimal operations are obtained using the $1$-fold left infinitesimal bimodule operation over $\vec{O}=O$ together with the action of the symmetric group (seen as an operad):\vspace{4pt}
$$
\begin{array}{cccl}
\circ_{i}: & O(n)\times N(m) & \longrightarrow & N(n+m-1); \\ 
 & (x\,;\,y) & \longmapsto & \mu\big(x\cdot (1\,;\,i)\,;\,y\big)\cdot \big((1\,;\,i)\circ_{i}id_{\Sigma_{m}}\big)^{-1},
\end{array} \vspace{4pt}
$$
where $(1\,;\,i)\in \Sigma_{n}$ is the permutation between $1$ and $i$. In particular, the left operation $\circ_{1}$ corresponds to $\mu$. This $1:1$ correspondence is well defined and implies an equivalence of categories.   
\end{proof}

\begin{rmk}\label{N6}\textbf{A $k$-fold Ibimodule as a $1$-fold Ibimodule over a colored operad}

\noindent A colored operad with set of colors $S=\{c_{1},\ldots,c_{k}\}$ is given by a family of spaces $O:=\{O(A_{c_{1}},\ldots, A_{c_{k}}\,;\,c)\}$, with $A_{c_{i}}$ a finite set associated to the color $c_{i}$ and $c\in S$, together with operadic compositions $\circ_{a_{i}}$, with $a_{i}\in A_{c_{i}}$, of the form\vspace{2pt}
$$
\circ_{a_{i}}:O(A_{c_{1}},\ldots, A_{c_{k}}\,;\,c)\times O(B_{c_{1}},\ldots, B_{c_{k}}\,;\,c_{i})\longrightarrow O(A_{c_{1}}\sqcup B_{c_{1}},\ldots, A_{c_{i}}\cup_{a_{i}} B_{c_{i}},\ldots, A_{c_{k}}\sqcup B_{c_{k}} \,;\,c).\vspace{3pt}
$$

Besides, we can also consider the category of $1$-fold infinitesimal bimodules over a colored operad $O$, also denoted by $Ibimod_{O}$, whose objects are family of topological spaces $N:=\{N(A_{c_{1}},\ldots, A_{c_{k}}\,;\,c)\}$ together with operations of the form \vspace{2pt}
$$
\circ_{a_{i}}:O(A_{c_{1}},\ldots, A_{c_{k}}\,;\,c)\times N(B_{c_{1}},\ldots, B_{c_{k}}\,;\,c_{i})\longrightarrow N(A_{c_{1}}\sqcup B_{c_{1}},\ldots, A_{c_{i}}\cup_{a_{i}} B_{c_{i}},\ldots, A_{c_{k}}\sqcup B_{c_{k}} \,;\,c),\vspace{5pt}
$$
$$
\circ^{a_{i}}:N(A_{c_{1}},\ldots, A_{c_{k}}\,;\,c)\times O(B_{c_{1}},\ldots, B_{c_{k}}\,;\,c_{i})\longrightarrow N(A_{c_{1}}\sqcup B_{c_{1}},\ldots, A_{c_{i}}\cup_{a_{i}} B_{c_{i}},\ldots, A_{c_{k}}\sqcup B_{c_{k}} \,;\,c),\vspace{5pt}
$$
again satisfying some relations similar to the usual notion of infinitesimal bimodule over an operad (we refer the reader to \cite{Ducoulombier16} for more details). \vspace{7pt}

Then, from a family of topological operads $O_{1},\ldots,O_{k}$ relative to topological monoid $X$, we build the following colored operad with set of colors $S=\{c_{1},\ldots , c_{k+1}\}$:\vspace{5pt}
\begin{equation}\label{N7}
\tilde{O}(A_{c_{1}},\ldots, A_{c_{k+1}}\,;\,c):=\left\{
\begin{array}{cl}\vspace{3pt}
O_{i}(A_{c_{i}}) & \text{if } c=c_{i} \text{ for } i\leq k \text{ and } A_{c_{j}}=\emptyset \text{ for } j\neq i, \\ \vspace{3pt}
O_{1}(A_{c_{1}}^{\ast})\times_{X}\cdots \times _{X} O_{k}(A_{c_{k}}^{\ast}) & \text{if } c=c_{k+1} \text{ and } A_{c_{k+1}}=\ast, \\ 
\emptyset & \text{otherwise}.
\end{array} 
\right.\vspace{2pt}
\end{equation}
The colored operadic structure on $\tilde{O}$ comes from the operadic structures of  $O_{1},\ldots , O_{k}$ as well as the operation (\ref{B6}). Then, we consider a particular $1$-fold infinitesimal bimodule $N_{\ast}$ over the colored operad (\ref{N7}) defined as follows:\vspace{7pt}
$$
N_{\ast}(A_{c_{1}},\ldots, A_{c_{k+1}}\,;\,c):= \left\{
\begin{array}{cl}\vspace{5pt}
\ast & \text{if } A_{c_{k+1}}=\emptyset \text{ and } c=c_{k+1},\\ 
\emptyset & \text{otherwise}.
\end{array} 
\right.\vspace{7pt}
$$
Finally, one has the following equivalence of categories:
$$
Ibimod_{\vec{O}} \cong Ibimod_{\tilde{O}}\downarrow N_{\ast}.
$$
\end{rmk}\vspace{5pt}

\begin{expl}\label{E1}\textbf{The $k$-fold infinitesimal bimodule $\mathbb{O}$}

\noindent Let $O_{1},\ldots,O_{k}$ be a family of reduced operads relative to a topological monoid $X$. Then, we consider the $k$-fold sequence given by the formula\vspace{2pt}
\begin{equation}\label{J8}
\mathbb{O}(A_{1},\ldots,A_{k})=O_{1}(A_{1})\times \cdots \times O_{k}(A_{k}), \hspace{15pt}\text{for } (A_{1},\ldots A_{k})\in \Sigma^{\times k}.\vspace{2pt}
\end{equation}
The $k$-fold infinitesimal right operations are obtained using the operadic structures of $O_{1},\ldots,O_{k}$. The $k$-fold infinitesimal left operations are given  by \vspace{2pt}
$$
\begin{array}{ccc}\vspace{5pt}
\mu:\vec{O}(B_{1},\ldots,B_{k})\times \mathbb{O}(A_{1},\ldots,A_{k}) & \longrightarrow & \mathbb{O}(A_{1}\sqcup B_{1},\ldots, A_{k}\sqcup B_{k}); \\ 
(x_{1},\ldots,x_{k})\,;\, (y_{1},\ldots,y_{k}) & \longmapsto & (x\circ_{\ast}y_{1},\ldots,x_{k}\circ_{\ast}y_{k}).
\end{array} \vspace{5pt}
$$

More generally, if $N_{1},\ldots,N_{k}$ are infinitesimal bimodules over the reduced operads $O_{1},\ldots,O_{k}$, respectively, then the $k$-fold sequence 
$$
\mathbb{N}(A_{1},\ldots,A_{k})=N_{1}(A_{1})\times \cdots \times N_{k}(A_{k}),\hspace{15pt}\text{for } (A_{1},\ldots A_{k})\in \Sigma^{\times k},
$$ 
inherits a $k$-fold infinitesimal bimodule structure over the family of reduced operads $O_{1},\ldots,O_{k}$ relative to any topological monoid $X$.
\end{expl}

\begin{expl}\label{D7}\textbf{The $k$-fold infinitesimal bimodule $\mathcal{R}_{n}^{k}$}

\noindent In the following, $\vec{O}$ is associated to the family of reduced operads $\mathcal{C}_{d_{1}},\ldots,\mathcal{C}_{d_{k}}$ relative to the topological monoid $\mathcal{C}_{n}(1)$. Without loss of generality, we assume that $d_{1}\leq \cdots\leq d_{k}<n$. Then, we introduce the $k$-fold sequence 
$$
\mathcal{R}_{n}^{k}(A_{1},\ldots,A_{k})=\mathcal{R}_{n}(A_{1}\sqcup\cdots \sqcup A_{k}), \hspace{15pt} \forall \, (A_{1},\ldots,A_{k})\in \Sigma^{\times k}.\vspace{4pt}
$$
The $k$-fold sequence so obtained inherits a $k$-fold infinitesimal bimodule structure over $\vec{O}$. The $k$-fold infinitesimal right operations are defined using the composite map $\kappa_{i}:\mathcal{C}_{d_{i}}\rightarrow \mathcal{C}_{d_{k}}\hookrightarrow \mathcal{R}_{d_{k}}\rightarrow \mathcal{R}_{n}$ and the operadic structure of $\mathcal{R}_{n}$. In order to define the $k$-fold infinitesimal left operations, we need a map of the form\vspace{3pt}
\begin{equation}\label{D4}
\varepsilon:\mathcal{C}_{d_{1}}(A_{1}^{\ast}) \underset{\mathcal{C}_{n}(1)}{\times} \cdots \underset{\mathcal{C}_{n}(1)}{\times}\mathcal{C}_{d_{k}}(A_{k}^{\ast}) \longrightarrow \mathcal{R}_{n}(A_{1}\sqcup \cdots \sqcup A_{k}\sqcup \{\ast\}).\vspace{4pt}
\end{equation}
For this purpose, we consider the element $c_{k}\in \mathcal{R}_{n}(k)$ which subdivides the unit cubes into $k$ equal little rectangles along the last coordinate. From a point $(x_{1},\ldots, x_{k})$ in the product space over $\mathcal{C}_{n}(1)$, the composition $(\ldots ( c_{k}\circ_{k}\kappa_{k}(x_{k}))\ldots )\circ_{1}\kappa_{1}(x_{1})$ is an element in $\mathcal{R}_{n}(A_{1}^{\ast}\sqcup \cdots \sqcup A_{k}^{\ast})$.  Nevertheless, since the product in $(\ref{D4})$ is over the space $\mathcal{C}_{n}(1)$ and since $d_{i}<n$, the marked rectangles in the composition are aligned and such an element can be identified with a point in $\mathcal{R}_{n}(A_{1}\sqcup \cdots \sqcup A_{k}\sqcup \{\ast\})$ by gluing together all the marked rectangles.\vspace{-5pt}

\begin{figure}[!h]
\begin{center}
\includegraphics[scale=0.315]{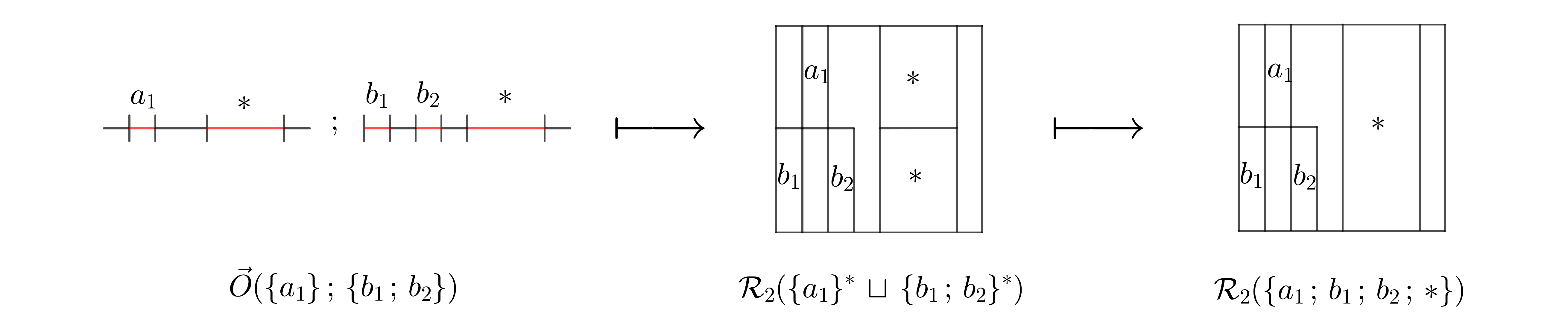}\vspace{-9pt}
\caption{Illustration of the map $\varepsilon$.}\vspace{-11pt}
\end{center}
\end{figure}

\noindent Finally, the $k$-fold  infinitesimal left operation is defined by the formula \vspace{2pt}
$$
\begin{array}{cccc}\vspace{3pt}
\mu: & \vec{O}_{n}(A_{1},\ldots,A_{k})\times \mathcal{R}^{k}_{n}(B_{1},\ldots,B_{k}) & \longrightarrow  & \mathcal{R}^{k}_{n}(A_{1}\sqcup B_{1},\ldots,A_{k}\sqcup B_{k}); \\ 
 & (x_{1},\ldots,x_{k})\,;\,y & \longmapsto & \varepsilon(x_{1},\ldots,x_{k})\circ_{\ast}y.
\end{array} \vspace{-15pt}
$$
\begin{figure}[!h]
\begin{center}
\includegraphics[scale=0.29]{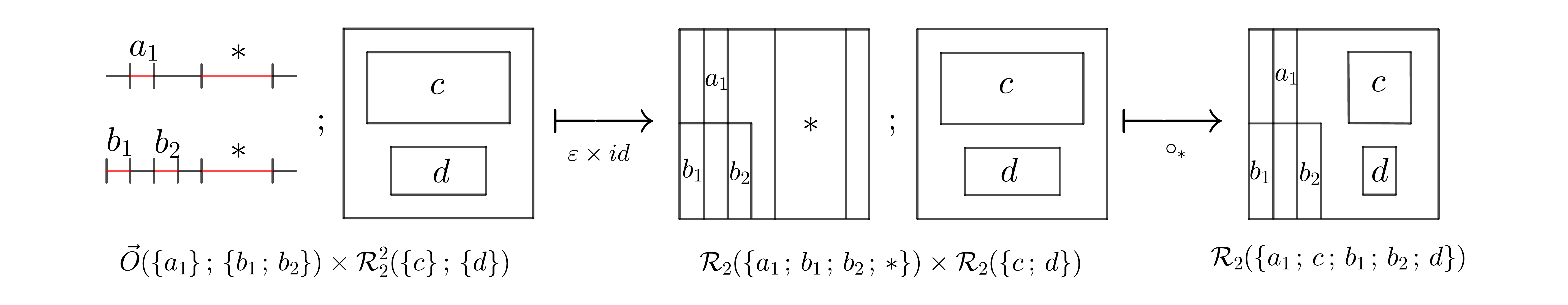}\vspace{-5pt}
\caption{Illustration of the $2$-fold left infinitesimal operation.}
\end{center}\vspace{-25pt}
\end{figure}
\end{expl}

\subsection{The Reedy and projective model category structures on $Ibimod_{\vec{O}}$}\label{C1}

First, we fix a family of reduced operads $O_{1},\ldots,O_{k}$ relative to a topological monoid $X$. The purpose of this section is to define a model category structure on the category of $k$-fold infinitesimal bimodules over $\vec{O}$. More precisely, we introduce two different model category structures: projective and Reedy. Both model category structures have advantages and inconveniences. With the Reedy model structure, we need to deal with fibrant replacements but it makes the main theorem of Section \ref{D5} easier to prove. On the other hand, the projective model category structure is related to the Goodwillie-Weiss manifold calculus theory and all the objects are fibrants. In the following, we compare these two structures and we give properties needed to prove the main results in Section \ref{N4}. 

\newpage

\subsubsection{The projective model category structure} 

Similarly to \cite{Ducoulombier16}, the category of $k$-fold sequences can be endowed with a cofibrantly generated model category structure, called \textit{projective} model structure, in which all the objects are fibrant. More precisely, a map is a weak equivalence (resp. a fibration) if each of its components is a weak homotopy equivalence (resp. a Serre fibration). In order to define a model category structure on the category of $k$-fold infinitesimal bimodules, we need an adjunction\vspace{1pt}
\begin{equation}\label{A7}
\mathcal{F}_{Ib\,;\,\vec{O}}:Seq_{k}\leftrightarrows Ibimod_{\vec{O}}:\mathcal{U},\vspace{2pt}
\end{equation}  
where $\mathcal{U}$ is the forgetful functor. As usual in the operadic theory, the free $k$-fold infinitesimal bimodule functor can be described in terms of coproduct indexed by a set of trees.\vspace{7pt}

\begin{defi}\textbf{The set of $k$-fold reduced pearled trees}\label{B4}

\noindent A \textit{pearled tree} $T=(T\,;\,p)$ is a planar rooted tree $T$ with a particular vertex $p\in V(T)$ in the path joining the first leaf or univalent vertex (according to the planar order) and the root. A pearled tree is said to be \textit{reduced} if the vertices other than the pearl are connected to the pearl by an inner edge. Furthermore, we need the following  notation:
\begin{itemize}
\item[$\blacktriangleright$] $V^{rp}(T)$ is the set of vertices other than the pearl composing the path from the pearl to the root of $T$,
\item[$\blacktriangleright$] $V^{c}(T)$ is the set of vertices other than the pearl and other than the vertices in $V^{rp}(T)$.
\end{itemize}\vspace{-9pt}

\begin{figure}[!h]
\begin{center}
\includegraphics[scale=0.4]{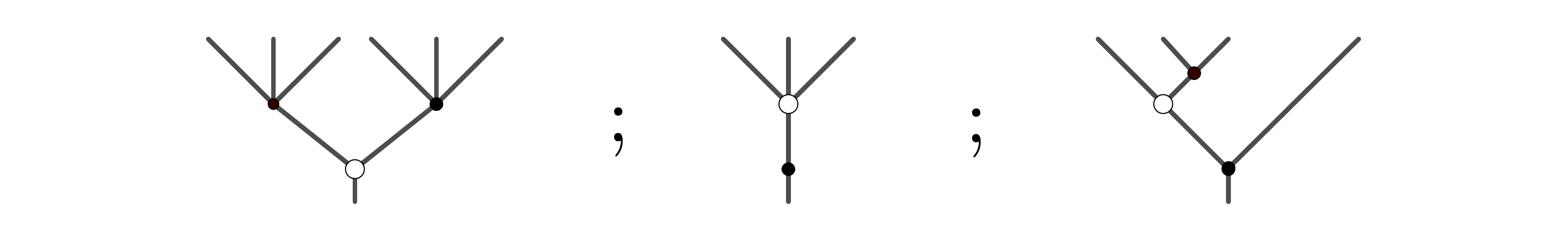}\vspace{-7pt}
\caption{Illustration of reduced pearled trees.}\vspace{-10pt}
\end{center}
\end{figure}

For $\vec{n}=(n_{1},\ldots,n_{k})\in \mathbb{N}^{k}$, we introduce the set $rpTree[\,\vec{n}\,]$ of $k$-fold reduced pearled trees $\vec{T}=(T_{1},\ldots,T_{k},\vec{\sigma})$ where $T_{i}$ is a reduced pearled tree having $n_{i}$ leaves and $\vec{\sigma}\in \Sigma_{n_{1}}\times \cdots \times \Sigma_{n_{k}}$ is a permutation labelling the leaves of the reduced pearled trees. Furthermore, we assume that $|V^{rp}(T_{1})|=|V^{rp}(T_{i})|$ for all $i\leq k$. In particular, if $v^{1}\in V^{rp}(T_{1})$, then we denote by $v^{i}$ the corresponding vertex in $T_{i}$ according to the planar structure. \vspace{-9pt}

\begin{figure}[!h]
\begin{center}
\includegraphics[scale=0.38]{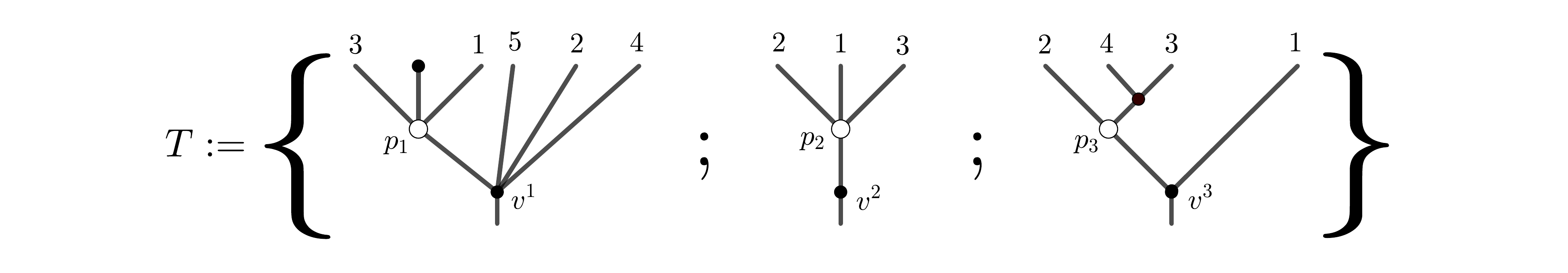}\vspace{-11pt}
\caption{Illustration of an element in $rpTree[5\,;\,3\,;\,4]$.}\vspace{-12pt}
\end{center}
\end{figure}
\end{defi}

\begin{const}\label{A8}
Let $N=\{N(n_{1},\ldots,n_{k})\}$ be a $k$-fold sequence. The free $k$-fold infinitesimal bimodule $\mathcal{F}_{Ib\,;\,\vec{O}}(N)$, also denoted by $\mathcal{F}_{Ib}(N)$ when $\vec{O}$ is understood, consists in labelling the vertices of $k$-fold reduced pearled trees by elements in $N$ and elements in the operads $O_{1},\ldots,O_{k}$. More precisely, one has\vspace{3pt}
\begin{equation}\label{F6}
\mathcal{F}_{Ib}(N)(\,\vec{n}\,)= \left.\left(\underset{\vec{T} \in\, rpTree[\,\vec{n}\,]}{\coprod} \!\!\!\!\!\! N(|p_{1}|,\ldots, |p_{k}|)\times \underset{v^{1}\in V^{rp}(T_{1})}{\prod} \vec{O}(|v^{1}|,\ldots,|v^{k}|) \,\times \underset{\substack{i\in\{1,\ldots,k\}\\v\in V^{c}(T_{i})}}{\prod}\!\!\!\!\!\! O_{i}(\,|v|\,)\right)\,\, \right/ \!\!\sim \vspace{3pt}
\end{equation}
where the equivalence relation is generated by the usual unit axiom and the compatibility relations with the group action $ \Sigma_{n_{1}}\times \cdots \times \Sigma_{n_{k}}$ preserving the position of pearls. We denote by $[T\,;\,m\,;\,\{x_{v}\}]$ a point in $\mathcal{F}_{Ib}(N)$. The reader can easily check that the construction so obtained for $1$-fold infinitesimal bimodules is homeomorphic to the construction of the free infinitesimal bimodule introduced in \cite{Ducoulombier16}.

\newpage

\begin{figure}[!h]
\hspace{-25pt}\includegraphics[scale=0.58]{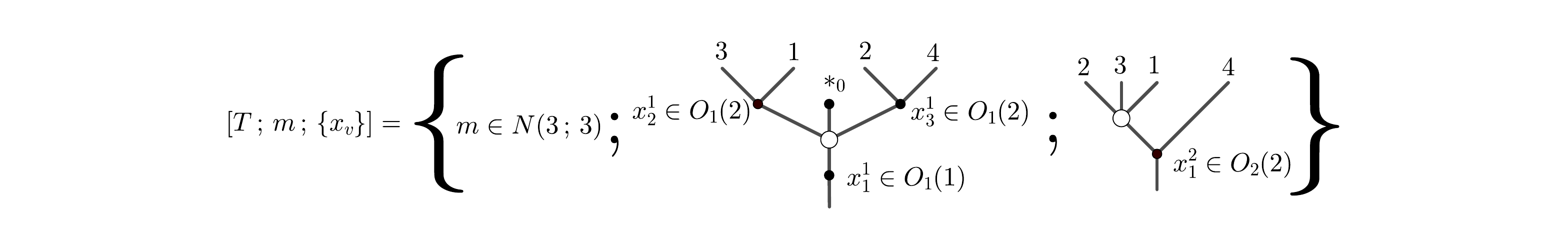}\vspace{-12pt}
\caption{Illustration of a point in the space $\mathcal{F}_{Ib}(N)(4\,;\,4)$ with $f_{1}(x_{^{1}})=f_{2}(x_{1}^{2}\circ_{2}\ast_{0})$.}\label{A6}\vspace{-1pt}
\end{figure} 

 The $k$-fold infinitesimal right operation $\circ_{i}^{j}$ with an element $x\in O_{i}(n)$ consists in grafting a corolla labelled by $x$ into the $j$-th leaf of the tree $T_{i}$. If the element so obtained contains an inner edge connecting two vertices other than a pearl, then we contract it using the operadic structure of $O_{i}$. For instance, the composition $\circ_{1}^{2}$ of the point represent in Figure \ref{A6} with an element $x\in O_{1}(2)$ gives rise to \vspace{3pt}

\hspace{-60pt}\includegraphics[scale=0.58]{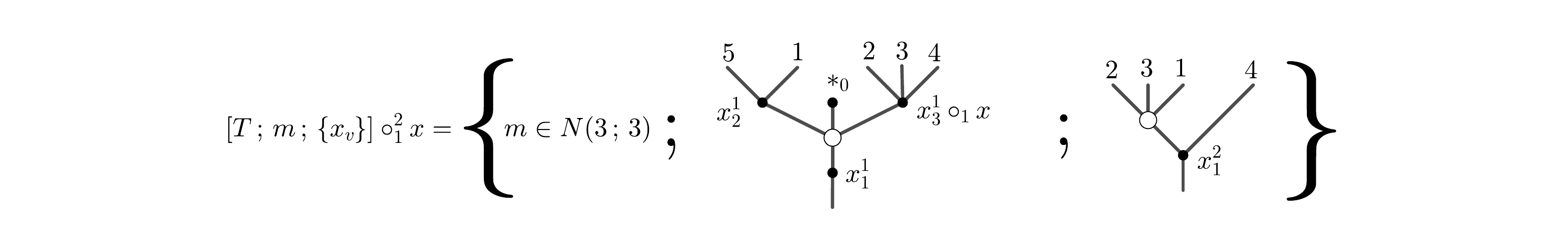}\vspace{2pt}

 Similarly, from an element $(x_{1},\ldots,x_{k})\in \vec{O}$, the $k$-fold infinitesimal left operation consists in grafting each reduced pearled tree $T_{i}$ into the first leaf of the corolla labelled by $x_{i}$. If the element so obtained contains an inner edge connecting two vertices other than a pearl, then we contract it using the operadic structures of $O_{1},\ldots,O_{k}$ as illustrated in the picture below. For instance, the composition of a point $(x_{1}\,,\,x_{2})\in \vec{O}(2\,,\,2)$ with the element represented in Figure \ref{A6} gives rise to \vspace{1pt}

\hspace{-85pt}\includegraphics[scale=0.58]{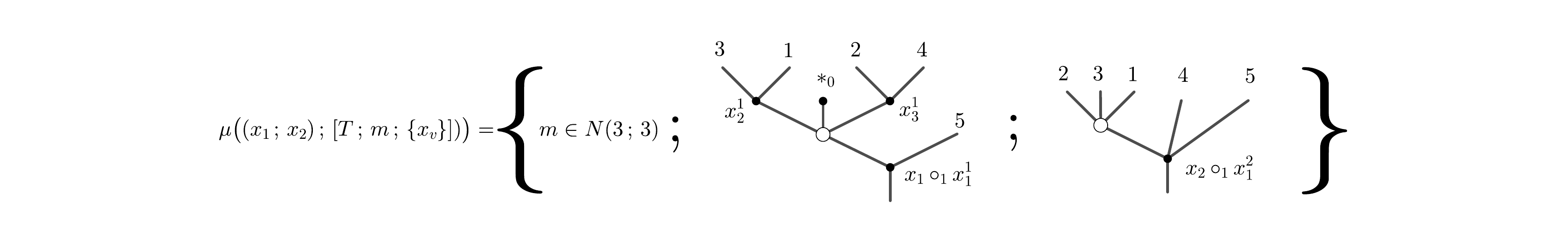}\vspace{7pt}

\end{const}

\begin{thm}\label{A9}
The pair of functors $(\mathcal{F}_{Ib}\,;\,\mathcal{U})$ forms an adjunction. Furthermore, the category of $k$-fold infinitesimal bimodules inherits a cofibrantly generated model category structure in which all the objects are fibrant and making the adjunction (\ref{A7}) into a Quillen adjunction. More precisely, a $k$-fold infinitesimal bimodule map $f$ is a weak equivalence (resp. a fibration) if the induced map $\mathcal{U}(f)$ is a weak equivalence (resp. a fibration) in the category of $k$-fold sequences. This model category structure is called projective model category structure. 
\end{thm}

\begin{proof}
There are many ways to prove the theorem. One of them consists in using Remark \ref{N6} and the fact that the category of infinitesimal bimodules over a colored operad is already equipped with a projective model category structure (see \cite{Ducoulombier19,Ducoulombier18}). The structure so obtained coincides with the structure described in the theorem since the free $k$-fold infinitesimal bimodule functor is an alternative description of the free infinitesimal bimodule functor over the corresponding colored operad. \vspace{7pt}

Another way to prove the theorem is the transfer Theorem \cite[Section 2.5]{Berger03} applied to the adjunction (\ref{A7}). One has to check that there exists a functorial fibrant replacement and a functorial factorization of the diagonal map in the category $Ibimod_{\vec{O}}$. Since all the objects are fibrant in the category of $k$-fold sequences, the identity functor provides a functorial fibrant replacement. So, for any $N\in Ibimod_{\vec{O}}$, we need to prove the existence of an object $Path(N)\in Ibimod_{\vec{O}}$ such that there is a factorization of the diagonal map\vspace{5pt}
$$
\xymatrix{
\Delta: N \ar[r]^{\hspace{-5pt}\simeq}_{\hspace{-5pt}f_{1}} & Path(N) \ar@{->>}[r]_{f_{2}} & N\times N,
}
$$
where $f_{1}$ is a weak equivalence and $f_{2}$ is a fibration. Let us consider\vspace{3pt}
$$
Path(N)(n_{1},\ldots,n_{k})=Map\big( [0\,,\,1]\,;\,N(n_{1},\ldots,n_{k})\big).\vspace{3pt}
$$
This $k$-fold sequence inherits a $k$-fold infinitesimal bimodule structure from $N$. The map from $N$ to $Path(N)$, sending a point to the constant path, is clearly a homotopy equivalence. Furthermore, the map \vspace{3pt}
$$
f_{2}:Map\big( [0\,,\,1]\,;\,N(n_{1},\ldots,n_{k})\big)\longrightarrow Map\big( \partial [0\,,\,1]\,;\,N(n_{1},\ldots,n_{k})\big)=(N\times N)(n_{1},\ldots,n_{k})\vspace{3pt}
$$
induced by the inclusion $i:\partial [0\,,\,1]\rightarrow [0\,,\,1]$ is a fibration since the map $i$ is a cofibration.
\end{proof}\vspace{3pt}

\begin{rmk}
Construction \ref{A8} and Theorem \ref{A9} admit an analogue version for $\vec{r}$-truncated $k$-fold infinitesimal bimodules. In that case, we need to consider the set $rpTree[\vec{n}\leq \vec{r}]$ instead of $rpTree[\vec{n}]$ in Construction \ref{A8} where $rpTree[\vec{n}\leq \vec{r}]$ is the set of elements $(T_{1},\ldots,T_{k},\vec{\sigma})\in rpTree[\vec{n}]$ for which the number of leaves plus the number of univalent vertices of $T_{i}$ is smaller than $r_{i}$.
\end{rmk}

\subsubsection{The Reedy model category structure}\label{G0}

We refer the reader to \cite{Fresse17} for this section. Let $\Lambda^{\times k}$ be the category whose objects are families of $k$ finite sets $(A_{1},\ldots, A_{k})$ and morphisms are families of injective maps between them. In particular, $\Sigma^{\times k}$ is the sub-category of isomorphisms of $\Lambda^{\times k}$. By a $k$-fold $\Lambda$-sequence, we understand a contravariant functor from $\Lambda^{\times k}$ to spaces and we denote the corresponding category by $\Lambda Seq_{k}$. In practice, such an object is given by a $k$-fold sequence $N$ together with maps generated by applications of the form \vspace{3pt}
$$
\lambda^{\ast}_{i\,;\,j}:N(n_{1},\ldots, n_{k})\longrightarrow N(n_{1},\ldots, n_{j}-1, \ldots, n_{k}), \vspace{1pt}
$$
associated to the map \vspace{3pt}
$$
\begin{array}{rcl}\vspace{5pt}
\lambda_{i\,;\,j}:(\{1,\ldots,n_{1}\},\ldots,\{1,\ldots,n_{j}-1\},\ldots , \{1,\ldots,n_{k}\}) & \longrightarrow &  (\{1,\ldots,n_{1}\},\ldots,\{1,\ldots,n_{j}\},\ldots , \{1,\ldots,n_{k}\}); \\ 
 (l_{1},\ldots,l_{k}) & \longmapsto & 
 \left\{
 \begin{array}{ll}\vspace{5pt}
 (l_{1},\ldots,l_{k}) & \text{if } l_{j}<i, \\ 
 (l_{1},\ldots,l_{j}+1,\ldots,l_{l}) & \text{if } l_{j}\geq i.
 \end{array} 
 \right. 
\end{array} \vspace{1pt}
$$

Given an element $\vec{r}=(r_{1},\ldots,r_{k})\in \mathbb{N}^{k}$, we also consider the sub-category $T_{\vec{r}}\,\Lambda^{\times k}$ whose objects are families of finite sets $(A_{1},\ldots,A_{k})$ with $|A_{i}|\leq r_{i}$. An $\vec{r}$-truncated $\Lambda$-sequence is a contravariant functor from $T_{\vec{r}}\,\Lambda^{\times k}$ to spaces and we denote by $T_{\vec{r}}\,\Lambda Seq_{k}$ the corresponding category. There is an obvious functor\vspace{3pt}
$$
T_{\vec{r}}(-):\Lambda Seq_{k}\longrightarrow T_{\vec{r}}\,\Lambda Seq_{k}.\vspace{3pt}
$$ 

The categories $\Lambda Seq_{k}$ and $T_{\vec{r}}\,\Lambda Seq_{k}$ are endowed with the so called Reedy model structure. For a $k$-fold (possibly truncated) $\Lambda$-sequence $X$, we denote by $\mathcal{M}(X)$ the $k$-fold (truncated) $\Sigma$-sequence, called \textit{matching object} of $X$, defined by \vspace{5pt}
$$
\mathcal{M}(X)(\,\vec{l}\,)=\underset{\substack{u\in \Lambda^{\times k}_{<}(\vec{i}\,;\,\vec{l})\\ \vec{i}< \vec{l}}}{lim} X(\,\vec{i}\,)\vspace{3pt}
$$
where $\Lambda^{\times k}_{<}$ is the subcategory of $\Lambda^{\times k}$ consisting of family of order preserving maps. According to \cite[Theorem 8.3.19]{Fresse17}, the categories $\Lambda Seq_{k}$ and $T_{\vec{r}}\,\Lambda Seq_{k}$ are endowed with a cofibrantly generated model category structure for which weak equivalences are objectwise weak homotopy equivalences while fibrations are morphisms $f:X\rightarrow Y$ for which any induced map \vspace{3pt}
$$
X(\,\vec{r}\,)\longrightarrow \mathcal{M}(X)(\,\vec{l}\,)\times_{\mathcal{M}(Y)(\,\vec{l}\,)}Y(\,\vec{l}\,)\vspace{3pt}
$$
is a Serre fibration in every arity where defined. As shown in \cite[Theorem 8.3.20]{Fresse17}, a morphism in $\Lambda Seq_{k}$ and $T_{\vec{r}}\,\Lambda Seq_{k}$ is a cofibration if and only if it is a cofibration as a morphism in $Seq_{k}$ and $T_{\vec{r}}\, Seq_{k}$, respectively.

Let $O_{1},\ldots, O_{k}$ be a family of reduced operads relative to a topological monoid $X$. From now on, we denote by $\Lambda Ibimod_{\vec{O}}$  and $T_{\vec{r}}\, \Lambda Ibimod_{\vec{O}}$ the categories of $k$-fold infinitesimal bimodules and $\vec{r}$-truncated $k$-fold infinitesimal bimodules, respectively, equipped with the Reedy model category structure. This structure is transferred from the categories $\Lambda Seq_{k}$ and $T_{\vec{r}}\,\Lambda Seq_{k}$, respectively, along the adjunctions\vspace{4pt}
\begin{equation}\label{F7}
\begin{array}{rcl}\vspace{7pt}
\mathcal{F}_{Ib\,;\,\vec{O}}^{\Lambda}:\Lambda Seq_{k} & \leftrightarrows & \Lambda Ibimod_{\vec{O}}:\mathcal{U}, \\ 
T_{\vec{r}}\mathcal{F}_{Ib\,;\,\vec{O}}^{\Lambda}:T_{\vec{r}}\Lambda Seq_{k}& \leftrightarrows & T_{\vec{r}}\Lambda Ibimod_{\vec{O}}:\mathcal{U},
\end{array} \vspace{2pt}
\end{equation}
where the free functors, also denoted by $\mathcal{F}_{Ib}^{\Lambda}$ and $T_{\vec{r}}\mathcal{F}_{Ib}^{\Lambda}$ when $\vec{O}$ is understood, are obtained from the functors $\mathcal{F}_{Ib}$ and its truncated version by taking the restriction of the coproduct (\ref{F6}) to the $k$-fold reduced pearled trees without univalent vertices other than the pearls. \vspace{7pt}

In other words, let us consider the category of $k$-fold infinitesimal bimodules associated to the family of $\Lambda$-operads $O_{1}^{>0},\ldots, O_{k}^{>0}$ relative the monoid $X$. The $\Lambda$-operad $O_{i}^{>0}$ is obtained from $O_{i}$ by removing the unique point in arity $0$ (see \cite{Fresse17} for more details). The forgetting maps are still well defined using the $\Lambda$-structure of the operad $O_{i}^{>0}$. This category is denoted by $Ibimod_{\vec{O}_{>0}}$. Then, one has \vspace{3pt}
$$
\mathcal{F}_{Ib\,;\,\vec{O}}^{\Lambda}(N)\coloneqq\mathcal{F}_{Ib\,;\,\vec{O}_{>0}}(N), \hspace{15pt} \text{and}\hspace{15pt} T_{\vec{r}}\,\mathcal{F}_{Ib\,;\,\vec{O}}^{\Lambda}(N)\coloneqq T_{\vec{r}}\,\mathcal{F}_{Ib\,;\,\vec{O}_{>0}}(N). \vspace{3pt}
$$
By construction, these objects are $k$-fold (truncated) infinitesimal bimodules over $\vec{O}_{>0}$. We can extend this structure in order to get $k$-fold (truncated) infinitesimal bimodules over $\vec{O}$ using the operadic structure of $O_{1},\ldots, O_{k}$ or the $\Lambda^{\times k}$ structure of $N$. For instance, the image of the point in  $\mathcal{F}_{Ib}^{\Lambda}(N)(4\,;\,2)$\vspace{-10pt}
\begin{figure}[!h]
\begin{center}
\includegraphics[scale=0.17]{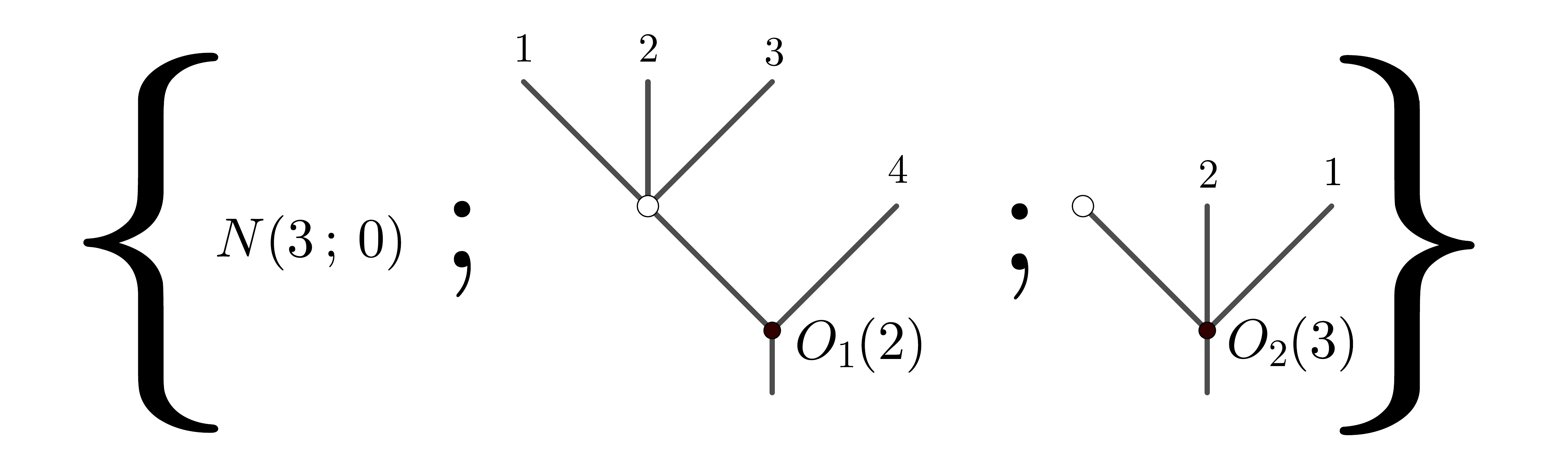}\vspace{-10pt}
\end{center}
\end{figure}

\noindent through the composite map $(-\circ_{2}^{1} \ast)  \circ_{1}^{1} \ast:\big( \mathcal{F}_{Ib}^{\Lambda}(N)(4\,;\,2)\times O_{1}(0)\big)\times O_{2}(0)\rightarrow \mathcal{F}_{Ib}^{\Lambda}(N)(3\,;\,1)$ is \vspace{3pt}

\begin{figure}[!h]
\begin{center}
\includegraphics[scale=0.17]{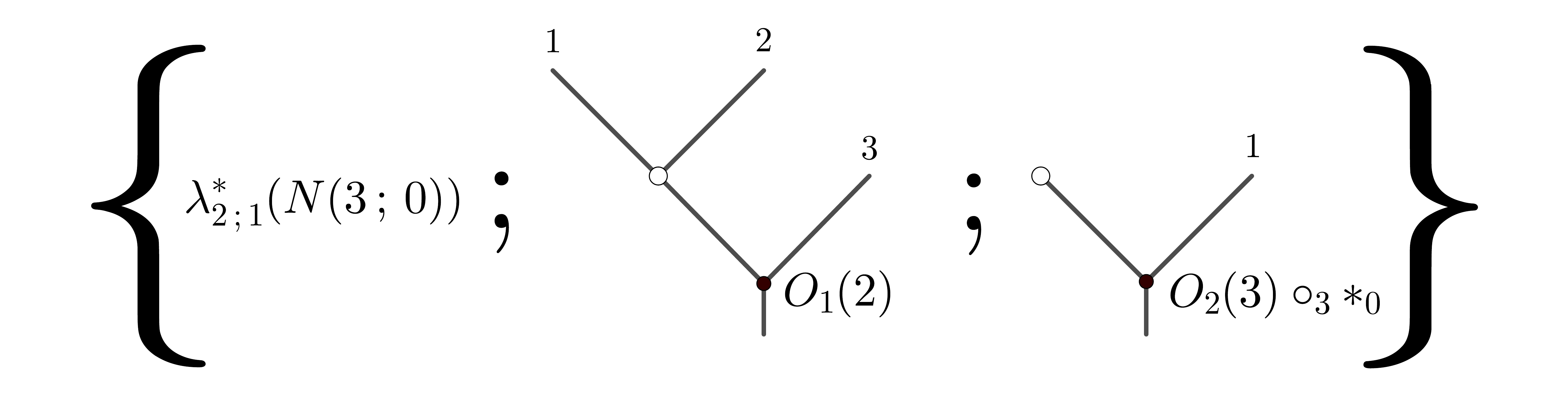}\vspace{-7pt}
\end{center}
\end{figure}

\begin{thm}{\cite{Ducoulombier19}}\label{K8} One has the following properties on the Reedy model category structure: 
\begin{itemize}
\item[$(i)$] The categories $\Lambda Ibimod_{\vec{O}}$ and $T_{\vec{r}}\Lambda Ibimod_{\vec{O}}$, with $\vec{r}\in \mathbb{N}^{k}$, admit a cofibrantly generated model category structure, called Reedy model category structure, transferred from $\Lambda Seq_{k}$ and $T_{\vec{r}}\Lambda Seq_{k}$, respectively, along the adjunctions $(\ref{F7})$.
\item[$(ii)$] A morphism in the category of $k$-fold (possibly truncated) infinitesimal bimodules over $\vec{O}$ is a cofibration for the Reedy model category structure if and only if it is a cofibration as a morphism of $k$-fold (truncated) infinitesimal bimodules over $\vec{O}_{>0}$ equipped with the projective model structure.     
\end{itemize}
\end{thm}\vspace{7pt}

\begin{sproof}
We already know that this properties are trues in the context of infinitesimal bimodules over a colored operad \cite{Ducoulombier19}. So, the theorem is a consequence of the description of $k$-fold infinitesimal bimodules in terms of $1$-fold infinitesimal bimodules over a colored operad introduced in Remark \ref{N6}.
\end{sproof}

\subsubsection{Connections between the two model category structures and properties}\label{K6}

In the previous sections, we introduce two model category structures, projective and Reedy, on the category of $k$-fold (possibly truncated) infinitesimal bimodules over $\vec{O}$. In what follows, we show that these two structures are more or less the same homotopically speaking and induce the same derived mapping space up to a homeomorphism (see the identification (\ref{G7})). For this reason,  we won't distinguish between the two mapping spaces and we will simply write $Ibimod_{\vec{O}}^{h}(-\,;\,-)$ and $T_{\vec{r}}\,Ibimod_{\vec{O}}^{h}(-\,;\,-)$.\vspace{5pt}

\begin{thm}{\cite{Ducoulombier19}}\label{H6} One has the following relations between the Reedy and projective model structures:
\begin{itemize}
\item[$(i)$] One has Quillen equivalences:
$$
\begin{array}{rcl}\vspace{5pt}
id: Ibimod_{\vec{O}} & \leftrightarrows & \Lambda Ibimod_{\vec{O}}:id, \\ 
id: T_{\vec{r}}\,Ibimod_{\vec{O}} & \leftrightarrows & T_{\vec{r}}\,\Lambda Ibimod_{\vec{O}}:id.
\end{array} 
$$
\item[$(ii)$] For any pair $M$ and $N$ of $k$-fold (possibly truncated) infinitesimal bimodules over $\vec{O}$, one has equivalence of mapping spaces
\begin{equation}\label{G7}
\begin{array}{rcl}\vspace{5pt}
Ibimod_{\vec{O}}^{h}(M\,;\,N) & \cong & \Lambda Ibimod_{\vec{O}}^{h}(M\,;\,N), \\ 
T_{\vec{r}}\,Ibimod_{\vec{O}}^{h}(M\,;\,N) & \cong & T_{\vec{r}}\,\Lambda Ibimod_{\vec{O}}^{h}(M\,;\,N).
\end{array} 
\end{equation}
\end{itemize}
\end{thm}\vspace{5pt}

\begin{sproof}
We already know that this kind of properties are trues in the context of infinitesimal bimodules over a colored operad. So, the theorem is a consequence of the description of $k$-fold infinitesimal bimodules in terms of $1$-fold infinitesimal bimodules over a colored operad introduced in Remark \ref{N6}.\vspace{7pt}
\end{sproof}

A map $\vec{\alpha}:\vec{O}\rightarrow \vec{O}'$ between two families of operads $O_{1},\ldots O_{k}$ and $O_{1}',\ldots O_{k}'$ relative to topological monoids $X$ and $X'$, respectively, is a family of operadic maps $\alpha_{i}:O_{i}\rightarrow O_{i}'$ and a  map of monoids $\alpha:X\rightarrow X'$ such that the following diagram commutes:\vspace{3pt}
$$
\xymatrix{
O_{i}(1) \ar[r]^{\alpha_{i}} \ar[d]^{f_{i}} & O_{i}'(1) \ar[d]^{f_{i}'} \\
X \ar[r]^{\alpha} & X' 
}\vspace{3pt}
$$
Such a map $\vec{\alpha}$ is said to be a weak equivalence if the maps $\alpha$, $\alpha_{i}$ and the induced map \vspace{3pt}
$$
O_{1}(A_{1}^{\ast})\underset{X}{\times}\cdots \underset{X}{\times}O_{k}(A_{k}^{\ast}) \longrightarrow O'_{1}(A_{1}^{\ast})\underset{X}{\times}\cdots \underset{X}{\times}O'_{k}(A_{k}^{\ast}), \hspace{15pt} \forall (A_{1},\ldots,A_{k})\in \Sigma^{\times k},\vspace{3pt} 
$$
are weak homotopy equivalences. \vspace{7pt}

\begin{thm}{\cite{Ducoulombier19}}
For any weak equivalence $\vec{\alpha}:\vec{O}\rightarrow \vec{O}'$ of families of reduced operads relative to monoids with cofibrant components, one has Quillen equivalences\vspace{3pt}
$$
\begin{array}{rcl}\vspace{7pt}
\alpha_{Ib}^{!}:\Lambda Ibimod_{\vec{O}} & \leftrightarrows & \Lambda Ibimod_{\vec{O}'}:\alpha^{\ast}_{Ib}, \\ 
\alpha_{Ib}^{!}: T_{\vec{r}}\,\Lambda Ibimod_{\vec{O}} & \leftrightarrows & T_{\vec{r}}\,\Lambda Ibimod_{\vec{O}'}:\alpha^{\ast}_{Ib},
\end{array} \vspace{1pt}
$$
where $\alpha_{Ib}^{\ast}$ is the restriction functor and $\alpha_{Ib}^{!}$ is the induction functor. \vspace{5pt} 
\end{thm}

\begin{sproof}
According to the notation introduced in Remark \ref{N6}, the weak equivalence $\vec{\alpha}:\vec{O}\rightarrow \vec{O}'$ induced a weak equivalence of colored operads $\tilde{\alpha}:\tilde{O}\rightarrow \tilde{O}'$ where $\tilde{O}$ and $\tilde{O}'$ are the colored operads associated with the families $O_{1},\ldots O_{k}$ and $O_{1}',\ldots O_{k}'$, respectively, obtained from the formula (\ref{N7}). Nevertheless, a weak equivalence of colored operads with cofibrant components induces a Quillen equivalence between the corresponding categories of infinitesimal bimodules. In other words, one has\vspace{5pt}
$$
\xymatrix{
Ibimod_{\vec{O}}\cong Ibimod_{\tilde{O}}\downarrow N_{\ast} \ar@<0.5ex>[r] & Ibimod_{\tilde{O}'}\downarrow N_{\ast}\cong Ibimod_{\vec{O}'} \ar@<0.5ex>[l]^{\substack{\text{Quillen} \\ \text{equivalence}}}.
}
$$
\end{sproof}

\subsection{The Boardman-Vogt resolution for $k$-fold infinitesimal bimodules}\label{C6}

As explained in the previous section, the category of $k$-fold infinitesimal bimodules is endowed with a projective model category structure in which all the objects are fibrant. Consequently, in order to compute the derived mapping space \vspace{3pt}
$$
Ibimod_{\vec{O}}^{h}(N_{1}\,;\,N_{2})\coloneqq Ibimod_{\vec{O}}(N_{1}^{c}\,;\,N_{2}),\vspace{3pt}
$$
we only need an explicit cofibrant replacement $N_{1}^{c}$ for any $k$-fold infinitesimal bimodule $N_{1}$. For this purpose, we use a kind of Boardman-Vogt resolution which is well known in the context of operads \cite{Berger06,Boardman68,Boardman73} as well as the context of (infinitesimal) bimodules \cite{Ducoulombier16,Ducoulombier18}. Thereafter, we deduce from this cofibrant replacement another cofibrant replacement of $N_{1}$ in the Reedy model category of $k$-fold infinitesimal bimodules over $\vec{O}$.\vspace{7pt}

\begin{defi}\label{C8}\textbf{The set of $k$-fold pearled trees}

\noindent For $\vec{n}=(n_{1},\ldots,n_{k})\in \mathbb{N}^{k}$, we consider the set $pTree[\,\vec{n}\,]$ of $k$-fold pearled trees $\vec{T}=(T_{1},\ldots,T_{k},\vec{\sigma})$ where $T_{i}$ is a pearled tree having $n_{i}$ leaves and $\vec{\sigma}\in \Sigma_{n_{1}}\times \cdots \times \Sigma_{n_{k}}$ is a permutation labelling the leaves of the pearled trees.  According to the notation introduced in Definition \ref{B4}, we also assume that $|V^{rp}(T_{1})|=|V^{rp}(T_{i})|$  for all $i\leq k$. \vspace{-15pt}   

\begin{figure}[!h]
\begin{center}
\includegraphics[scale=0.33]{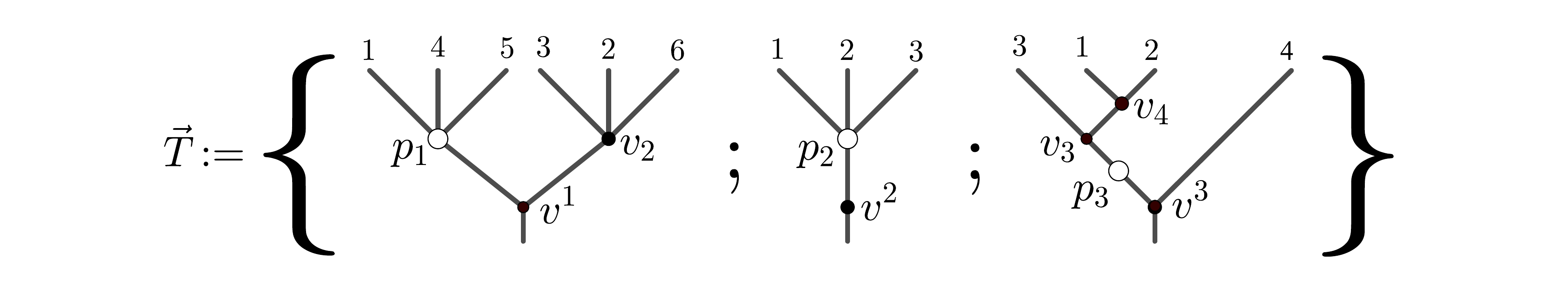}\vspace{-13pt}
\caption{Illustration of an element in $pTree[6\,;\,3\,;\,4]$.}\label{K7}\vspace{-10pt}
\end{center}
\end{figure}
\end{defi}

\begin{const}\label{E2}
Let $N$ be a $k$-fold infinitesimal bimodule over $\vec{O}$. For any $\vec{n}=(n_{1},\ldots,n_{k})\in \mathbb{N}^{k}$, the space $\mathcal{I}b_{\vec{O}}(N)(\,\vec{n}\,)$, also denoted by $\mathcal{I}b(N)(\,\vec{n}\,)$ when $\vec{O}$ is understood, is defined as the quotient of the subspace of the coproduct\vspace{3pt}
\begin{equation}\label{F8}
\left.\left(\underset{\vec{T} \in\, pTree[\vec{n}]}{\coprod} \!\!\!\!\!\! N(|p_{1}|,\ldots, |p_{k}|) \,\times \underset{v^{1}\in V^{rp}(T_{1})}{\prod} \big[ \vec{O}(|v^{1}|,\ldots,|v^{k}|)\times [0\,,\,1]\big] \times \underset{\substack{i\in \{1,\ldots, n\}\\ v\in V^{c}(T_{i})}}{\prod}\!\!\!\! \big[ O_{i}(|v|)\times [0\,,\,1]\big]\right)\,\,\right/ \!\!\sim\vspace{2pt}
\end{equation}
formed by points satisfying the following condition: if $e$ is an inner edge from $v$ to $v'$, according to the orientation towards the pearl, then one has $t_{v}\leq t_{v'}$ where $t_{v}$ and $t_{v'}$ are the real numbers associated to $v$ and $v'$, respectively. By convention, the pearls are indexed by $0$. For instance, any point indexed by the element represented in Figure \ref{K7} must satisfy to the conditions $t_{v^{1}}= t_{v^{2}}=t_{v^{3}}$, $t_{v^{1}}\leq t_{v_{2}}$ and $t_{v_{3}}\leq t_{v_{4}}$. \vspace{7pt}

The equivalent relation is generated by the compatibility relations with the symmetric group action $\Sigma_{n_{1}}\times \cdots \times \Sigma_{n_{k}}$ and the unit axiom. Furthermore, if two vertices connected by an inner edge $e$ are indexed by the same real number $t_{e}$, then we contract $e$ using the operadic structures of $O_{1},\ldots,O_{k}$ or the $k$-fold infinitesimal bimodule structure of $N$. The vertex so obtained is indexed by $t_{e}$. For instance, the point in $\mathcal{I}b(N)(4\,;\,4)$ \vspace{-10pt} 
\begin{figure}[!h]
\begin{center}
\includegraphics[scale=0.47]{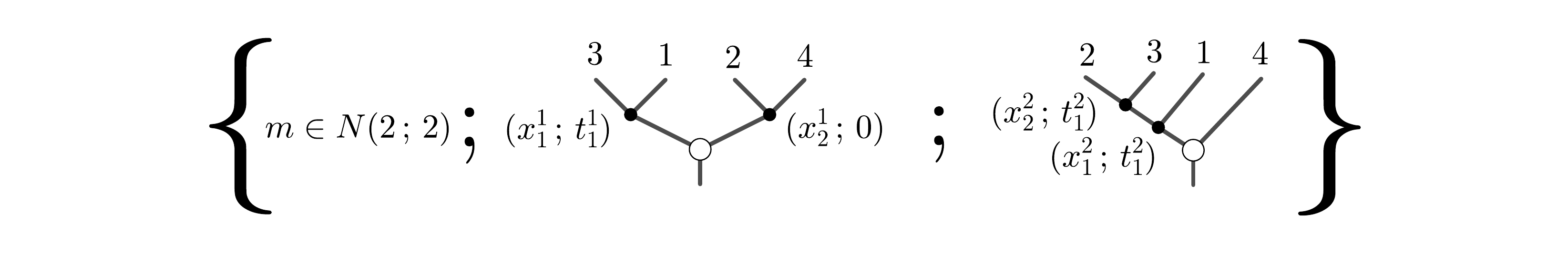}\vspace{-22pt}
\end{center}
\end{figure}

\noindent is equivalent to the following one \vspace{-10pt}

\begin{figure}[!h]
\begin{center}
\includegraphics[scale=0.47]{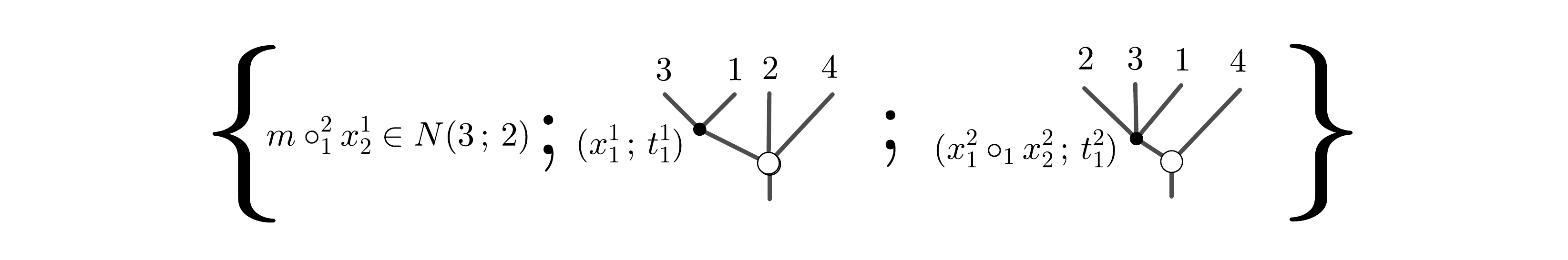}\vspace{-54pt}
\end{center}
\end{figure}

\newpage

The $k$-fold infinitesimal right operation $\circ_{i}^{j}$ with an element $x\in O_{i}(n)$ consists in grafting a $n$-corolla indexed by the pair $(x\,;\,1)$ into the $j$-th leaf of the tree $T_{i}$. Similarly, given an element $(x_{1},\ldots, x_{k})$ in the space  $\vec{O}(n_{1},\ldots,n_{k})$, the $k$-fold infinitesimal left operation consists in grafting each pearled tree $T_{i}$ into the first leaf of the $n_{i}$-corolla indexed by $(x_{i}\,;\,1)$. Moreover, one considers two maps of $k$-fold infinitesimal bimodules\vspace{3pt}
\begin{equation}\label{B0}
\eta:\mathcal{I}b(N)\longrightarrow N \hspace{15pt}\text{and}\hspace{15pt} \tau:\mathcal{F}_{Ib}(N)\longmapsto \mathcal{I}b(N)\vspace{5pt}
\end{equation}
where $\eta$ sends the real numbers to $0$ while $\tau$ indexes the vertices other than the pearls by $1$. The reader can easily check that the above construction for $1$-fold infinitesimal bimodules is homeomorphic to the Boardman-Vogt resolution for infinitesimal bimodules introduced in \cite{Ducoulombier18}.\vspace{5pt}
\end{const}

From now on, we provide a filtration of the resolution $\mathcal{I}b(N)$ according to the number of geometrical inputs which is the number of leaves plus the number of univalent vertices other than the pearls. A point in $\mathcal{I}b(N)$ is said to be \textit{prime} if the real numbers indexing the vertices are strictly smaller than $1$. Besides, a point is said to be \textit{composite} if one of the real numbers is $1$ and such a point can be associated to a \textit{prime component}. More precisely, the prime component is obtained by cutting the vertices indexed by $1$. For instance, the prime component associated to the composite point \vspace{-15pt}
\begin{figure}[!h]
\begin{center}
\includegraphics[scale=0.6]{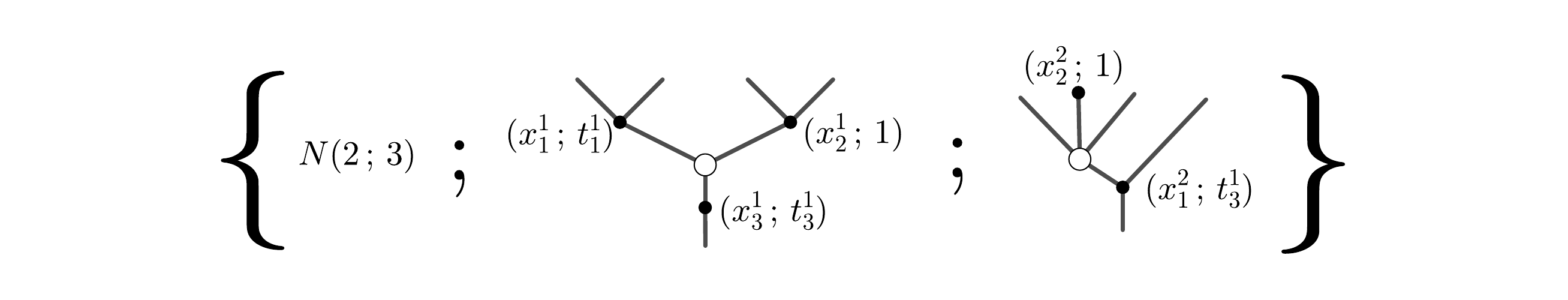}\vspace{-21pt}
\end{center}
\end{figure}

\noindent is given by the following element \vspace{-10pt}

\begin{figure}[!h]
\begin{center}
\includegraphics[scale=0.6]{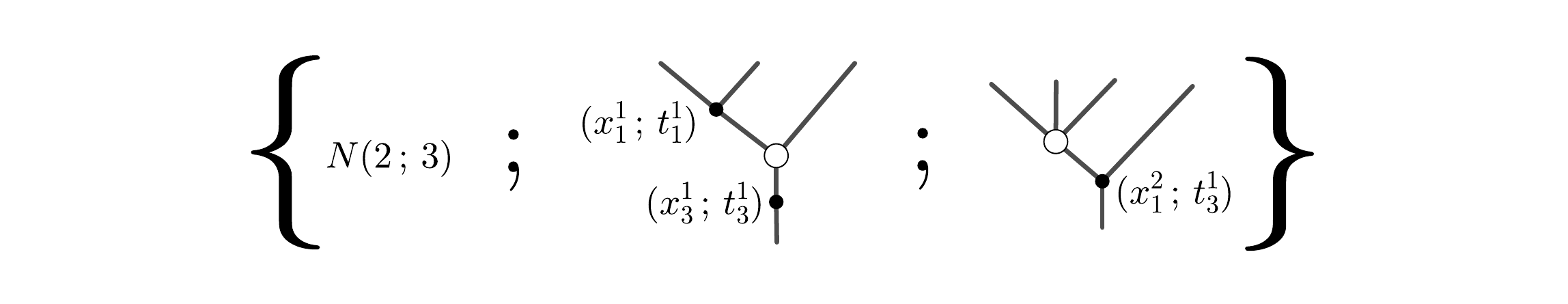}
\vspace{-20pt}
\end{center}
\end{figure}

 For $\vec{n}=(n_{1},\ldots,n_{k})\in \mathbb{N}^{k}$, a prime point, indexed by an element $(T_{1},\ldots,T_{k},\sigma)$, is in the $\vec{n}$-th filtration term $\mathcal{I}b(N)[\,\vec{n}\,]$ if $T_{i}$ has at most $n_{i}$ geometrical inputs. Then, a composite point is in the $\vec{n}$-th filtration term if its prime component is in $\mathcal{I}b(N)[\,\vec{n}\,]$. For instance the composite point above is in $\mathcal{I}b(N)[3\,;\,4]$.  By construction, $\mathcal{I}b(N)[\,\vec{n}\,]$ is a $k$-fold infinitesimal bimodule and, for each pair $\vec{m}\leq \vec{n}$, there is an inclusion of $k$-fold infinitesimal bimodules\vspace{3pt}
\begin{equation}\label{G1}
\iota[\vec{m}\leq \vec{n}]:\mathcal{I}b(N)[\,\vec{m}\,]\longrightarrow \mathcal{I}b(N)[\,\vec{n}\,].\vspace{7pt}
\end{equation}

\begin{thm}\label{K9}
Let $N$ be a $k$-fold infinitesimal bimodule over $\vec{O}$. If $\vec{O}$ and $N$ are cofibrant as $k$-fold sequences, then $\mathcal{I}b(N)$ and $T_{\vec{r}\,}(\mathcal{I}b(N)[\,\vec{r}\,])$ are cofibrant replacements of $N$ and $T_{\vec{r}}\,(N)$ in the projective model categories $Ibimod_{\vec{O}}$ and $T_{\vec{r}}\,Ibimod_{\vec{O}}$, respectively. In particular, the maps $\eta$ and $\tau$ (see (\ref{B0})) are, respectively, a weak equivalence and a cofibration.\vspace{7pt}
\end{thm}

\begin{sproof}
There are many ways to prove this theorem. One of them consists in using Remark \ref{N6} and the fact that the Boardman-Vogt resolution introduced in Construction \ref{E2} is homeomorphic to the Boardman-Vogt resolution  introduced \cite{Ducoulombier18} in the category of infinitesimal bimodules over the appropriated colored operad. \vspace{5pt}

We can also check by hand the theorem. In that case the strategy is the following: the map $\eta$ induces a homotopy equivalence in the category of $k$-fold sequences in which the homotopy consists in bringing the real numbers indexing the vertices to $0$. For the truncated case, we need first to contract the inner edges which are not connected to a leaf because the homotopy sending the real numbers to $0$ doesn't necessarily preserve the number of geometrical inputs. On the other hand, the maps $\mu$ and $\iota[\vec{m}\leq \vec{n}]$ can be proved to be cofibrations in the category of $k$-fold infinitesimal bimodules by induction on the number of vertices of $k$-fold pearled trees. 
\end{sproof}


\subsubsection{Boardman-Vogt resolution in the $\Lambda^{\times k}$ setting}

From now on, we fix a family of reduced operads $O_{1},\ldots,O_{k}$  relative to a topological monoid $X$ and we adapt the construction in the previous section to produce cofibrant replacements in the category of $k$-fold infinitesimal bimodules equipped with the Reedy model category structure. According to the notation introduced in Section \ref{G0}, we set\vspace{3pt}
$$
\mathcal{I}b_{\vec{O}}^{\Lambda}(N)\coloneqq \mathcal{I}b_{\vec{O}_{>0}}(N).\vspace{3pt}
$$
In other words, this $k$-fold sequence, also denoted by $\mathcal{I}b^{\Lambda}(N)$ when $\vec{O}$ is understood, is obtained by taking the restriction of the coproduct (\ref{F8}) to the $k$-fold pearled trees without univalent vertices other than the pearls. By construction, it is a $k$-fold infinitesimal bimodule over $\vec{O}_{>0}$. Similarly to the free functor in the previous section, we can extend this structure in order to get a $k$-fold infinitesimal bimodule over $\vec{O}$ using the operadic structure of $O_{1},\ldots,O_{k}$ and the $\Lambda^{\times k}$ structure of $M$. Furthermore, the $k$-fold infinitesimal bimodule maps over $\vec{O}_{>0}$
\begin{equation}\label{G2}
\eta:\mathcal{I}b^{\Lambda}(N)\longrightarrow N \hspace{15pt}\text{and}\hspace{15pt} \tau:\mathcal{F}_{Ib}^{\Lambda}(N)\longmapsto \mathcal{I}b^{\Lambda}(N),\vspace{1pt}
\end{equation}
respect the $\Lambda^{\times k}$ structure. In the same way, the filtration (\ref{G1}) gives rise to a filtration of $k$-fold infinitesimal bimodules over $\vec{O}_{>0}$  \vspace{3pt}
$$
\iota[\vec{m}\leq \vec{n}]:\mathcal{I}b^{\Lambda}(N)[\,\vec{m}\,]\longrightarrow \mathcal{I}b^{\Lambda}(N)[\,\vec{n}\,],\vspace{6pt}
$$
compatible with the right action by $O_{i}(0)$. So, this is a filtration of $k$-fold infinitesimal bimodules over $\vec{O}$. \vspace{7pt}

\begin{thm}
Let $N$ be a $k$-fold infinitesimal bimodule over $\vec{O}$. If $\vec{O}$ and $N$ are cofibrant as $k$-fold sequences, then $\mathcal{I}b^{\Lambda}(N)$ and $T_{\vec{r}\,}(\mathcal{I}b^{\Lambda}(N)[\,\vec{r}\,])$ are cofibrant replacements of $N$ and $T_{\vec{r}}\,(N)$ in the Reedy model categories $\Lambda Ibimod_{\vec{O}}$ and $T_{\vec{r}}\Lambda Ibimod_{\vec{O}}$, respectively. In particular, the maps $\eta$ and $\tau$ (see (\ref{G2})) are, respectively, a weak equivalence and a cofibration.
\end{thm}  \vspace{3pt}

\begin{proof}
It is a consequence of Theorems \ref{K8} and \ref{K9}.
\end{proof}\vspace{-9pt}

\section{Multivariable manifold calculus and embeddings spaces}\label{L3}

In \cite{Arone14}, Arone and Turchin show that the Goodwillie-Weiss tower associated to embedding spaces is related to derived mapping spaces of truncated $1$-fold infinitesimal bimodules over the little cubes operad. In this section, we extend this result to embedding spaces whose source object is a manifold with many components (not necessarily the same dimension) using a multivariable version of the Goodwillie-Weiss manifold calculus theory introduced by Munson and Volic in \cite{Munson12}. More precisely, the main result of this section is the following one:\vspace{7pt}

\begin{thm}\label{K2}
Let $d_{1}\leq \ldots \leq d_{k} <n$ be a family of integers. Then, one has\vspace{3pt}
$$
\begin{array}{rcll}\vspace{8pt}
T_{\vec{r}}\,\mathcal{L}(d_{1},\ldots,d_{k}\,;\,n) & \simeq & T_{\vec{r}}\, Ibimod_{\vec{O}}^{h}\big( \mathbb{O}\,;\, \mathcal{R}_{n}^{k}\,\big),& \text{for any } \vec{r}\in (\mathbb{N}\sqcup \{\infty\})^{k},\\ 
\mathcal{L}(d_{1},\ldots,d_{k}\,;\,n) & \simeq & Ibimod_{\vec{O}}^{h}\big( \mathbb{O}\,;\, \mathcal{R}_{n}^{k}\,\big), & \text{if } d_{k}+3\leq n,
\end{array} \vspace{-1pt}
$$ 
where $\mathcal{R}_{n}^{k}$ is the $k$-fold infinitesimal bimodule introduced in Example \ref{D7} ; $\vec{O}$ and $\mathbb{O}$ are the $k$-fold sequences associated to the family of reduced operads $\mathcal{C}_{d_{1}},\ldots,\mathcal{C}_{d_{k}}$ relative to $\mathcal{C}_{n}(1)$ and the functor $T_{\vec{r}}$ associated to the embedding space (which is not the truncated functor) is the $\vec{r}$-th polynomial approximation which will be introduced in the next subsections.  
\end{thm}\vspace{7pt}

The first subsection is devoted to the multivariable manifold calculus theory associated to a contravariant functor of the form $F:\mathcal{O}(\sqcup_{i}M_{i})\rightarrow Top$. We consider two cases: $M_{i}$ is a compact smooth manifold and $M_{i}=\mathbb{R}^{d_{i}}$. Then, we give an alternative definition of $k$-fold infinitesimal bimodules in terms of contravariant functors. This description is used to show that, under some conditions, a contravariant functor $F$ produces a $k$-fold infinitesimal bimodule. In the last subsection, we identify $F$ with a derived mapping space of $k$-fold infinitesimal bimodules and we prove the above theorem.

\subsection{Multivariable manifold calculus theory}

As recalled in the introduction, the theory of manifold calculus of functors developed by Weiss \cite{Weiss99} and Goodwillie-Weiss \cite{Weiss99.2} provides an approximation of contravariant functors $F:\mathcal{O}(M)\rightarrow Top$, where $M$ is a smooth manifold and $\mathcal{O}(M)$ the poset of its open subsets. More precisely, they build the following tower of fibrations\vspace{5pt}
$$
\xymatrix{
& F \ar[dl] \ar[d] \ar[dr] \ar[drr] & & & \\
T_{0}F & T_{1}F\ar[l] & T_{2}F \ar[l] & T_{3}F \ar[l] & \cdots \ar[l]
}\vspace{5pt}
$$ 
which converges to $F$ under some conditions on the functor. By converging, we mean that the induced natural transformation $F\rightarrow T_{\infty}\,F$, where $T_{\infty}\,F$ is the limit of the tower, is pointwise weak homotopy equivalence. The $r$-th polynomial approximation $T_{r}F$ has the advantage to be easier to understand than the initial functor $F$. For instance, Arone and Turchin \cite[Theorem 5.9]{Arone14} prove that $T_{r}F$ can be identified with derived mapping space of infinitesimal bimodules (or just right modules depending on the context) over the little cubes operad. \vspace{7pt}

Thereafter, Munson and Volic \cite{Munson12} extend this theory to the context of contravariant functors $F:\mathcal{O}(M)\rightarrow Top$ where $M$ is now a disjoint union of compact smooth manifolds $M=\coprod_{i}M_{i}$ and they use it to build an explicit cosimplicial replacement for the space of links \cite{Munson14}. This section is subdivides into two parts. The first one is devoted to recall the theory for compact smooth manifolds. Then, we consider a "compactly supported" version in order to study the particular case $M=\coprod_{i}\mathbb{R}^{d_{i}}$.

\subsubsection{Case 1: Compact smooth manifolds $M=\coprod_{i}M_{i}$} 

From now one, we fix $M_{1},\ldots, M_{k}$ to be compact smooth manifolds of dimension $d_{1},\ldots,d_{k}$, respectively. We denote by $\mathcal{O}(\coprod_{i}M_{i})$ the poset of open subsets of $\coprod_{i}M_{i}$. By using the identification \vspace{3pt}
$$
\xymatrix@R=5pt{
\mathcal{O}\big(\, \underset{1\leq i \leq k}{\displaystyle\coprod} M_{i} \big) \ar[r]^{\cong} & \underset{1\leq i\leq k}{\displaystyle\prod} \mathcal{O}(M_{i});\\
\mathcal{U} \ar@{|->}[r] & \vec{\mathcal{U}}=(\mathcal{U}\cap M_{1},\ldots, \mathcal{U}\cap M_{k}),
}\vspace{3pt}
$$
any contravariant functor $F:\mathcal{O}(\coprod_{i}M_{i})\rightarrow Top$ can be seen as a functor on a single variable or as a multivariable functor. Following the notation introduced by Munson and Volic, we say that such a functor is \textit{good} if the following two conditions hold:\vspace{5pt}
\begin{itemize}
\item[$(i)$] $F$ takes isotopy equivalences to weak equivalences,\vspace{3pt}
\item[$(ii)$] for any sequences $\vec{\mathcal{U}_{0}}\subset \vec{\mathcal{U}_{1}}\subset \cdots$, with $\vec{\mathcal{U}_{i}}=(\mathcal{U}_{i}^{1},\ldots, \mathcal{U}_{i}^{k})\in \prod_{i}\mathcal{O}(M_{i})$, one has a homotopy equivalence\vspace{3pt}
$$
F(\cup \vec{\mathcal{U}_{i}})\longrightarrow holim_{i}\, F(\vec{\mathcal{U}}_{i}). 
$$
\end{itemize}\vspace{3pt}

A good contravariant functor $F:\mathcal{O}(\coprod_{i}M_{i})\rightarrow Top$ is said to be \textit{polynomial of degree} $\vec{r}=(r_{1},\ldots,r_{k})$ if $F$ is polynomial of degree $r_{i}$ along the $i$-th variable. In other words, for any open subsets $\mathcal{U}_{j}\in \mathcal{O}(M_{j})$, with $j\neq i$, and any cubical diagram $D:\{0\,;\,1\}^{r_{i}+1}\rightarrow \mathcal{O}(M_{i})$ for which the natural map \vspace{3pt}
$$
\underset{\hspace{30pt}\{0\,;\,1\}^{r_{i}+1}_{\emptyset}}{colim}\, D(-)\longrightarrow D(1,\ldots,1),\hspace{15pt} \text{with } \{0\,;\,1\}^{r_{i}+1}_{\emptyset}\coloneqq\{0\,;\,1\}^{r_{i}+1}\setminus (1,\ldots,1),\vspace{3pt}
$$
is an isotopy equivalence, then the following map must be a homotopy equivalence:\vspace{5pt}
$$
F(\mathcal{U}_{1},\ldots,D(1,\ldots,1),\ldots, \mathcal{U}_{k})\longrightarrow\hspace{-18pt} \underset{\hspace{30pt}\{0\,;\,1\}^{r_{i}+1}_{\emptyset}}{holim}\, F(\mathcal{U}_{1},\ldots,D(-),\ldots, \mathcal{U}_{k}).
$$

\begin{defi}\textbf{The polynomial approximation for compact manifolds}

\noindent For any $\vec{r}=(r_{1},\ldots,r_{k})\in \mathbb{N}^{k}$ and $\vec{\mathcal{U}}=(\mathcal{U}_{1},\ldots , \mathcal{U}_{k})\in\mathcal{O}(\coprod_{i}M_{i})$, we denote by $\mathcal{O}_{\vec{r}}\,(\vec{\mathcal{U}})$ the poset formed by elements $(V_{1},\ldots,V_{k})\in \mathcal{O}(\vec{\mathcal{U}})$ for which $V_{i}$ is diffeomorphic to at most $r_{i}$ disjoint open cubes of dimension $d_{i}$. Equivalently, $\mathcal{O}_{\vec{r}}\,(\vec{\mathcal{U}})$ can be defined as the poset consisting of families of smooth embeddings $A_{i}\times I_{d_{i}}\hookrightarrow \mathcal{U}_{i}$, where $A_{i}$ is a set with at most $r_{i}$ elements and $I_{d_{i}}$ is the open unit cube of dimension $d_{i}$. The $\vec{r}$-th polynomial approximation of $F$ is the contravariant functor $T_{\vec{r}}\, F:\mathcal{O}(\coprod_{i}M_{i})\rightarrow Top$ given by the formula\vspace{5pt}
$$
T_{\vec{r}}\, F(\vec{\mathcal{U}})\coloneqq\hspace{-10pt} \underset{\hspace{20pt}\mathcal{O}_{\vec{r}}\,(\vec{\mathcal{U}})}{holim}\, F(-).\vspace{5pt}
$$
\end{defi}

\begin{thm}{\cite[Theorems 4.14 and 4.16]{Munson12}}\label{J6}
One has the following properties:\vspace{2pt}
\begin{itemize}
\item[$(1)$] $T_{\vec{r}}\,F$ is polynomial of degree $\vec{r}$.\vspace{2pt}
\item[$(2)$] If $F$ is polynomial of degree $\vec{r}$, then $F\rightarrow T_{\vec{r}}\,F$ is a homotopy equivalence.\vspace{2pt}
\item[$(3)$] If $F$ is polynomial of degree $\vec{r}$, then $F$ is also polynomial of degree $\vec{s}$, with $\vec{r}\leq \vec{s}$. \vspace{1pt}
\item[$(4)$] The map $F(\vec{\mathcal{U}})\rightarrow T_{\vec{r}}\,F(\vec{\mathcal{U}})$ is a weak equivalence for any $\vec{U}\in \mathcal{O}_{\vec{r}}(M)$. Moreover, $T_{\vec{r}}\,F$ is characterized (up to equivalence) as the polynomial functor of degree $\vec{r}$ with this property.
\end{itemize}
\end{thm}\vspace{7pt}

\begin{expl}
As explained in \cite[Theorem 4.19]{Munson12}, for any smooth manifold $N$ of dimension $n$, the contravariant functor associated  to the space of immersions $Imm(\coprod_{i}M_{i}\,;\, N)$ is polynomial of degree $(1,\ldots,1)$ while the contravariant functor associated to the embedding space $Emb(\coprod_{i}M_{i}\,;\, N)$\vspace{3pt}
$$
Emb(-\,;\,N):\mathcal{O}(\textstyle\coprod_{i}M_{i})\longrightarrow Top\vspace{3pt}
$$
is not polynomial of degree $\vec{r}$ for any $\vec{r}\in \mathbb{N}^{k}$. However, if $d_{i}+3\leq n$ for all $i\in \{1,\ldots,k\}$, then the following map is a homotopy equivalence:
$$
Emb(\textstyle\coprod_{i}M_{i}\,;\, N)\longrightarrow T_{\vec{\infty}}Emb(\textstyle\coprod_{i}M_{i}\,;\, N).\vspace{2pt}
$$
\end{expl}

\subsubsection{Case 2: Compactly supported version $M=\coprod_{i}\mathbb{R}^{d_{i}}$} \label{M9}

For the present work, we need to consider a "compactly supported" version of multivariable manifold calculus for subsets of $M=\coprod_{i}\mathbb{R}^{d_{i}}$, in order to study functors that are invariant with respect to isotopies with bounded support.\vspace{7pt}

In what follows, we denote by $\mathcal{O}_{\partial}(M)=\mathcal{O}_{\partial}(\coprod_{i}\mathbb{R}^{d_{i}})$ the poset of open subsets of $\coprod_{i}\mathbb{R}^{d_{i}}$ whose complement is bounded. A morphism $(\mathcal{U}_{1},\ldots,\mathcal{U}_{k})=\mathcal{U}\rightarrow V=(V_{1},\ldots,V_{k})$ is said to be an isotopy equivalence if there are smooth embeddings $V_{i}\rightarrow U_{i}$, that coincide with the identity outside a bounded subset of $V_{i}$ such that both compositions are isotopic to the identity via an isotpy  which is constant outside a bounded subset.  \vspace{7pt}

The notion of \textit{good} contravariant functor $F:\mathcal{O}_{\partial}(\coprod_{i}\mathbb{R}^{d_{i}})\rightarrow Top$ as well as the notion of polynomial functor of degree $\vec{r}$ are defined similarly to the previous case. We only need to consider the poset $\mathcal{O}_{\partial}(\coprod_{i}\mathbb{R}^{d_{i}})$ instead of the poset $\mathcal{O}(\coprod_{i}M_{i})$ in the definitions.\vspace{7pt} 

\begin{defi}\label{J5}\textbf{Compactly supported version of the polynomial approximation}

\noindent For any $\vec{r}=(r_{1},\ldots,r_{k})\in \mathbb{N}^{k}$ and $\vec{\mathcal{U}}=(\mathcal{U}_{1},\ldots , \mathcal{U}_{k})\in\mathcal{O}_{\partial}(\coprod_{i}\mathbb{R}^{d_{i}})$, we denote by $\mathcal{O}_{\partial\,;\,\vec{r}}\,(\vec{\mathcal{U}})$ the poset formed by elements $(V_{1},\ldots,V_{k})\in \mathcal{O}(\vec{\mathcal{U}})$ for which $V_{i}$ is diffeomorphic to at most $r_{i}$ disjoint open cubes and one anti-cube (complementary of the closed unit cube) of dimension $d_{i}$. \vspace{7pt}

Equivalently, the poset consists of families of compactly supported smooth embeddings $A_{i}^{\ast}\boxtimes I_{d_{i}}\hookrightarrow \mathcal{U}_{i}$, where $A_{i}^{\ast}=A_{i}\sqcup \{\ast\}$ is a pointed set marked by $\ast$. Here, $A_{i}^{\ast}\boxtimes I_{d_{i}}$ is the disjoint union of $|A_{i}|$ open cubes labelled by the set $A_{i}$ and one anti-cube associated to the marked element. By "compactly supported", we means that the embeddings are required to agree with the identity on the anti-cube outside a bounded set. So, the $\vec{r}$-th polynomial approximation of $F$ is given by the formula\vspace{5pt}
\begin{equation}\label{N8}
T_{\vec{r}}\, F(\vec{\mathcal{U}})\coloneqq\hspace{-10pt} \underset{\hspace{20pt}\mathcal{O}_{\partial\,;\,\vec{r}}\,(\vec{\mathcal{U}})}{holim}\, F(-).
\end{equation}
\end{defi}\vspace{-20pt}

\newpage

\begin{thm}
The properties in Theorem \ref{J6} are true for good contravariant functors $F:\mathcal{O}_{\partial}(\coprod_{i}\mathbb{R}^{d_{i}})\rightarrow Top$.
\end{thm}\vspace{5pt}

\begin{sproof}
Similarly to \cite{Munson12}, the properties can be proved by induction on the number of variables using the fact that the polynomial approximation (\ref{N8}) can be rewritten as follows: \vspace{3pt}
$$
T_{\vec{r}}\, F(\vec{\mathcal{U}})\coloneqq\hspace{-10pt} \underset{\hspace{20pt}\mathcal{O}_{\partial\,;\,r_{1}}\,(\mathcal{U}_{1})}{holim}\, \big( \cdots \big(\hspace{-10pt} \underset{\hspace{20pt}\mathcal{O}_{\partial\,;\,r_{k}}\,(\mathcal{U}_{k})}{holim}\, F(-)\,\big)\cdots \big).
$$
\end{sproof}\vspace{3pt}

\begin{defi}\label{J9}\textbf{Standard embeddings}

\noindent A \textit{standard isomorphism} of $\mathbb{R}^{n}$ is a self homeomorphism that is the composition of a translation and a multiplication by a positive scalar. Let $X$ be a subspace of $\mathbb{R}^{n}$. An embedding $f:X\rightarrow \mathbb{R}^{n}$ is called a \textit{standard embedding} if the restriction of $f$ to any connected component of $X$ coincides with the inclusion followed by a standard isomorphism of $\mathbb{R}^{n}$.\vspace{5pt} 

More generally, if $Y$ is another subset of $\mathbb{R}^{n}$, then a standard embedding from $X$ to $Y$ is a standard embedding of $X$ into $\mathbb{R}^{n}$ whose image lies in $Y$. We denote by $sEmb(X\,;\,Y)$ the space of standard embeddings from $X$ into $Y$. For instance, as a symmetric sequence, one has the following identification:
$$
\mathcal{C}_{d}(A)=sEmb\big(\, A\times [0\,,\,1]^{d}\,;\, [0\,,\,1]^{d}\,\big), \hspace{15pt} \text{for any finite set }A.
$$
\end{defi}\vspace{3pt}

\begin{defi}\label{K1}\textbf{Context-free functors}

\noindent Let $\mathcal{M}$ be the topologically enriched category whose objects are families of pairs $((\mathcal{U}_{1}\,,\,\mathcal{U}_{1}^{\ast}),\ldots,(\mathcal{U}_{k}\,,\,\mathcal{U}_{k}^{\ast}))$, where $\mathcal{U}_{i}$ is  disjoint unions (possibly empty) of open subsets of $\mathbb{R}^{d_{i}}$ and $\mathcal{U}_{i}^{\ast}$ is a connected component of $\mathcal{U}_{i}$ that is the complement of a compact subset. Morphisms from $((\mathcal{U}_{1}\,,\,\mathcal{U}_{1}^{\ast}),\ldots,(\mathcal{U}_{k}\,,\,\mathcal{U}_{k}^{\ast}))$ to $((\mathcal{V}_{1}\,,\,\mathcal{V}_{1}^{\ast}),\ldots,(\mathcal{V}_{k}\,,\,\mathcal{V}_{k}^{\ast}))$ are standard embeddings from $\mathcal{U}_{i}$ to $\mathcal{V}_{i}$ that take $\mathcal{U}_{i}^{\ast}$ to $\mathcal{V}_{i}^{\ast}$. \vspace{5pt}

Let $F$ be a good contravariant functor from $\mathcal{O}_{\partial}(\coprod_{i}\mathbb{R}^{d_{i}})$ to spaces. We say that $F$ is \textit{context-free} if it factors (up to natural equivalence) through the category $\mathcal{M}$. In other words, there is a contravariant functor $F':\mathcal{M}\rightarrow Top$ such that $F$ is weakly equivalent to the composed functor
$$
\xymatrix{
\mathcal{O}_{\partial}(\coprod_{i}\mathbb{R}^{d_{i}}) \ar[r] & \mathcal{M} \ar[r]^{F'} & Top.
}
$$ 
\end{defi}\vspace{1pt}

\begin{expl}\label{Q3}
Similarly to the compact case, the reader can check that the contravariant functor associated to the space of immersions compactly supported $Imm_{c}(\coprod_{i}\mathbb{R}^{d_{i}}\,;\, \mathbb{R}^{n})$ is polynomial of degree $(1,\ldots,1)$ and context-free. If there no ambiguity about the dimensions $d_{1},\ldots,d_{k},n$, then we denote by $\mathcal{L}(-)$ the contravariant functor associated to the space $\mathcal{L}(d_{1},\ldots,d_{k}\,;\,n)$ given by \vspace{3pt}
$$
\mathcal{L}(\vec{U})=hofib\left(\,
Emb_{c}\left(\,\vec{\mathcal{U}}\,;\, \mathbb{R}^{n}\,\right)\longrightarrow 
Imm_{c}\left(\,\vec{\mathcal{U}}\,;\, \mathbb{R}^{n}\,\right)\,
\right). \vspace{3pt}
$$
Similarly to the compact case, we can prove by induction that the functors \vspace{3pt}
$$
Emb_{c}(-\,;\,\mathbb{R}^{n}),\,\, \mathcal{L}(-):\mathcal{O}_{\partial}(\textstyle\coprod_{i}\mathbb{R}^{d_{i}})\longrightarrow Top \vspace{3pt}
$$
are context-free but there are not polynomial of degree $\vec{r}$ for any $\vec{r}\in \mathbb{N}^{k}$. However, if $d_{i}+3\leq n$ for all $i\in \{1,\ldots, k\}$, then the following maps are homotopy equivalences:\vspace{3pt}
$$
\begin{array}{rcl}\vspace{10pt}
Emb_{c}(\textstyle\coprod_{i}\mathbb{R}^{d_{i}}\,;\, \mathbb{R}^{n}) & \longrightarrow & T_{\vec{\infty}}Emb_{c}(\textstyle\coprod_{i}\mathbb{R}^{d_{i}}\,;\, \mathbb{R}^{n}), \\ 
\mathcal{L}(d_{1},\ldots,d_{k}\,;\,n) & \longrightarrow & T_{\vec{\infty}}\mathcal{L}(d_{1},\ldots,d_{k}\,;\,n).
\end{array} 
$$
\end{expl}

\subsection{A $k$-fold infinitesimal bimodule as a contravariant functor}

In order to compare Goodwillie-Weiss calculus towers with derived mapping spaces of $1$-fold infinitesimal bimodules, Arone and Turchin \cite[Proposition 3.9]{Arone14} identify the category $Ibimod_{O}$ with an enriched category of contravariant functors from a small category $\Gamma(O)$ to topological spaces. In the following, we adapt this method for $k$-fold infinitesimal bimodules and the multivariable Goodwillie-Weiss calculus. First, we recall the definition of the category $\Gamma(O)$. 

\begin{defi}\textbf{The small category $\Gamma(O)$}

\noindent Let $O$ be a reduced operad and let $\Gamma(O)$ be the category enriched over $Top$ whose objects are finite pointed set. By convention, we denote by $\ast_{A}$ the marked element in the pointed set $A$. The space of morphisms between the two pointed sets $A$ and $B$ is given by \vspace{5pt}
$$
\Gamma(O)(A\,;\,B)=\underset{\alpha:A\rightarrow B}{\coprod} \hspace{5pt}\underset{b\in B}{\prod} \hspace{5pt} O(\,\alpha^{-1}(b)\,) \vspace{5pt}
$$
where the coproduct is indexed by pointed maps. The composition law is not the usual one. Let us fix two pointed maps  $\alpha:A\rightarrow B$ and $\beta:B\rightarrow C$. We need a map of the form \vspace{5pt}
$$
\underset{c\in C}{\prod}\hspace{5pt}O(\,\beta^{-1}(c)\,)\times \underset{b\in B}{\prod}\hspace{5pt}O(\,\alpha^{-1}(b)\,) \longrightarrow \underset{c\in C}{\prod}\hspace{5pt}O(\,(\beta\circ\alpha)^{-1}(c)\,). \vspace{5pt}
$$
For this purpose, we rewrite the left hand term as follows: \vspace{5pt}
$$
\underset{\text{Part 1}}{\underbrace{O(\,\alpha^{-1}(\ast_{B})\,)\times O(\,\beta^{-1}(\ast_{C})\,) \times \underset{b\in \beta^{-1}(\ast_{C})\setminus \{\ast_{B}\}}{\prod}\hspace{-5pt}O(\,\alpha^{-1}(b)\,)}} \times \underset{c\in C\setminus \{\ast_{C}\}}{\prod}\hspace{5pt} \underset{\text{Part 2}}{\underbrace{ O(\,\beta^{-1}(c)\,) \times \underset{b\in \beta^{-1}(c)}{\prod} \hspace{5pt} O(\,\alpha^{-1}(b)\,) }}. \vspace{3pt}
$$
In Part 2, we use the operadic structure of $O$ in order to get an element in $O(\, (\beta\circ\alpha)^{-1}(c)\,)$ with $c\neq \ast_{C}$. In Part 1, first we use the composition $O(\,\alpha^{-1}(\ast_{B})\,)\circ_{\ast_{A}} O(\,\beta^{-1}(\ast_{C})\,)$ followed by the operadic composition $\circ_{b}$, with $b\in  \beta^{-1}(\ast_{C})\setminus \{\ast_{B}\}$ in order to compose the other points.  
\end{defi} \vspace{5pt}

\begin{defi}\textbf{The small category $\Gamma(\vec{O})$}

\noindent Let $O_{1},\ldots,O_{k}$ be a family of reduced operads relative to a topological monoid $X$. Let  $\Gamma(\vec{O})$ be the category enriched over $Top$ whose objects are families of pointed sets of the form $\vec{A}=(A_{1},\ldots,A_{k})$. By convention, we denote by $\ast_{A_{i}}$ the marked element in the pointed set $A_{i}$. The space of morphisms between two families $\vec{A}=(A_{1},\ldots,A_{k})$ and $\vec{B}=(B_{1},\ldots,B_{k})$ is given by the formula  \vspace{3pt}
\begin{equation}\label{Z1}
\Gamma(\vec{O})(\vec{A}\,;\,\vec{B})\coloneqq\Gamma(O_{1})(A_{1}\,;\,B_{1}) \underset{X}{\times} \cdots \underset{X}{\times}\Gamma(O_{k})(A_{k}\,;\,A_{k}), \vspace{1pt}
\end{equation}
where the product over $X$ is obtained using the composite map \vspace{5pt}
$$
\Gamma(O_{i})(A_{i}\,;\,B_{i})= \underset{\alpha:A_{i}\rightarrow B_{i}}{\coprod} \hspace{5pt}\underset{b\in B_{i}}{\prod} \hspace{5pt} O_{i}(\,\alpha^{-1}(b)\,)\longrightarrow \underset{\alpha:A_{i}\rightarrow B_{i}}{\coprod} O_{i}(\,\alpha^{-1}(\ast_{B_{i}})\,)\longrightarrow \underset{\alpha:A_{i}\rightarrow B_{i}}{\coprod} O_{i}(\ast_{A_{i}})\longrightarrow X \vspace{3pt}
$$
The composition law is obtained from the composition law on each factor $\Gamma(O_{i})(A_{i}\,;\,B_{i})$. The reader can easily check that the composition so obtained is still well defined.
\end{defi} \vspace{5pt}

\begin{pro}
The category of $k$-fold infinitesimal bimodules over $\vec{O}$ is equivalent to the category of contravariant functors from $\Gamma(\vec{O})$ to topological spaces. \vspace{5pt}
\end{pro}

\begin{proof}
First, let $N$ be a $k$-fold infinitesimal bimodule. We will associate with $N$ an enriched contravariant functor $\overline{N}:\Gamma(\vec{O})\rightarrow Top$. It is defined on objects of $\Gamma(\vec{O})$ by the formula \vspace{3pt}
$$
\overline{N}(A_{1},\ldots,A_{k})\coloneqq N(A_{1}\setminus \{\ast_{A_{1}}\},\ldots,A_{k}\setminus \{\ast_{A_{k}}\}).
$$
Let $\vec{A}=(A_{1},\ldots,A_{k})$ and $\vec{B}=(B_{1},\ldots,B_{k})$ be two objects in the category $\Gamma(\vec{O})$. To describe the action of $\overline{N}$ on morphisms, we need to construct continuous maps \vspace{3pt}
$$
\overline{N}(\vec{B})\times \Gamma(\vec{O})(\vec{A}\,;\,\vec{B})\longrightarrow \overline{N}(\vec{A})
$$ 
that are associative and unital with respect to composition in $\Gamma(\vec{O})$. Since the category of topological spaces is obviously  enriched, tensored and cotensored over itself, it follows that the product distributes over coproducts. So, our task is equivalent to constructing morphisms  \vspace{3pt}
$$
\underset{
\begin{array}{c}
\parallel \\ 
\vec{O}(\,\alpha^{-1}_{1}(\ast_{B_{1}})\setminus \{\ast_{A_{1}}\},\ldots,\alpha^{-1}_{k}(\ast_{B_{k}})\setminus \{\ast_{A_{k}}\} \,)
\end{array} 
 }{\underbrace{O(\,\alpha^{-1}_{1}(\ast_{B_{1}}) \,)\underset{X}{\times} \cdots \underset{X}{\times} O(\,\alpha^{-1}_{k}(\ast_{B_{k}}) \,)}}\hspace{-15pt} \times \overline{N}(\vec{B}) \times \underset{\substack{ i\in \{1,\ldots, k\}\\ b\in B_{i}\setminus \{\ast_{B_{i}}\}}}{\prod} O(\,\alpha_{i}^{-1}(b)\,)\longrightarrow \overline{N}(\vec{A}), \vspace{3pt}
$$
where $\alpha_{i}:A_{i}\rightarrow B_{i}$ is a pointed map. Nevertheless, this map can be defined using the $k$-fold infinitesimal bimodule structure of $N$. Furthermore, the operation so obtained respect the composition in $\Gamma(\vec{O})$. \vspace{5pt}

Conversely, suppose that $\overline{N}:\Gamma(\vec{O})\rightarrow Top$ is an enriched contravariant functor. We need to associate with it a $k$-fold infinitesimal bimodule $N$. Objectwise, $N$ is defined  as \vspace{1pt}
$$
N(A_{1},\ldots, A_{k})=\overline{N}(A_{1}^{\ast},\ldots , A_{k}^{\ast}),\hspace{15pt} \text{for any }(A_{1},\ldots,A_{k})\in \Sigma^{\times k}. \vspace{1pt}
$$ 
The $k$-fold left infinitesimal bimodule operations  \vspace{1pt}
$$
\mu:\vec{O}(B_{1},\ldots,B_{k})\times N(A_{1},\ldots,A_{k})\longrightarrow N(A_{1}\sqcup B_{1},\ldots, A_{k}\sqcup B_{k}) \vspace{1pt}
$$ 
are obtained using the pointed maps $\alpha_{i}:(A_{i}\sqcup B_{i})^{\ast}\rightarrow A_{i}^{\ast}$ sending $a\in A_{i}$ to itself and the other elements to the base point $\ast\in A_{i}^{\ast}$.  The functor $\overline{N}$, applied to the family of maps $\{\alpha_{i}\}$, gives rise to an operation of the form  \vspace{2pt}
$$
\underset{
\begin{array}{c}
\parallel \\ 
\vec{O}(B_{1},\ldots,B_{k})
\end{array} 
 }{\underbrace{O(\,\alpha^{-1}_{1}(\ast) \,)\underset{X}{\times} \cdots \underset{X}{\times} O(\,\alpha^{-1}_{k}(\ast) \,)}}\hspace{2pt} \times 
 \underset{
\begin{array}{c}\vspace{2pt}
\parallel \\ 
N(A_{1},\ldots,A_{k})
\end{array} 
 }{\underbrace{\overline{N}(A_{1}^{\ast},\ldots,A_{k}^{\ast})\hspace{-5pt}\phantom{\underset{X}{\times}}}} \times \underset{\substack{ i\in \{1,\ldots, k\}\\ a\in A_{i}}}{\prod} O_{i}(1)\longrightarrow \underset{
\begin{array}{c}\vspace{2pt}
\parallel \\ 
N(A_{1}\sqcup B_{1},\ldots, A_{k}\sqcup B_{k})
\end{array} 
 }{\underbrace{\overline{N}((A_{1}\sqcup B_{1})^{\ast},\ldots, (A_{k}\sqcup B_{k})^{\ast})\hspace{-5pt}\phantom{\underset{X}{\times}}}}. \vspace{2pt}
$$
By taking the unit $\ast_{1}\in O_{i}(1)$, we get exactly the operation researched. In the same way, we can define the $k$-fold right infinitesimal operations making $N$ into $k$-fold infinitesimal bimodule over $\vec{O}$. The reader can check that the composition law in the category $\Gamma(\vec{O})$ and the $k$-fold infinitesimal bimodule's axioms are compatible.
\end{proof}

\subsubsection{The particular case of the little cubes operads}

For the rest of this section, $\vec{O}$ and $\mathbb{O}$ are the $k$-fold sequences $(\ref{J7})$ and $(\ref{J8})$, respectively, associated to the family of reduced operads $\mathcal{C}_{d_{1}},\ldots, \mathcal{C}_{d_{k}}$ relative to the monoid $\mathcal{C}_{n}(1)$ with $d_{i}<n$ for all $i\in \{1,\ldots,k\}$. Without loss of generality, we assume that $d_{1}\leq \cdots \leq d_{k}$. This section is devoted to give an alternative description of the category $\Gamma(\vec{O})$ using open cubes and anti-cubes. \vspace{7pt} 

First, let us notice a property which is specific to the little cubes operads. Indeed the following proposition implies that any $1$-fold infinitesimal bimodule over the $n$-dimensional little rectangles operad $\mathcal{R}_{d}$ inherits a $k$-fold infinitesimal bimodule structure over the families of the little cubes operads $\mathcal{C}_{d_{1}},\ldots, \mathcal{C}_{d_{k}}$ relative to the monoid $\mathcal{C}_{n}(1)$ as illustrated in  Example \ref{D7}. \vspace{5pt}

\begin{pro}\label{L8}
There exist a covariant functor $D:\Gamma(\,\vec{O}\,)\rightarrow \Gamma(\,\mathcal{R}_{n}\,)$. \vspace{3pt}
\end{pro}

\begin{proof}
On the objects, the functor $D$ sends a family of pointed sets $\vec{A}=(A_{1},\ldots,A_{k})$ to the pointed set \vspace{4pt}
$$
D(\,\vec{A}\,)=\{\ast\}\sqcup \underset{1\leq i \leq k}{\bigsqcup} A_{i}\setminus \{\ast_{A_{i}}\}.
$$
Let $\vec{A}$ and $\vec{B}$ be two families of pointed sets. We consider a point $\vec{x}=(\{x_{b_{1}}\},\ldots , \{x_{b_{k}}\})\in \Gamma(\vec{O})(\,\vec{A}\,;\,\vec{B}\,)$ indexed by the family of pointed maps $\vec{\alpha}\coloneqq\{\alpha_{i}:A_{i}\rightarrow B_{i}\}$. According to our notation, one has\vspace{5pt}
$$
\{x_{b_{i}}\}\in \prod_{b_{i}\in B_{i}}\mathcal{C}_{d_{i}}(\alpha_{i}^{-1}(b_{i})).\vspace{-25pt}
$$

\newpage

\noindent Finally, $D(\,\vec{x}\,)=\{y_{b}\}_{b\in D(\,\vec{B}\,)}$ is the element in $\Gamma(\mathcal{R}_{n})(D(\,\vec{A}\,)\,;\,D(\,\vec{B}\,))$ indexed by the pointed map  \vspace{2pt}
$$
\begin{array}{rcl}\vspace{7pt}
D(\,\vec{\alpha}\,):\{\ast\}\sqcup \underset{1\leq i \leq k}{\displaystyle\bigsqcup} A_{i}\setminus \{\ast_{A_{i}}\} & \longrightarrow & \{\ast\}\sqcup \underset{1\leq i \leq k}{\displaystyle\bigsqcup} B_{i}\setminus \{\ast_{B_{i}}\} \\ \vspace{5pt}
\ast & \longmapsto & \ast \\ 
a_{i}\in A_{i}\setminus\{\ast_{A_{i}}\} & \longmapsto & \left\{
\begin{array}{cl}\vspace{7pt}
\alpha_{i}(a_{i}) & \text{if } \alpha_{i}(a_{i})\neq \ast_{B_{i}}, \\ 
\ast & \text{otherwise}.
\end{array} \right.
\end{array}  \vspace{3pt}
$$
If $b=b_{i}\in B_{i}\setminus \{\ast_{B_{i}}\}$, then $y_{b}=\kappa_{i}(x_{b_{i}})$ where $\kappa_{i}$ is the composite map of operads $\kappa_{i}:\mathcal{C}_{d_{i}}\rightarrow \mathcal{C}_{d_{k}}\hookrightarrow \mathcal{R}_{d_{k}}\rightarrow \mathcal{R}_{n}$. However, if $b=\ast$, then $y_{b}$ is obtained using the map $(\ref{D4})$: \vspace{3pt}
$$
y_{\ast}=\varepsilon(x_{\ast_{B_{1}}},\ldots, x_{\ast_{B_{k}}}), \hspace{15pt}\text{since } (x_{\ast_{B_{1}}},\ldots, x_{\ast_{B_{k}}})\in \vec{O}\big(\alpha_{1}^{-1}(\ast_{B_{1}}),\ldots,\alpha_{k}^{-1}(\ast_{B_{k}})\big). \vspace{3pt}
$$
Thus finishes the proof of the proposition.
\end{proof}
\vspace{9pt}

From now on, we denote by $I_{n}$ the unit open cube of dimension $n$ and we denote by $I_{n}^{c}\coloneqq\mathbb{R}^{n}\setminus [0\,,\,1]^{n}$ the anti-cube of dimension $n$. By definition, one has a description of the space of morphisms (\ref{Z1}) in terms of standard embeddings (see Definition \ref{J9})\vspace{6pt}
\begin{equation}\label{O2}
\Gamma(\,\vec{O}\,)(\,\vec{A}\,;\,\vec{B}\,)\cong 
sEmb\big(A_{1}\times I_{d_{1}}\,;\, B_{1}\times I_{d_{1}}\big) \underset{sEmb\big( I_{d_{1}}\,;\, I_{d_{2}}\big)}{\varprod}\cdots \cdots \underset{sEmb\big( I_{d_{k-1}}\,;\, I_{d_{k}}\big)}{\varprod} sEmb\big(A_{k}\times I_{d_{k}}\,;\, B_{k}\times I_{d_{k}}\big),\vspace{4pt}
\end{equation}
where the map from $sEmb(A_{i}\times I_{d_{i}}\,;\, B_{i}\times I_{d_{i}})$ to $sEmb(I_{d_{i}}\,;\, I_{d_{i+1}})$ is obtained by taking the standard embedding associated to the marked element followed by the inclusion $\mathbb{R}^{d_{i}}\rightarrow \mathbb{R}^{d_{i+1}}$. Similarly, the map from $sEmb(A_{i+1}\times I_{d_{i+1}}\,;\, B_{i+1}\times I_{d_{i+1}})$ to $sEmb(I_{d_{i}}\,;\, I_{d_{i+1}})$ is obtained by taking the inclusion $\mathbb{R}^{d_{i}}\rightarrow \mathbb{R}^{d_{i+1}}$ followed by the standard embedding associated to the marked element. The composite law is still the same. \vspace{7pt}

The next proposition shows that if we change slightly the definition of the space $(\ref{O2})$ using the anti-cube, then we can simplify the composition of the category  $\Gamma(\,\vec{O}\,)$. In that case, the new composition law is the ordinary composition of standard embeddings. For any pointed set $A$, we consider the following notation for the disjoint union of $|A|-1$ open cubes and one anti-cube:\vspace{3pt}
$$
A\boxtimes I_{n}\coloneqq \big[(A\setminus \{\ast_{A}\})\times I_{n}\big] \displaystyle \coprod I_{n}^{c}.\vspace{5pt}
$$

\begin{pro}\label{K0}
The category $\Gamma(\,\vec{O}\,)$ is equivalent to the category ( also denoted by $\Gamma(\,\vec{O}\,)$ by abuse of notation) whose objects are families of pointed set $\vec{A}=(A_{1},\ldots,A_{k})$ and the space of morphisms from $\vec{A}$ to $\vec{B}$ is \vspace{6pt}
$$
sEmb\big(A_{1}\boxtimes I_{d_{1}}\,;\, B_{1}\boxtimes I_{d_{1}}\big) \underset{sEmb\big( I_{d_{1}}^{c}\,;\, I_{d_{2}}^{c}\big)}{\varprod}\cdots \cdots \underset{sEmb\big( I_{d_{k-1}}^{c}\,;\, I_{d_{k}}^{c}\big)}{\varprod} sEmb\big(A_{k}\boxtimes I_{d_{k}}\,;\, B_{k}\boxtimes I_{d_{k}}\big),\vspace{4pt}
$$
where the map from $sEmb(A_{i+1}\boxtimes I_{d_{i+1}}\,;\, B_{i+1}\boxtimes I_{d_{i+1}})$ is obtained by taking the inclusion $\mathbb{R}^{d_{i}}\rightarrow \mathbb{R}^{d_{i+1}}$ followed by the the standard embedding associated to the anti-cube. The map from $sEmb(A_{i}\boxtimes I_{d_{i}}\,;\, B_{i}\boxtimes I_{d_{i}})$ to $sEmb(I_{d_{i}}^{c}\,;\, I_{d_{i+1}}^{c})$ is obtained by taking the standard embedding associated to the anti-cube followed by the usual inclusion of spaces $\mathbb{R}^{d_{i}}\rightarrow \mathbb{R}^{d_{i+1}}$. In that case, the composite law is just ordinary composition of standard embeddings. 
\end{pro}\vspace{4pt}

\begin{proof}
In \cite[Proposition 4.9]{Arone14}, Arone and Turchin prove the above statement for $\Gamma(\mathcal{C}_{n})$. The same arguments work for the family of reduced operads. 
\end{proof}

\newpage

\begin{expl}\label{K3}In what follows, we list the main examples of contravariant functor from $\Gamma(\,\vec{O}\,)$ to spaces:\vspace{7pt}

\noindent $(1)$ Let $F$ be a good, contravariant and context-free functor from $\mathcal{O}_{\partial}(\sqcup_{i}\mathbb{R}^{d_{i}})$ to spaces. So, $F$ can be seen as a contravariant functor from $\mathcal{M}$ (see Definition \ref{K1}) to spaces. Since the category described in the previous proposition can be interpreted as a subcategory of $\mathcal{M}$ by sending any family of pointed sets $\vec{A}$ to the element \vspace{3pt}
$$
\big(\,(A_{1}\boxtimes I_{d_{1}}\,;\, I_{d_{1}}^{c}),\ldots,(A_{k}\boxtimes I_{d_{k}}\,;\, I_{d_{k}}^{c})\,\big) \in \mathcal{M},\vspace{3pt}
$$  
the functor $F$ produces a contravariant functor from $\Gamma(\,\vec{O}\,)$ to spaces. \vspace{7pt}

\noindent $(2)$ Let $\vec{\mathcal{U}}$ be an element in $\mathcal{O}_{\partial}(\sqcup_{i}\mathbb{R}^{d_{i}})$. We consider the contravariant functor $sEmb(-\,;\,\vec{\mathcal{U}})$  from $\Gamma(\,\vec{O}\,)$ to spaces given by the formula
$$
\vec{A}=(A_{1},\ldots,A_{k})\longmapsto \displaystyle \underset{1\leq i \leq k}{\prod} sEmb(A_{i}\boxtimes I_{d_{i}}\,;\,\mathcal{U}_{i}).\vspace{3pt}
$$ 
Furthermore, in the context of $1$-fold infinitesimal bimodules, Arone and Turchin \cite[Lemma 4.11]{Arone14} show that $sEmb(-\,;\,\mathbb{R}^{d_{i}})$ and the little cubes operad $\mathcal{C}_{d_{i}}$ are weakly equivalent as infinitesimal bimodules over $\mathcal{C}_{d_{i}}$. Similarly, $sEmb(-\,;\,\sqcup_{i}\mathbb{R}^{d_{i}})$ and $\mathbb{O}$ are weakly equivalent as $k$-fold infinitesimal bimodules.\vspace{7pt}

\noindent $(3)$ Let $\vec{\mathcal{U}}$ be an element in $\mathcal{O}_{\partial}(\sqcup_{i}\mathbb{R}^{d_{i}})$. We consider the contravariant functor $sEmb_{\ast}(-\,;\,\vec{\mathcal{U}})$  from $\Gamma(\,\vec{O}\,)$ to spaces given by the formula\vspace{3pt}
$$
\vec{A}=(A_{1},\ldots,A_{k})\longmapsto \displaystyle 
sEmb\big(A_{1}\boxtimes I_{d_{1}}\,;\,\mathcal{U}_{1}\big) \underset{sEmb\big( I_{d_{1}}^{c}\,;\, I_{d_{2}}^{c}\big)}{\varprod}\cdots \cdots \underset{sEmb\big( I_{d_{k-1}}^{c}\,;\, I_{d_{k}}^{c}\big)}{\varprod} sEmb\big(A_{k}\boxtimes I_{d_{k}}\,;\, \mathcal{U}_{k}\big),\vspace{3pt}
$$ 
where the product is defined as in Proposition $\ref{K0}$. Furthermore, there is a natural transformation induced by the inclusion from $sEmb_{\ast}(-\,;\,\vec{\mathcal{U}})$  to $sEmb(-\,;\,\vec{\mathcal{U}})$ which is obviously a weak equivalence objectwise. So, $sEmb_{\ast}(-\,;\,\vec{\mathcal{U}})$ and $sEmb(-\,;\,\vec{\mathcal{U}})$ are weakly equivalent as $k$-fold infinitesimal bimodules.
\end{expl}

\subsection{Embedding spaces as mapping spaces of $k$-fold infinitesimal bimodules}

This section is devoted to the proof of Theorem \ref{K2}. First, we give a general statement which is valuable for any good, contravariant context-free functor. Then, we show that Theorem \ref{K2} is a consequence of the proposition below applied to the contravariant functors $\mathcal{L}$ (see Example \ref{Q3}).\vspace{7pt}

\begin{pro}\label{K5}
Let $F$ be a good, contravariant and context-free functor from $\mathcal{O}_{\partial}(\sqcup_{i}\mathbb{R}^{d_{i}})$ to spaces. Due to Example \ref{K3}, $F$ can be seen as a $k$-fold infinitesimal bimodule. Then, for any $\vec{r}\in \mathbb{N}^{k}$ and $\vec{\mathcal{U}}\in \mathcal{O}_{\partial}(\sqcup_{i}\mathbb{R}^{d_{i}})$, one has the weak homotopy equivalence \vspace{3pt}
\begin{equation}\label{K4}
T_{\vec{r}}\,F(\vec{\mathcal{U}})\simeq T_{\vec{r}}\,Ibimod_{\vec{O}}^{h}\big(\, sEmb(-\,;\,\vec{\mathcal{U}})\,;\,F(-)\,\big).\vspace{5pt}
\end{equation}
\end{pro}

\begin{proof}
We adapt the proof of \cite[Proposition 5.9]{Arone14}. Let $\vec{\mathcal{V}}=(\mathcal{V}_{1},\ldots,\mathcal{V}_{k})$ be an object in $\Gamma(\,\vec{O}\,)$, where $V_{i}$ is the disjoint union of one anti-cube and $s_{i}$ open cubes of dimension $d_{i}$, with $s_{i}\leq r_{i}$. Then we consider the functor from $\mathcal{M}$ to spaces $\vec{\mathcal{U}}\mapsto sEmb(\vec{\mathcal{V}}\,;\,\vec{\mathcal{U}})$ (see the second point in Example \ref{K3}). Notice that it is naturally equivalent to the functor that associated to $\vec{\mathcal{U}}$ the product of configuration spaces of $s_{i}$-tuples of points in $\mathcal{U}_{i}$, with $i\in \{1,\ldots,k\}$. In other words, one has \vspace{5pt}
$$
sEmb(\vec{\mathcal{V}}\,;\,\vec{\mathcal{U}})\simeq \displaystyle \underset{1\leq i\leq k}{\prod} Conf(s_{i}\,;\,\mathcal{U}_{i}).\vspace{3pt}
$$ 
It follows that it is an isotopy functor in the sense that if $\vec{\mathcal{U}}\rightarrow \vec{\mathcal{U}'}$ is a standard embedding that happens to be an isotopy equivalence, then the induced map \vspace{4pt}
$$
sEmb(\vec{\mathcal{V}}\,;\,\vec{\mathcal{U}})\rightarrow sEmb(\vec{\mathcal{V}}\,;\,\vec{\mathcal{U}'})\vspace{-13pt}
$$ 
\newpage

\noindent is a homotopy equivalence. It also takes filtered unions to filtered homotopy colimits. Furthermore, it is easy to check that this functor of $\vec{\mathcal{U}}$ is polynomial of degree $\vec{r}$. As a consequence, for any space $Y$, the contravariant functor\vspace{4pt}
$$
\vec{\mathcal{U}}\longmapsto Map\big(\, sEmb(\vec{\mathcal{V}}\,;\,\vec{\mathcal{U}})\,;\,Y\,\big)\vspace{5pt}
$$ 
is also polynomial of degree $\vec{r}$. Moreover, using the explicit description $(\ref{F8})$ of the cofibrant replacement in the projective model category of $k$-fold infinitesimal bimodules, we can identified the derived mapping space of $k$-infinitesimal bimodules in $(\ref{K4})$ with the "homotopy" limit of mapping space of the form $Map(\, sEmb(\vec{\mathcal{V}}\,;\,\vec{\mathcal{U}})\,;\,Y\,)$. Consequently, the derived mapping space is also a good, contravariant functor that is polynomial of degree $\vec{r}$.\vspace{7pt}

According to point $(4)$ of Theorem \ref{J6} and since the derived mapping space is polynomial of degree $\vec{r}$, we only need to check that the natural transformation from $F$ to the derived mapping space is a weak equivalence for any element in $\mathcal{O}_{\partial\,;\,\vec{r}}\,(\sqcup_{i}\mathbb{R}^{d_{i}})$. Let us start with an element $\vec{\mathcal{U}}\in \mathcal{O}_{\partial\,;\,\vec{r}}\,(\sqcup_{i}\mathbb{R}^{d_{i}}) \cap \Gamma(\vec{O})$, where $\Gamma(\vec{O})$ is seen as a subcategory of $\mathcal{M}$. In other words, $\mathcal{U}_{i}$ is the disjoint union of $r_{i}$ open cubes and one anti-cube. Then, we consider the following commutative diagram
$$
\xymatrix@R=15pt{
F(\,\vec{\mathcal{U}}\,) \ar[r]\ar[rd] & T_{\vec{r}}\,Ibimod_{\vec{O}}^{h}\big(\, sEmb_{\ast}(-\,;\,\vec{\mathcal{U}})\,;\,F(-)\,\big)  \\
& T_{\vec{r}}\,Ibimod_{\vec{O}}^{h}\big(\, sEmb(-\,;\,\vec{\mathcal{U}})\,;\,F(-)\,\big)\ar[u]
}
$$ 
The right vertical map is a weak equivalence since $sEmb_{\ast}(-\,;\,\vec{\mathcal{U}})$ and $sEmb(-\,;\,\vec{\mathcal{U}})$ are weakly equivalent as $k$-fold infinitesimal bimodules. The upper horizontal map is also a weak equivalence using the enriched Yoneda lemma together with the fact that $sEmb_{\ast}(-\,;\,\vec{\mathcal{U}})$ is a representable contravariant functor from $\Gamma(\,\vec{O}\,)$ to spaces. So, the diagonal map in the above diagram is a weak equivalence too. Finally, since $F$ and the functor associated to the derived mapping space are good and since any element in $\mathcal{O}_{\partial\,;\,\vec{r}}\,(\sqcup_{i}\mathbb{R}^{d_{i}})$ is isotopy equivalent to an element in $\mathcal{O}_{\partial\,;\,\vec{r}}\,(\sqcup_{i}\mathbb{R}^{d_{i}}) \cap \Gamma(\vec{O})$, the map from $F$ to the derived mapping space is a weak equivalence for any element in $\mathcal{O}_{\partial\,;\,\vec{r}}\,(\sqcup_{i}\mathbb{R}^{d_{i}})$. 
\end{proof}  \vspace{5pt}

\begin{proof}[Proof of Theorem \ref{K2}]
We recall that $\mathcal{L}$ is the contravariant functor associated to the space $\mathcal{L}(d_{1},\ldots,d_{k}\,;\,n)$. It is a good and context-free functor. As a consequence of Theorem \ref{K5}, one has the following weak homotopy equivalence:\vspace{3pt}
$$
T_{\vec{r}}\,\mathcal{L}(d_{1},\ldots,d_{k}\,;\,n)=T_{\vec{r}}\, \mathcal{L}(\underset{1\leq i\leq k}{\coprod} \mathbb{R}^{d_{i}}) \simeq 
T_{\vec{r}}\,Ibimod_{\vec{O}}^{h}\big( sEmb(-\,;\,\underset{1\leq i\leq k}{\coprod} \mathbb{R}^{d_{i}}) \,;\, \mathcal{L}(-)\big).\vspace{3pt}
$$
As explained in Example \ref{K3}, the functor $sEmb(-\,;\,\sqcup_{i}\mathbb{R}^{d_{i}})$ is weakly equivalent to $\mathbb{O}$ as $k$-fold infinitesimal bimodules. Unfortunately, the functor $\mathcal{L}$ is not directly related to the $k$-fold infinitesimal bimodule $\mathcal{R}_{n}^{k}$. To solve this problem, let us consider some hyperplans $H_{1},\ldots,H_{k}$ of $\mathbb{R}^{n}$ of dimension $d_{1},\ldots,d_{k}$, respectively. Since $d_{1}\leq \cdots \leq d_{k}<n$, we can suppose that the hyperplans are parallel along the last coordinate. Then, we introduce the functor\vspace{3pt} 
$$
\overline{Emb}_{st}(-\,;\,\mathbb{R}^{n}):\mathcal{O}_{\partial\,;\,\vec{r}}\,(\sqcup_{i}\mathbb{R}^{d_{i}}) \longrightarrow Top,\vspace{3pt}
$$ 
defined as the homotopy fiber of the inclusion $Emb_{st}(\vec{\mathcal{U}}\,;\,\mathbb{R}^{n})\rightarrow Imm_{st}(\vec{\mathcal{U}}\,;\,\mathbb{R}^{n})$, where $Emb_{st}(\vec{\mathcal{U}}\,;\,\mathbb{R}^{n})$ is the space of smooth embeddings $f\coloneqq\{f_{i}:\mathcal{U}_{i}\rightarrow\mathbb{R}^{n}\}$ in which $f_{i}$ agree, outside a bounded set, with a standard embedding into $H_{i}$, followed by the inclusion into $\mathbb{R}^{n}$. The space $Imm_{st}(\vec{\mathcal{U}}\,;\,\mathbb{R}^{n})$ is defined in the same way. The functor so obtained is context-free and the natural transformation induced by the inclusion \vspace{3pt}
$$
\mathcal{L}(-)\longrightarrow \overline{Emb}_{st}(-\,;\,\mathbb{R}^{n})\vspace{3pt}
$$ 
is a weak equivalence of $k$-fold infinitesimal bimodules. Similarly to \cite[Theorem 5.10]{Arone14}, one has a weak equivalence of $k$-fold infinitesimal bimodules between $\overline{Emb}_{st}(-\,;\,\mathbb{R}^{n})$ and $\mathcal{R}_{n}^{k}$. Thus finishes the proof of the first part of the theorem. The second part is a consequence of Example \ref{Q3}.
\end{proof}


\section{The model category of $k$-fold bimodules}\label{N3}

In \cite{Ducoulombier18}, we prove that the derived mapping space of $1$-fold infinitesimal bimodules over the little cubes operad $\mathcal{C}_{d}$ is weakly equivalent to an explicit $d$-iterated loop space using mapping space of bimodules. In order to get a similar statement in the context of $k$-fold infinitesimal bimodules, we introduce in this section the notion of $k$-fold bimodule over a family of reduced operads $O_{1},\ldots,O_{k}$ relative to another operad $O$. Similarly to Section \ref{B5}, we consider two model category structures and we build explicit cofibrant resolutions in both cases.

\subsection{The category of "truncated" $k$-fold bimodules}

We consider a particular set of cardinality $-1$ denoted by $+$ and called the \textit{augmented set}. Let $\Sigma^{\times k}_{+}$ be the category of families $(A_{1},\ldots, A_{k})$ of $k$ finite sets (possibly augmented) different to $(+,\ldots, +)$ and families of isomorplisms between them. A $k$-fold augmented sequence is a  functor $M$ from $\Sigma^{\times k}_{+}$ to spaces. We will write $M(n_{1},\ldots,n_{k})$, with $n_{i}\in \mathbb{N}\sqcup \{+\}$ and $(n_{1},\ldots,n_{k})\neq (+,\ldots,+)$, for the space $M(A_{1},\ldots,A_{k})$ where \vspace{5pt}
$$
A_{i}=\left\{
\begin{array}{cc}\vspace{5pt}
\{1,\ldots,n_{i}\} & \text{if } n_{i}\neq +, \\ 
+ & \text{if } n_{i}= +.
\end{array} 
\right.\vspace{7pt}
$$

In practice, a $k$-fold augmented sequence is given by a family of topological spaces $M(n_{1},\ldots,n_{k})$, with $n_{i}\in \mathbb{N}\sqcup\{+\}$ and $(n_{1},\ldots,n_{k})\neq (+,\ldots,+)$, together with actions of the symmetric groups: for each element $\sigma\in \Sigma_{n_{1}}\times\cdots\times \Sigma_{n_{k}}$, there is a continuous map \vspace{4pt}
$$
\begin{array}{rcl}\vspace{3pt}
\sigma^{\ast}:M(n_{1},\ldots,n_{k}) & \longrightarrow & M(n_{1},\ldots,n_{k}); \\ 
 x & \longmapsto & x\cdot \sigma,
\end{array} \vspace{3pt}
$$
satisfying some relations. By convention, we assume that $\Sigma_{+}=\ast$ and we denote by $Seq_{k}^{+}$ the category of $k$-fold augmented sequences. \vspace{7pt}

Given an element $\vec{r}=(r_{1},\ldots,r_{k})\in \mathbb{N}^{k}$, we also consider the category $T_{\vec{r}}\,\Sigma^{\times k}_{+}$ whose objects are families of finite sets, possibly augmented, $(A_{1},\ldots,A_{k})$ with $|A_{i}|\leq r_{i}$. The category of $\vec{r}$-truncated $k$-fold augmented sequences, denoted by $T_{\vec{r}}\,Seq_{k}^{+}$, is formed by covariant functors from $T_{\vec{r}}\,\Sigma^{\times k}_{+}$ to spaces. Furthermore, there is an obvious functor\vspace{5pt}
$$
T_{\vec{r}}\,(-):Seq_{k}^{+}\longrightarrow T_{\vec{r}}\,Seq_{k}^{+}.\vspace{10pt}
$$

\begin{defi}\label{Z3}\textbf{The family of spaces $\vec{O}_{S}$}

\noindent Let $f_{i}:O_{i}\rightarrow O$, with $i\in \{1,\ldots,k\}$, be a family of operadic maps between reduced operads. The family of operads $O_{1},\ldots,O_{k}$ is said to be relative to the operad $O$. Let $\mathcal{P}_{+}(A)$ be the poset of subsets of $A$ plus the augmented set $+$. We also consider the set $\mathcal{P}_{k}(A)$ formed by families $S=\{S_{1},\ldots,S_{k}\}$, with $S_{i}\in \mathcal{P}_{+}(A)$ such that, for each $a\in A$, one of the following condition is satisfied: \vspace{5pt}
\begin{itemize}
\item[$i)$] $\forall i\in \{1,\ldots,k\}$ the element $a$ is contained in $S_{i}$;\vspace{5pt}
\item[$ii)$] there exists a unique $i\in \{1,\ldots,k\}$ such that  $a\in S_{i}$.\vspace{5pt}
\end{itemize}
For any element $S\in \mathcal{P}_{k}(A)$, we introduce the space\vspace{5pt}
\begin{equation}\label{D8}
\vec{O}_{S}(S_{1},\ldots,S_{k})\coloneqq\left\{\left.
\{x_{l}\}\in \underset{\substack{1\leq i\leq k\\S_{i}\neq +}}{\displaystyle \prod} O_{i}(S_{i})\,\,\right| \,\, \forall S_{i}\neq + \text{ and } S_{j}\neq +,\,\, f_{i}[S_{i}\,;\,S_{j}](x_{i})=f_{j}[S_{i}\,;\,S_{j}](x_{j})\right\},\vspace{5pt}
\end{equation}
where the composite map 
$
f_{i}[S_{i}\,;\,S_{j}]:O_{i}(S_{i})\rightarrow O_{i}(S_{i}\cap S_{j})\rightarrow O(S_{i}\cap S_{j})
$ 
is obtained by composing all inputs other than $S_{i}\cap S_{j}$ with the unique point in $O_{i}(0)$ followed by the operadic map $f_{i}$. By abuse of notation, we denote by $(x_{1},\ldots,x_{k})$ a point in the space (\ref{D8}) with $x_{i}=+$ if $S_{i}=+$.  \newpage

Let $a\in A$, $S^{1}=(S_{1}^{1},\ldots, S_{k}^{1})\in \mathcal{P}_{k}(A)$ and $S^{2}=(S_{1}^{2},\ldots, S_{k}^{2})\in \mathcal{P}_{k}(B)$ with the condition $S_{i}^{2}=+$ iff $a\notin S_{i}^{1}$. The family (\ref{D8}) inherits an algebraic structure from the operads $O_{1},\ldots, O_{k}$ in the sense that one has associative operations of the form\vspace{5pt}
\begin{equation}\label{D6}
\begin{array}{rclcl}\vspace{7pt}
\mu_{a}:\vec{O}_{S^{1}}(S_{1}^{1},\ldots, S_{k}^{1}) & \times & \vec{O}_{S^{2}}(S_{1}^{2},\ldots, S_{k}^{2}) & \longrightarrow & \vec{O}_{S^{1}\cup_{a}S^{2}}(S_{1}^{1}\cup_{a}S_{1}^{2},\ldots, S_{k}^{1}\cup_{a}S_{k}^{2}); \\ 
(x_{1}^{1},\ldots,x_{k}^{1}) & ; & (x_{1}^{2},\ldots,x_{k}^{2}) & \longmapsto & (x_{1}^{3},\ldots,x_{k}^{3}),
\end{array} \vspace{5pt}
\end{equation}
where the points $x_{i}^{3}\in O_{i}(S_{i}^{1}\cup_{a}S_{i}^{2})$ and the family $S^{1}\cup_{a}S^{2}\in \mathcal{P}_{k}(A\cup_{a}B)$ are defined as follows:\vspace{5pt}
$$
S_{i}^{1}\cup_{a}S_{i}^{2}=\left\{
\begin{array}{cc}\vspace{7pt}
S_{i}^{1} & \text{if } a\notin S_{i}^{1}, \\ 
(S_{i}^{1}\sqcup S_{i}^{2})\setminus \{a\} & \text{if } a\in S_{i}^{1},
\end{array} 
\right. \hspace{15pt} \text{and}\hspace{15pt}
x_{i}^{3}=\left\{
\begin{array}{cc}\vspace{7pt}
x_{i}^{1} & \text{if } a\notin S_{i}^{1}, \\ 
x_{i}^{1}\circ_{a}x_{i}^{2} & \text{if } a\in S_{i}^{1}.
\end{array} 
\right. \vspace{5pt}
$$
For any $S=\{S_{1},\ldots,S_{k}\}$, with $S_{i}\in \mathcal{P}_{+}(\{1,\ldots,n\})$, we will write $\vec{O}_{S}(n_{1},\ldots, n_{k})$ for the space $\vec{O}_{S}(S_{1},\ldots, S_{k})$ with \vspace{5pt}
$$
n_{i}=\left\{
\begin{array}{cc}\vspace{5pt}
|S_{i}| & \text{if } S_{i}\neq +, \\ 
+ & \text{if } S_{i}=+.
\end{array} \right.\vspace{10pt}
$$
\end{defi}

\begin{rmk}\label{D2}
The family $S=\{S_{1},\ldots,S_{k}\}\in \mathcal{P}_{k}(\{1,\ldots,n\})$ in the above definition can be interpreted as a $n$-coralla $T_{n}$ together with a family of applications $f_{i}:E(T_{n})\rightarrow \{external \,;\,internal\}$, with $1\leq i\leq k$, labelling the edges of the corolla as follows: if $S_{i}=+$, then the edges are indexed by $external$ ; 
if $S_{i}\neq +$, then the trunk is indexed by $internal$ and, according to the planar order of the corolla, the subset of leaves, whose images by the application $f_{i}$ are $internal$, is in bijection with the set $S_{i}$. By convention the edges indexed by $external$ are represented by dotted edges.  \vspace{0pt}

\hspace{-92pt}\includegraphics[scale=0.44]{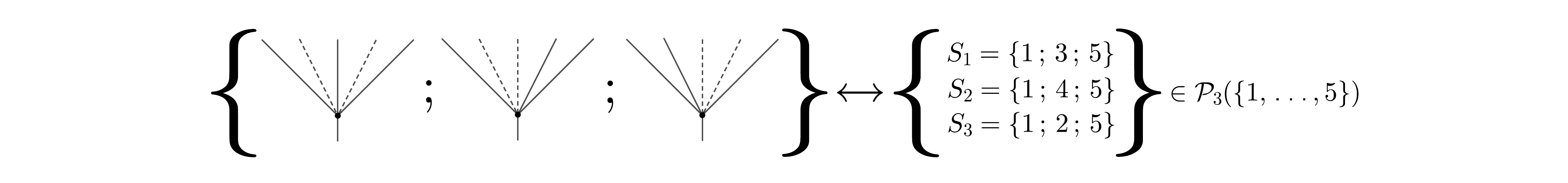}
\end{rmk}

\begin{defi} \textbf{The category of $k$-fold bimodules over $\vec{O}$}

\noindent A $k$-fold bimodule over $\vec{O}$, or just $\vec{O}$-\textit{bimodule}, is a $k$-fold augmented sequence $M$ together with operations called $k$-fold right operations of the form\vspace{5pt}
\begin{equation*}
\begin{array}{ll}
\circ_{i}^{a}:M(A_{1},\ldots,A_{k})\times O_{i}(B) 
\longrightarrow M(A_{1},\ldots,A_{i}\cup_{a}B,\ldots, A_{k}), & \text{with } 1\leq i\leq k \text{ and } a\in A_{i}\neq +,
\end{array} \vspace{5pt}
\end{equation*}
and, for any $S\in \mathcal{P}_{k}(A)$, $k$-fold left operations of the form\vspace{5pt}
\begin{equation*}
\begin{array}{ll}
\mu_{S}:\vec{O}_{S}(S_{1},\ldots,S_{k})\times \,\,\,
\underset{a\in A}{\displaystyle\prod}\,\,\, M(B_{1}^{a},\ldots,B_{k}^{a}) \longrightarrow M\left(\,\,\, \underset{\substack{a\in A\\ B_{1}^{a}\neq +}}{\displaystyle\coprod}B_{1}^{a},\ldots ,  \underset{\substack{a\in A\\ B_{k}^{a}\neq +}}{\displaystyle\coprod}B_{k}^{a} \,\,\, \right),
 & \text{with } B_{i}^{a}=+ \text{ iff } a\notin S_{i}.
\end{array} \vspace{5pt}
\end{equation*}
By convention, if $S_{i}=+$, then the coproduct $\coprod_{a}B_{i}^{a}$ is the augmented set $+$. These operations satisfy compatibility relations with the symmetric group as well as associativity and unit axioms (see Appendix \ref{W1}). Due to the condition $(A_{1},\ldots, A_{k})\neq (+,\ldots,+)$, the notion of $1$-fold bimodule is equivalent to the usual notion of bimodule over an operad. A map of $k$-fold bimodules should respect the operations. We denote by $Bimod_{\vec{O}}$ the category of $k$-fold bimodules. Finally, let us notice that a $k$-fold bimodule $M$ is equipped with maps of the form \vspace{7pt}
$$
\gamma_{\vec{A}}:\underset{\substack{1\leq i\leq k\\ A_{i}\neq +}}{\prod}O_{1}(\emptyset)\rightarrow M(\vec{A}),\hspace{15pt}\forall \vec{A}=(A_{1},\ldots,A_{k})\in \Sigma^{\times k}_{+}\text{ with } A_{i}\in \{+\,,\,\emptyset\}.\vspace{-20pt}
$$

\newpage

In practice, a $k$-fold bimodule over $\vec{O}$ is determined by a family of topological spaces $M(n_{1},\ldots,n_{k})$, with $n_{i}\in \mathbb{N}\sqcup\{+\}$ and $(n_{1},\ldots,n_{k})\neq (+,\ldots,+)$, together with actions of the symmetric groups $\Sigma_{n_{1}}\times \cdots \times \Sigma_{n_{k}}$, $k$-fold right operations of the form\vspace{1pt}
$$
\begin{array}{ll}
\circ_{i}^{j}:M(n_{1},\ldots,n_{k})\times O_{i}(m)\rightarrow M(n_{1},\ldots, n_{i}+m-1,\ldots n_{k}), & \text{with } i\leq k, \, n_{i}\neq + \text{ and } j\leq n_{i},
\end{array} \vspace{3pt}
$$
and, for any $S=(S_{1},\ldots,S_{k})\in\mathcal{P}_{k}(\{1,\ldots,n\})$, $k$-fold left operations of the form \vspace{5pt}
$$
\begin{array}{ll}
\mu_{S}:\vec{O}_{S}(n_{1},\ldots,n_{k})\times \underset{1\leq i\leq n}{\displaystyle\prod} M(m_{1}^{i},\ldots,m_{k}^{i})\rightarrow M\left(\,\,\, \underset{\substack{1\leq i\leq k\\ m_{1}^{i}\neq +}}{\displaystyle\sum}m_{1}^{i},\ldots,\underset{\substack{1\leq i\leq k\\ m_{k}^{i}\neq +}}{\displaystyle\sum}m_{k}^{i}\,\,\,\right), & \text{with } m_{j}^{i}=+ \text{ iff } j\notin S_{i}. \vspace{5pt}
\end{array} 
$$
By convention, if $S_{j}=+$, then one has $\sum_{i}m_{j}^{i}=+$. For the rest of the paper, we use also the  notation
$$
\begin{array}{ll}\vspace{7pt}
x\circ_{i}^{j}y=\circ_{i}^{j}(x\,;\,y) & \text{for } x\in M(n_{1},\ldots,n_{k}) \,\text{ and } y\in O_{i}(m).
\end{array}\vspace{-2pt}
$$

Given an element $\vec{r}=(r_{1},\ldots,r_{k})\in \mathbb{N}^{k}$, we also consider the category of $\vec{r}$-truncated $k$-fold  bimodule over $\vec{O}$ denoted by $T_{\vec{r}}\,Bimod_{\vec{O}}$. An object is an $\vec{r}$-truncated $k$-fold augmented sequence together with operations as above. One has an obvious functor\vspace{2pt}
$$
T_{\vec{r}}\,(-):Bimod_{\vec{O}}\longrightarrow T_{\vec{r}}\,Bimod_{\vec{O}}.
$$ 
\end{defi}\vspace{5pt}

\begin{expl}\label{M1}\textbf{The $k$-fold bimodule $\mathbb{O}^{+}$}

\noindent Let $O_{1},\ldots,O_{k}$ be a family of reduced operads relative to $O$. Similarly to Example \ref{E1}, the $k$-fold augmented sequence $\mathbb{O}^{+}$ defined as follows:\vspace{5pt}
\begin{equation}\label{Q9}
\mathbb{O}^{+}(A_{1},\ldots , A_{k}):= \underset{\substack{1\leq i \leq k\\ A_{i}\neq +}}{\displaystyle \prod} O(A_{i}),\hspace{15pt} \forall (A_{1},\ldots,A_{k})\in \Sigma_{+}^{\times k},\vspace{3pt}
\end{equation}
inherits $k$-fold right operations from the operadic structures of $O_{1},\ldots,O_{k}$. For any $S=(S_{1},\ldots,S_{k})\in \mathcal{P}_{k}(A)$ and $S_{i}=\{a_{1}^{i},\ldots , a_{l_{i}}^{i}\}$, the $k$-fold left operations are given by the formula \vspace{3pt}
$$
\begin{array}{rcl}\vspace{5pt}
\mu_{S}:\vec{O}_{S}(S_{1},\ldots,B_{k})\times \underset{a\in A}{\displaystyle\prod} \,\,\,\mathbb{O}^{+}(B_{1}^{a},\ldots,B_{k}^{a}) & \longrightarrow & \mathbb{O}^{+}\left(\,\,\, \underset{\substack{a\in A\\ B_{1}^{a}\neq +}}{\displaystyle\coprod} B_{1}^{a},\ldots, \underset{\substack{a\in A\\ B_{k}^{a}\neq +}}{\displaystyle\coprod} B_{k}^{a}\,\,\,\right); \\ 
(x_{1},\ldots,x_{k})\,;\, \big\{(y_{1}^{a},\ldots,y_{k}^{a})\big\}_{a\in A} & \longmapsto & (z_{1},\ldots,z_{k}),
\end{array} \vspace{3pt}
$$
with $z_{i}$ defined using the operadic structures of $O_{1},\ldots,O_{k}$:
$$
z_{i}=\big(\,\cdots\big(\, x_{i}\circ_{a_{1}^{i}}y_{i}^{a_{1}^{i}}\,\big)\circ_{a_{2}^{i}}y_{i}^{a_{2}^{i}}\cdots\,\big)\circ_{a_{l_{i}}^{i}}y_{i}^{a_{l_{i}}^{i}}.\vspace{5pt}
$$

More generally, if $M_{1},\ldots,M_{k}$ are bimodules over the reduced operads $O_{1},\ldots,O_{k}$, respectively, then the $k$-fold augmented sequence $\mathbb{M}^{+}$, given by the formula \vspace{5pt}
$$
\mathbb{M}^{+}(A_{1},\ldots,A_{k}):= \underset{\substack{1\leq i \leq k\\ A_{i}\neq +}}{\displaystyle \prod} M(A_{i}), \hspace{15pt} \forall (A_{1},\ldots,A_{k})\in \Sigma_{+}^{\times k},
$$ inherits a $k$-fold bimodule structure over the family of reduced operads $O_{1},\ldots,O_{k}$ relative to the terminal operad (i.e. the operad with only one point in each arity). Consequently, $\mathbb{M}^{+}$ is also a $k$-fold bimodule over the family $O_{1},\ldots,O_{k}$ relative to any other operad $O$.
\end{expl}

\begin{notat}
 Since $\Sigma^{\times k}$ is a subcategory of $\Sigma^{\times k}_{+}$, there is a functor\vspace{3pt}
$$(-)^{-}:Seq_{k}^{+}\rightarrow Seq_{k}.$$ 
\end{notat}\vspace{2pt}

\begin{pro}\label{D9}
Let $\eta:\mathbb{O}^{+}\rightarrow M$ be a map of $k$-fold bimodules over $\vec{O}$. Then $M^{-}$ inherits a $k$-fold infinitesimal bimodule structure over the family of reduced operad $O_{1},\ldots,O_{k}$ relative to the topological monoid $O(1)$ and the induced map $\eta:\mathbb{O}\rightarrow M^{-}$ is a map of $k$-fold infinitesimal bimodules.\vspace{3pt}
\end{pro}

\begin{proof}
Since the $k$-fold right infinitesimal operations and the $k$-fold right operations are the same, we only need to build maps of the form \vspace{3pt}
$$
\mu:\vec{O}(n_{1},\ldots,n_{k})\times M^{-}(m_{1},\ldots,m_{k})\longrightarrow M^{-}(n_{1}+m_{1}-1,\ldots,n_{k}+m_{k}-1).\vspace{5pt}
$$
For this purpose, we consider the family  $S=\{S_{1},\ldots,S_{k}\}\in\mathcal{P}_{k}(\{1,\ldots,n_{1}+\cdots+n_{k}-k+1\})$ with \vspace{3pt}
$$
S_{i}=\{1,n_{1}+\cdots+n_{i-1}-i+3,\ldots, n_{1}+\cdots+n_{i}-i+1\}.\vspace{3pt}
$$
By construction, one has an operation of the form\vspace{3pt}
$$
\mu_{S}: \vec{O}_{S}(n_{1},\ldots,n_{k})\times  M^{-}(m_{1},\ldots,m_{k})\times\!\!\underset{j=2}{\overset{n_{1}+\cdots+n_{k}-k+1}{\prod}}\!\!M(m_{1}^{j},\ldots,m_{k}^{j}) \longrightarrow M^{-}(n_{1}+m_{1}-1,\ldots,n_{k}+m_{k}-1),
$$
where $(m_{1}^{j},\ldots,m_{k}^{j})$, with $m_{i}^{j}\in \mathbb{N}\sqcup \{+\}$, is given by 
$$
m^{j}_{i}=\left\{
\begin{array}{cc}\vspace{5pt}
1 & \text{if } j\in S_{i}, \\ 
+ & \text{if } j\notin S_{i}.
\end{array} 
\right.
$$
By construction, one has $(m_{1}^{j},\ldots,m_{k}^{j})\neq (+,\ldots,+)$. Then, we use the image of the units $\ast_{1}\in O_{i}(1)$  by the map $\eta:\mathbb{O}^{+}\rightarrow M$ in order to get the operation researched. The reader can check that the induced map $\eta:\mathbb{O}\rightarrow M^{-}$ is a morphism of $k$-fold infinitesimal bimodules. \vspace{5pt}
\end{proof}

\begin{expl}\label{L7}\textbf{The $k$-fold bimodule $\mathcal{R}_{n}^{k}$}

\noindent Let $d_{1}\leq \cdots \leq d_{k}<n$ be integers and let $\vec{O}$ be the object (\ref{D8}) associated to the family of reduced operads $\mathcal{C}_{d_{1}},\ldots, \mathcal{C}_{d_{k}}$ relative to the operad $\mathcal{C}_{n}$. By abuse of notation, we denote by  $\mathcal{R}_{n}^{k}$ the $k$-fold augmented sequence given by the formula\vspace{3pt}
$$
\mathcal{R}_{n}^{k}(A_{1},\ldots,A_{k})=\mathcal{R}_{n}\left(\,\,\, \underset{\substack{1\leq i\leq k\\ A_{i}\neq +}}{\displaystyle\coprod}A_{i}\,\,\,\right),\hspace{15pt} \forall (A_{1},\ldots,A_{k})\in \Sigma^{\times k}_{+}.\vspace{5pt}
$$ 
The $k$-fold right operations are defined using the composite map of operads $\kappa_{i}:\mathcal{C}_{d_{i}}\rightarrow \mathcal{C}_{d_{k}}\hookrightarrow \mathcal{R}_{d_{k}}\rightarrow \mathcal{R}_{n}$ together with the operadic structure of $\mathcal{R}_{n}$. Let $S=\{S_{1},\ldots,S_{k}\}$ be an element in $\mathcal{P}_{k}(A)$ with $A=\{a_{1},\ldots,a_{n}\}$. In a first time, we assume that $S_{i}\neq +$, for any $i\in \{1,\ldots,k\}$. We want to define the map\vspace{5pt}
\begin{equation}\label{Z5}
\mu_{S}: \vec{O}_{S}(S_{1},\ldots,S_{k})\times \underset{a\in A}{\displaystyle \prod} \mathcal{R}_{n}^{k}(B_{1}^{a},\ldots, B_{k}^{a})\longrightarrow \mathcal{R}_{n}^{k}\left(\,\,\, \underset{\substack{a\in A\\ B_{1}^{a}\neq +}}{\displaystyle\coprod}B_{1}^{a},\ldots,\underset{\substack{a\in A\\ B_{k}^{a}\neq +}}{\displaystyle\coprod}B_{k}^{a}\,\,\,\right),\hspace{15pt}\text{with } B_{i}^{a}=\emptyset \text{ iff } a\notin S_{i}.\vspace{5pt}
\end{equation}
For this purpose, we need a map of the form\vspace{5pt}
$$
\varepsilon: \vec{O}_{S}(S_{1},\ldots,S_{k})\longrightarrow \mathcal{R}_{n}\left(\,\,\, \underset{1\leq i \leq k}{\coprod}S_{i}\,\,\,\right) \longrightarrow \mathcal{R}_{n}\left( A \right). \vspace{-15pt}
$$
\newpage

\noindent The first map sends a point $(x_{1},\ldots,x_{k})$ to the operadic composition $(\cdots ( c_{k}\circ_{k}\kappa_{k}(x_{k}))\cdots ) \circ_{1}\kappa_{1}(x_{1})$ where the element $c_{k}\in \mathcal{C}_{n}(k)$ is introduced in Example \ref{D7}. Due to the definition $(\ref{D8})$, two little rectangles indexed by the same element $a$ in configurations $x_{i}$ and $x_{i+1}$, respectively, share a face of codimension $1$ in the operadic composition. Consequently, the second map consists in gluing together such pair of rectangles as illustrated in Figure \ref{Z4}. \vspace{-5pt} 
\begin{figure}[!h]
\begin{center}
\includegraphics[scale=0.32]{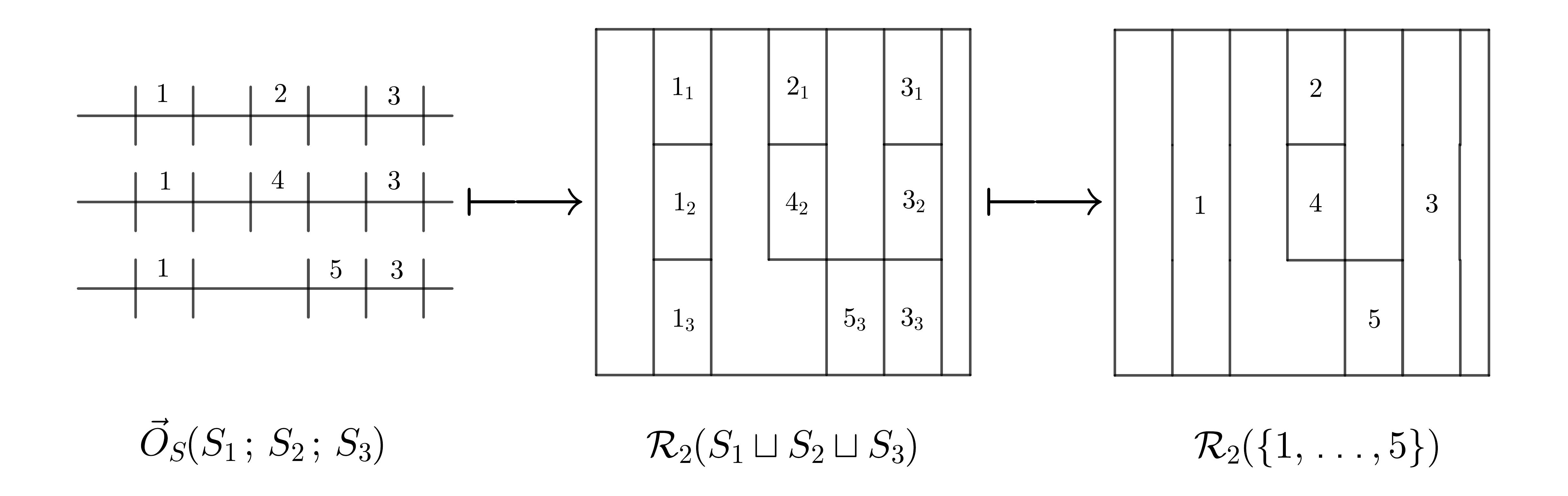}\vspace{-10pt}
\caption{Illustration of $\varepsilon$ with $S_{1}=\{1\,;\,2\,;\,3\}$, $S_{2}=\{1\,;\,3\,;\,4\}$ and $S_{3}=\{1\,;\,3\,;\,5\}$ in $\mathcal{P}_{+}(\{1,\ldots,5\})$.}\label{Z4}\vspace{-7pt}
\end{center}
\end{figure}

\noindent Finally, the map $(\ref{Z5})$ is defined using the operadic structure of $\mathcal{R}_{n}$ as follows:\vspace{5pt}
$$
\mu_{S}\big( (x_{1},\ldots,x_{k})\,;\,\{y_{a}\}_{a\in A} \big)=\big(\cdots \big(\varepsilon(x_{1},\ldots,x_{k})\circ_{a_{1}}y_{a_{1}}\big)\cdots \big)\circ_{a_{n}}y_{a_{n}}.\vspace{5pt}
$$
For instance, the $k$-fold left operation applied to the points \vspace{-1pt}
\begin{figure}[!h]
\begin{center}
\includegraphics[scale=0.32]{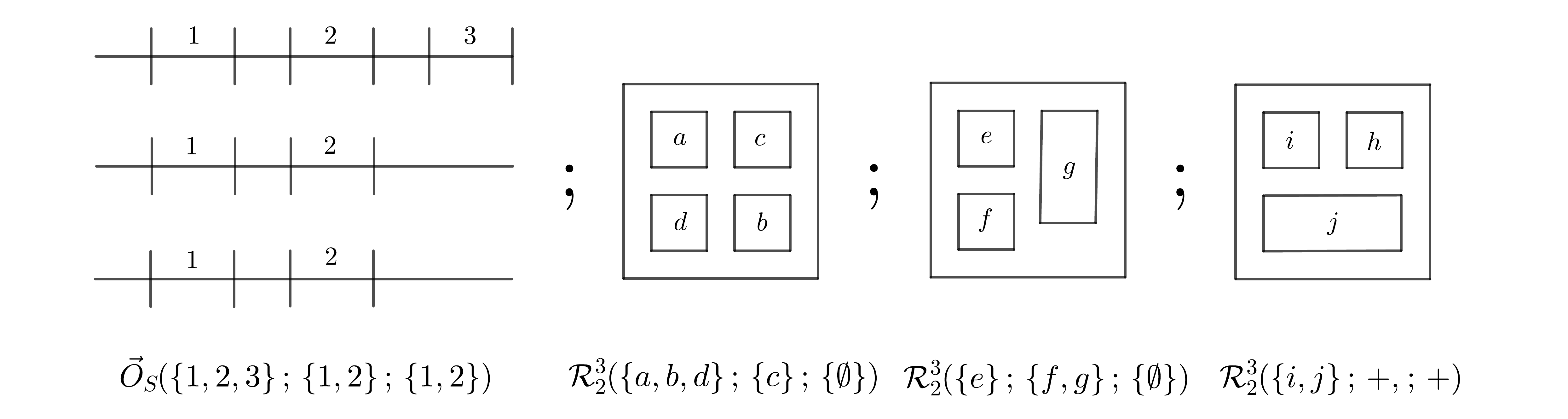}
\end{center}\vspace{-10pt}
\end{figure}

\noindent gives rise to the following element:\vspace{-5pt}
\begin{figure}[!h]
\begin{center}
\includegraphics[scale=0.32]{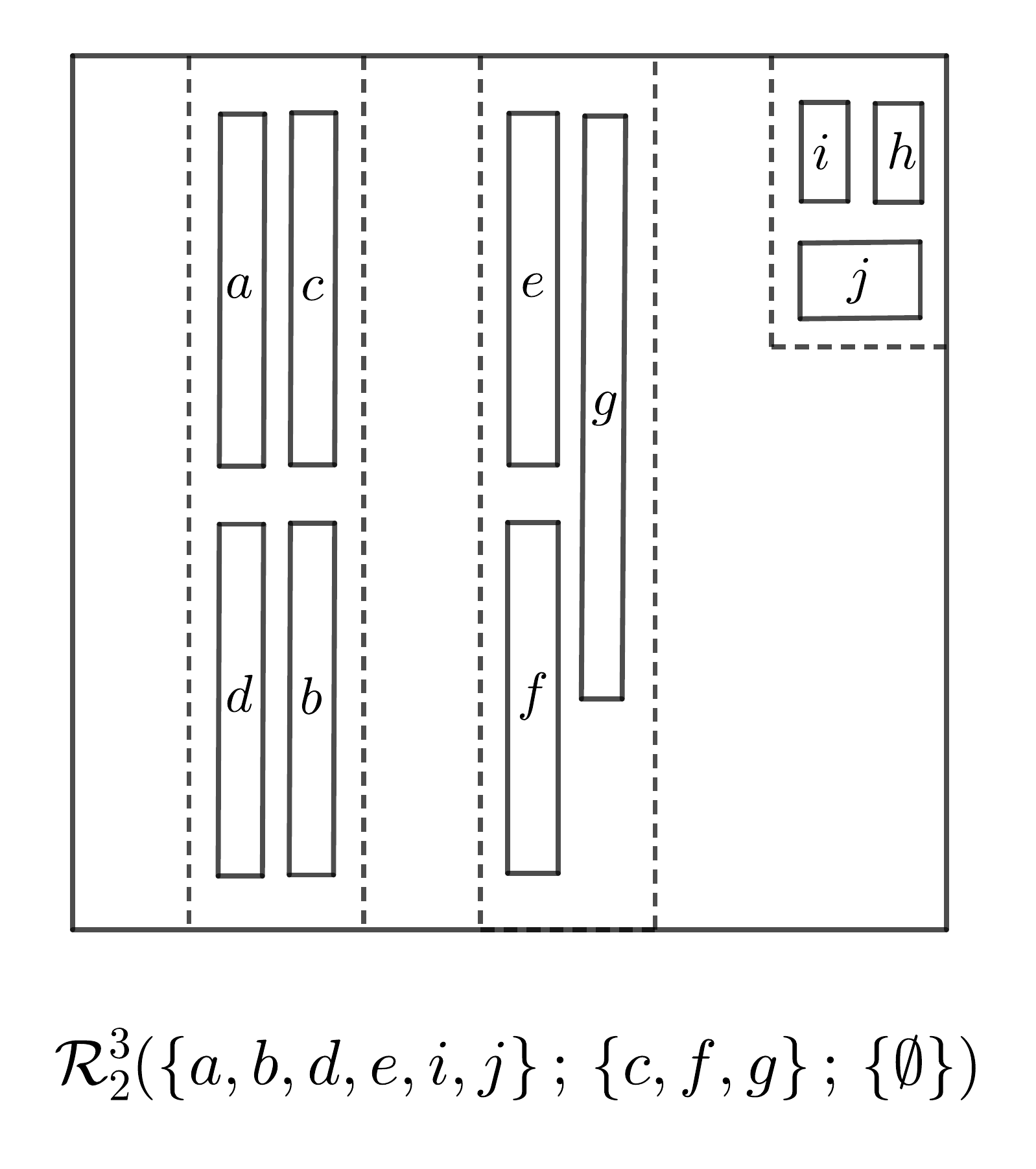}
\end{center}\vspace{-40pt}
\end{figure}

\newpage

If $S_{i}=+$ for some $i\in \{1,\ldots,k\}$, then the construction of the left operations is similar. In that case, one has to consider the element $c_{l}\in \mathcal{C}_{n}(l)$ instead of $c_{k}$ in the definition of $\varepsilon$ where $l$ is the number of sets in $(A_{1},\ldots,A_{k})$ different to the augmented set $+$. We recall that $c_{l}$ is the family of rectangles splitting the unit cube into $l$ equal rectangles along the last coordinate.  Furthermore,  there is a map of $k$-fold bimodules \vspace{3pt}
$$
\begin{array}{rcl}\vspace{5pt}
\eta:\mathbb{O}^{+}(A_{1},\ldots,A_{k}) & \longrightarrow &  \mathcal{R}_{n}^{k}(A_{1},\ldots,A_{k}); \\ 
(x_{1},\ldots, x_{k}) & \longmapsto & c_{l}(\kappa_{i_{1}}(x_{i_{1}}),\ldots,\kappa_{i_{l}}(x_{i_{l}})),
\end{array} \vspace{3pt}
$$  
where $\{i_{1}<\ldots <i_{l}\}$ is the subset of $\{1,\ldots,k\}$ such that $A_{i_{j}}\neq +$ for $j\in \{1,\ldots,l\}$. The $k$-fold infinitesimal bimodule structure induced by the maps $\eta$ and Proposition \ref{D9} coincides with the structure introduced in Example \ref{D7}.\vspace{9pt}
\end{expl}

\begin{rmk}\label{Z2} \textbf{A $k$-fold bimodule as a $1$-fold bimodule over a colored operad}

\noindent From the family of reduced operads $O_{1},\ldots,O_{k}$ relative to the operad $O$, we consider the colored operad with set of colors $S=\{c_{1},\ldots,c_{k+1}\}$ defined as follows:   \vspace{5pt}
$$
\overline{O}(A_{c_{1}},\ldots,A_{c_{k+1}}\,;\,c)=\left\{
\begin{array}{cl}\vspace{5pt}
O_{i}(A_{c_{i}}) & \text{if } c=c_{i} \text{ for } i\leq k \text{ and } A_{c_{j}}=\emptyset \text{ for } j\neq i, \\ \vspace{9pt}
\vec{O}_{S}(S_{1},\ldots,S_{k}) & \text{if } c=c_{k+1}, \\ 
\emptyset & \text{otherwise},
\end{array} 
\right.\vspace{5pt}
$$
where $S=(S_{1},\ldots,S_{k})\in \mathcal{P}_{k}\left(\underset{1\leq i \leq k+1}{\displaystyle\bigsqcup}A_{c_{i}}\right)$ and $S_{i}=A_{c_{i}}\sqcup A_{c_{k+1}}$.\vspace{9pt}

\noindent Finally, there is an equivalence of categories \vspace{3pt}
$$
Bimod_{\vec{O}}\cong Bimod_{\overline{O}}\downarrow M_{\ast}\vspace{1pt}
$$
where $M_{\ast}$ is the following bimodule over the colored operad $\overline{O}$:\vspace{5pt}
$$
M_{\ast}(A_{c_{1}},\ldots,A_{c_{k+1}}\,;\,c)=\left\{
\begin{array}{cl}\vspace{7pt}
\ast & \text{if } A_{c_{k+1}}=\emptyset \text{ and } c=c_{k+1}, \\ 
\emptyset & \text{otherwise}.
\end{array} 
\right.\vspace{7pt}
$$
\end{rmk}

\subsection{The Reedy and projective model category structures on $Bimod_{\vec{O}}$}

The purpose of this section is to define a model category structure on the category of $k$-fold  bimodules over $\vec{O}$. More precisely, we introduce two different model category structures: projective and Reedy. Similarly to the $k$-fold infinitesimal bimodule case, both structures have advantages and inconveniences. In the following, we compare these two structures and we give the properties needed to prove the main results in Section \ref{N4}. 

\subsubsection{The projective model category structure} 

Similarly to Section \ref{C1}, in order to introduce model category structures on the category of $k$-fold bimodules, we build an adjunction\vspace{5pt}
\begin{equation}\label{C0}
\mathcal{F}_{B\,;\,\vec{O}}:Seq_{k}^{\emptyset}\leftrightarrows Bimod_{\vec{O}}:\mathcal{U},\vspace{5pt}
\end{equation}
where $\mathcal{U}$ is the forgetful functor and $Seq_{k}^{\emptyset}$ is the category of $k$-fold augmented sequences $M$ with based points in the spaces of the form $M(n_{1},\ldots,n_{k})$, with $n_{i}\in \{+\,;\,0\}$. As usual, the free $k$-fold bimodule can be described in terms of coproduct indexed by a set of trees.

\begin{defi}\label{C2}\textbf{The set of $k$-fold reduced trees with section}

\noindent A \textit{tree with section} is a pair $T=(T\,;\,V^{p}(T))$ where $T$ is a planar rooted tree and $V^{p}(T)$ is a subset of vertices, called \textit{pearls}, satisfying the following relation: each path from a leaf or univalent vertex to the root passes through a unique pearl. The set of pearls forms a section cutting the tree into two parts. We denote by $V^{u}(T)$ (resp. $V^{d}(T)$) the vertices above the section (resp. below the section). A tree with section is said to be \textit{reduced} if each inner edge is connected to a pearl. \vspace{-2pt}
\begin{figure}[!h]
\begin{center}
\includegraphics[scale=0.22]{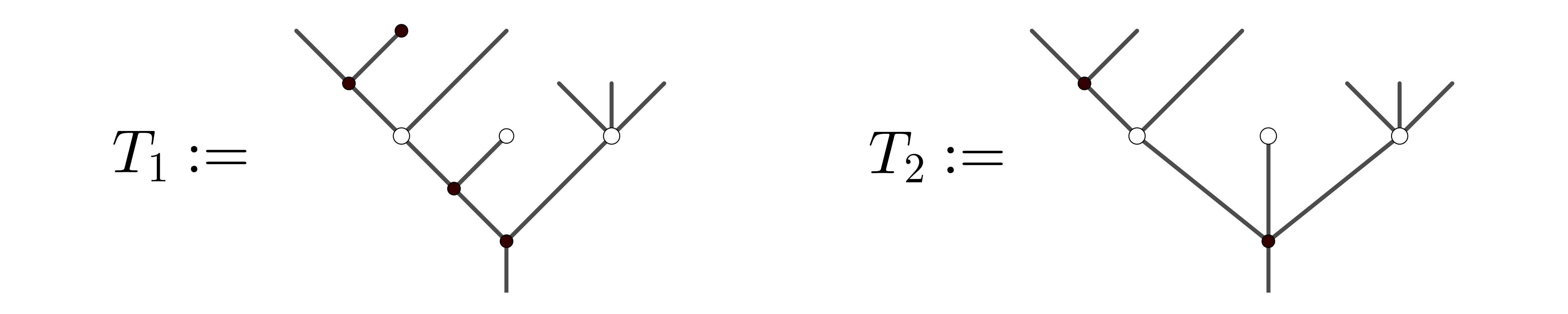}\vspace{-5pt}
\caption{Examples of a tree with section $T_{1}$ and a reduced tree with section $T_{2}$.}\vspace{-8pt}
\end{center}
\end{figure}

For $\vec{n}=(n_{1},\ldots,n_{k})$, with $n_{i}\in \mathbb{N}\sqcup \{+\}$ and $(n_{1},\ldots,n_{k})\neq (+,\ldots, +)$, the set of reduced $k$-fold trees with section $rsTree[\,\vec{n}\,]$ is formed by families $\vec{T}=(T_{1},\ldots,T_{k},f_{1},\ldots,f_{k},\sigma)$ where $T_{i}$ is a reduced tree with section having $n_{i}$ leaves (without leaves if $n_{i}=+$) and $\sigma\in \Sigma_{n_{1}}\times \cdots \times \Sigma_{n_{k}}$ is a permutation labelling the leaves. Furthermore, the sub-trees $T'_{1},\ldots,T'_{k}$ obtained from $T_{1},\ldots,T_{k}$ by removing vertices and edges above the sections are assumed to be the same. Finally, $f_{i}:E(T'_{i})\rightarrow \{internal \,;\, external\}$ is a function labelling the inner edges below the section of $T_{i}$ satisfying the following conditions:\vspace{5pt}
\begin{itemize}
\item[$\blacktriangleright$] a pearl in $T_{i}$ with an $external$ output edge is necessarily univalent,\vspace{5pt}
\item[$\blacktriangleright$] the tree $T_{i}$ is trivial in the sense that all the edges are $external$ if and only if $n_{i}=+$,\vspace{5pt}
\item[$\blacktriangleright$] if the output edge of a vertex below the section of $T_{i}$ is $external$, then its input edges are also $external$,\vspace{5pt}
\item[$\blacktriangleright$] if $e_{1}$ is an inner edge in $T_{1}'$ and $e_{2},\ldots, e_{k}$ are the corresponding edges in $T_{2}',\ldots,T_{k}'$, respectively, then $f_{i}(e_{i})=internal$ for all $i\in \{1,\ldots,k\}$, or there exists a unique $i\in \{1,\ldots,k\}$ such that $f_{i}(e_{i})=internal$ and $f_{j}(e_{j})=external$ for any $j\neq i$.\vspace{7pt}
\end{itemize}

Since the sub-trees $T'_{1},\ldots,T'_{k}$ are the same, if $v^{1}$  is a vertex  of $T_{1}$ below the section, then we denote by $v^{i}$  the corresponding vertex in $V(T_{i})$. As usual, if $v$ is a vertex, then we denote by $|v|$ the number of incoming edges. Nevertheless, if $v$ is a vertex below the section, then we denote by $|v|_{i}$ the number of incoming $internal$ edges. By convention, an $external$ edge is represented by a dotted line and $|p|=+$ or $|v|_{i}=+$ if the output edge of the pearl $p$ or the vertex below the section $v$ is labelled by $external$. 

\begin{figure}[!h]
\begin{center}
\includegraphics[scale=0.2]{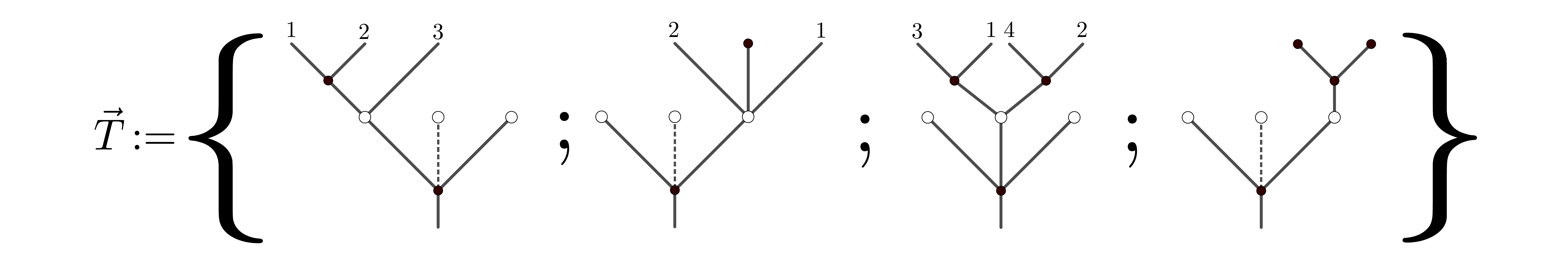}\vspace{-9pt}
\caption{Illustration of a point in $rsTree[3\,;\,2\,;\,4\,;\,0]$}\vspace{-5pt}
\end{center}
\end{figure}
\end{defi}

\begin{const}\label{F9}
Let $M=\{M(n_{1},\ldots,n_{k})\}\in Seq_{k}^{\emptyset}$. Then the free $k$-fold bimodule $\mathcal{F}_{B\,;\,\vec{O}}(M)$, also denoted by $\mathcal{F}_{B}(M)$ when the operads are understood, is defined as follows:\vspace{5pt}
\begin{equation*}
\mathcal{F}_{B}(M)(\,\vec{n}\,):=\left. \left(
\underset{\vec{T}\in rsTree[\vec{n}]}{\coprod}\hspace{5pt} \underset{v^{1}\in V^{d}(T_{1})}{\prod} \vec{O}_{S_{v^{1}}}(|v^{1}|_{i}, \cdots, |v^{k}|_{i})\times \underset{p^{1}\in V^{p}(T_{1})}{\prod} M(|p^{1}|,\cdots,|p^{k}|)\times \underset{\substack{i\in\{1,\ldots,k\}\\ v\in V^{u}(T_{i})}}{\prod} O_{i}(|v|) \right)\,\,
\right/\!\!\sim\vspace{3pt}
\end{equation*}
where $S_{v^{1}}=(S_{v^{1}}^{1},\ldots ,S_{v^{1}}^{k} )\in \mathcal{P}_{k}(\{1,\ldots,|v^{1}|\})$ is given by \vspace{5pt}
$$
S_{v^{1}}^{i}\coloneqq
\left\{
\begin{array}{ll}\vspace{7pt}
\text{the augmented set } + & \text{if the output edge of } v^{i} \text{ is } external, \\ 
\big\{ j\in \{1,\ldots,|v^{1}|\}\,\,\big|\,\, \text{the } j\text{-th incoming input of } v^{i} \text{ is } internal\big\} & \text{otherwise}.
\end{array} 
\right.\vspace{-10pt}
$$

\begin{figure}[!h]
\begin{center}
\includegraphics[scale=0.2]{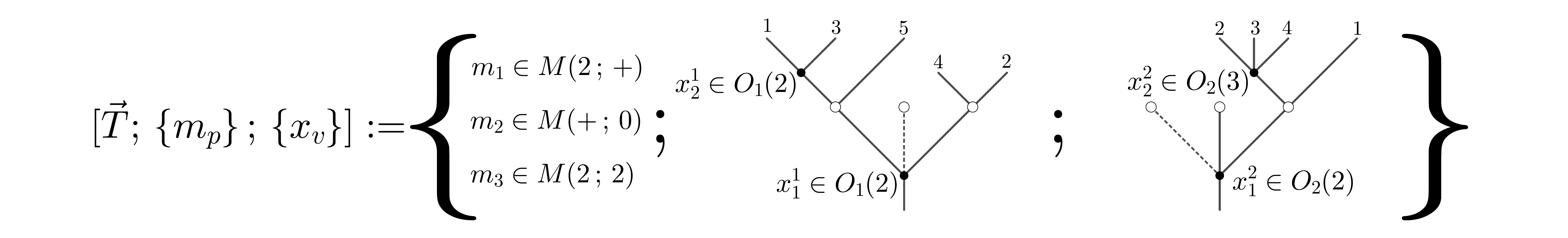}\vspace{-5pt}
\caption{Illustration of a point in $\mathcal{F}_{B}(M)(5\,;\,4)$.}\vspace{-8pt}
\end{center}
\end{figure}

The equivalence relation is generated by the unit axiom and  compatibility relations with the symmetric group action.  Furthermore, if $p^{1}\in V^{p}(T_{1})$ is indexed by the based point in $M(n_{1},\ldots,n_{k})$, with $n_{i}\in \{0\,;\,+\}$, then we contract the corresponding pearls using the operadic structures of $O_{1},\ldots,O_{k}$ as illustrated in Figure \ref{G3}. We denote by $[\vec{T}\,;\,\{m_{p}\}\,;\,\{x_{v}\}]$ a point in the free bimodule. Furthermore, the construction of the free $1$-fold bimodule so obtained is homeomorphic to the usual construction  introduced in \cite{Ducoulombier16}.\vspace{2pt}
\begin{figure}[!h]
\begin{center}
\includegraphics[scale=0.107]{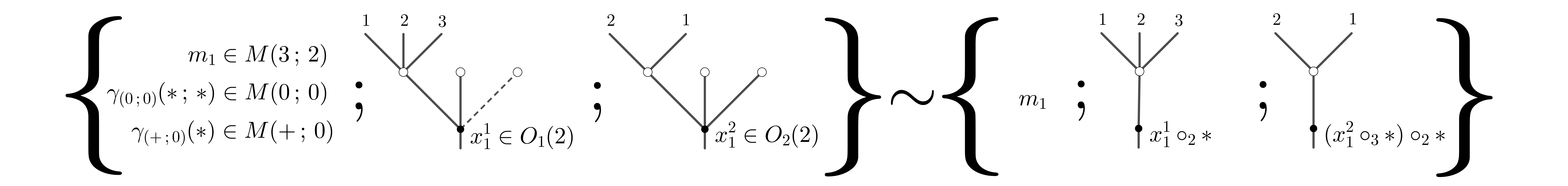}
\caption{Illustration of the equivalence relation.}\label{G3}\vspace{-13pt}
\end{center}
\end{figure}

The right operation $\circ_{i}^{j}$ with an element $x\in O_{i}(m)$ consists in grafting the $m$-corolla whose vertex is indexed by $x$ into the $i$-th leaf of the reduced tree with section $T_{i}$. If the element so obtained contains an inner edge joining two consecutive vertices other than a pearl, then we contract it using the operadic structure of $O_{i}$.\vspace{9pt}

Let $(x_{1},\ldots,x_{k})\in \vec{O}_{S}(n_{1},\ldots,n_{k})$, with $S=(S_{1},\ldots,S_{k})\in \mathcal{P}_{k}(\{1,\ldots,m\})$, and let $[\vec{T}^{i}\,;\,\{m_{p}^{i}\}\,;\,\{x_{v}^{i}\}]$ be a family of points in $\mathcal{F}_{B}(M)$. The left operation is defined as follows: each tree $T_{u}^{i}$, with $1\leq i\leq m$, is grafted from left to right to a leaf of the $m$-corolla whose vertex is indexed by $x_{u}$ and whose $l$-th leaf is $external$ if and only if $l\notin S_{u}$. If the element so obtained contains inner edges joining two consecutive vertices other than pearls, then we contract them using the operations (\ref{D6}). For instance, the left operation between the two points in $\mathcal{F}_{B}(M)(0\,;\,4)$ and $\mathcal{F}_{B}(M)(2\,;\,0)$, respectively,\vspace{9pt}

\hspace{-48pt}\includegraphics[scale=0.185]{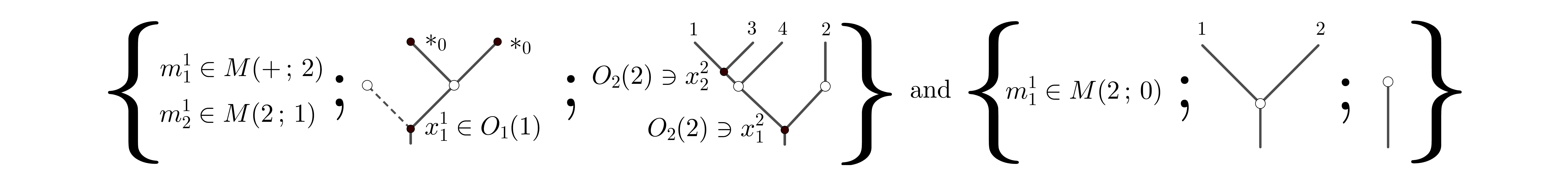}\vspace{9pt}

\noindent and a point $(x_{1}\,;\,x_{2})\in \vec{O}_{S}(2\,;\,2)$, with $S_{1}=S_{2}=\{1\,;\,2\}$, gives rise to\vspace{5pt}

\begin{center}
\includegraphics[scale=0.23]{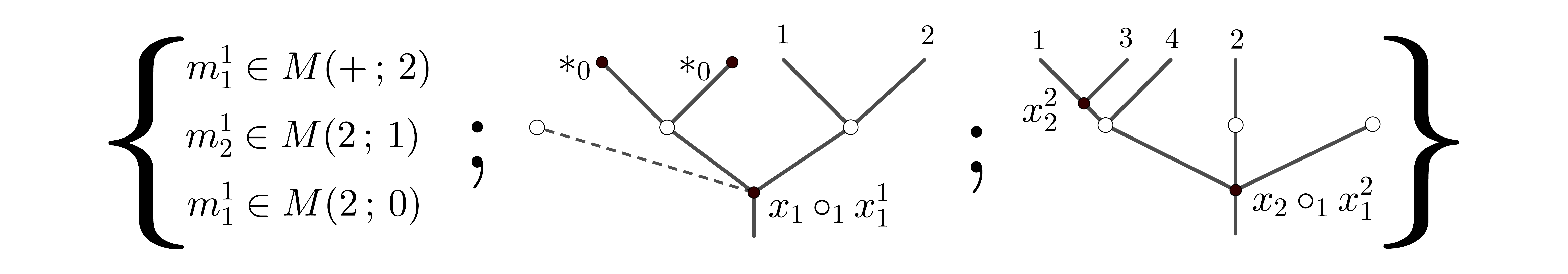}\vspace{-3pt}
\end{center}

\noindent Finally, since the operads $O_{1},\ldots,O_{k}$ are assumed to be reduced, the map \vspace{5pt}
$$
\gamma_{\vec{n}}:\underset{\substack{1\leq i\leq k\\ n_{i}\neq +}}{\prod}O_{i}(0) \longrightarrow \mathcal{F}_{B}(M)(n_{1},\ldots,n_{k}), \hspace{15pt} \text{with } n_{i}\in \{+\,;\,0\},\vspace{-3pt}
$$
sends the unique point in the product to the element $[\vec{T}\,;\,\{m_{p}\}\,;\,\{x_{v}\}]$ where $T_{i}$ is the pearled $0$-corolla with an $external$ output edge if and only if $n_{i}=+$. \vspace{5pt}
\end{const}

\begin{thm}\label{C5}
The pair of functors $(\mathcal{F}_{B}\,;\,\mathcal{U})$ forms an adjunction. Furthermore, the category of $k$-fold bimodules inherits a cofibrantly generated model category structure in which all the objects are fibrant and making the adjunction (\ref{C0}) into a Quillen adjunction. More precisely, a map $f$ is a weak equivalence (resp. a fibration) if the induced map $\mathcal{U}(f)$ is a weak equivalence (resp. a fibration) in the category of $k$-fold sequences with based points. This model category structure is called projective model category structure. \vspace{5pt}
\end{thm}

\begin{proof}
Similar to the proof of Theorem \ref{A9}.\vspace{5pt}
\end{proof}

\begin{rmk}
Construction \ref{F9} and Theorem \ref{C5} admit analogue versions for $\vec{r}$-truncated bimodules. In that case, we only need to consider the set $rsTree[\,\vec{n}\leq \vec{r}\,]$ instead of $rsTree[\,\vec{n}\,]$ in Construction \ref{F9} where $rsTree[\,\vec{n}\leq \vec{r}\,]$ is the set of elements $(T_{1},\ldots,T_{k},\vec{\sigma})\in rsTree[\,\vec{n}\,]$ for which the number of leaves plus the number of univalent $internal$ vertices of $T_{i}$ is smaller than $r_{i}$.
\end{rmk}

\subsubsection{The Reedy model category structure}

Let $\Lambda^{\times k}_{+}$ be the category  whose objects are families $(A_{1},\ldots,A_{k})\neq (+,\ldots,+)$ where $A_{i}$ is a finite set or the augmented set and morphisms are families of injective maps. Given an element $\vec{r}=(r_{1},\ldots,r_{k})\in \mathbb{N}^{k}$, we also consider the sub-category $T_{\vec{r}}\,\Lambda^{\times k}_{+}$ whose objects are families of finite sets $(A_{1},\ldots,A_{k})$ with $|A_{i}|\leq r_{i}$. One has the two categories\vspace{3pt}
\begin{itemize}
\item[$\blacktriangleright$] $\Lambda Seq_{k}^{+}$: the category of contravariant functor from $\Lambda^{\times k}_{+}$ to spaces;\vspace{3pt}
\item[$\blacktriangleright$] $T_{\vec{r}}\,\Lambda Seq_{k}^{+}$: the category of contravariant functor from $T_{\vec{r}}\,\Lambda^{\times k}_{+}$ to spaces.\vspace{3pt}
\end{itemize}
Similarly, let $\Lambda^{\times k}_{>0}$ be the subcategory of $\Lambda^{\times k}_{+}$ whose objects are families $(A_{1},\ldots,A_{k})$ different to elements of the form $(B_{1},\ldots,B_{k})$ with $B_{i}\in \{+\,;\,\emptyset\}$. We also consider the sub-category $T_{\vec{r}}\,\Lambda^{\times k}_{>0}$ whose objects are families of finite sets $(A_{1},\ldots,A_{k})$ with $|A_{i}|\leq r_{i}$. One has the two categories\vspace{3pt}
\begin{itemize}
\item[$\blacktriangleright$] $\Lambda Seq_{k}^{>0}$: the category of contravariant functor from $\Lambda^{\times k}_{>0}$ to spaces;\vspace{3pt}
\item[$\blacktriangleright$] $T_{\vec{r}}\,\Lambda Seq_{k}^{>0}$: the category of contravariant functor from $T_{\vec{r}}\,\Lambda^{\times k}_{>0}$ to spaces.\vspace{3pt}
\end{itemize}
Then, we consider the following adjunctions:\vspace{3pt}
$$
\begin{array}{rcl}\vspace{9pt}
(-)_{>0}: \Lambda Seq_{k}^{+} & \leftrightarrows & \Lambda Seq_{k}^{>0}:(-)_{+}, \\ 
(-)_{>0}: T_{\vec{r}}\,\Lambda Seq_{k}^{+} & \leftrightarrows & T_{\vec{r}}\,\Lambda Seq_{k}^{>0}:(-)_{+},
\end{array} \vspace{3pt}
$$
where $M_{>0}$ is obtained from $M$ by forgetting the components of the form  $M(B_{1},\ldots,B_{k})$, with $B_{i}\in \{+\,;\,\emptyset\}$. On the other hand, $N_{+}$ is obtained from $N$ by defining $N_{+}(B_{1},\ldots,B_{k})=\ast$, for any element $(B_{1},\ldots,B_{k})$ with $B_{i}\in \{+\,;\,\emptyset\}$, and keeping all the other components the same. As in Section \ref{G0}, these two categories are endowed with a Reedy model structure.\vspace{7pt}

Let $O_{1},\ldots,O_{k}$ be a family of operads relative to another operad $O$. A reduced $k$-fold bimodules over $\vec{O}$ is a $k$-fold bimodule $M$ satisfying $M(B_{1},\ldots,B_{k})=\ast$ for any elements $(B_{1},\ldots,B_{k})$ with $B_{i}\in \{+\,;\,\emptyset\}$. We denote by the category of reduced $k$-fold bimodules over $\vec{O}$ by $\Lambda Bimod_{\vec{O}}$ and one has a unitarization-inclusion adjunctions\vspace{5pt}
\begin{equation}\label{G6}
\begin{array}{rcl}\vspace{9pt}
\tau:Bimod_{\vec{O}} & \leftrightarrows & \Lambda Bimod_{\vec{O}}:\iota, \\ 
\tau:T_{\vec{r}}\, Bimod_{\vec{O}} & \leftrightarrows & T_{\vec{r}}\, \Lambda Bimod_{\vec{O}}:\iota,
\end{array} \vspace{-17pt}
\end{equation}

\newpage

\noindent where $\iota$ is the inclusion functor and $\tau$ its adjoint which consists in collapsing the arities indexed by elements of the form $(B_{1},\ldots,B_{k})$, with $B_{i}\in \{+\,;\,\emptyset\}$, to a point and adjusting the other components according to the equivalence relation induced by this collapse.  The category $\Lambda Bimod_{\vec{O}}$ and its truncated version admit a Reedy model category structure transferred from the adjunctions \vspace{5pt}
\begin{equation}\label{G5}
\begin{array}{rcl}\vspace{5pt}
\mathcal{F}_{B\,;\,\vec{O}}^{\Lambda}:\Lambda Seq_{k}^{>0} & \leftrightarrows & \Lambda Bimod_{\vec{O}}:\mathcal{U}, \\ 
T_{\vec{r}}\mathcal{F}_{B\,;\,\vec{O}}^{\Lambda}:T_{\vec{r}}\,\Lambda Seq_{k}^{>0}& \leftrightarrows & T_{\vec{r}}\,\Lambda Bimod_{\vec{O}}:\mathcal{U},
\end{array} \vspace{2pt}
\end{equation}
where the free functors, also denoted by $\mathcal{F}_{B}^{\Lambda}$ and $T_{\vec{r}}\mathcal{F}_{B}^{\Lambda}$ when $\vec{O}$ is understood, are obtained from $\mathcal{F}_{B}$ and its truncated version, respectively, by taking the restriction of the coproduct in Construction \ref{F9} to the $k$-fold trees with section without univalent vertices above the sections. According to the notation introduced in Section \ref{G0}, as $k$-fold augmented sequences, one has \vspace{5pt}
$$
\mathcal{F}_{B\,;\,\vec{O}}^{\Lambda}(M)\coloneqq\mathcal{F}_{B\,;\,\vec{O}_{>0}}(M_{>0})_{+}, \hspace{15pt} \text{and}\hspace{15pt} T_{\vec{r}}\mathcal{F}_{B\,;\,\vec{O}}^{\Lambda}(M)\coloneqq T_{\vec{r}}\mathcal{F}_{B\,;\,\vec{O}_{>0}}(M_{>0})_{+}. \vspace{3pt}
$$
By construction, the objects above are equipped with a $k$-fold (truncated) bimodule structures over $\vec{O}_{>0}$. We can extend this structure in order to get $k$-fold (truncated) bimodule over $\vec{O}$ using the operadic structures of $O_{1},\ldots, O_{k}$ and the $\Lambda^{\times k}_{+}$ structure of $M$.\vspace{7pt}

\begin{thm}{\cite{Ducoulombier19}}\label{J0} One has the following properties on the Reedy model category structure:

\begin{itemize}
\item[$(i)$] The categories $\Lambda Bimod_{\vec{O}}$ and $T_{\vec{r}}\,\Lambda Bimod_{\vec{O}}$, with $\vec{n}\in \mathbb{N}^{k}$, admit a cofibrantly generated model category structure, called Reedy model structure, transferred from $\Lambda Seq_{k}^{>0}$ and $T_{\vec{r}}\,\Lambda Seq_{k}^{>0}$, respectively, along the adjunctions $(\ref{G5})$.\vspace{5pt}

\item[$(ii)$] A morphism in the category of $k$-fold (possibly truncated) bimodules over $\vec{O}$ is a cofibration for the Reedy model category structure if and only if it is a cofibration as a morphism of $k$-fold (truncated) bimodules over $\vec{O}_{>0}$ equipped with the projective model category structure.\vspace{5pt}

\item[$(iii)$] In case $O_{1},\ldots,O_{k},O$ are Reedy cofibrant operads, the model structure on $\Lambda Bimod_{\vec{O}}$ is left proper. 
\end{itemize}
\end{thm}\vspace{5pt}

\begin{sproof}
We already know that this properties are trues in the context of bimodules over a colored operad \cite{Ducoulombier19}. So, the theorem is a consequence of the description of $k$-fold  bimodules in terms of $1$-fold bimodules over a colored operad introduced in Remark \ref{Z2}.
\end{sproof}

\subsubsection{Connections between the two structures and properties}

In the previous sections, we introduce two model category structures, projective and Reedy, on the category of $k$-fold (possibly truncated) bimodules over $\vec{O}$. In what follows, we show that these two structures are more or less the same homotopically speaking and induce the same derived mapping space up to a homeomorphism (see the identifications (\ref{G8})). For this reason,  we won't distinguish between the two mapping spaces and we will simply write $Bimod_{\vec{O}}^{h}(-\,;\,-)$ and $T_{\vec{r}}\,Bimod_{\vec{O}}^{h}(-\,;\,-)$.\vspace{5pt}

\begin{thm}{\cite{Ducoulombier19}}\label{H7} One has the following relations between the Reedy and projective model structures:

\begin{itemize}
\item[$(i)$] The adjunctions (\ref{G6}) are Quillen adjunctions.\vspace{3pt}
\item[$(ii)$] For any pair $M$ and $N$ of reduced $k$-fold (truncated) bimodules, one has equivalences of mapping spaces:\vspace{3pt}
\begin{equation}\label{G8}
\begin{array}{rcl}\vspace{5pt}
Bimod_{\vec{O}}^{h}(\iota M\,;\,\iota N) & \cong & \Lambda Bimod_{\vec{O}}^{h}(M\,;\,N), \\ 
T_{\vec{r}}\,Bimod_{\vec{O}}^{h}(\iota M\,;\,\iota N) & \cong & T_{\vec{r}}\,\Lambda Ibimod_{\vec{O}}^{h}(M\,;\,N).
\end{array} 
\end{equation}
\end{itemize}
\end{thm}\vspace{-2pt}

\noindent \textit{Sketch of proof.} The proof is similar to the proof in the context of operads \cite{FresseTWsmall}. It is straightforward that $\tau$ sends the generating cofibrations to cofibrations, which implies that the adjunctions (\ref{G6}) are Quillen adjunctions. In order to prove (\ref{G8}), it is enough to show that the natural map $\tau M^{c}\rightarrow M$, where $M^{c}$ is a cofibrant replacement of $\iota M$ (which can be chosen to be the one developed in the next section), is a weak equivalence. \vspace{-20pt}

\newpage

A map $\vec{\alpha}:\vec{O}\rightarrow \vec{O}'$ between two families of operads $O_{1},\ldots O_{k}$ and $O_{1}',\ldots O_{k}'$ relative to topological operads $O$ and $O'$, respectively, is a family of operadic maps $\alpha_{i}:O_{i}\rightarrow O_{i}'$ and  $\alpha:O\rightarrow O'$ such that $f_{i}'\circ \alpha_{i}=\alpha\circ f_{i}$. Such a map $\vec{\alpha}$ is said to be a weak equivalence if the maps $\alpha$, $\alpha_{i}$ and the induced maps \vspace{3pt}
$$
\vec{O}_{S}(S_{1},\ldots,S_{k})\longrightarrow \vec{O}'_{S}(S_{1},\ldots,S_{k}),\hspace{15pt} \forall S\in \mathcal{P}_{k}(A),
$$
are weak homotopy equivalences. In particular, such a map $\vec{\alpha}:\vec{O}\rightarrow \vec{O}'$ produces a map between the families of operads $O_{1},\ldots O_{k}$ and $O_{1}',\ldots O_{k}'$ relative to the topological monoids $O(1)$ and $O'(1)$, respectively, which is a weak equivalence if  $\vec{\alpha}$ is a weak equivalence. The above theorem is true in the context of bimodules over a colored operad \cite{Ducoulombier19} and is a consequence of Remark \ref{Z2}. \vspace{7pt}

\begin{thm}{\cite{Ducoulombier19}}
Let $\vec{\alpha}:\vec{O}\rightarrow \vec{O}'$ be a weak equivalence with cofibrant components between families of reduced operads relative to another one. One has Quillen equivalences
$$
\begin{array}{rcl}\vspace{5pt}
\alpha_{B}^{!}:\Lambda Bimod_{\vec{O}} & \leftrightarrows & \Lambda Bimod_{\vec{O}'}:\alpha^{\ast}_{B}, \\ 
\alpha_{B}^{!}: T_{\vec{r}}\,\Lambda Bimod_{\vec{O}} & \leftrightarrows & T_{\vec{r}}\,\Lambda Bimod_{\vec{O}'}:\alpha^{\ast}_{B},
\end{array} 
$$
where $\alpha_{B}^{\ast}$ is the restriction functor and $\alpha_{B}^{!}$ is the induction one.  
\end{thm}

\subsection{The Boardman-Vogt resolution for $k$-fold bimodules}\label{F5}

As explained in the previous section, the category of $k$-fold  bimodules is endowed with a projective model category structure in which all the objects are fibrant. Consequently, in order to compute the derived mapping space \vspace{5pt}
$$
Bimod_{\vec{O}}^{h}(M_{1}\,;\,M_{2})\coloneqq Bimod_{\vec{O}}(M_{1}^{c}\,;\,M_{2}),\vspace{5pt}
$$
we only need an explicit cofibrant replacement $M_{1}^{c}$ for any $k$-fold bimodule $M_{1}$. Similarly to Section \ref{C6}, we use a kind of Boardman-Vogt resolution. Then, we deduce from this resolution another cofibrant replacement of $M_{1}$ in the Reedy model category.\vspace{7pt}

\begin{defi}\textbf{The set of $k$-fold trees with section}

\noindent Let $\vec{n}=(n_{1},\ldots,n_{k})$ with $n_{i}\in \mathbb{N}\sqcup\{+\}$ and $(n_{1},\ldots,n_{k})\neq (+,\ldots,+)$. According to the notation introduced in Definition \ref{C2}, the set $sTree[\,\vec{n}\,]$ of $k$-fold trees with section is formed by families $\vec{T}=(T_{1},\ldots,T_{k},f_{1},\ldots,f_{k},\sigma)$ where $T_{i}$ is a tree with section having $n_{i}$ leaves (without leaves if $n_{i}=+$) and $\sigma\in \Sigma_{n_{1}}\times \cdots \times \Sigma_{n_{k}}$ is a permutation labelling the leaves. Furthermore, we assume that the sub-trees $T'_{1},\ldots,T'_{k}$ obtained from $T_{1},\ldots,T_{k}$, respectively, by removing the vertices and the edges above the sections are the same. Finally, $f_{i}:E(T'_{i})\rightarrow \{internal \,;\, external\}$ is a function labelling the inner edges below the section of $T_{i}$ satisfying the same conditions introduced in Definition \ref{C2}.\vspace{5pt}

\begin{figure}[!h]
\begin{center}
\includegraphics[scale=0.123]{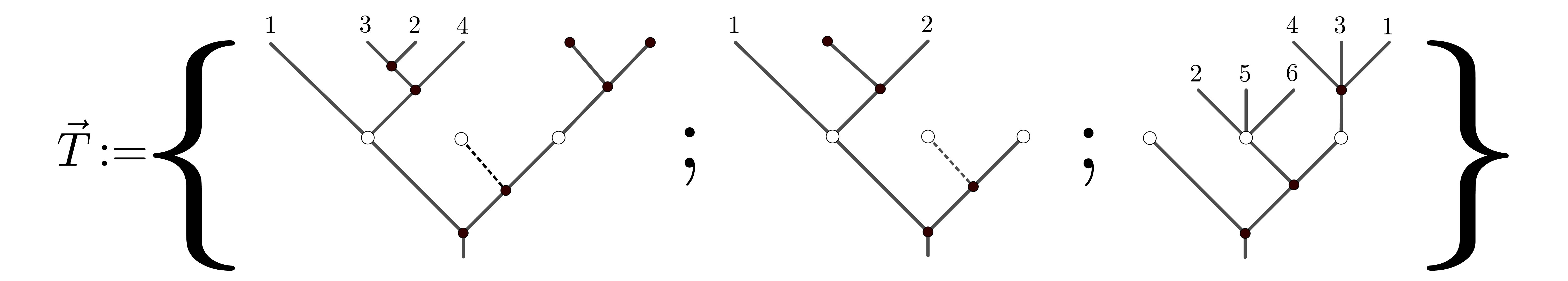}\vspace{-5pt}
\caption{Illustration of an element in $sTree[4\,;\,2\,;\,6]$.}\label{O0}\vspace{-10pt}
\end{center}
\end{figure}

Since the sub-trees $T'_{1},\ldots,T'_{k}$ are the same, if $v^{1}$ is a vertex of $T_{1}$ below the section, then we denote by $v^{j}$ the corresponding vertex of $T_{j}$. Furthermore, we denote by $|v|_{i}$ the number of incoming $internal$ vertices of a vertex $v$ below the section. By convention, an $external$ edge is represented by a dotted line and $|p|=+$ or $|v|_{i}=+$ if the output edge of the pearl $p$ or the vertex below the section $v$ is labelled by $external$.

\newpage

\end{defi}

\begin{const}\label{H1}
Let $M=\{M(n_{1},\ldots,n_{k})\}$ be a $k$-fold bimodule over $\vec{O}$. The resolution $\mathcal{B}_{\vec{O}}(M)$, also denoted by $\mathcal{B}(M)$ when the operads are understood,  consists in labelling the vertices of $k$-fold trees with section by real numbers in the interval $[0\,,\,1]$ and elements in the bimodule $M$ or in the operads $O_{1},\ldots,O_{k}$. For $\vec{n}=(n_{1},\ldots,n_{k})$, with $n_{i}\in\mathbb{N}\sqcup\{+\}$ and $(n_{1},\ldots,n_{k})\neq (+,\ldots,+)$, the space $\mathcal{B}(M)(\vec{n})$ is defined as a quotient of the subspace\vspace{5pt}
$$
\left.\left(
\underset{\vec{T}\in sTree[\vec{n}]}{\coprod}\hspace{5pt} \underset{v^{1}\in V^{d}(T_{1})}{\prod} \big[ \vec{O}_{S_{v^{1}}}(|v^{1}|_{i}, \cdots, |v^{k}|_{i})\times [0\,,\,1]\big]\times \underset{p_{1}\in V^{p}(T_{1})}{\prod} M(|p_{1}|,\cdots,|p_{k}|)\times \underset{\substack{i\in\{1,\ldots,k\}\\ v\in V^{u}(T_{i})}}{\prod} \big[ O_{i}(|v|)\times [0\,,\,1]\big]\right)\,\,
\right/\!\!\sim\vspace{5pt}
$$
with the following condition: if two vertices $v$ and $v'$ above the section (resp. below the section) are connected by an inner edge from $v$ to $v'$ according to the direction toward the root, then the real numbers $t_{v}$ and $t_{v'}$ associated to $v$ and $v'$ satisfy the condition $t_{v}\geq t_{v'}$ (resp. $t_{v}\leq t_{v'}$). By convention the pearls are indexed by $0$ and $S_{v^{1}}=(S_{v^{1}}^{1},\ldots ,S_{v^{1}}^{k} )\in \mathcal{P}_{k}(\{1,\ldots,|v^{1}|\})$ is given by \vspace{5pt}
$$
S_{v^{1}}^{i}\coloneqq
\left\{
\begin{array}{ll}\vspace{7pt}
\text{the augmented set } + & \text{if the output edge of } v^{i} \text{ is } external, \\ 
\big\{ j\in \{1,\ldots,|v^{1}|\}\,\,\big|\,\, \text{the } j\text{-th incoming input of } v^{i} \text{ is } internal\big\} & \text{otherwise}.
\end{array} 
\right.\vspace{-7pt}
$$

\begin{figure}[!h]
\begin{center}
\includegraphics[scale=0.18]{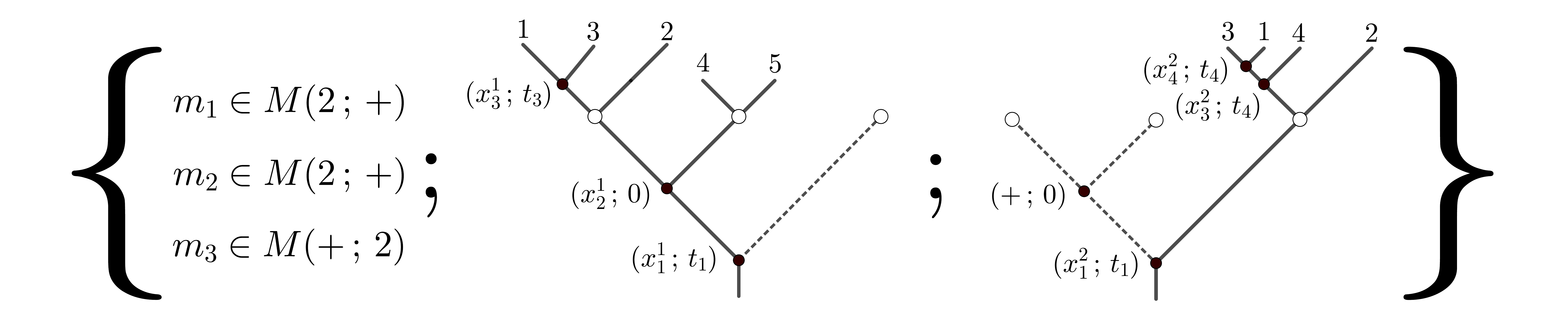}\vspace{-5pt}
\caption{Illustration of a point in $\mathcal{B}(M)(5\,;\,4)$.}\label{C4}\vspace{-7pt}
\end{center}
\end{figure}

The equivalence relation is generated by the unit axiom, the compatibility with the symmetric group action as well as the relation contracting the pearls indexed by the based point in the spaces of the form $M(B_{1},\ldots,B_{k})$, with $B_{i}\in \{+\,;\,\emptyset\}$. Furthermore, if two vertices connected by an inner edge are indexed by the same real number $t$, then we contract the inner edge using the $k$-fold bimodule structure of $M$ or the operadic structures of $O_{1},\ldots,O_{k}$. The new vertex so obtained is indexed by $t$. We denote by $[\vec{T}\,;\,\{m_{p}\}\,;\,\{x_{v}\}\,;\,\{t_{v}\}]$ a point in the above space. For instance, the point represented in Figure \ref{C4} is equivalent to the following one in which $S=\{\{1\,;\,2\}\,,\,+\}\in \mathcal{P}_{2}(\{1\,;\,2\})$: \vspace{-4pt}
\begin{figure}[!h]
\begin{center}
\includegraphics[scale=0.16]{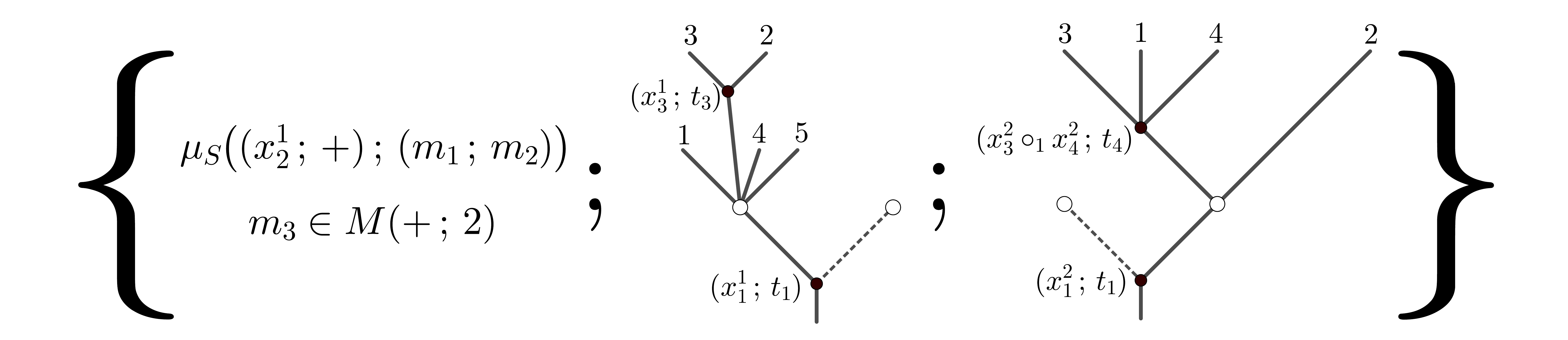}\vspace{-10pt}
\end{center}
\end{figure}

The left and right operations are defined as in Construction \ref{F9}. They consist in grafting corollas labelled by elements in the operads $O_{1},\ldots,O_{k}$ and indexing the new vertices by the real number $1$. Moreover, one has two maps of $k$-fold bimodules over $\vec{O}$\vspace{5pt}
\begin{equation}\label{C7}
\eta':\mathcal{B}(M)\longrightarrow M \hspace{15pt}\text{and}\hspace{15pt} \tau':\mathcal{F}_{B}(M)\longrightarrow \mathcal{B}(M),\vspace{5pt}
\end{equation}
where $\eta'$ is the map sending the real number to $0$ while the map $\tau'$ indexes the vertices other than the pearls by $1$. Furthermore, the above  construction for $1$-fold bimodules is hoemeomorphic to the usual Boardman-Vogt resolution for bimodules introduced in \cite{Ducoulombier16}.
\end{const}

From now on, we introduce a filtration of the resolution $\mathcal{B}(M)$ according to the number of geometrical inputs which is the number of leaves plus the number of univalent vertices other than the pearls. Similarly to Section \ref{C6}, a point in $\mathcal{B}(M)$ is said to be \textit{prime} if the real numbers indexing the vertices are strictly smaller than $1$. Besides, a point is said to be \textit{composite} if one of the real numbers is $1$ and such a point can be associated to prime components. More precisely, the prime components are obtained by cutting the vertices indexed by $1$. For instance, the prime components associated to the point\vspace{-5pt}

\begin{figure}[!h]
\begin{center}
\includegraphics[scale=0.16]{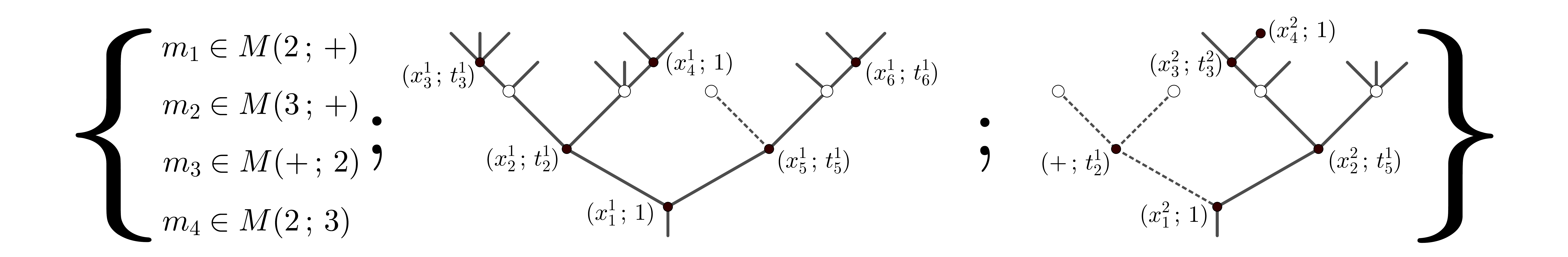}\vspace{-10pt}
\end{center}
\end{figure}

\noindent are the following ones:\vspace{-3pt}

\begin{figure}[!h]
\begin{center}
\includegraphics[scale=0.14]{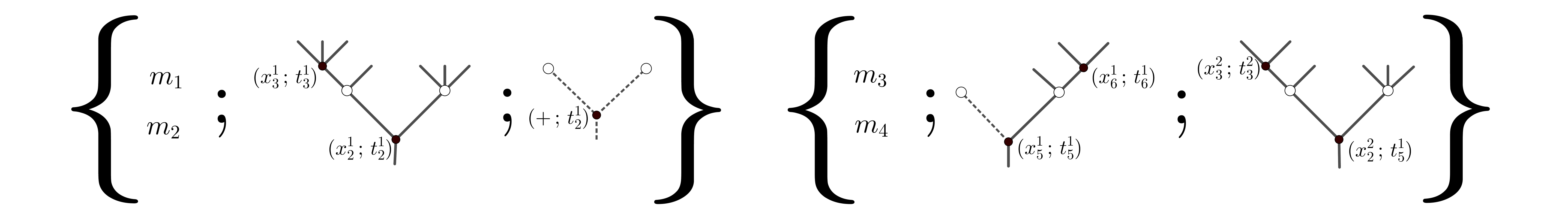}\vspace{-7pt}
\end{center}
\end{figure}

For $\vec{n}=(n_{1},\ldots,n_{k})\in \mathbb{N}^{k}$, a prime point, indexed by an element $\vec{T}\in sTree[\vec{n}]$, is in the $\vec{n}$-th filtration term $\mathcal{B}(M)[\vec{n}]$ if $T_{i}$ has at most $n_{i}$ geometrical inputs. Then, a composite point is in the $\vec{n}$-th filtration term if its prime components are in $\mathcal{B}(M)[\vec{n}]$. For instance the composite point in the example above is in $\mathcal{B}(M)[7\,;\,6]$.  By construction, $\mathcal{B}(M)[\vec{n}]$ is a $k$-fold bimodule and, for each pair $\vec{m}\leq \vec{n}$, there is a map of $k$-fold bimodules induced by the inclusion\vspace{5pt}
\begin{equation}\label{G9}
\iota[\vec{m}\leq \vec{n}]:\mathcal{B}(M)[\vec{m}]\longrightarrow \mathcal{B}(M)[\vec{n}].\vspace{7pt}
\end{equation}

\begin{thm}\label{L0}
Let $M$ be a $k$-fold bimodule over the family of reduced operads $O_{1},\ldots,O_{k}$ relative to $O$. If $\vec{O}$ and $M$ are cofibrant as $k$-fold sequences and the maps $\gamma_{\vec{n}}$, with $n_{i}\in \{+\,;\,0\}$, are cofibrations, then $\mathcal{B}(M)$ and $T_{\vec{r}}\,(\mathcal{B}(M)[\,\vec{r}\,])$ are cofibrant replacements of $M$ and $T_{\vec{r}\,}(M)$, respectively. In particular, the maps $\eta'$ and $\tau'$ (see (\ref{C7})) are, respectively, a weak equivalence and a cofibration in the projective model category $Bimod_{\vec{O}}$.\vspace{3pt}
\end{thm}

\begin{sproof}There are many ways to prove this theorem. One of them consists in using Remark \ref{Z2} and the fact that the Boardman-Vogt resolution introduced in Construction \ref{H1} is homeomorphic to the Boardman-Vogt resolution  introduced \cite{Ducoulombier16} in the category of bimodules over the appropriated colored operad. \vspace{7pt}

We can also check by hand the theorem. Without going into details, the map $\eta'$ induces a homotopy equivalence in the category of $k$-fold augmented sequences in which the homotopy consists in bringing the real numbers indexing the vertices to $0$. For the truncated case, we need first to contract the inner edges which are not connected to a leaf because the homotopy sending the real numbers to $0$ doesn't necessarily preserve the number of geometrical inputs. On the other hand, the maps $\mu'$ and $\iota[\vec{m}\leq \vec{n}]$ can be proved to be cofibrations in the category of $k$-fold bimodules by induction on the number of vertices of $k$-fold trees with sections.
\end{sproof}

\subsubsection{Boardman-Vogt resolution in the $\Lambda^{\times k}_{>0}$ setting.}\label{I1}

From now on, we assume that $O_{1},\ldots,O_{k}$ is a family of reduced operads relative to an operad $O$ and we adapt the construction in the previous section to produce cofibrant replacements in the category of reduced $k$-fold bimodules equipped with the Reedy model category structure. According to the notation introduced in Section \ref{G0}, we set \vspace{7pt}
$$
\mathcal{B}_{\vec{O}}^{\Lambda}(M)\coloneqq \mathcal{B}_{\vec{O}_{>0}}(M_{>0})_{+}.
$$
In other words, $\mathcal{B}_{\vec{O}}^{\Lambda}(M)$, also denoted by $\mathcal{B}^{\Lambda}(M)$ when $\vec{O}$ is understood, is obtained from the restriction of the coproduct in Construction \ref{H1} to the $k$-fold  trees with section without univalent vertices other than the pearls. By construction, the $k$-fold augmented sequence so obtained is a $k$-fold bimodule over $\vec{O}_{>0}$. Similarly to the free functor in the previous section, we can extend this structure in order to get a $k$-fold bimodule over $\vec{O}$ using the operadic structure of the operads $O_{1},\ldots,O_{k}$ and the $\Lambda^{\times k}_{+}$ structure of $M$. Furthermore, the $\vec{O}_{>0}$-bimodule maps $(\ref{C7})$ induce the maps\vspace{6pt}
\begin{equation}\label{H0}
\eta':\mathcal{B}^{\Lambda}(M)\longrightarrow M \hspace{15pt}\text{and}\hspace{15pt} \tau':\mathcal{F}_{B}^{\Lambda}(M)\longmapsto \mathcal{B}^{\Lambda}(M),\vspace{5pt}
\end{equation}
 which respect the $\Lambda^{\times k}_{+}$ structure and thus are $k$-fold bimodule maps over $\vec{O}$. Finally, the filtration (\ref{G9}) gives rise to a filtration of $k$-fold bimodules over $\vec{O}_{>0}$  \vspace{4pt}
$$
\iota[\vec{m}\leq \vec{n}]:\mathcal{B}^{\Lambda}(M)[\,\vec{m}\,]\longrightarrow \mathcal{B}^{\Lambda}(M)[\,\vec{n}\,],\vspace{4pt}
$$
compatible with the right action by $O_{1}(0),\ldots,O_{k}(0)$. So, this is also a filtration of $k$-fold bimodules over $\vec{O}$. \vspace{1pt}

\begin{thm}\label{L9}
Let $M$ be a $k$-fold bimodule over $\vec{O}$. If $\vec{O}$ and $M$ are cofibrant as $k$-fold sequences and the map $\gamma_{\vec{n}}$, with $n_{i}\in \{+\,;\,0\}$, are cofibrations, then $\mathcal{B}^{\Lambda}(M)$ and $T_{\vec{r}\,}(\mathcal{B}^{\Lambda}(M)[\,\vec{r}\,])$ are cofibrant replacements of $M$ and $T_{\vec{r}}\,(M)$ in the Reedy model categories $\Lambda Bimod_{\vec{O}}$ and $T_{\vec{r}}\,\Lambda Bimod_{\vec{O}}$,respectively. In particular, the maps $\eta'$ and $\tau'$ (see (\ref{H0})) are, respectively, a weak equivalence and a cofibration.\vspace{5pt}
\end{thm}

\begin{proof}
It is a consequence of Theorems \ref{J0} and \ref{L0}. Nevertheless, we recall the arguments showing that $\mathcal{B}^{\Lambda}(M)$ and its truncated versions are cofibrant in order to introduced some notation used in Section \ref{D5}. The idea is to check that each map $\iota[\vec{m}\leq \vec{n}]$ is a cofibration in the projective model category of $k$-fold bimodules over $\vec{O}_{>0}$. Without loss of generality, we assume that $k=2$, $\vec{m}=(m_{1}\,;\,m_{2})$ and $\vec{n}=(m_{1}+1\,;\,m_{2})$. In that case the map   $\iota[\vec{m}\leq \vec{n}]$ can be obtained as a sequence \vspace{3pt}
\begin{equation}\label{O3}
\xymatrix{
\mathcal{B}^{\Lambda}(M)[\,\vec{m}\,] \ar[r] & \mathcal{B}^{\Lambda}_{+}(M)[\,\vec{n}\,] \ar[r] & \mathcal{B}_{0}^{\Lambda}(M)[\,\vec{n}\,]\ar[r] & \cdots \ar[r] & \mathcal{B}_{m_{2}}^{\Lambda}(M)[\,\vec{n}\,]=\mathcal{B}^{\Lambda}(M)[\,\vec{n}\,]
}\vspace{3pt}
\end{equation}
where the $k$-fold bimodule $\mathcal{B}_{+}^{\Lambda}(M)[\,\vec{n}\,]$ is defined using the pushout diagram \vspace{5pt}
$$
\xymatrix{
\mathcal{F}_{B\,;\,\vec{O}_{>0}}\big(\,\partial\mathcal{B}^{\Lambda}(M)(m_{1}+1\,;\,+)\,\big)\ar[r] \ar[d] & \mathcal{F}_{B\,;\,\vec{O}_{>0}}\big(\,\mathcal{B}^{\Lambda}(M)(m_{1}+1\,;\,+)\,\big)\ar[d]\\
\mathcal{B}^{\Lambda}(M)[\,\vec{m}\,] \ar[r] & \mathcal{B}_{+}^{\Lambda}(M)[\,\vec{n}\,]
}\vspace{5pt}
$$
%
The $k$-fold bimodules $\mathcal{B}_{i}^{\Lambda}(M)[\,\vec{n}\,]$, with $0\leq i\leq m_{2}$ and $\mathcal{B}_{-1}^{\Lambda}(M)[\,\vec{n}\,]=\mathcal{B}_{+}^{\Lambda}(M)[\,\vec{n}\,]$, are built by induction using pushout diagrams of the form\vspace{5pt}
$$
\xymatrix{
\mathcal{F}_{B\,;\,\vec{O}_{>0}}\big(\,\partial\mathcal{B}^{\Lambda}(M)(m_{1}+1\,;\,i)\,\big)\ar[r] \ar[d] & \mathcal{F}_{B\,;\,\vec{O}_{>0}}\big(\,\mathcal{B}^{\Lambda}(M)(m_{1}+1\,;\,i)\,\big)\ar[d]\\
\mathcal{B}^{\Lambda}_{i-1}(M)[\,\vec{n}\,] \ar[r] & \mathcal{B}_{i}^{\Lambda}(M)[\,\vec{n}\,]
}\vspace{5pt}
$$
The space $\mathcal{B}^{\Lambda}(M)(m_{1}+1\,;\,i)$ is seen as an augmented sequence concentrated in arity $(m_{1}+1\,;\,i)$. Similarly, $\partial\mathcal{B}^{\Lambda}(M)(m_{1}+1\,;\,i)$ is the augmented sequence formed by points in $\mathcal{B}^{\Lambda}(M)(m_{1}+1\,;\,i)$  having a vertex indexed by $1$. By induction on the number of vertices, the inclusion from $\partial\mathcal{B}^{\Lambda}(M)(m_{1}+1\,;\,i)$  to $\mathcal{B}^{\Lambda}(M)(m_{1}+1\,;\,i)$  is a cofibration in the category of augmented sequences. Since the free $k$-fold bimodule functor and pushout diagrams preserve cofibrations, the horizontal maps in the above diagrams are cofibrations in the projective model category of $k$-fold bimodules over $\vec{O}_{>0}$. Thus proves that (\ref{O3}) is a sequence of cofibrations and $\iota[\vec{m}\leq \vec{n}]$ is a cofibration in the Reedy model category of $k$-fold bimodules over $\vec{O}$.
\end{proof}

\newpage

\subsubsection{Properties of the Boardman-Vogt resolution associated to $\mathbb{O}^{+}$}

For $\vec{n}=(n_{1},\ldots,n_{k})$, with $n_{i}\in \mathbb{N}\sqcup \{+\}$ and $(n_{1},\ldots,n_{k})\neq (+,\ldots, +)$, we denote by $\overline{\mathbb{O}}(n_{1},\ldots,n_{k})$ the subspace of $\mathcal{B}^{\Lambda}(\mathbb{O}^{+})(n_{1},\ldots,n_{k})$ formed by points indexed by $k$-fold trees with section $(T_{1},\ldots,T_{k},f_{1},\ldots,f_{k},\sigma)$ satisfying the following conditions: 
\begin{itemize}
\item[$\blacktriangleright$] the roots are the only vertices which are not pearls and they are indexed by $1$,
\item[$\blacktriangleright$] the roots have exactly $\underset{\substack{1\leq i\leq k\\ n_{i}\neq +}}{\sum} n_{i}$ incoming edges,
\item[$\blacktriangleright$] vertices other than the roots are univalent pearls or bivalent pearls labelled by the units,
\item[$\blacktriangleright$] $f_{j}:E(T'_{i})\rightarrow \{internal\,;\,internal\}$ labels the $l$-th incoming edge of the root by $internal$ if and only if
$$l\in \left\{\,\,\,\,\left(\underset{\substack{1\leq i\leq j-1\\ n_{i}\neq +}}{\sum}n_{i} \right)+1,\ldots, \underset{\substack{1\leq i\leq j\\ n_{i}\neq +}}{\sum} n_{i}\right\}.$$ 
\end{itemize}\vspace{-10pt}
 
\begin{figure}[!h]
\begin{center}
\includegraphics[scale=0.17]{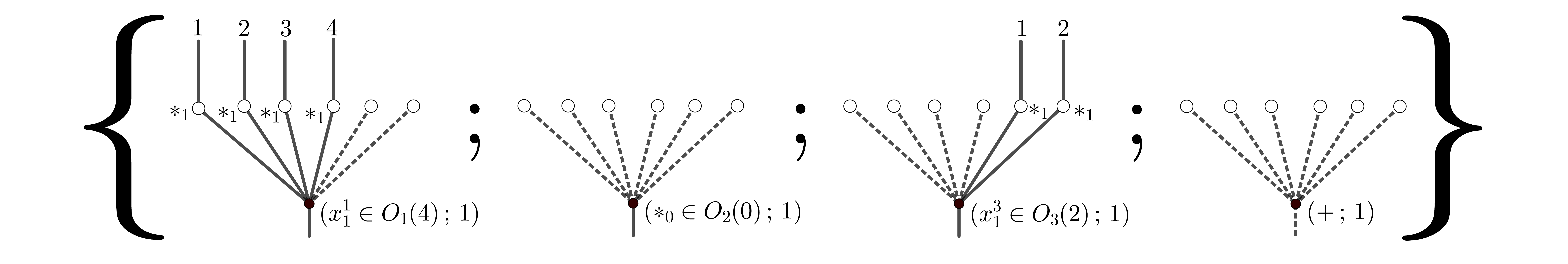}\vspace{-5pt}
\caption{Illustration of a point in $\overline{\mathbb{O}}(4\,;\,0\,;\,2\,;\,+)$.}\vspace{-11pt}
\end{center}
\end{figure}

Let us remark that the $k$-fold augmented sequence $\overline{\mathbb{O}}$ is homeomorphic to the $k$-fold bimodule $\mathbb{O}^{+}$. The map from $\overline{\mathbb{O}}$ to $\mathbb{O}^{+}$ is the projection on the parameters indexing the roots of the $k$-fold trees with section. From this identification and the weak equivalence of $k$-fold bimodules $(\ref{C7})$, we deduce a map of $k$-fold augmented sequences\vspace{5pt}
\begin{equation}\label{E3}
\lambda:\mathcal{B}^{\Lambda}(\mathbb{O}^{+})\longrightarrow \overline{\mathbb{O}}\cong\mathbb{O}^{+}.\vspace{7pt}
\end{equation}

\begin{lmm}\label{E9}
The map $(\ref{E3})$ is a homotopy equivalence.\vspace{3pt}
\end{lmm}

\begin{proof}
First, we consider the $k$-fold augmented sequence $A$ formed by points in $\mathcal{B}^{\Lambda}(\mathbb{O}^{+})$ without vertices above the sections. The two $k$-fold augmented sequences $\mathcal{B}^{\Lambda}(\mathbb{O}^{+})$ and $A$ are homotopically equivalent and the homotopy consists in bringing the real numbers indexing the vertices above the sections to $0$. For instance, for the element in $\mathcal{B}^{\Lambda}(\mathbb{O}^{+})(8\,;\,4)$ represented by (the points in $(m_{1},\ldots,m_{k})\in \mathbb{O}^{+}$ indexing the pearls are directly represented on the trees)\vspace{-5pt}
\begin{figure}[!h]
\begin{center}
\includegraphics[scale=0.17]{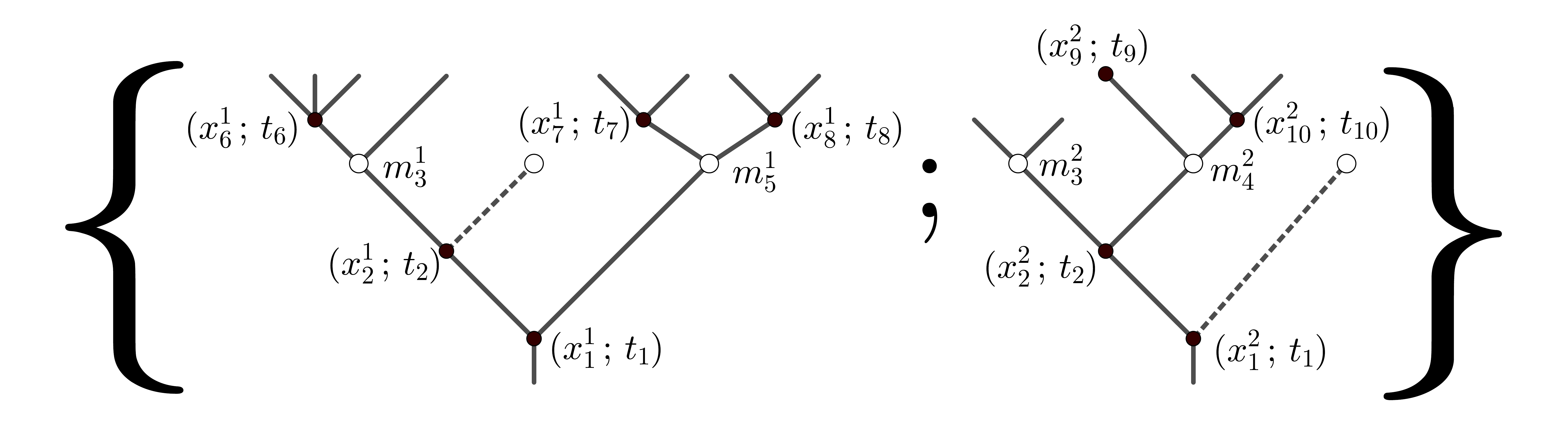}\vspace{-15pt}
\end{center}
\end{figure}

\noindent the homotopy sends the parameters $t_{5}$, $t_{6}$, $t_{7}$, $t_{12}$ and $t_{13}$ to $0$ and one gets the following point in $A$:
\begin{figure}[!h]
\begin{center}
\includegraphics[scale=0.17]{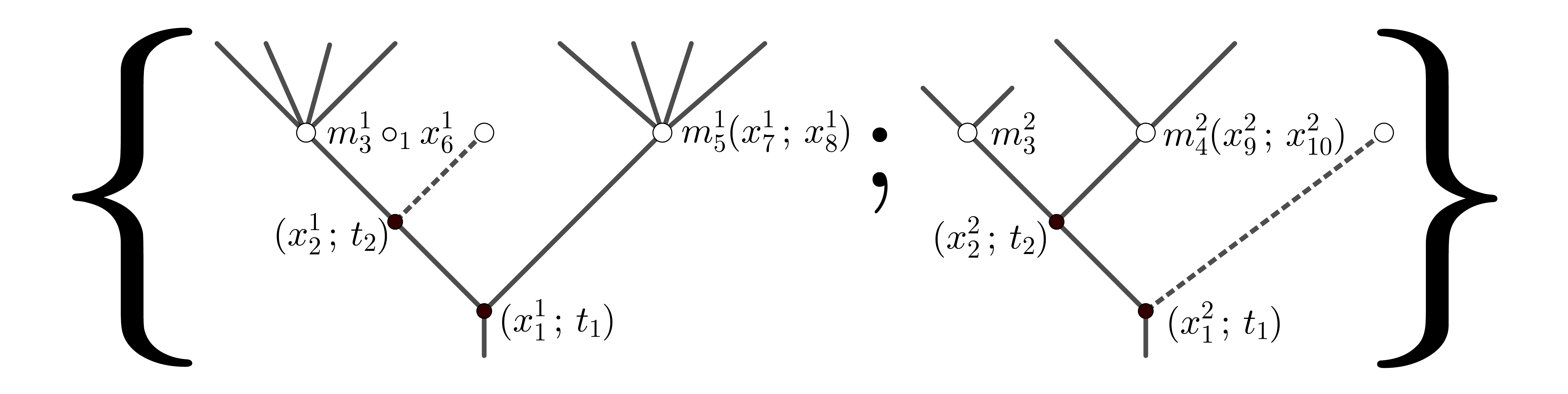}\vspace{-40pt}
\end{center}
\end{figure}

\newpage

Let $B$ be the $k$-fold augmented sequence formed by points in $A$ for which the pearls are $external$ pearls or bivalent pearls indexed by the unit of the operads $O_{1},\ldots,O_{k}$. Using the identification $\theta=\gamma_{l}(\theta;\ast_{1},\ldots ,\ast_{1})$, for any $\theta\in O_{i}(n)$, and the compatibility with the symmetric group axiom, we are able to check that the $k$-fold augmented sequences $A$ and $B$ are homeomorphic. Locally, the identifications used are of the form\vspace{1pt}
\begin{figure}[!h]
\begin{center}
\includegraphics[scale=0.13]{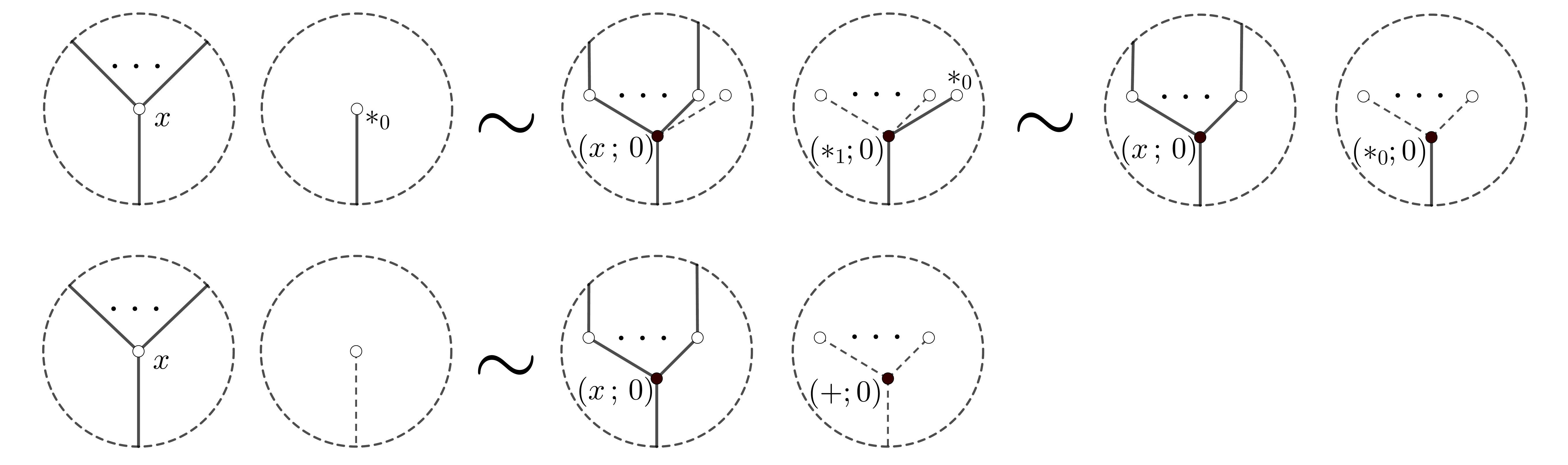}\vspace{-11pt}
\end{center}
\end{figure}

\noindent For instance, the point in $B$ associated to the example above is the following one:\vspace{-1pt}
\begin{figure}[!h]
\begin{center}
\includegraphics[scale=0.17]{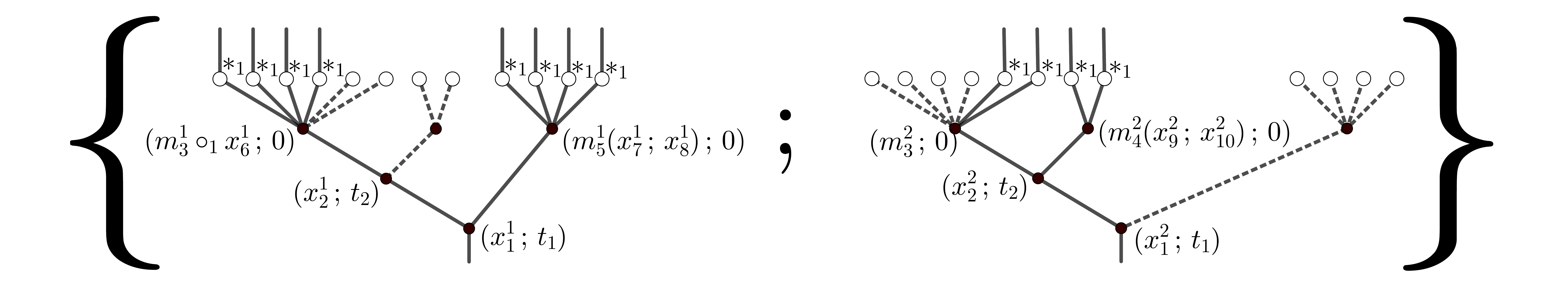}
\end{center}\vspace{-10pt}
\end{figure}

Finally, the homotopy between $B$ and $\overline{\mathbb{O}}$ consists in bringing the real numbers indexing the vertices below the sections to $1$. Then, using the action of the symmetric group, such a point can be identified with an element in $\overline{\mathbb{O}}$. In particular, the point in $\overline{\mathbb{O}}$ associated to the example above is the following one:
\begin{figure}[!h]
\begin{center}
\includegraphics[scale=0.22]{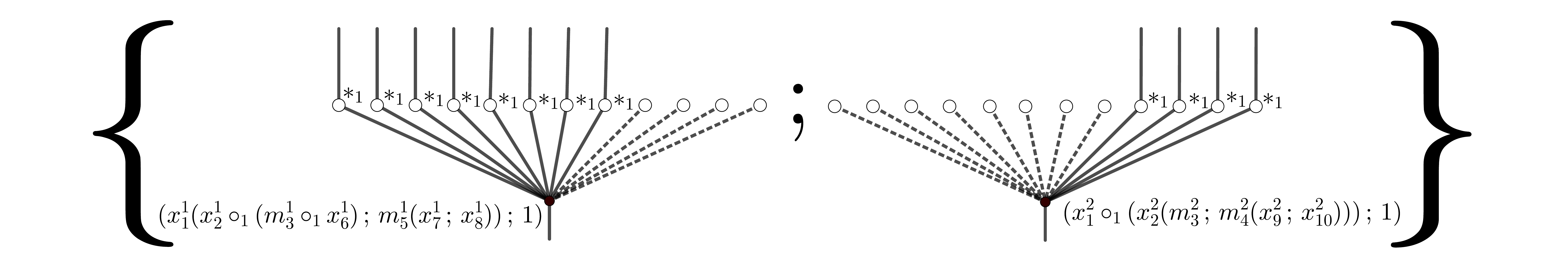}\vspace{-6pt}
\caption{Illustration of the identification.}\label{O1}
\end{center}\vspace{-20pt}
\end{figure}

\end{proof}

As explained in Example \ref{M1}, if $M_{1},\ldots,M_{k}$ are bimodules over $O_{1},\ldots,O_{k}$, respectively, then the $k$-fold augmented sequence $\mathbb{M}^{+}(n_{1},\ldots,n_{k})$ inherits a $k$-fold bimodule structure over the family $O_{1},\ldots,O_{k}$ relative to any reduced operad $O$. In particular, we can consider the following $k$-fold bimodule obtained from the Boardman-Vogt resolution of each operads $O_{i}$ as a $1$-fold bimodule over itself:\vspace{6pt}
\begin{equation}\label{H3}
\mathbb{B}^{\Lambda}(O)(n_{1},\ldots,n_{k})=\underset{\substack{1\leq i\leq k\\ n_{i}\neq +}}{\displaystyle \prod}\mathcal{B}_{O_{i}}^{\Lambda}(O_{i})(n_{i}), \hspace{15pt} \text{with } n_{i}\in \mathbb{N}\sqcup\{+\} \text{ and } (n_{1},\ldots,n_{k})\neq (+,\ldots,+).
\end{equation}
By construction $\mathbb{B}^{\Lambda}(O)$ is equipped with a map to $\mathbb{O}^{+}$ which is a weak equivalence of $k$-fold bimodules. The homotopy consists in bringing the parameters to $0$. Unfortunately, (\ref{H3}) is not cofibrant in the Reedy (or projective) model category of $k$-fold bimodules. However, there is a map of $k$-fold bimodules\vspace{5pt}
\begin{equation}\label{H4}
\xi:\mathcal{B}^{\Lambda}(\mathbb{O}^{+})\longrightarrow \mathbb{B}^{\Lambda}(O),\vspace{5pt}
\end{equation}
which consists in removing the $external$ edges. The map $\xi$ is obviously a weak equivalence since the two $k$-fold bimodules are both weakly equivalent to $\mathbb{O}^{+}$. \vspace{-10pt}

\newpage

\begin{lmm}\label{L2}
The map $(\ref{H4})$ induces a homotopy retract in the category of $k$-fold augmented sequences.\vspace{4pt}
\end{lmm}

\begin{proof}
We give an explicit description of the homotopy retract between the two objects. For this purpose, we consider the map of $k$-fold augmented sequences \vspace{3pt}
\begin{equation}\label{H5}
\phi:\mathbb{B}^{\Lambda}(O)\longrightarrow \mathcal{B}^{\Lambda}(\mathbb{O}^{+}),\vspace{1pt}
\end{equation}
sending a family $\{y_{i}\}_{1\leq i \leq k}$, with $y_{i}=[T_{i}\,;\,\{m_{p}^{i}\}\,;\,\{x_{v}^{i}\}\,;\,\{t_{v}^{i}\}]\in \mathcal{B}^{\Lambda}(O_{i})$, to the element indexed be the $k$-fold tree with section $(T_{1}^{+},\ldots,T_{k}^{+},f_{1},\ldots,f_{k},\sigma)$. The tree with section $T_{i}^{+}$ is obtained by grafting from left to right to a leaf of the $k$-corolla the trees with section  $T_{1}',\ldots,T_{i-1}',T_{i},T_{i+1}',\ldots,T_{k}'$ where $T'_{u}$ is the tree with section obtained from $T_{u}$ by removing the edges above the section. The application $f_{i}$ labels the edges associated to $T_{i}$ by $internal$ and the other edges by $external$. The new vertex produced by the $k$-corolla is indexed by $(\ast_{1}\,;\,1)$ while the other vertices are indexed by real numbers in the interval and points in the operads according to the parameters of $\{y_{i}\}$. For instance, the point in $\mathbb{B}^{\Lambda}(O)(6\,;\,4)$\vspace{-3pt}

\begin{figure}[!h]
\begin{center}
\includegraphics[scale=0.15]{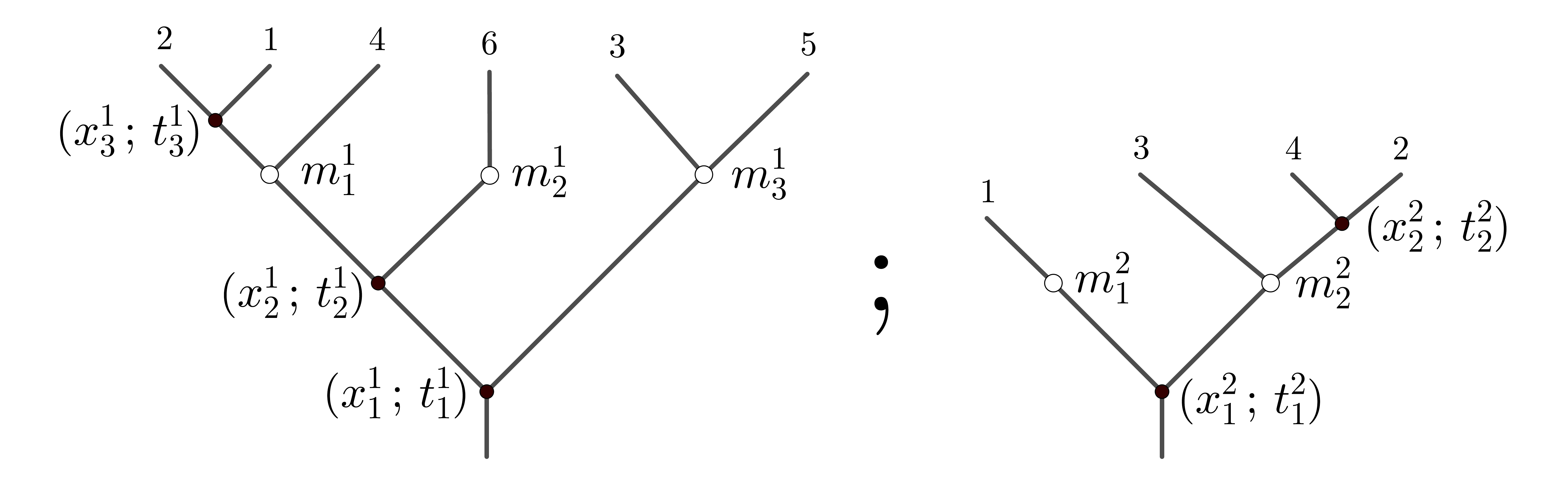}
\end{center}\vspace{-13pt}
\end{figure}

\noindent is sent to the following element in $\mathcal{B}^{\Lambda}(\mathbb{O}^{+})(6\,;\,4)$:

\begin{figure}[!h]
\begin{center}
\includegraphics[scale=0.083]{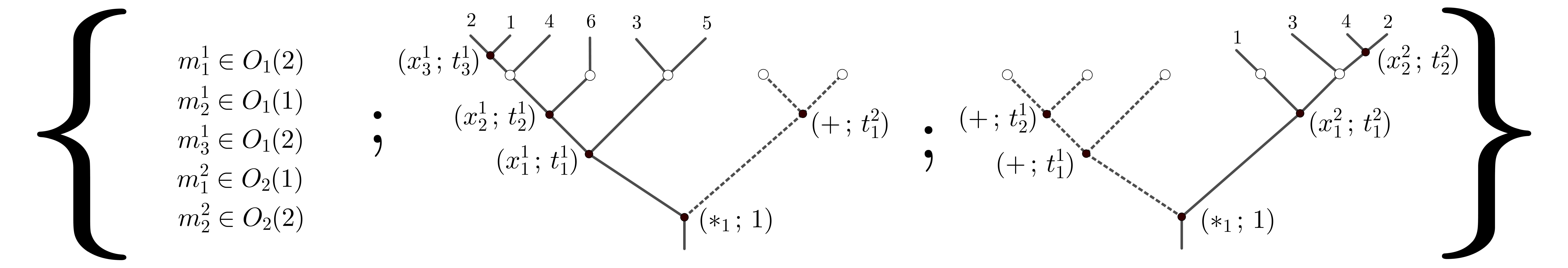}
\end{center}\vspace{-10pt}
\end{figure}

Now, one has to check that the maps (\ref{H4}) and (\ref{H5}) give rise to a deformation retract in the category of $k$-fold augmented sequences. The identification $\xi\circ\phi = id$ is clear. The homotopy $\phi\circ\xi\simeq id$ is obtained by using the identification defined locally as follows:
\begin{figure}[!h]
\begin{center}
\includegraphics[scale=0.14]{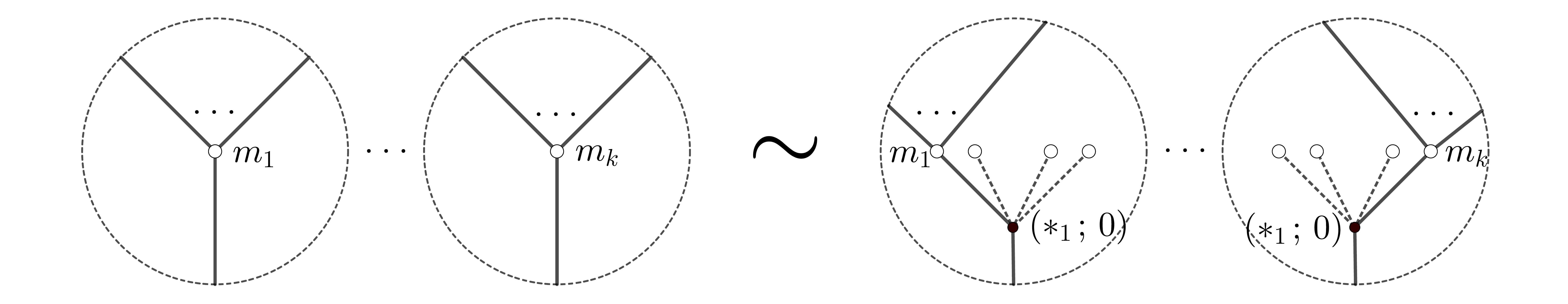}
\end{center}\vspace{-13pt}
\end{figure} 

\noindent  where $(m_{1},\ldots,m_{k})\in \mathbb{O}^{+}(n_{1},\ldots,n_{k})$. The vertices just below the section indexed by $(\ast_{1}\,;\,0)$ are called marked vertices. So, the homotopy sends the real numbers indexing the marked vertices to $1$. Each time marked vertices meet another vertices below the section, we use the identification defined locally as follows:\vspace{-5pt}
\begin{figure}[!h]
\begin{center}
\includegraphics[scale=0.13]{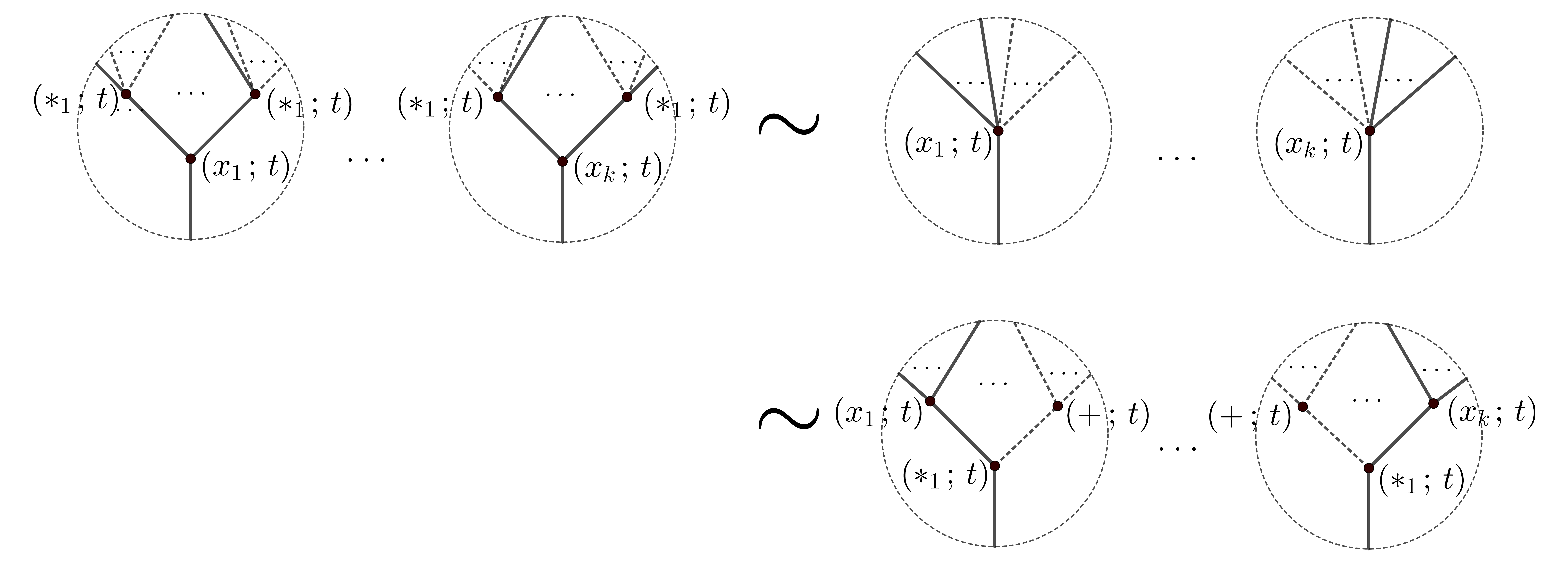}\vspace{-45pt}
\end{center}
\end{figure} 

\end{proof}

\newpage

\section{Delooping of multivariable manifold calculus functors}\label{D5}

In \cite{Ducoulombier18}, we explain the relation between the usual notions of bimodule and infinitesimal bimodule over an operad. More precisely, from a map of bimodules $\eta:O\rightarrow M$ over a $2$-reduced operad $O$, we build explicit maps\vspace{5pt}
\begin{equation}\label{O4}
\begin{array}{lcl}\vspace{7pt}
\xi:Map_{\ast}\big(\,\Sigma O(2)\,;\, Bimod^{h}_{O}(O\,;\,M)\,\big) & \longrightarrow & Ibimod_{O}^{h}(O\,;\,M), \\ 
\xi_{r}:Map_{\ast}\big(\,\Sigma O(2)\,;\, T_{r}\,Bimod^{h}_{O}(O\,;\,M)\,\big)  & \longrightarrow & T_{r}\,Ibimod_{O}^{h}(O\,;\,M),
\end{array} \vspace{5pt}
\end{equation}
where the spaces $Ibimod_{O}^{h}(O\,;\,M)$ and $T_{r}\,Ibimod_{O}^{h}(O\,;\,M)$ are pointed by the equivalence classes of the maps $\eta$ and $T_{r}(\eta)$, respectively, and are equipped with the Reedy model category structure (see Theorems \ref{H6} and \ref{H7}).\vspace{9pt}

The maps $\xi$ and $\xi_{r}$ always exist and they are weak equivalences under some conditions on the operad $O$. In \cite{Ducoulombier18}, we introduce the notion of \textit{coherent operad} (resp. $r$\textit{-coherent operad}) which are operads with exactly the properties needed to get weak equivalences (see Section \ref{F1}). The rest of this section is devoted to extend the above result to the context of $k$-fold bimodules and $k$-fold infinitesimal bimodules. More precisely, we prove the following statement: \vspace{9pt}

\begin{thm}\label{D3}
Let $O_{1},\ldots,O_{k}$ be a family of well-pointed $\Sigma$-cofibrant $2$-reduced operads relative to a reduced operad $O$ and let $\eta:\mathbb{O}^{+}\rightarrow M$ be a map of $k$-fold bimodules such that $M(A_{1},\ldots,A_{k})=\ast$ for any elements $(A_{1},\ldots,A_{k})$ with $A_{i}\in \{+\,;\,\emptyset\}$. We also assume the maps of the form\vspace{5pt}
\begin{equation}\label{P9}
\iota_{i}:M(+,\ldots,+,A_{i},+,\ldots,+)\longrightarrow M(\emptyset,\ldots,\emptyset,A_{i},\emptyset,\ldots,\emptyset),\vspace{5pt}
\end{equation}
induced by the $k$-fold bimodule structure, are weak equivalences. If $M$ is $\Sigma$-cofibrant, then there exist explicit continuous maps \vspace{5pt}
\begin{equation}\label{D1}
\begin{array}{rll}
\gamma:\mathbb{F}_{\vec{O}}(M) &  \longrightarrow &  Ibimod_{\vec{O}}^{h}(\mathbb{O}\,;\,M^{-}),
\end{array}
\end{equation}
\begin{equation}\label{D0}
\begin{array}{rll}
 \gamma_{\vec{r}}: T_{\vec{r}}\,\mathbb{F}_{\vec{O}}(M) & \longrightarrow  &  T_{\vec{r}}\,Ibimod_{\vec{O}}^{h}(\mathbb{O}\,;\,M^{-}),
\end{array} \vspace{5pt}
\end{equation}
with $\vec{r}=(r_{1},\ldots,r_{k})\in \mathbb{N}^{k}$. Furthermore, if the operads $O_{1},\ldots,O_{k}$ are $j_{1}$-coherent,..., $j_{k}$-coherent, respectively, then the map (\ref{D0}) is a weak equivalence for any $\vec{r}\leq \vec{j}=(j_{1},\ldots,j_{k})$. In particular, if the operads are coherent, then the map (\ref{D1}) is a weak equivalence.\vspace{7pt}
\end{thm}

In the next sub-section we give a description of the spaces $\mathbb{F}_{\vec{O}}(M)$ and $T_{\vec{r}}\,\mathbb{F}_{\vec{O}}(M)$ in terms of derived mapping spaces of $k$-fold bimodules. Then we introduce an alternative cofibrant resolution of $\mathbb{O}$ as a $k$-fold infinitesimal bimodule in order to build explicitly the maps (\ref{D1}) and (\ref{D0}). In Section \ref{O5}, we recall the notion of \textit{coherent operad}. Finally, the last subsections are devoted to the proof of Theorem \ref{D3}. As a direct consequence of Theorem \ref{D3} and the definitions of the spaces $\mathbb{F}_{\vec{O}}(M)$ and $T_{\vec{r}}\,\mathbb{F}_{\vec{O}}(M)$, one has also the following corollary:\vspace{7pt}

\begin{cor}
If $O_{1}=\cdots=O_{k}=O$ is a well pointed $\Sigma$-cofibrant operad and $\eta:\mathbb{O}^{+}\rightarrow M$ is a map of $k$-fold bimodules satisfying the condition (\ref{P9}) and such that $M(A_{1},\ldots,A_{k})=\ast$ for any elements $(A_{1},\ldots,A_{k})$ with $A_{i}\in \{+\,;\,\emptyset\}$. If $M$ is $\Sigma$-cofibrant, then there exist explicit maps\vspace{3pt}
$$
\begin{array}{rll}\vspace{9pt}
 \gamma: Map_{\ast}\big( \Sigma O(2) \,;\, ,Bimod_{\vec{O}}^{h}(\mathbb{O}^{+}\,;\,M)\big) & \longrightarrow  &  Ibimod_{\vec{O}}^{h}(\mathbb{O}\,;\,M^{-}),\\
 \gamma_{\vec{r}}: Map_{\ast}\big( \Sigma O(2) \,;\, T_{\vec{r}}\,Bimod_{\vec{O}}^{h}(\mathbb{O}^{+}\,;\,M)\big) & \longrightarrow  &  T_{\vec{r}}\,Ibimod_{\vec{O}}^{h}(\mathbb{O}\,;\,M^{-}),
\end{array} \vspace{3pt}
$$
with $\vec{r}=(r_{1},\ldots,r_{k})\in \mathbb{N}^{k}$. Furthermore, if the operad $O$ is $j$-coherent, then the map $\gamma_{\vec{r}}$ is a weak equivalence for any $\vec{r}\leq \vec{j}=(j,\ldots,j)$. In particular, if the operad is coherent, then the map $\gamma$ is a weak equivalence.
\end{cor}

\subsection{Construction of the spaces $\mathbb{F}_{\vec{O}}(M)$ and $T_{\vec{r}}\,\mathbb{F}_{\vec{O}}(M)$}\label{P1}

Let $O_{1},\ldots,O_{k}$ be a family of $2$-reduced operads relative to a reduced operad $O$ and let $\eta:\mathbb{O}^{+}\rightarrow M$ be a map of $k$-fold bimodules satisfying the conditions of Theorem \ref{D3}. For any $i\in \{1,\ldots,k\}$, the $k$-fold sequences given by \vspace{5pt}
\begin{equation*}\label{P5}
M_{i}(A)= M( \underset{i-1}{\underbrace{+,\ldots,+}},\,A,\,\underset{k-i}{\underbrace{+,\ldots,+}}),\hspace{15pt} \forall A\in \Sigma,
\vspace{6pt}
\end{equation*}
inherits a $1$-fold bimodule structure over $O_{i}$ from $M$. Furthermore, the map $\eta$ provides a bimodule map\vspace{5pt}
$$
\eta_{i}:O_{i}\longrightarrow M_{i}.\vspace{5pt}
$$
We also introduce the following notation:\vspace{5pt}
$$
\vec{O}(2):= \vec{O}_{S}(\{1\,;\,2\},\ldots,\{1\,;\,2\}). \vspace{4pt}
$$
Moreover, for any $i\in \{1,\ldots,k\}$, there is a map between suspension spaces \vspace{5pt}
\begin{equation}\label{P3}
\begin{array}{clcl}\vspace{5pt}
 \delta_{i}: & \textstyle\sum \vec{O}(2) & \longrightarrow & \textstyle\sum O_{i}(2); \\ 
 & [x_{1},\ldots,x_{k}\,;\,t] & \longmapsto & [x_{i}\,;\,t].
\end{array} 
\vspace{5pt}
\end{equation}

Let $C_{k}$ be the category whose objects are the integers $0,1,\ldots, k$ and the pairs of the form $(0\,,\,i)$ with $i\in \{1,\ldots,k\}$. The set of morphism $C_{k}(\alpha\,;\,\beta)$ is reduced to one element if $\alpha=\beta$ or $\alpha\in \{0\,;\,i\}$ and $\beta=(0\,,\,i)$. Otherwise it is the empty set. Then we consider the functor $F_{\vec{O}\,;\,M}$ from the category $C_{k}$ to spaces given by \vspace{5pt}
\begin{equation}\label{F2}
\begin{array}{rcl}\vspace{12pt}
F_{\vec{O}\,;\,M}(0) & := & Map_{\ast}\left( 
\textstyle\sum \vec{O}(2) \,;\, Bimod_{\vec{O}}^{h}(\mathbb{O}^{+}\,;\, M)\right), \\ \vspace{12pt}
F_{\vec{O}\,;\,M}(i) & := & Map_{\ast}\left( 
\textstyle\sum O_{i}(2) \,;\, Bimod_{O_{i}}^{h}(O_{i}\,;\, M_{i})\right), \hspace{15pt} \text{with } i\in \{1,\ldots,k\}, \\ 
F_{\vec{O}\,;\,M}(0\,,\,i) & := &  Map_{\ast}\left( 
\textstyle\sum \vec{O}(2) \,;\, Bimod_{O_{i}}^{h}(O_{i}\,;\, M_{i})\right),
\end{array} \vspace{5pt}
\end{equation}
where the derived mapping spaces of bimodules are pointed by the equivalence class of the maps $\eta$ and $\eta_{i}$. On morphisms, there are two cases to consider. If the morphism is of the form $f\in C_{k}(0\,;\,(0\,,\,i))$, then $F_{\vec{O}\,;\,M}(f)$ is obtained from the map  (\ref{P3}).  If the morphism is of the form $f\in C_{k}(i\,;\,(0\,,\,i))$, then $F_{\vec{O}\,;\,M}(f)$ is obtained from the restriction map\vspace{5pt}
$$
rest_{i}:Bimod_{\vec{O}}^{h}(\mathbb{O}^{\,+}\,;\, M) \longrightarrow Bimod_{O_{i}}^{h}(O_{i}\,;\, M_{i}).\vspace{5pt}
$$

\noindent In particular, if $k=2$, then one has the diagram\vspace{5pt}
$$
\xymatrix@R=10pt@C=90pt{
 Map_{\ast}\left( 
 \sum O_{1}(2)\,;\, Bimod_{O_{1}}^{h}(O_{1}\,;\, M_{1})\right)\ar[rd]
&  \\
&  Map_{\ast}\left( 
\sum  \vec{O}(2)\,;\, Bimod_{O_{1}}^{h}(O_{1}\,;\, M_{1})\right)\\
 Map_{\ast}\left( 
\sum \vec{O}(2)\,;\, Bimod_{\vec{O}}^{h}(\mathbb{O}^{+}\,;\, M)\right) \ar[ru]\ar[rd]
 & \\
 &  Map_{\ast}\left( 
\sum \vec{O}(2) \,;\, Bimod_{O_{2}}^{h}(O_{2}\,;\, M_{2})\right)\\
  Map_{\ast}\left( 
\sum O_{2}(2)\,;\, Bimod_{O_{2}}^{h}(O_{2}\,;\, M_{2})\right)\ar[ru] &
}\vspace{-5pt}
$$

\newpage

Similarly, for $\vec{r}=(r_{1},\ldots,r_{k})\in \mathbb{N}^{k}$, we define the truncated functor $T_{\vec{r}}\,F_{\vec{O}\,;\,M}$ from the category $C_{k}$ to spaces by considering derived mapping spaces of $\vec{r}$-truncated $k$-fold bimodules instead of $k$-fold bimodules in the definition (\ref{F2}). Finally, the two spaces researched are the following ones:\vspace{5pt}
$$
\mathbb{F}_{\vec{O}}(M)=\hspace{-7pt}\underset{\hspace{20pt}C_{k}}{lim}\,F_{\vec{O}\,;\,M} \hspace{15pt} \text{and}\hspace{15pt} 
T_{\vec{r}}\,\mathbb{F}_{\vec{O}}(M)=\hspace{-7pt}\underset{\hspace{20pt}C_{k}}{lim}\,T_{\vec{r}}\,F_{\vec{O}\,;\,M}.\vspace{-3pt}
$$   

In particular, if $O_{1}=\cdots=O_{k}=O$, then, for any $i\in \{1,\ldots,k\}$, the spaces $\vec{O}(2)=O(2)$ and the maps (\ref{P3}) are the identity maps. Consequently, given a morphism of the form $f\in C_{k}(i\,;\,(0\,,\,i))$, the continuous map $F_{\vec{O}\,;\,M}(f)$ is also the identity map and the above limits can be simplified and take the following form:
$$
\begin{array}{rcl}\vspace{9pt}
\mathbb{F}_{\vec{O}}(M) & = & Map_{\ast}\big( \Sigma O(2)\,;\, Bimod_{\vec{O}}^{h}(\mathbb{O}^{+}\,;\,M)\big), \\ 
T_{\vec{r}}\,\mathbb{F}_{\vec{O}}(M) & = & Map_{\ast}\big( \Sigma O(2)\,;\, T_{\vec{r}}\,Bimod_{\vec{O}}^{h}(\mathbb{O}^{+}\,;\,M)\big).
\end{array} 
$$

\subsection{An alternative resolution for the $k$-fold infinitesimal bimodule $\mathbb{O}$}

In order to give an explicit description of the maps $(\ref{D1})$ and $(\ref{D0})$, we need to change slightly the cofibrant resolution of $\mathbb{O}$ in the category of $k$-fold infinitesimal bimodules introduced in Construction \ref{E2}. This alternative resolution is obtained as a semi-direct product of the resolution $\mathcal{B}^{\Lambda}(\mathbb{O}^{+})$ and a sequence $\mathcal{I}$.  \vspace{7pt}

\begin{defi}\label{E4}\textbf{The intermediate sequence $\mathcal{I}$}

\noindent The space $\mathcal{I}_{\vec{O}}(n)$, also denoted by $\mathcal{I}(n)$ when $\vec{O}$ is understood, consists in labelling by elements in the operads $O_{1},\ldots,O_{k}$ and points in the interval $[0\,,\,1]$ the vertices of an alternative version of the set of $k$-fold pearled trees (see Definition \ref{C8}). More precisely, we denote by $pTree'[n]$ the set formed by families $(T,f_{1},\ldots,f_{k})$ where $T$ is a pearled tree having $n$ leaves and without univalent vertices other than the pearl. The application $f_{i}:E(T)\rightarrow\{external \,;\, internal\}$  satisfies the following conditions:\vspace{3pt}
\begin{itemize}
\item[$\blacktriangleright$] the output edge of the pearl is necessarily $internal$,\vspace{3pt}
\item[$\blacktriangleright$] if the output edge of a vertex is $external$, then its incoming edges are also $external$,\vspace{3pt}
\item[$\blacktriangleright$] if $f_{i}(e)=internal$, then $f_{j}(e)=internal$ for all $j\neq i$ or  $f_{j}(e)=external$ for all $j\neq i$.\vspace{3pt}
\end{itemize}
For any $v\in V(T)$, we denote by $S_{v}=(S_{v}^{1},\ldots,S_{v}^{k})\in \mathcal{P}_{k}(\{1,\ldots,|v|\})$ the element given by the formula\vspace{5pt}
$$
S_{v}^{i}\coloneqq \left\{
\begin{array}{l}\vspace{7pt}
\text{the augmented set } + \text{ if the image of the output edge of } v \text{ by } f_{i} \text{ is } external, \\ 
\big\{ j\in \{1,\ldots,|v|\}\,\,\big|\,\, \text{the image of the } j\text{-th incoming edge of } v \text{ by } f_{i} \text{ is }internal\,\big\} \text{ otherwise}.
\end{array} 
\right.\vspace{5pt}
$$
Finally, the space $\mathcal{I}(n)$ is given by the quotient of the subspace of\vspace{5pt}
$$
\left.\left( \,\,\, \underset{pTree'[n]}{\coprod} \vec{O}_{S_{p}}(S_{p}^{1},\ldots,S_{p}^{k})\times \underset{v\in V(T)\setminus \{p\}}{\prod} \big[ \vec{O}_{S_{v}}(S_{v}^{1},\ldots,S_{v}^{k})\times [0\,,\,1]\big]\,\,\,\right)\,\, \right/\!\!\sim\vspace{5pt}
$$
with the following condition: if two vertices $v$ and $v'$ are connected by an inner edge from $v$ to $v'$ according to the direction toward the pearl, then the real numbers $t_{v}$ and $t_{v'}$ associated to $v$ and $v'$,respectively, satisfy the relation $t_{v}\geq t_{v'}$. By convention, the pearl is indexed by $0$ and the $external$ edges are represented by dotted lines.\vspace{-9pt}
\begin{figure}[!h]
\begin{center}
\includegraphics[scale=0.215]{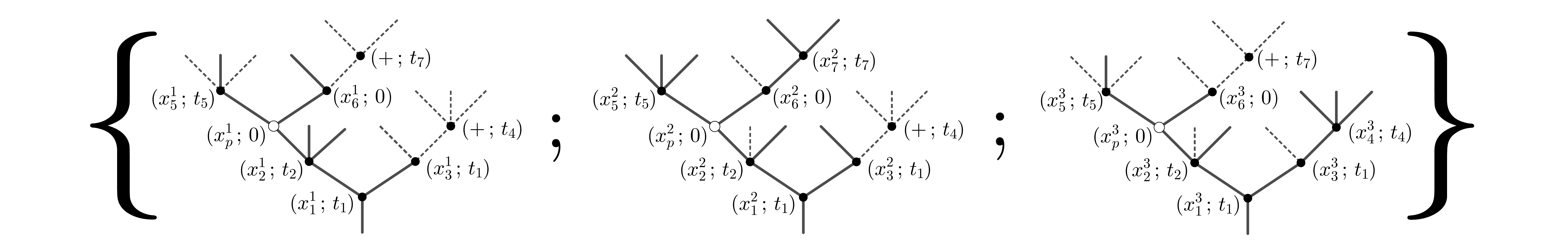}\vspace{-9pt}
\caption{Illustration of a point in $\mathcal{I}(12)$ with $k=3$.}\vspace{-40pt}\label{C9}
\end{center}
\end{figure}

\newpage

The equivalence relation is generated by the unit axiom and the compatibility with the symmetric group action preserving the position of the pearl. Furthermore, if two vertices $v$ and $v'$ are connected by an inner edge $e$ (from $v$ to $v'$ according to the orientation toward the pearl) and if they are indexed by the same real number $t$, then we contract $e$ using the operations (\ref{D6}). The vertex obtained from the contraction is indexed by $t$. We denote by $[T\,;\,\{x_{v}\}\,;\,\{t_{v}\}]$ a point in $\mathcal{I}(n)$. For instance, the point represented in Figure \ref{C9} is equivalent to the following one: \vspace{-5pt}
\begin{figure}[!h]
\begin{center}
\includegraphics[scale=0.215]{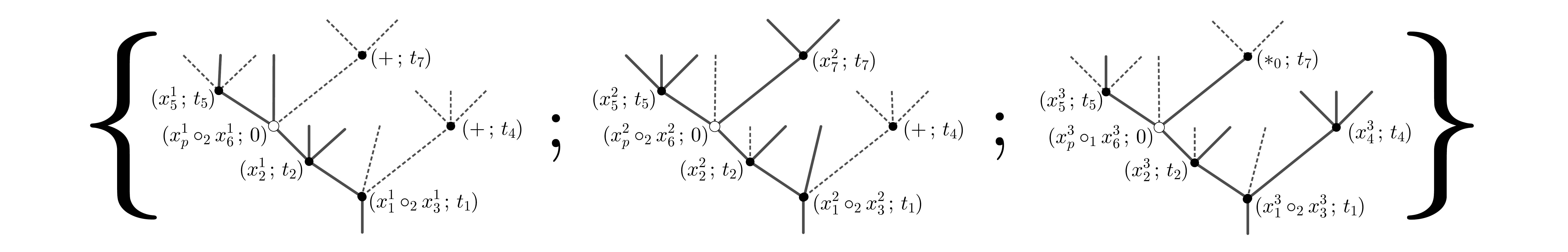}.\vspace{-10pt}
\end{center}
\end{figure}
\end{defi}

\noindent \textbf{The $k$-fold right operation over $\vec{O}_{S}$:} 

\noindent Let $x=[T\,;\,\{x_{v}\}\,;\,\{t_{v}\}]$ be a point in $\mathcal{I}(n)$. For any $i\in \{1,\ldots,n\}$ and $S=\{S_{1},\ldots,S_{k}\}\in\mathcal{P}_{k}(\{1,\ldots,m\})$ such that $S_{j}=+$ if the image of $i$-th leaf of $T$ by the application $f_{j}$ is $external$, there exists a map\vspace{5pt}
$$
\circ^{i}[x]: \vec{O}_{S}(S_{1},\ldots,S_{k})\longrightarrow \mathcal{I}(n+m-1),\vspace{5pt}
$$
which consists in grafting the $m$-corolla indexed by the point in $\vec{O}_{S}(S_{1},\ldots,S_{k})$ and the real number $1$ into the $i$-th leaf of $T$. Then, we extend the application $f_{i}:E(T)\rightarrow \{external\,;\,internal\}$ by labelling the $j$-th leaf the corolla by $internal$ if and only if $j\in S_{i}$. For instance, the composition $\circ^{2}[x]$, where $x$ is the point represented in Figure \ref{C9}, with an element $(\theta_{1},\theta_{2},\theta_{3})\in \vec{O}_{S}(S_{1},S_{2},S_{3})$ associated to the elements $S_{1}=\{1\,;\,3\}$, $S_{2}=\{2\}$ and $S_{3}=\emptyset$ in $\mathcal{P}_{3}(\{1,2,3\})$ gives rise to

\begin{figure}[!h]
\begin{center}
\includegraphics[scale=0.215]{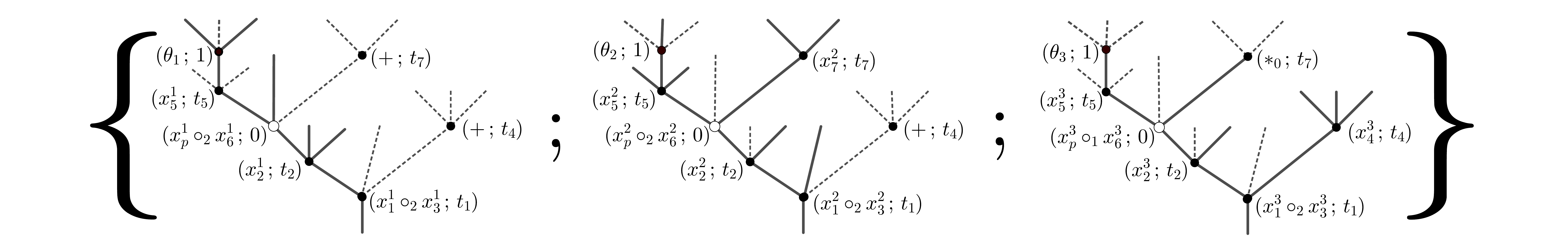}\vspace{-10pt}
\end{center}
\end{figure}

\noindent \textbf{The left $k$-fold infinitesimal operations over $\vec{O}$:} 

\noindent Given $(n_{1},\ldots,n_{k})\in \mathbb{N}^{k}$, there is a map of the form\vspace{5pt}
$$
\mu:\vec{O}(n_{1},\ldots,n_{k}) \times \mathcal{I}(m)\longrightarrow \mathcal{I}(m+n_{1}+\cdots + n_{k}-k),\vspace{5pt}
$$
which consists in grafting the point in $\mathcal{I}(m)$ into the first leaf of the $(n_{1}+\cdots+n_{k}-k+1)$-corolla indexed by the point in $\vec{O}(n_{1},\ldots,n_{k})$ and the real number $1$. Then, we extend the map $f_{i}:E(T)\rightarrow \{internal \,;\, external\}$, associated to the point in $\mathcal{I}(m)$, by labelling the $j$-th leaf of the corolla by $internal$ if and only if one has $j\in\{1,n_{1}+\cdots + n_{i-1}-i+3,\ldots,n_{1}+\cdots+ n_{i}-i+1\}$. For instance, the left $k$-fold infinitesimal  operation, applied to the point in $\mathcal{I}(12)$ represented in Figure \ref{C9} and the point $(\theta_{1},\theta_{2},\theta_{3})\in \vec{O}(2;3;1)$, gives rise to the following element in $\mathcal{I}(15)$:\vspace{0pt}
\begin{figure}[!h]
\begin{center}
\includegraphics[scale=0.215]{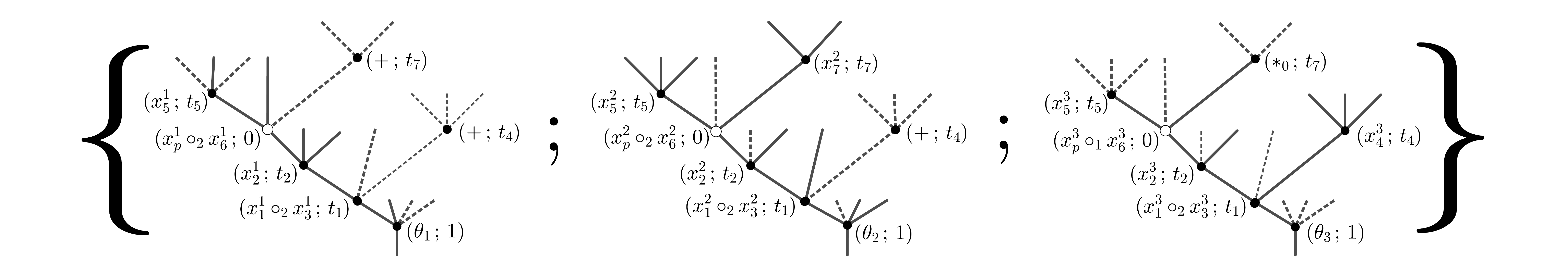}
\end{center}\vspace{-25pt}
\end{figure}

\newpage

\noindent Finally, there is a map \vspace{5pt}
\begin{equation}\label{E8}
\eta'':\mathcal{I}(n)\longrightarrow \underset{S\in \mathcal{P}_{k}(\{1,\ldots,n\})}{\displaystyle \coprod} \vec{O}_{S}(S_{1},\ldots,S_{k}),\vspace{5pt}
\end{equation}
sending a point $[T\,;\,\{\theta_{v}\}\,;\,\{t_{v}\}]$ to the element obtained by fixing the parameters $\{t_{v}\}$ to $0$ and  indexed by the family of power sets $S=(S_{1},\ldots,S_{k})\in \mathcal{P}_{k}(\{1,\ldots,n\})$ with \vspace{3pt}
$$
S_{i}=\big\{j\in \{1,\ldots,n\}\,\big|\, \text{the image of the } j \text{-th leaf of } T \text{ by } f_{i} \text{ is } internal\big\}.\vspace{7pt}
$$

\begin{const} \label{L1}
The aim of the following construction is to introduce an alternative cofibrant resolution of the $k$-fold infinitesimal bimodule $\mathbb{O}$ by using the $k$-fold bimodule structure of $\mathbb{O}^{+}$. The $k$-fold sequence $\overline{\mathcal{I}b}(\mathbb{O})$ is obtained as a combination of the intermediate sequence $\mathcal{I}$ introduced in Definition \ref{E4} and the Boardman-Vogt resolution $\mathcal{B}^{\Lambda}(\mathbb{O}^{+})$. More precisely, one has\vspace{5pt}
$$
\left.\overline{\mathcal{I}b}(\mathbb{O})(n_{1},\ldots,n_{k})\subset \left( \underset{\substack{l>0,\,\{m_{u}^{i}\}_{1\leq u \leq k}^{1\leq i \leq l}}}{\displaystyle\coprod} \hspace{-5pt}\Sigma_{\vec{n}}\times\mathcal{I}(l)\times \underset{1\leq i\leq l}{\displaystyle \prod} \mathcal{B}^{\Lambda}(\mathbb{O}^{+})(m_{1}^{i},\ldots,m_{k}^{i})\right)\,\,\right/ \!\!\sim , \hspace{10pt} \text{with } \underset{\substack{1\leq i \leq l\\ m_{u}^{i}\neq +}}{\sum} m_{u}^{i}=n_{u},\vspace{5pt}
$$
such that $m_{u}^{i}\in \mathbb{N}\sqcup\{+\}$ and $(m_{1}^{i},\ldots,m_{k}^{i})\neq (+,\ldots, +)$. A point in $\overline{\mathcal{I}b}(\mathbb{O})$ is denoted by $[x\,;\,y^{1},\ldots,y^{l};\sigma]$ where $x\in \mathcal{I}(l)$, $y^{i}=(y_{1}^{i},\ldots,y_{k}^{i})\in \mathcal{B}^{\Lambda}(\mathbb{O}^{+})$ and $\sigma\in \Sigma_{n_{1}}\times \cdots \times \Sigma_{n_{k}}=\Sigma_{\vec{n}}$. Furthermore, such a point satisfies the following condition:\vspace{5pt} 
\begin{itemize}
\item[$\blacktriangleright$] the image of the $i$-th leaf of $x$ by $f_{u}$ is $external$ if and only if $m_{u}^{i}=+$.\vspace{1pt}
\end{itemize}
 
\begin{figure}[!h]
\begin{center}
\includegraphics[scale=0.25]{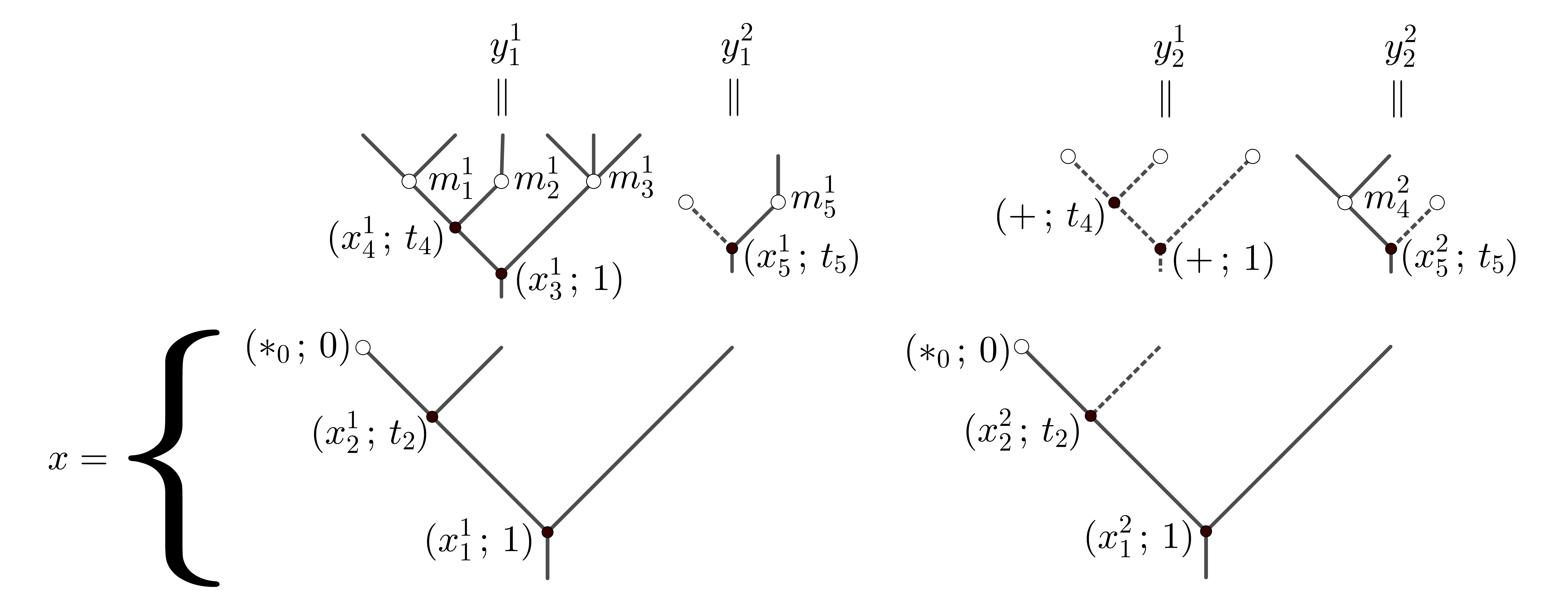}\vspace{-1pt}
\caption{Illustration of a point in $\overline{\mathcal{I}b}(\mathbb{O})(7;2)$.}\label{E5}
\end{center}\vspace{-5pt}
\end{figure}

The equivalence relation is generated by the compatibility with the symmetric group action (preserving the position of the pearl for the pearled trees) as well as the unit axiom. We also extend the relation illustrated in Figure \ref{G3}. If one the leaf of $x$ is associated to a based point in $\mathbb{O}^{+}(n_{1},\ldots,n_{k})$, with $n_{i}\in \{+\,;\,0\}$, then we remove the based point and we contract the corresponding leaf of $x$ using the operadic structures of $O_{1},\ldots,O_{k}$. Furthermore, the equivalence relation is also generated by the following conditions: \vspace{3pt}
\begin{itemize}
\item[$\blacktriangleright$] the relation equalizing the right $\vec{O}$-action on $\mathcal{I}$ and the left $\vec{O}$-action on $\mathcal{B}^{\Lambda}(\mathbb{O}^{+})$: for any element $\vec{\theta}=\vec{O}_{S}(S_{1},\ldots, S_{k})$, one has the identification\vspace{5pt}
\begin{equation}\label{F0}
\big[ x,y^{1},\ldots,\vec{\theta}(y^{i\,;\,1},\ldots,y^{i\,;\,m}),\ldots,y^{l}\big]\sim \big[ \circ^{i}[x](\vec{\theta}),y^{1},\ldots,y^{i\,;\,1},\ldots,y^{i\,;\,m},\ldots,y^{l}\big].\vspace{5pt}
\end{equation}

For instance, the element represented in Figure \ref{E5} is equivalent to the following one:
\begin{figure}[!h]
\begin{center}
\includegraphics[scale=0.25]{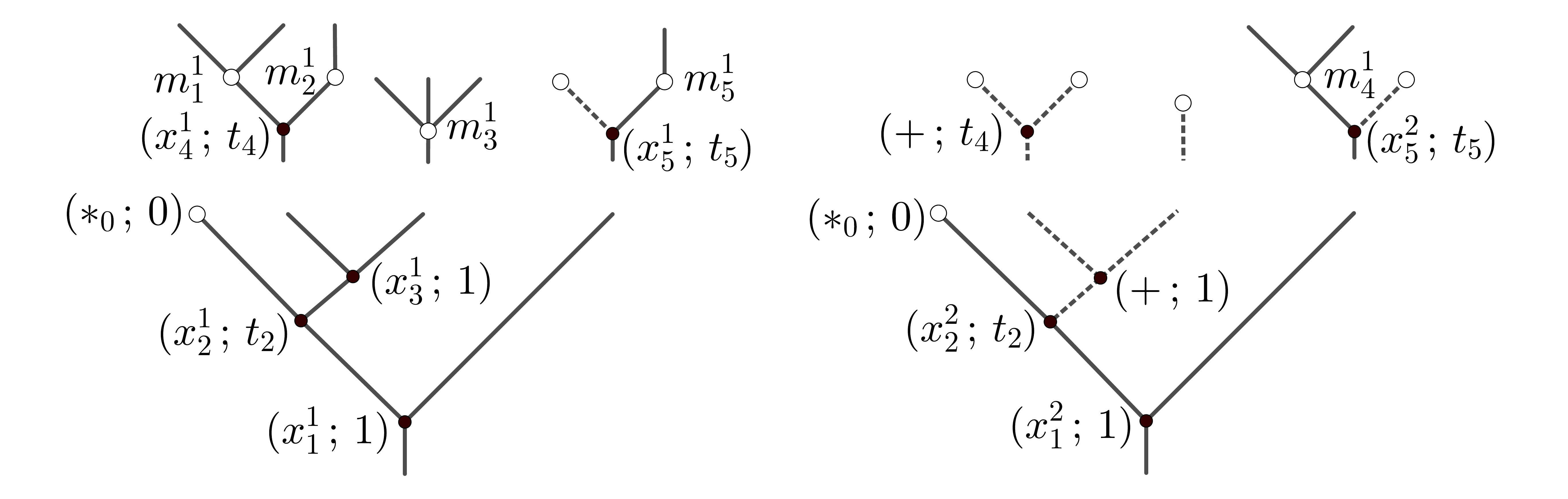}
\end{center}\vspace{-10pt}
\end{figure}

\item[$\blacktriangleright$] the relation used to get a $k$-fold infinitesimal bimodule: if the root of $x\in \mathcal{I}$ is indexed by $1$, then the elements $y_{u}^{i}$ associated to a leaf directly connected to the root of $x$ are identified with their image through the map $(\ref{E3})$. For instance, the point represented in Figure \ref{E5} is equivalent to 
\begin{figure}[!h]
\begin{center}
\includegraphics[scale=0.25]{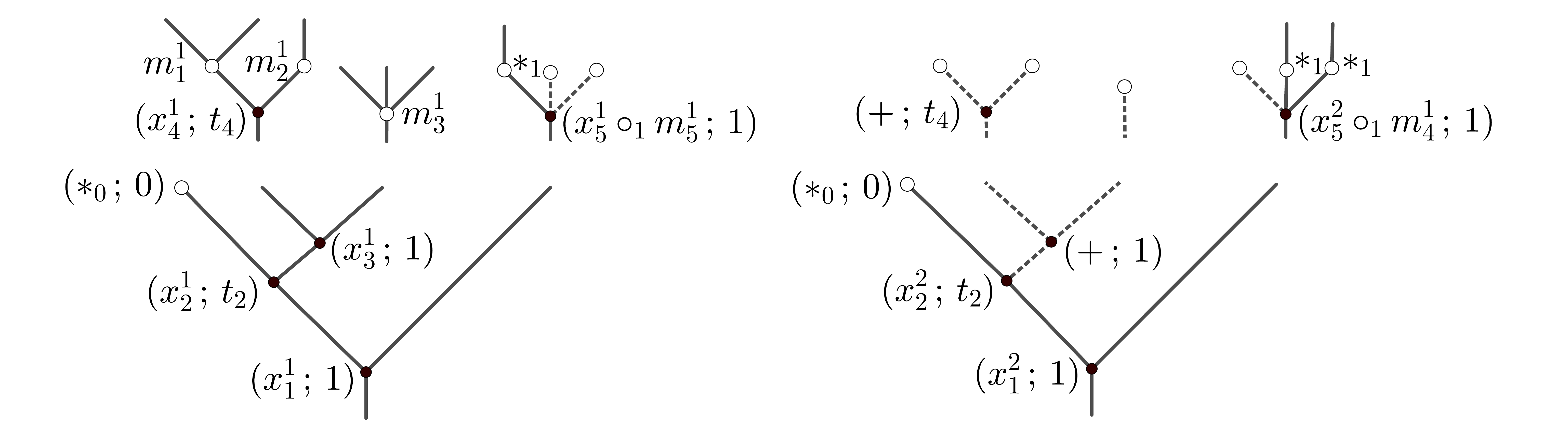}
\end{center}\vspace{-10pt}
\end{figure}
\end{itemize}

The $k$-fold right infinitesimal operations are obtained by using the $k$-fold right bimodule structure of $\mathcal{B}^{\Lambda}(\mathbb{O}^{+})$. The $k$-fold infinitesimal left operations are defined using the left structure of the intermediate sequence $\mathcal{I}$. More precisely, one has \vspace{5pt}
$$
\begin{array}{ccl}\vspace{7pt}
\mu:\vec{O}(n_{1}+1,\ldots,n_{k}+1)\times \overline{\mathcal{I}b}(\mathbb{O})(m_{1},\ldots,m_{k})  & \longrightarrow  & \overline{\mathcal{I}b}(\mathbb{O})(n_{1}+m_{1},\ldots,n_{k}+m_{k}); \\ 
\vec{\theta}\,;\, \big[ x;y^{1},\ldots,y^{l};\sigma\big] & \longmapsto & \big[ \mu(\vec{\theta}\,;\,x)\,;\,y^{1},\ldots,y^{l},z^{1},\ldots,z^{n_{1}+\cdots+n_{k}};\sigma\big],
\end{array} \vspace{5pt}
$$
with $z^{i}=(z^{i}_{1},\ldots,z^{i}_{k})$ given by $z_{l}^{i}=
\left\{
\begin{array}{cl}\vspace{5pt}
\includegraphics[scale=0.09]{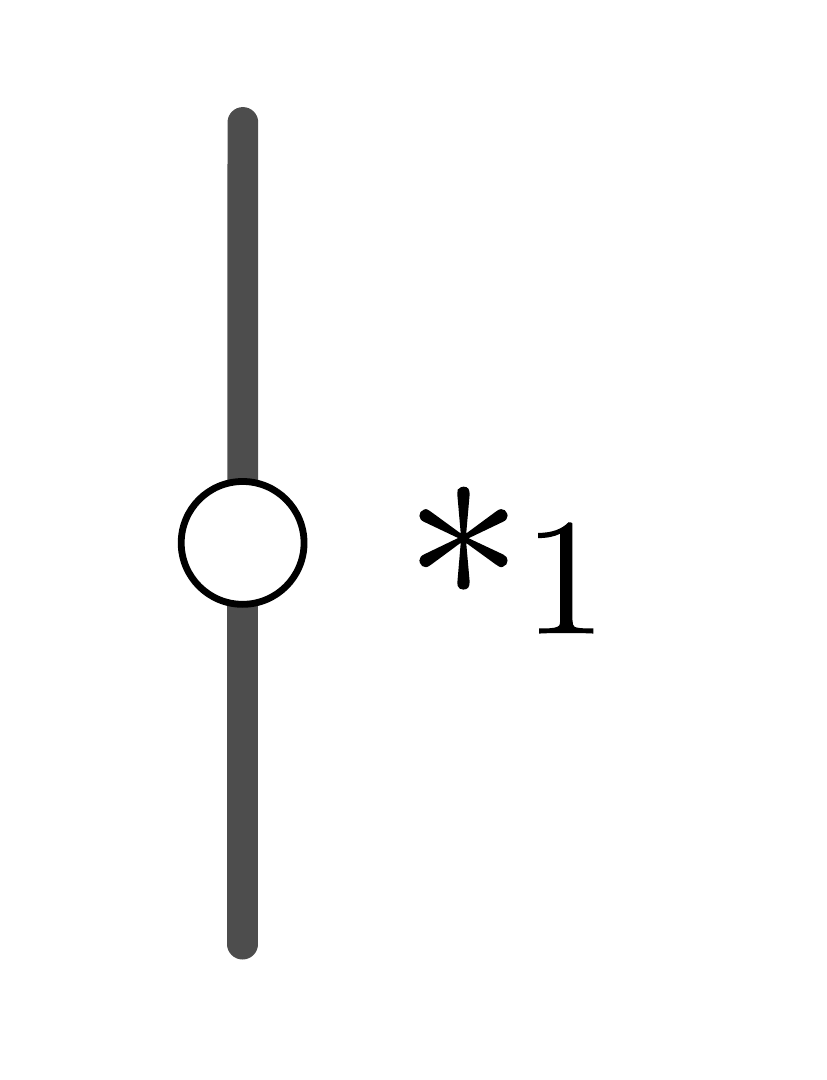}&  \text{if } i\in\{n_{1}+\cdots + n_{l-1}+1,\ldots,n_{1}+\cdots+n_{l}\}, \\ 
\includegraphics[scale=0.09]{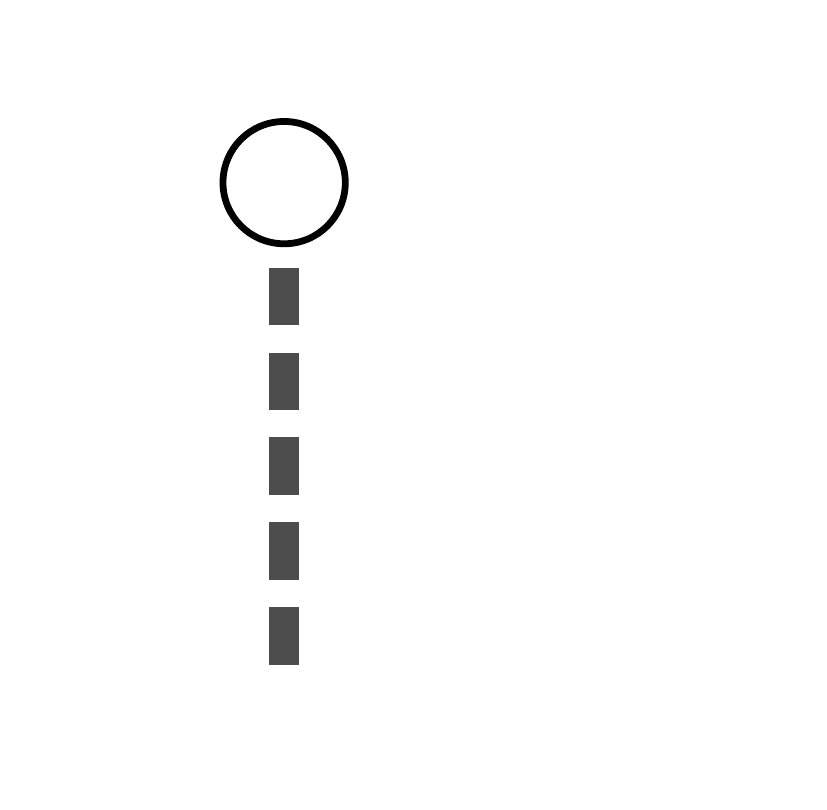} & \text{otherwise}.
\end{array} 
\right.\vspace{9pt}
$

Finally, we consider a map of $k$-fold infinitesimal bimodules from $\overline{\mathcal{I}b}(\mathbb{O})$ to $\mathbb{O}$ using the map of $k$-fold bimodules (\ref{C7}) and the map of $k$-fold sequences (\ref{E8}):\vspace{5pt}
\begin{equation}\label{E6}
\begin{array}{rcl}\vspace{7pt}
\overline{\eta}:\overline{\mathcal{I}b}(\mathbb{O})(n_{1},\ldots,n_{k}) & \longrightarrow & \mathbb{O}(n_{1},\ldots,n_{k}); \\ 
\left[ x\,;\,y^{1},\ldots,y^{l}\,;\,\sigma\right] & \longmapsto & \big( \eta''(x)(\eta'(y^{1}),\ldots,\eta'(y^{l}))\big)\cdot\sigma.
\end{array} 
\end{equation}\vspace{1pt}
\end{const}

Similarly to the previous sections, we introduce a filtration of $\overline{\mathcal{I}b}(\mathbb{O})$ according to the number of  inputs. We say that an element is prime if it is not obtained as a result of the infinitesimal bimodule structure. In other words, the root of $x$ as well as the vertices above the sections of $y^{1},\ldots y^{l}$ are indexed by real numbers strictly smaller than $1$. Otherwise, the element is said to be composite and can be assigned to a prime component by removing the root of $x$ (if the letter one is indexed by $1$) together with all $y^{i}$ such that the $i$-th leaf of $x$ is directly connected to the root; and, by removing all vertices from the element $y^{j}$ that are above the sections and indexed by $1$. For instance, the prime component associated to the composite point represented in Figure \ref{E5} is the following one:
\begin{figure}[!h]
\begin{center}
\includegraphics[scale=0.32]{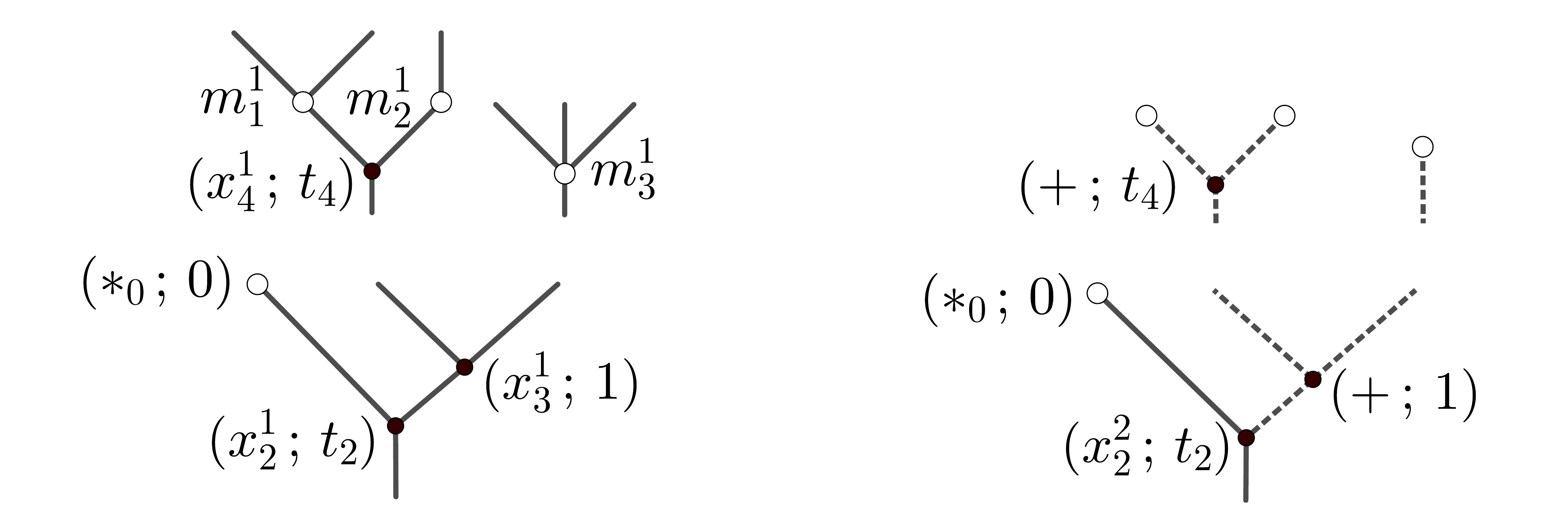}\vspace{-5pt}
\end{center}
\end{figure}

Let $\vec{r}=(r_{1},\ldots,r_{k})\in \mathbb{N}^{k}$ and $[x\,;\,y^{1},\ldots,y^{l}\,;\,\sigma]$ be a prime point where $y_{u}^{i}$ is indexed by the $k$-fold tree with section $(T^{i}_{1},\ldots,T^{i}_{k})$. The prime point is in the $\vec{r}$-th filtration term $\overline{\mathcal{I}b}(\mathbb{O})[\,\vec{r}\,]$ if, for each $u\in\{1,\ldots,k\}$, the sum of the inputs of the trees with section $T^{1}_{u},\ldots,T^{l}_{u}$ is smaller than $r_{u}$. Besides, a composite point is in the $\vec{r}$-th filtration term if its prime component is in $\overline{\mathcal{I}b}(\mathbb{O})[\,\vec{r}\,]$. For instance, the composite point represented in Figure \ref{E5} is in the filtration term $\overline{\mathcal{I}b}(\mathbb{O})[\,3\,;\,3\,]$. By construction, $\overline{\mathcal{I}b}(\mathbb{O})[\,\vec{r}\,]$ is a $k$-fold infinitesimal bimodule and, for each pair $\vec{m}\leq \vec{n}$, there is a map \vspace{5pt}
\begin{equation}\label{P6}
\iota[\,\vec{m}\leq \vec{n}\,]:\overline{\mathcal{I}b}(\mathbb{O})[\,\vec{m}\,]\longrightarrow \overline{\mathcal{I}b}(\mathbb{O})[\,\vec{n}\,].\vspace{10pt}
\end{equation}

\begin{thm}\label{M4}
If the operads $O_{1},\ldots,O_{k}$ are well pointed and $\Sigma$-cofibrant, then the $k$-fold (truncated) infinitesimal bimodules $\overline{\mathcal{I}b}(\mathbb{O})$ and $T_{\vec{r}}\,(\overline{\mathcal{I}b}(\mathbb{O})[\,\vec{r}\,])$ are cofibrant replacements of $\mathbb{O}$ and $T_{\vec{r}}\,\mathbb{O}$ in the Reedy model categories $Ibimod_{\vec{O}}$ and $T_{\vec{r}}\,Ibimod_{\vec{O}}$, respectively. In particular, the map (\ref{E6}) is a weak equivalence and the maps of the form (\ref{P6}) are cofibrations.  \vspace{5pt}
\end{thm}

\begin{proof}
First, we show that the map (\ref{E6}) is a weak equivalence in the category of $k$-fold infinitesimal bimodules. By construction of the Reedy model category structure, we only need to check that the map (\ref{E6}) is a weak equivalence as a map of $k$-fold sequences. For this purpose, we consider the map of $k$-fold sequences\vspace{3pt}
\begin{equation}\label{E7}
\begin{array}{rcl}\vspace{7pt}
\iota:\mathbb{O}(n_{1},\ldots,n_{k}) & \longrightarrow & \overline{\mathcal{I}b}(\mathbb{O})(n_{1},\ldots,n_{k}); \\ 
(x_{1},\ldots,x_{k}) & \longmapsto & \left[ x\,;\,y^{1},\ldots,y^{n_{1}+\cdots+n_{k}}\,;\,id\right].
\end{array} \vspace{3pt}
\end{equation}
The point $x$ is the pearled $(n_{1}+\cdots+n_{k})$-corolla indexed by  $(x_{1},\ldots,x_{k})$ seen as a point in $\vec{O}_{S}(S_{1},\ldots,S_{k})$ where $S=(S_{1},\ldots,S_{k})$ is the element in $\mathcal{P}_{k}(n_{1}+\cdots+n_{k})$ given by \vspace{3pt}
$$
S_{i}=\{n_{1}+\cdots+n_{i-1}+1,\ldots,n_{1}+\cdots+n_{i}\}.\vspace{3pt}
$$
The point $y^{i}=(y_{1}^{i},\ldots,y_{k}^{i})\in \mathcal{B}^{\Lambda}(\mathbb{O}^{+})$ is defined as follows:\vspace{3pt}
$$
y_{u}^{i}=
\left\{
\begin{array}{cl}\vspace{5pt}
\includegraphics[scale=0.1]{dessin41.pdf}&  \text{if } i\in\{n_{1}+\cdots + n_{u-1}+1,\ldots,n_{1}+\cdots+n_{u}\}, \\ 
\includegraphics[scale=0.1]{dessin42.pdf} & \text{otherwise}.
\end{array} 
\right.\vspace{3pt}
$$

The maps (\ref{E6}) and (\ref{E7}) give rise to a homotopy retract. Indeed, let $A$ be the $k$-fold sequence formed by points $[x\,;\,y^{1},\ldots,y^{l}\,;\,id]\in \overline{\mathcal{I}b}(\mathbb{O})$ where $y^{i}=(y^{i}_{1},\ldots,y^{i}_{k})$ has only one parameter which is a $1$-corolla indexed by the unit of the operads whereas the other parameters are $0$-corollas with an $external$ output. Furthermore, if $y^{i}_{u}$ is a $1$-corolla, then the other elements $y^{i'}_{u}$ , with $i'\neq i$, are necessarily $0$-corollas. As a consequence of the homotopy introduced in the proof of Lemma \ref{E9} and the relation $(\ref{F0})$, the $k$-fold sequences $\overline{\mathcal{I}b}(\mathbb{O})$ and $A$ are homotopically equivalent. For instance, the homotopy sends the point represented in Figure \ref{E5} to the following element in $A$:\vspace{-3pt}

\begin{figure}[!h]
\begin{center}
\includegraphics[scale=0.27]{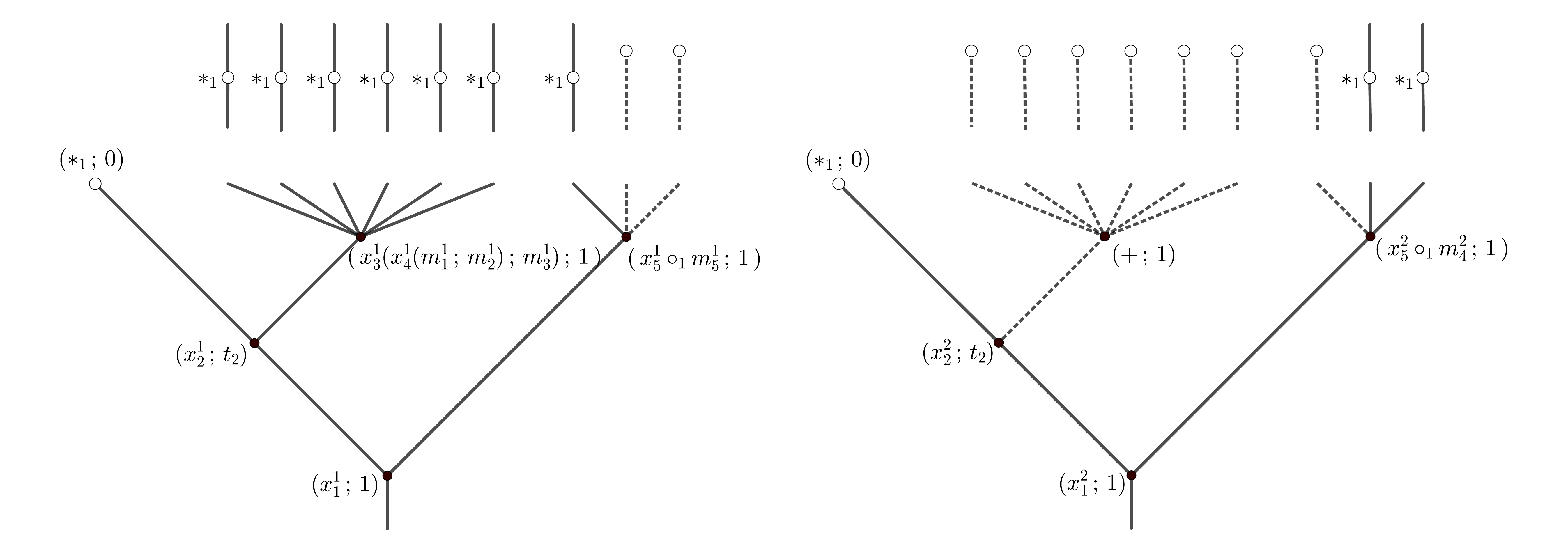}\vspace{-10pt}
\end{center}
\end{figure}

Then, the second part of the homotopy consists in bringing the real numbers indexing the point in $A$ to $0$. Such a point can be identified with an element in the image of the application $(\ref{E7})$. In particular, the homotopy sends the point represented in Figure \ref{E5} to the following element:

\begin{figure}[!h]
\begin{center}
\includegraphics[scale=0.2]{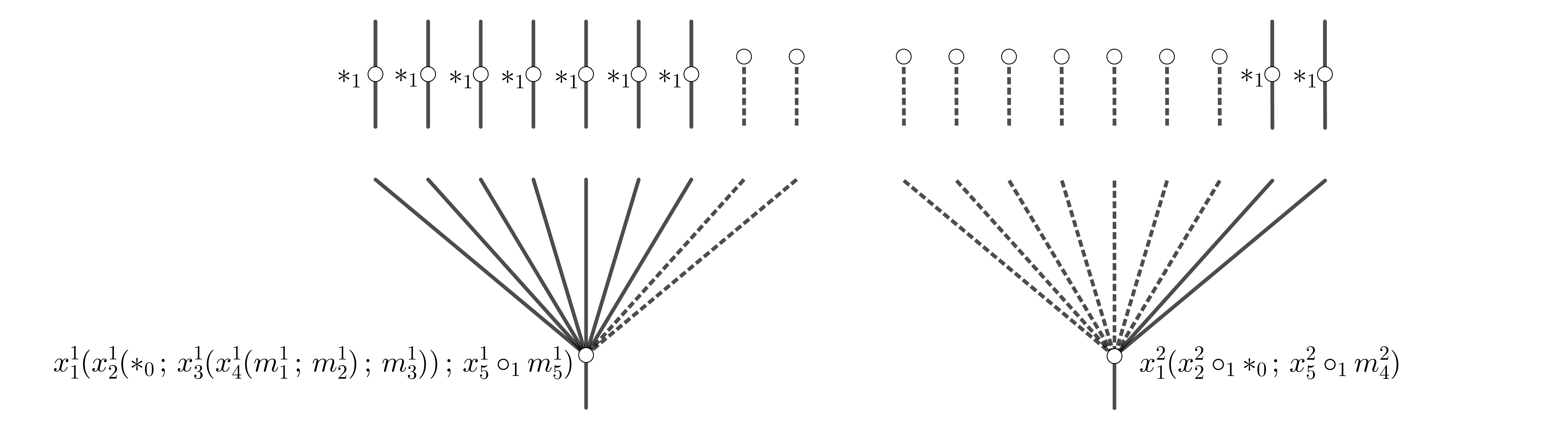}\vspace{-10pt}
\end{center}
\end{figure}

We refer the reader to \cite{Ducoulombier18} for a detailed account of the fact that $\overline{\mathcal{I}b}(\mathbb{O})$ and its truncated version are cofibrant in the appropriate model category. The idea is to check that the map $\iota[\,\vec{m}\leq \vec{n}\,]$ is a cofibration in the category of $k$-fold infinitesimal bimodules over $\vec{O}_{>0}$ (see Theorem \ref{K8}). Similarly to the proof of Theorem \ref{L9}, we assume that $k=2$, $\vec{m}=(m_{1}\,;\,m_{2})$ and $\vec{n}=(m_{1}+1\,;\,m_{2})$. In that case, the map $\iota[\,\vec{m}\leq \vec{n}\,]$ is obtained as a sequence of $k$-fold infinitesimal bimodule maps\vspace{3pt}
\begin{equation}\label{O6}
\xymatrix{
\overline{\mathcal{I}b}(\mathbb{O})[\,\vec{m}\,] \ar[r] & \overline{\mathcal{I}b}_{+}(\mathbb{O})[\,\vec{n}\,] \ar[r] & \overline{\mathcal{I}b}_{0}(\mathbb{O})[\,\vec{n}\,]\ar[r] & \cdots \ar[r] & \overline{\mathcal{I}b}_{m_{2}}(\mathbb{O})[\,\vec{n}\,]=\overline{\mathcal{I}b}(\mathbb{O})[\,\vec{n}\,]
}
\end{equation}
where the $k$-fold infinitesimal bimodule $\overline{\mathcal{I}b}_{+}(\mathbb{O})[\,\vec{n}\,]$ is defined using the pushout diagram
\begin{equation*}
\xymatrix{
\mathcal{F}_{\mathcal{I}b\,;\,\vec{O}_{>0}}\big(\, \partial\overline{\mathcal{I}b}(\mathbb{O})(\,m_{1}+1\,;\,+\,)\,\big) \ar[r] \ar[d] & \mathcal{F}_{Ib\,;\,\vec{O}_{>0}}\big(\, \overline{\mathcal{I}b}(\mathbb{O})(\,m_{1}+1\,;\,+\,)\big) \ar[d]\\
\overline{\mathcal{I}b}(\mathbb{O})[\,\vec{m}\,] \ar[r] & \overline{\mathcal{I}b}_{+}(\mathbb{O})[\,\vec{n}\,]
}
\end{equation*}
where the space $\overline{\mathcal{I}b}(\mathbb{O})(\,m_{1}+1\,;\,+\,)$, seen as a sequence concentrated in arity $(m_{1}+1\,;\,0)$, is formed by points $[x\,;\,y^{1},\ldots,y^{l}\,;\,\sigma]$ in $\overline{\mathcal{I}b}(\mathbb{O})(\,m_{1}+1\,;\,0\,)$ for which the application $f_{2}$ associated to $x$ labels by $external$ the edges other than the edges composing the path joining the pearl to the root. Then, $\partial\overline{\mathcal{I}b}(\mathbb{O})(\,m_{1}+1\,;\,+\,)$ is the sub-sequence formed by points $[x\,;\,y^{1},\ldots,y^{l}\,;\,\sigma]$ having a decomposition of the form \vspace{5pt}
\begin{equation*}
[x\,;\,y^{1},\ldots,y^{l}\,;\,\sigma]=\left\{
\begin{array}{ll}\vspace{5pt}
\mu\big(\theta\,;\,[\tilde{x}\,;\,\tilde{y}^{1},\ldots,\tilde{y}^{l}\,;\,id]\big)\cdot \sigma & \text{with } \theta\in \vec{O}, \\ \vspace{5pt}
\text{or} \\ 
\left[\tilde{x}\,;\,\tilde{y}^{1},\ldots,\tilde{y}^{l}\,;\,id\right]\circ_{i}^{j}
\theta & \text{with } \theta\in O_{i},
\end{array} 
\right.\vspace{5pt}
\end{equation*}
where $[\tilde{x}\,;\,\tilde{y}^{1},\ldots,\tilde{y}^{l}\,;\,id]\in \overline{\mathcal{I}b}(\mathbb{O})(\,l_{1}\,;\,l_{2}\,)$ with $l_{1}\leq m_{1}$. \vspace{9pt}

The $k$-fold infinitesimal bimodule $\overline{\mathcal{I}b}_{0}(\mathbb{O})[\,\vec{n}\,]$ is built using the pushout diagram\vspace{5pt}
\begin{equation*}
\xymatrix{
\mathcal{F}_{\mathcal{I}b\,;\,\vec{O}_{>0}}\big(\, \partial\overline{\mathcal{I}b}(\mathbb{O})(\,m_{1}+1\,;\,0\,)\,\big) \ar[r] \ar[d] & \mathcal{F}_{Ib\,;\,\vec{O}_{>0}}\big(\, \overline{\mathcal{I}b}(\mathbb{O})(\,m_{1}+1\,;\,0\,)\big) \ar[d]\\
\overline{\mathcal{I}b}_{+}(\mathbb{O})[\,\vec{n}\,] \ar[r] & \overline{\mathcal{I}b}_{0}(\mathbb{O})[\,\vec{n}\,]
}\vspace{3pt}
\end{equation*} 
where the space $\partial\overline{\mathcal{I}b}(\mathbb{O})(\,m_{1}+1\,;\,0\,)$ is the sub-sequence of $\overline{\mathcal{I}b}(\mathbb{O})(\,m_{1}+1\,;\,0\,)$ formed by points $[x\,;\,y^{1},\ldots,y^{l}\,;\,\sigma]$ in $\overline{\mathcal{I}b}(\mathbb{O})(\,m_{1}+1\,;\,+\,)$ or having a decomposition of the form \begin{equation*}
[x\,;\,y^{1},\ldots,y^{l}\,;\,\sigma]=\left\{
\begin{array}{ll}\vspace{5pt}
\mu\big(\theta\,;\,[\tilde{x}\,;\,\tilde{y}^{1},\ldots,\tilde{y}^{l}\,;\,id]\big)\cdot \sigma & \text{with } \theta\in \vec{O}, \\ \vspace{5pt}
\text{or} \\ 
\left[\tilde{x}\,;\,\tilde{y}^{1},\ldots,\tilde{y}^{l}\,;\,id\right]\circ_{i}^{j}
\theta & \text{with } \theta\in O_{i},
\end{array} 
\right.\vspace{5pt}
\end{equation*}
where $[\tilde{x}\,;\,\tilde{y}^{1},\ldots,\tilde{y}^{l}\,;\,id]\in \overline{\mathcal{I}b}(\mathbb{O})(\,l_{1}\,;\,l_{2}\,)$ with $l_{1}\leq m_{1}$ or $l_{1}= m_{1}+1$ and $l_{2}=+$. \vspace{5pt}

Finally, the $k$-fold infinitesimal bimodules $\overline{\mathcal{I}b}_{i}(\mathbb{O})[\,\vec{n}\,]$ are obtained by induction using pushout diagrams of the form \vspace{5pt}
\begin{equation*}
\xymatrix{
\mathcal{F}_{\mathcal{I}b\,;\,\vec{O}_{>0}}\big(\, \partial\overline{\mathcal{I}b}(\mathbb{O})(\,m_{1}+1\,;\,i\,)\,\big) \ar[r] \ar[d] & \mathcal{F}_{Ib\,;\,\vec{O}_{>0}}\big(\, \overline{\mathcal{I}b}(\mathbb{O})(\,m_{1}+1\,;\,i\,)\big) \ar[d]\\
\overline{\mathcal{I}b}_{i-1}(\mathbb{O})[\,\vec{n}\,] \ar[r] & \overline{\mathcal{I}b}_{i}(\mathbb{O})[\,\vec{n}\,]
}\vspace{5pt}
\end{equation*} 
where the space $\overline{\mathcal{I}b}(\mathbb{O})(\,m_{1}+1\,;\,i\,)$ is seen as a sequence concentrated in arity $(m_{1}+1\,;\,i)$ and $\partial\overline{\mathcal{I}b}(\mathbb{O})(\,m_{1}+1\,;\,i\,)$ is the sub-sequence formed by points $[x\,;\,y^{1},\ldots,y^{l}\,;\,\sigma]$ having a decomposition of the form \vspace{5pt}
\begin{equation*}
[x\,;\,y^{1},\ldots,y^{l}\,;\,\sigma]=\left\{
\begin{array}{ll}\vspace{5pt}
\mu\big(\theta\,;\,[\tilde{x}\,;\,\tilde{y}^{1},\ldots,\tilde{y}^{l}\,;\,id]\big)\cdot \sigma & \text{with } \theta\in \vec{O}, \\ \vspace{5pt}
\text{or} \\ 
\left[\tilde{x}\,;\,\tilde{y}^{1},\ldots,\tilde{y}^{l}\,;\,id\right]\circ_{i}^{j}
\theta & \text{with } \theta\in O_{i},
\end{array} 
\right.\vspace{5pt}
\end{equation*}
where $[\tilde{x}\,;\,\tilde{y}^{1},\ldots,\tilde{y}^{l}\,;\,id]\in \overline{\mathcal{I}b}(\mathbb{O})(\,l_{1}\,;\,l_{2}\,)$ with $l_{1}\leq m_{1}$ or $l_{1}= m_{1}+1$ and $l_{2}<i$. By induction on the number of vertices, the inclusion from $\partial\overline{\mathcal{I}b}(\mathbb{O})(m_{1}+1\,;\,i)$  to $\overline{\mathcal{I}b}(\mathbb{O})(m_{1}+1\,;\,i)$  is a cofibration in the category of sequences. Since the free $k$-fold infinitesimal bimodule functor and pushout diagrams preserve cofibrations, the horizontal maps in the above diagrams are cofibrations in the projective model category of $k$-fold infinitesimal bimodules over $\vec{O}_{>0}$. Thus proves that (\ref{O6}) is a sequence of cofibrations and $\iota[\vec{m}\leq \vec{n}]$ is a cofibration in the Reedy model category.
\end{proof}

\newpage

\subsection{Relation between $k$-fold infinitesimal bimodules and $k$-fold bimodules}\label{F1}

As a consequence of the explicit resolution $\mathcal{B}^{\Lambda}(\mathbb{O}^{+})$ of $\mathbb{O}^{+}$ as a $k$-fold bimodule, introduced in Section \ref{F5}, as well as the alternative resolution $\overline{\mathcal{I}b}(\mathbb{O})$ of $\mathbb{O}$ as a $k$-fold infinitesimal bimodule, introduced in the previous section, we are able to describe the map \vspace{3pt}
\begin{equation}\label{Z6}
\gamma:\mathbb{F}_{\vec{O}}(M)\longrightarrow Ibimod_{\vec{O}}\big( \, \overline{\mathcal{I}b}(\mathbb{O}) \,;\, M^{-}\,\big). \vspace{3pt}
\end{equation}
By using the cofibrant replacement $\mathcal{B}^{\Lambda}(\mathbb{O}^{+})$, we can make the space $\mathbb{F}_{\vec{O}}(M)$ more explicit. A point is a family of continuous maps of the form\vspace{3pt}
\begin{equation}\label{P7}
\begin{array}{ll}\vspace{9pt}
f_{0}[\,\vec{n}\,]: \textstyle \sum \vec{O}(2)\times \mathcal{B}^{\Lambda}(\mathbb{O}^{+})(\,\vec{n}\,) \longrightarrow M(\,\vec{n}\,), & \text{with } \vec{n}\neq (+,\ldots,+), \\ 
f_{i}[\,n\,]: \textstyle \sum O_{i}(2)\times \mathcal{B}^{\Lambda}(O_{i})(\,n\,) \longrightarrow M_{i}(\,n\,), & \text{with } i\in \{1,\ldots,k\} \text{ and } n\geq 0,
\end{array} 
\end{equation}
satisfying some relations related to the $k$-fold right operations:\vspace{5pt}
\begin{itemize}[leftmargin=12pt]
\item[$\blacktriangleright$] For  $\vec{\theta}\in \vec{O}_{S}(\,\vec{m}\,)$, $S\in \mathcal{P}_{k}(\{1,\ldots,l\})$ and $y_{i}\in \mathcal{B}^{\Lambda}(\mathbb{O}^{+})(\,\vec{m_{i}}\,)$ such that $S_{u}=+$ and $m_{i}^{u}=+$ if $u\notin A$, one has\vspace{3pt}
$$
f_{0}[\,\vec{n}\,]\big( \,[x_{1},\ldots,x_{k}\,;\,t]\,;\,\vec{\theta}(y_{1},\ldots,y_{l})\,\big)=\vec{\theta}\big(\, f_{0}[\,\vec{m_{1}}\,]([x_{1},\ldots,x_{k}\,;\,t]\,;\,y_{1}),\cdots,f_{0}[\,\vec{m_{l}}]([x_{1},\ldots,x_{k}\,;\,t]\,;\,y_{l})\,\big),\vspace{3pt}
$$ 
\item[$\blacktriangleright$] For  $\theta\in O_{i}(\,l\,)$ and $y_{i}\in \mathcal{B}^{\Lambda}(O_{i})(\,m_{i}\,)$, one has\vspace{3pt}
$$
f_{i}[\,n\,]\big( \,[x\,;\,t]\,;\,\theta(y_{1},\ldots,y_{l})\,\big)=\theta\big(\, f_{i}[\,m_{1}\,]([x\,;\,t]\,;\,y_{1}),\cdots,f_{i}[\,m_{l}]([x\,;\,t]\,;\,y_{l})\,\big),\vspace{3pt}
$$ 

\end{itemize}

\noindent satisfying a relation related to the $k$-fold left operations:\vspace{5pt}

\begin{itemize}[leftmargin=12pt]

\item[$\blacktriangleright$] For $x\in \mathcal{B}^{\Lambda}(\mathbb{O}^{+})(n_{1},\ldots,n_{k})$ and $\theta\in O_{j}(l)$, one has\vspace{3pt}
$$
f_{0}[n_{1},\ldots ,n_{j}+l-1,\ldots,n_{k}]\big(\, [x_{1},\ldots,x_{k}\,;\,t]\,;\, x\circ_{j}^{i}\theta\,\big)=\big(\, f_{0}[n_{1},\ldots ,n_{j},\ldots,n_{k}]([x_{1},\ldots,x_{k}\,;\,t]\,;\, x)\,\big) \circ_{j}^{i}\theta,\vspace{3pt}
$$

\item[$\blacktriangleright$] For $y\in \mathcal{B}^{\Lambda}(O_{j})(n)$ and $\theta\in O_{j}(l)$, one has\vspace{3pt}
$$
f_{j}[n+l-1]\big(\, [x\,;\,t]\,;\, y\circ^{i}\theta\,\big)=\big(\, f_{j}[n]([x\,;\,t]\,;\, y)\,\big) \circ^{i}\theta,\vspace{3pt}
$$

\end{itemize}

\noindent satisfying a relation related to the based point $\eta$:

\begin{itemize}[leftmargin=12pt]
\item[$\blacktriangleright$] For $y\in \mathcal{B}^{\Lambda}(\mathbb{O}^{+})(\,\vec{n}\,)$ and $(x_{1},\ldots,x_{k})\in \vec{O}(2)$, one has\vspace{3pt}
$$
f_{0}[\,\vec{n}\,]\big(\, [x_{1},\ldots,x_{k}\,;\,0]\,;\, y\,\big)=\eta\circ\mu' (y),\vspace{3pt}
$$

\item[$\blacktriangleright$] For $y\in \mathcal{B}^{\Lambda}(O_{i}^{+})(\,n\,)$ and $x\in O_{i}(2)$, one has\vspace{3pt}
$$
f_{i}[\,n\,]\big(\, [x\,;\,0]\,;\, y\,\big)=\eta_{i}\circ\mu' (y),\vspace{3pt}
$$
\end{itemize}
satisfying a relation related to the limit where $\delta_{i}$ is the map (\ref{P3}):\vspace{3pt}

\begin{itemize}[leftmargin=12pt]
\item[$\blacktriangleright$] For $i\in \{1,\ldots, k\}$, $y\in \mathcal{B}^{\Lambda}(O_{i})(\,n\,)=\mathcal{B}^{\Lambda}(\mathbb{O}^{+})(+,\ldots,+,n,+,\ldots,+)$ and $[x_{1},\cdots,x_{k}\,;\,t]\in \Sigma\vec{O}(2)$, one has \vspace{3pt}
$$
f_{0}[+,\ldots,+,n,+,\ldots,+]\big(\, [x_{1},\cdots,x_{k}\,;\,t]\,;\,y\,\big)= f_{i}[\,n\,]\big(\, [x_{i}\,;\,t]\,;\,y\,\big).\vspace{3pt}
$$
\end{itemize}
Similarly, there is a description of the truncated space $T_{\vec{r}}\,\mathbb{F}_{\vec{O}}(M)$ in which the family of continuous maps (\ref{P7}) is indexed by elements $\vec{n}=(n_{1},\ldots,n_{k})$ such that $n_{i}\leq r_{i}$.\vspace{-5pt} 

\newpage

For $i\in \{1,\ldots, n\}$ and $x=[T\,;\,\{x_{v}\}\,;\,\{t_{v}\}]$ a point in $\mathcal{I}(n)$, we denote by $j_{i}(x)$ the first common vertex in $T$ shared by the path joining the $i$-th leaf to the root and the path joining the pearl to the root. We also denote by $e_{i}(x)$ the incoming edge of the vertex $j_{i}(x)$ composing the path joining the $i$-th leaf to the root. We consider the application\vspace{5pt}
$$
\begin{array}{rcl}\vspace{5pt}
\varrho_{i}:\mathcal{I}(n) & \longrightarrow & \{0,\ldots,k\}; \\ 
x & \longmapsto & \left\{
\begin{array}{ll}\vspace{5pt}
0 & \text{if } f_{l}(e_{i}(x))=internal \,\, \forall l\in \{1,\ldots,k\}, \\ 
l & \text{if } f_{l}(e_{i}(x))=internal \text{ and } f_{l'}(e_{i}(x))=external\,\, \forall l'\neq l.
\end{array} 
\right.
\end{array} 
$$
We also consider the map \vspace{3pt}
\begin{equation}\label{F3}
\xymatrix{
\tau_{i}: \mathcal{I}(n) \ar[r]^{\hspace{-47pt}p_{i}} & Cone\,\big( \,\vec{O}(2)\,\big) \sqcup \underset{1\leq i \leq k}{\displaystyle\coprod}\,Cone\,\big(\,  O_{i}(2)\,\big) \ar[r]^{\hspace{10pt}q} & \sum\,\vec{O}(2) \sqcup \underset{1\leq i \leq k}{\displaystyle\coprod}\,\sum\,O_{i}(2),
} \vspace{3pt}
\end{equation}
where $q$ is the quotient map from the cone to the suspension. Let $(x_{1},\ldots,x_{k}\,;\,t)$ be the parameters associated to the vertex $j_{i}(x)$. In order to define $p_{i}(x)$, we consider two cases:\vspace{5pt}
\begin{itemize}
\item[$i)$] if $\varrho_{i}(x)=0$, then $p_{i}(x)$ is the point in $Cone\,\big(\, \vec{O}(2)\,\big)$ obtained from $(x_{1},\ldots,x_{k}\,;\,t)$ by composing the inputs other than the first ones and the inputs corresponding to $e_{i}(x)$ with the unique points in $O_{1}(0),\ldots,O_{k}(0)$.\vspace{5pt} 
\item[$ii)$] if $\varrho_{i}(x)=l>0$, then $p_{i}(x)$ is the point in $Cone\,\big(\, O_{i}(2)\,\big)$ obtained from $(x_{i}\,;\,t)$ by composing the inputs other than the first one and the input corresponding to $e_{i}(x)$ with the unique point in $O_{i}(0)$. 
\end{itemize}
\begin{figure}[!h]
\begin{center}
\includegraphics[scale=0.135]{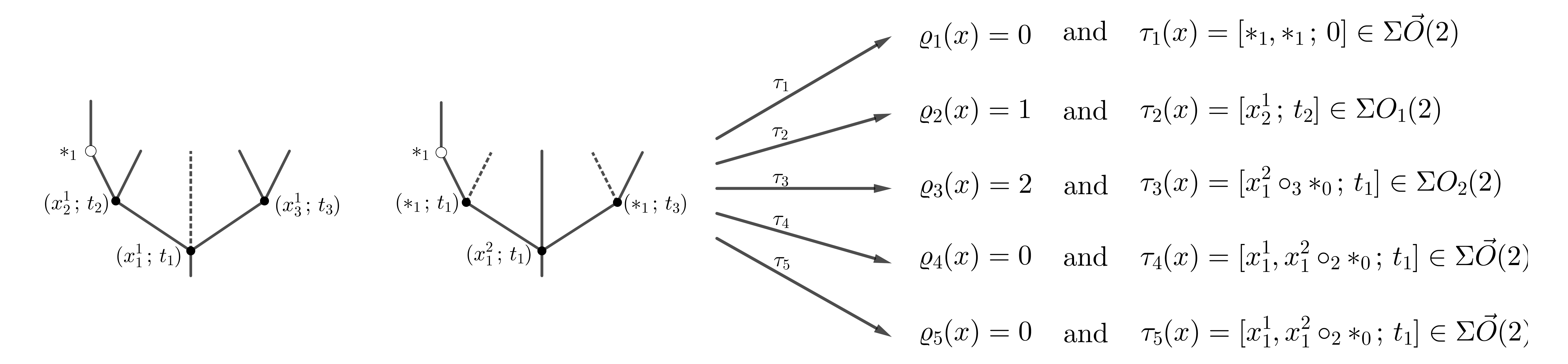}
\caption{Illustration of the applications $\varrho$ and $\tau$.}
\end{center}
\end{figure}

From now on, we use the above notation in order to describe the map  (\ref{Z6}). By abuse of notation, if $y\in \mathcal{B}^{\Lambda}(O_{i})(n)$, then we also denote by $y$ the corresponding point in  $\mathcal{B}^{\Lambda}(\mathbb{O}^{+})(+,\ldots,+,n,+,\ldots,+)=\mathcal{B}^{\Lambda}(O_{i})(n)$. Let $\{f_{i}\}_{0\leq i\leq k}$ be a point of the form (\ref{P7}). Then, one has  \vspace{3pt}
\begin{equation}\label{Q0}
\begin{array}{rcl}\vspace{7pt}
\gamma\big(\,\{f_{i}\}\,\big): \overline{\mathcal{I}b}(\mathbb{O})(n_{1},\ldots,n_{k}) & \longrightarrow & M^{-}(n_{1},\ldots,n_{k}); \\ 
\left[ x\,;\,y^{1},\ldots,y^{l}\,;\,\sigma\right] & \longmapsto & \eta''(x)\big(\, f_{\varrho_{1}(x)}[\,\vec{m_{1}}\,](\tau_{1}(x)\,;\, y^{1}),\cdots, f_{\varrho_{l}(x)}[\,\vec{m_{l}}\,](\tau_{l}(x)\,;\, y^{l})\,\big)\cdot \sigma.\vspace{5pt}
\end{array} 
\end{equation}
The reader can check that the condition (\ref{P9}) on $M$ makes (\ref{Q0}) into  a $k$-fold infinitesimal bimodule map.

\subsection{The notion of coherent operad}\label{O5}

In this subsection we introduce the notion of  {\it coherent operad} used in Theorem \ref{D3}. We refer the reader to \cite[Section 4.1]{Ducoulombier18} for more details and illustrations of this notion. In what follows, an object $c$ in a category $C$ is said to be \textit{partially terminal} if there is no morphism from $c$ to $c'$ for any objects $c'\neq c$. Given a category $C$, we denote by $\partial C$ its subset of non-partially terminal objects. \vspace{7pt}

\begin{defi}\textbf{The partial ordered sets $\Psi_{k}$ and $\Psi_{k}'$ }\vspace{3pt}
\begin{itemize}
\item[$(i)$] For any $k\geq 0$, the category $\Psi_k$ has for  objects non-planar pearled trees with $k$ leaves labelled bijectively by the set $\{1,\ldots,k\}$, whose pearl can have any arity $\geq 0$, and the other vertices are of arity $\geq 2$. The morphisms in $\Psi_k$ are inner edge contractions.\vspace{3pt}

\item[$(ii)$] For each $\Psi_k$, we denote by $c_k$ its terminal object, which is the pearled $k$-corolla, and by $c'_k$, $k\geq 2$, the tree with 2 vertices: a pearled root of arity one, whose only outgoing edge is attached to the other vertex of arity $k$.\vspace{3pt}

\item[$(iii)$] We denote by $\Psi'_k$, $k\geq 2$, the subcategory of  $\Psi_k$ of all morphisms except the morphism $c'_k\to c_k$.\vspace{-5pt}
\end{itemize}
\end{defi}

\begin{figure}[!h]
\begin{center}
\includegraphics[scale=0.11]{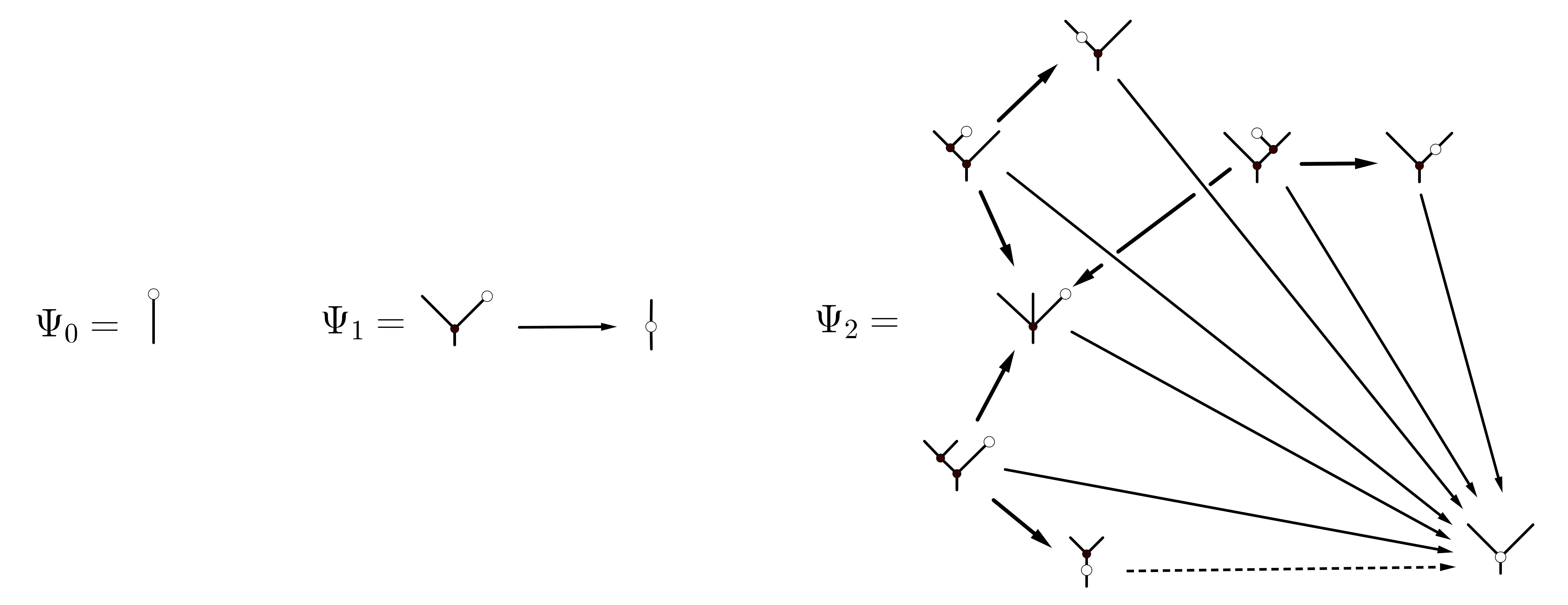}\vspace{-5pt}

\begin{minipage}{420pt}
\caption{Categories $\Psi_0$, $\Psi_1$, $\Psi_2$.  Morphism $c'_2\to c_2$ is shown by a dotted arrow. }
\end{minipage}\vspace{-12pt}
\end{center}
\end{figure}

\begin{defi}
For a topological operad $O$, an $O$-Ibimodule $M$, and $k\geq 0$, define a $\Psi_k$ shaped diagram\vspace{3pt}
$$
\begin{array}{rcl}
\rho_k^M\colon \Psi_k&\longrightarrow&\Top;\\
T&\mapsto&M\left(|p|\right)\times\displaystyle\prod_{v\in V(T)\setminus\{p\}}O\bigl(|v|\bigr).
\end{array} 
$$
On morphism it is defined by choosing a planar representative of each pearled tree, and then using the operadic composition and $O$-action on $M$ for each edge contraction. The choice of planar representatives won\rq{}t matter in the sense that the obtained diagrams are objectwise homeomorphic.\vspace{7pt}
\end{defi}

\begin{defi}\label{N1}\textbf{Coherent diagrams and operads}

\begin{itemize}
\item[$(i)$] A diagram $\rho\colon\Psi_k\to\Top$, $k\geq 2$, is called {\it coherent}, if the natural map below is a weak equivalence:\vspace{3pt}
\begin{equation}\label{I2}
\underset {\partial\Psi_k}{\hocolim}\,\rho\, \longrightarrow\, \underset{\Psi'_k}{\hocolim}\, {\rho}.\vspace{-1pt}
\end{equation}

\item[$(ii)$] A topological operad $O$ is called {\it coherent}, respectively, {\it $k$-coherent}, if the diagrams $\rho_i^O$ are coherent
for all $i\geq 2$, respectively, for all $i$ in the range $2\leq i\leq k$.
\end{itemize}\vspace{-10pt}
\end{defi}

\subsection{Proof of the main Theorem \ref{D3} }

Without loss of generality, we assume that $M$ is Reedy fibrant in the category of $k$-fold bimodules equipped with the Reedy model category structure. Indeed, if $M$ is not Reedy fibrant, then we  substitute $M$ by a fibrant replacement $M\rightarrow M^{f}$ which necessarily satisfies the conditions (\ref{P9}). Such a replacement is endowed with a map of $k$-fold bimodules $\mathbb{O}^{+}\rightarrow M \rightarrow M^{f}$. Furthermore, we assume that $k=2$. \vspace{7pt}

Let $O_{1}$ and $O_{2}$ be a pair of $2$-reduced operads relative to $O$. The operad $O_{1}$ is $j_{1}$-coherent whereas the operad $O_{2}$ is $j_{2}$-coherent. First, let us remark that the map\vspace{5pt}
$$
\gamma_{\vec{0}}:T_{\vec{0}}\,\mathbb{F}_{\vec{O}}(M)\longrightarrow T_{\vec{0}}\, Ibimod_{\vec{O}}\big( \, \overline{\mathcal{I}b}(\mathbb{O}) \,;\, M^{-}\,\big)\vspace{5pt}
$$
is an isomorphism since the $2$-fold bimodule $M$ satisfies the relations $M(0,0)=M(0,+)=M(+,0)=\ast$.\vspace{7pt}

From now on, we assume that the $\gamma_{\vec{m}}$ is a weak equivalence for some $\vec{m}=(m_{1}\,;\,m_{2})\in \mathbb{N}^{\times 2}$ such that $(0\,;\,0)< \vec{m}< (j_{1}\,;\,j_{2})$. We prove by induction that $\gamma_{\vec{n}}$, with $\vec{n}=(m_{1}+1\,;\, m_{2})$ is also a weak equivalence. Indeed, according to the notation introduced in the proofs of Theorem \ref{L9} and Theorem \ref{M4}, $\gamma_{\vec{n}}$ induces a map between two towers of fibrations\vspace{5pt}
$$
\xymatrix{
T_{\vec{n}}\mathbb{F}_{\vec{O}}(M)\ar@{=}[d]\ar[r]^{\hspace{-30pt}\gamma_{\vec{n}}} & T_{\vec{n}} Ibimod_{\vec{O}}\big( \, \overline{\mathcal{I}b}(\mathbb{O}) \,;\, M^{-}\,\big) \ar@{=}[d] &  \\
\mathbb{F}_{\vec{O}}(M)[m_{2}] \ar[d]\ar[r]^{\hspace{-30pt}\gamma_{m_{2}}} & Ibimod_{\vec{O}}\big( \, \overline{\mathcal{I}b}_{m_{2}}(\mathbb{O})[\vec{n}] \,;\, M^{-}\,\big)\ar[d] & Ibimod_{\vec{O}}\big( \, \overline{\mathcal{I}b}(\mathbb{O})[\vec{n}] \,;\, M^{-}\,\big) \ar@{=}[l]\\
\vdots \ar[d] & \vdots \ar[d] \\
\mathbb{F}_{\vec{O}}(M)[0]\ar[d]\ar[r]^{\hspace{-30pt}\gamma_{0}} & Ibimod_{\vec{O}}\big( \, \overline{\mathcal{I}b}_{0}(\mathbb{O})[\vec{n}] \,;\, M^{-}\,\big) \ar[d] & \\
\mathbb{F}_{\vec{O}}(M)[+]\ar[d]\ar[r]^{\hspace{-30pt}\gamma_{+}} & Ibimod_{\vec{O}}\big( \, \overline{\mathcal{I}b}_{+}(\mathbb{O})[\vec{n}] \,;\, M^{-}\,\big)\ar[d] &  \\
T_{\vec{m}}\mathbb{F}_{\vec{O}}(M)\ar[r]^{\hspace{-30pt}\gamma_{\vec{m}}}_{\hspace{-30pt}\simeq} & T_{\vec{m}} Ibimod_{\vec{O}}\big( \, \overline{\mathcal{I}b}(\mathbb{O}) \,;\, M^{-}\,\big) & Ibimod_{\vec{O}}\big( \, \overline{\mathcal{I}b}(\mathbb{O})[\vec{m}] \,;\, M^{-}\,\big) \ar@{=}[l]
} \vspace{5pt}
$$

\noindent where $\mathbb{F}_{\vec{O}}(M)[i]$, with $i\in\{+,0,\ldots,k\}$, is the limit of the  diagram\vspace{5pt}
$$
\xymatrix@R=9pt@C=-1pt{
 Map_{\ast}\left( 
 \sum O_{1}(2)\,;\, Bimod_{O_{1}}(\mathcal{B}^{\Lambda}(O_{1})[m_{1}+1]\,;\, M_{1})\right)\ar[rd]
&  \\
&  Map_{\ast}\left( 
\sum  \vec{O}(2)\,;\, Bimod_{O_{1}}(\mathcal{B}^{\Lambda}(O_{1})[m_{1}+1]\,;\, M_{1})\right)\\
 Map_{\ast}\left( 
\sum \vec{O}(2)\,;\, Bimod_{\vec{O}}(\mathcal{B}^{\Lambda}_{i}(\mathbb{O}^{+})[\vec{n}]\,;\, M)\right) \ar[ru]\ar[rd]
 & \\
 &  Map_{\ast}\left( 
\sum \vec{O}(2) \,;\, Bimod_{O_{2}}(\mathcal{B}^{\Lambda}(O_{2})[m_{2}]\,;\, M_{2})\right)\\
  Map_{\ast}\left( 
\sum O_{2}(2)\,;\, Bimod_{O_{2}}(\mathcal{B}^{\Lambda}(O_{2})[m_{2}]\,;\, M_{2})\right)\ar[ru] &
}\vspace{-1pt}
$$

We prove by induction that the map $\gamma_{\vec{n}}$ is a weak equivalence. For this purpose, we assume that $\gamma_{i-1}$ is a weak equivalence and we consider the following commutative diagram:\vspace{3pt}
$$
\xymatrix{
\mathbb{F}_{\vec{O}}(M)[i-1] \ar[d]^{\gamma_{i-1}}_{\simeq} & \mathbb{F}_{\vec{O}}(M)[i] \ar[l]\ar[d]^{\gamma_{i}} & F_{1}^{i}(g)\ar[l] \ar[d]^{\gamma_{g}}\\
Ibimod_{\vec{O}}\big( \, \overline{\mathcal{I}b}_{i-1}(\mathbb{O}) \,;\, M^{-}\,\big) & Ibimod_{\vec{O}}\big( \, \overline{\mathcal{I}b}_{i}(\mathbb{O}) \,;\, M^{-}\,\big) \ar[l] & F_{2}^{i}(g_{\ast})\ar[l]
}\vspace{3pt}
$$ 
where $F_{1}^{i}(g)$ is the fiber over an element $g\in \mathbb{F}_{\vec{O}}(M)[i-1]$ and $F_{2}^{i}(g_{\ast})$ is the fiber over $g_{\ast}=\gamma_{i-1}(g)$. Since the left horizontal maps are fibrations, $\gamma_{i}$ is a weak equivalence if the map $\gamma_{g}$ between the fibers is a weak equivalence. 

The first step is to describe the fibers $F_{1}^{i}(g)$ and $F_{2}^{i}(g_{\ast})$ as certain spaces of section extensions. As explained in \cite[Section 7.1]{Ducoulombier18}, for a Serre fibration $\pi:E\rightarrow B$, we denote by $\Gamma(\pi\,;\,B)$ the space of global sections of $\pi$. For $A\subset B$ and a partial section $s:A\rightarrow E$, we denote by \vspace{5pt}
$$
\Gamma_{s}(\pi\,;\,B\,;\,A)\vspace{5pt}
$$ 
the space of sections $f:B\rightarrow E$ extending $s$, i.e. $f_{|A}=s$. In order to avoid heavy notation, every map induced by $g$ will be denoted by $g_{\ast}$.

\subsubsection*{Description of the fiber $F_{2}^{i}(g_{\ast})$ as a space of section extensions}\vspace{5pt}

Let $i\in \{+,0,\ldots,m_{2}\}$. The fiber $F_{2}^{i}(g_{\ast})$ is a space of $\Sigma=(\Sigma_{m_{1}+1}\times \Sigma_{i})$-equivariant maps\vspace{5pt}
\begin{equation}\label{I0}
\begin{array}{ll}\vspace{9pt}
f:\overline{\mathcal{I}b}(\mathbb{O})(m_{1}+1\,;\,+)\longrightarrow M(m_{1}+1\,;\,0) & \text{if } i=+, \\ 
f:\overline{\mathcal{I}b}(\mathbb{O})(m_{1}+1\,;\,i)\longrightarrow M(m_{1}+1\,;\,i) & \text{otherwise},
\end{array}\vspace{3pt}
\end{equation}
satisfying two conditions. First, it is determined by $g$ on all composite point. In other words, it is determined on the subspace $\partial\overline{\mathcal{I}b}(\mathbb{O})(m_{1}+1\,;\,i)$ (see the proof of Theorem \ref{M4}). Indeed, we can apply $g_{\ast}$ on its prime component and then the $k$-fold infinitesimal bimodule structure. Secondly, it should respect the $\Lambda^{\times k}$-structure, i.e. the right actions by $O_{1}(0)$ and $O_{2}(0)$. \vspace{7pt}

The first condition means that the upper triangle below must commute\vspace{3pt}
\begin{equation}\label{H8}
\xymatrix{
\partial\overline{\mathcal{I}b}(\mathbb{O})(m_{1}+1\,;\,i) \ar[r] \ar[d] & M(m_{1}+1\,;\,i) \ar[d] \\
\overline{\mathcal{I}b}(\mathbb{O})(m_{1}+1\,;\,i) \ar[r] \ar@{-->}[ru]_{f} & \mathcal{M}(M)(m_{1}+1\,;\,i)
}\vspace{3pt}
\end{equation}
while the second condition means that the lower triangle commutes. Here the right arrow is the matching map (see Section \ref{G0}), which is a fibration as we assumed that $M$ is Reedy fibrant. The lower map is the composition\vspace{5pt}
$$
\overline{\mathcal{I}b}(\mathbb{O})(m_{1}+1\,;\,i) \longrightarrow \mathcal{M}(\overline{\mathcal{I}b}(\mathbb{O}))(m_{1}+1\,;\,i)\longrightarrow  \mathcal{M}(M)(m_{1}+1\,;\,i)\vspace{5pt}
$$ 
where the first map is the matching map and the second one is $\mathcal{M}(g)(m_{1}+1\,;\,i)$. \vspace{7pt}

Now we consider the pullback $E$ of the right arrow of (\ref{H8}) along its lower one. We get a fibration of $\Sigma$-spaces\vspace{5pt}
$$
\pi:E\longrightarrow \overline{\mathcal{I}b}(\mathbb{O})(m_{1}+1\,;\,i).\vspace{5pt}
$$
Since $\overline{\mathcal{I}b}(\mathbb{O})(m_{1}+1\,;\,i)$ is $\Sigma$-cofibrant and the property of being a Serre fibration is local, the map \vspace{5pt}
$$
\pi_{/\Sigma}:E_{/\Sigma}\longrightarrow \overline{\mathcal{I}b}(\mathbb{O})(m_{1}+1\,;\,i)_{/\Sigma}\vspace{5pt}
$$
is still a Serre fibration. \vspace{7pt}

Finally, the fiber $F_{2}(g_{\ast})$ is the space of section extensions\vspace{5pt}
$$
F_{2}^{i}(g_{\ast})\cong \Gamma_{g_{\ast}}\big( \, \pi_{/\Sigma} \,;\, \overline{\mathcal{I}b}(\mathbb{O})(m_{1}+1\,;\,i)_{/\Sigma}\,;\, \partial\overline{\mathcal{I}b}(\mathbb{O})(m_{1}+1\,;\,i)_{/\Sigma}\,\big)\vspace{5pt}
$$
where $g_{\ast}$ is the upper arrow in Diagram (\ref{H8}).

\subsubsection*{Description of the fiber $F_{1}^{i}(g_{\ast})$ as a space of section extensions}

There are two cases to consider. First, we assume that $i=+$. Then the points in the fiber $F_{1}^{+}(g)$ can be described as the space of $\Sigma=\Sigma_{m_{1}+1}$-equivariant maps \vspace{5pt}
$$
\begin{array}{lcl}\vspace{9pt}
f':\Sigma \vec{O}(2)\times \mathcal{B}^{\Lambda}(\mathbb{O}^{+})(m_{1}+1\,;\,+) & \longrightarrow &  M(m_{1}+1\,;\,+),   \\ 
f'':\Sigma O_{1}(2)\times \mathcal{B}^{\Lambda}(O_{1})(m_{1}+1) & \longrightarrow & M(m_{1}+1\,;\,+), 
\end{array} \vspace{5pt}
$$ 
satisfying in particular the relation\vspace{5pt}
$$
f'\big(\, [\alpha_{1}\,,\,\alpha_{2}\,;\,t]\,;\,x\,\big)=f''\big(\, [\alpha_{1}\,;\,t]\,;\,x\,\big),\hspace{15pt} \text{with } \left\{
\begin{array}{l}\vspace{7pt}
[\alpha_{1}\,,\,\alpha_{2}\,;\,t]\in \Sigma \vec{O}(2),  \\ 
x\in \mathcal{B}^{\Lambda}(\mathbb{O}^{+})(m_{1}+1\,;\,+).
\end{array} 
\right.\vspace{5pt}
$$
Consequently, such a point is entirely determined by the application $f''$. Similarly to the previous subsection, $f''$ satisfies two conditions. First, it is determined by $g$ on the subspace $\partial\mathcal{B}^{\Lambda}(O_{1})(m_{1}+1\,;\,+)$ and it is also determined by $\eta$ on the based point of the suspension. Secondly, it should respect the $\Lambda^{\times k}_{+}$-structure. These conditions can be represented by a commutative diagram\vspace{5pt}
\begin{equation}\label{H9}
\xymatrix{
\partial\big(\,\Sigma O_{1}(2) \times \mathcal{B}^{\Lambda}(O_{1})(m_{1}+1)\,\big) \ar[r]\ar[d] & M(m_{1}+1\,;\,+)\ar[d] \\
\Sigma O_{1}(2) \times \mathcal{B}^{\Lambda}(O_{1})(m_{1}+1) \ar@{-->}[ru]_{f''} \ar[r] & \mathcal{M}(M)(m_{1}+1\,;\,+)
}\vspace{5pt}
\end{equation}
The second condition coincides with the commutativity of the lower triangle. The first condition coincides with the commutativity of the upper triangle in which the boundary space is given by the pushout product\vspace{5pt}
$$
\big(\, \Sigma O_{1}(2)\times \partial\mathcal{B}^{\Lambda}(O_{1})(m_{1}+1)\,\big)
\underset{\ast \times \partial\mathcal{B}^{\Lambda}(O_{1})(m_{1}+1)}{\coprod}
\big(\, \ast \times \mathcal{B}^{\Lambda}(O_{1})(m_{1}+1)\,\big).\vspace{3pt}
$$ 
On the first part of the boundary, the upper map of (\ref{H9}) is determined by the application $g$ while the second part is given by the composition
$$
\xymatrix{
\ast \times \mathcal{B}^{\Lambda}(O_{1})(m_{1}+1) \ar[r]_{\hspace{10pt}pr_{2}} & \mathcal{B}^{\Lambda}(O_{1})(m_{1}+1) \ar[r]_{\hspace{10pt}\mu''} & O_{1}(m_{1}+1) \ar[r]_{\eta} & M(m_{1}+1\,;\,+).
}
$$

Now we consider the pullback $E$ of the right arrow of (\ref{H9}) along its lower one. We get a fibration of $\Sigma$-spaces \vspace{5pt}
$$
\pi:E\longrightarrow \Sigma O_{1}(2) \times \mathcal{B}^{\Lambda}(O_{1})(m_{1}+1).\vspace{5pt}
$$
Since $\Sigma O_{1}(2) \times \mathcal{B}^{\Lambda}(O_{1})(m_{1}+1)$ is $\Sigma$-cofibrant and the property of being a Serre fibration is local, the map \vspace{5pt}
$$
\pi_{/\Sigma}:E_{/\Sigma}\longrightarrow \Sigma O_{1}(2) \times \mathcal{B}^{\Lambda}(O_{1})(m_{1}+1)_{/\Sigma}\vspace{5pt}
$$
is still a Serre fibration. \vspace{7pt}

Finally, the $F_{1}(g_{\ast})$ is the space of section extensions\vspace{5pt}
$$
F_{1}^{+}(g_{\ast})\cong \Gamma_{g_{\ast}}\big( \, \pi_{/\Sigma} \,;\, \Sigma O_{1}(2) \times \mathcal{B}^{\Lambda}(O_{1})(m_{1}+1)_{/\Sigma}\,;\, \partial\big(\Sigma O_{1}(2) \times \mathcal{B}^{\Lambda}(O_{1})(m_{1}+1)\big)_{/\Sigma}\,\big)\vspace{5pt}
$$
where $g_{\ast}$ is the upper arrow in Diagram (\ref{H8}).\vspace{7pt}

From now on, we assume that $i\in \{0,\ldots,m_{2}\}$. Then the points in the fiber $F_{1}^{i}(g)$ is the space of $\Sigma=(\Sigma_{m_{1}+1}\times \Sigma_{i})$-equivariant maps \vspace{5pt}
$$
\begin{array}{lcl}
f':\Sigma \vec{O}(2)\times \mathcal{B}^{\Lambda}(\mathbb{O}^{+})(m_{1}+1\,;\,i) & \longrightarrow &  M(m_{1}+1\,;\,i),   
\end{array}\vspace{5pt} 
$$ 
satisfying two conditions. First, it is determined by $g$ on the subspace $\partial\mathcal{B}^{\Lambda}(\mathbb{O}^{+})(m_{1}+1\,;\,i)$ and it is also determined by $\eta$ on the based point of the suspension. Secondly, it should respect the $\Lambda^{\times k}_{+}$-structure. These conditions can be represented by a commutative diagram\vspace{3pt}
\begin{equation}\label{Q1}
\xymatrix{
\partial\big(\,\Sigma \vec{O}(2) \times \mathcal{B}^{\Lambda}(\mathbb{O}^{+})(m_{1}+1\,;\,i)\,\big) \ar[r]\ar[d] & M(m_{1}+1\,;\,i)\ar[d] \\
\Sigma \vec{O}(2) \times \mathcal{B}^{\Lambda}(\mathbb{O}^{+})(m_{1}+1\,;\,i) \ar@{-->}[ru]_{f'} \ar[r] & \mathcal{M}(M)(m_{1}+1\,;\,i)
}\vspace{3pt}
\end{equation}
The second condition coincides with the commutativity of the lower triangle. The first condition coincides with the commutativity of the upper triangle in which the boundary space in the above diagram is given by the pushout product\vspace{5pt}
$$
\big(\, \Sigma \vec{O}(2)\times \partial\mathcal{B}^{\Lambda}(\mathbb{O}^{+})(m_{1}+1\,;\,i)\,\big)
\underset{\ast \times \partial\mathcal{B}^{\Lambda}(\mathbb{O}^{+})(m_{1}+1\,;\,i)}{\coprod}
\big(\, \ast \times \mathcal{B}^{\Lambda}(\mathbb{O}^{+})(m_{1}+1\,;\,i)\,\big).\vspace{3pt}
$$ 
On the first part of the boundary, the upper map of (\ref{Q1}) is determined by the application $g$ while the second part is given by the composition\vspace{3pt}
$$
\xymatrix{
\ast \times \mathcal{B}^{\Lambda}(\mathbb{O}^{+})(m_{1}+1\,;\,i) \ar[r]_{\hspace{10pt}pr_{2}} & \mathcal{B}^{\Lambda}(\mathbb{O}^{+})(m_{1}+1\,;\,i) \ar[r]_{\hspace{10pt}\mu''} & \mathbb{O}^{+}(m_{1}+1\,;\,i) \ar[r]_{\eta} & M(m_{1}+1\,;\,i).
}
$$

Now we consider the pullback $E$ of the right arrow of (\ref{Q1}) along its lower one. We get a fibration of $\Sigma$-spaces\vspace{5pt}
$$
\pi:E\longrightarrow  \Sigma \vec{O}(2) \times \mathcal{B}^{\Lambda}(\mathbb{O}^{+})(m_{1}+1\,;\,i).\vspace{5pt}
$$
Since $\Sigma \vec{O}(2) \times \mathcal{B}^{\Lambda}(\mathbb{O}^{+})(m_{1}+1\,;\,i)$ is $\Sigma$-cofibrant and the property of being a Serre fibration is local, the map\vspace{5pt} 
$$
\pi_{/\Sigma}:E_{/\Sigma}\longrightarrow \Sigma \vec{O}(2) \times \mathcal{B}^{\Lambda}(\mathbb{O}^{+})(m_{1}+1\,;\,i)_{/\Sigma}\vspace{5pt}
$$
is still a Serre fibration. The fiber $F_{1}(g_{\ast})$ is the space of section extensions\vspace{3pt}
$$
F_{1}^{i}(g_{\ast})\cong \Gamma_{g_{\ast}}\big( \, \pi_{/\Sigma} \,;\, \Sigma \vec{O}(\{1\,;\,2\}) \times \mathcal{B}^{\Lambda}(\mathbb{O}^{+})(m_{1}+1\,;\,i)_{/\Sigma}\,;\, \partial\big(\Sigma \vec{O}(2) \times \mathcal{B}^{\Lambda}(\mathbb{O}^{+})(m_{1}+1\,;\,i)\big)_{/\Sigma}\,\big).
$$

\subsubsection*{Alternative description the spaces of section extensions associated to $F_{1}^{i}(g_{\ast})$}

In order to identify the fiber $F_{1}^{i}(g)$ with a subspace of $F_{2}^{i}(g_{\ast})$, we consider the subspace\vspace{5pt}
$$
\partial\overline{\mathcal{I}b}(\mathbb{O})(m_{1}+1\,;\,i)\subset \partial'\overline{\mathcal{I}b}(\mathbb{O})(m_{1}+1\,;\,i) \subset \overline{\mathcal{I}b}(\mathbb{O})(m_{1}+1\,;\,i)\vspace{5pt}
$$
containing $\partial\overline{\mathcal{I}b}(\mathbb{O})(m_{1}+1\,;\,i)$ together with all the elements of the form $[x\,;\,y^{1},\ldots,y^{l}\,;\,\sigma]$ where $l\geq 2$. So, we claim that the map $\gamma_{g}$ between the fibers sends a point in the fiber $F_{1}^{i}(g)$ to a $\Sigma$-equivariant map $f$ making the following diagram commutes:\vspace{3pt}
$$
\xymatrix{
\partial'\overline{\mathcal{I}b}(\mathbb{O})(m_{1}+1\,;\,i) \ar[r] \ar[d] & M(m_{1}+1\,;\,i) \ar[d] \\
\overline{\mathcal{I}b}(\mathbb{O})(m_{1}+1\,;\,i) \ar[r] \ar@{-->}[ru]_{f} & \mathcal{M}(M)(m_{1}+1\,;\,i)
}\vspace{3pt}
$$

\begin{pro}
The fiber $F_{1}^{i}(g)$ is weakly equivalent to the space of section extensions\vspace{3pt}
\begin{equation}\label{Z7}
\tilde{F}_{1}^{i}(g):=\Gamma_{g_{\ast}}\big( \, \pi_{/\Sigma} \,;\, \overline{\mathcal{I}b}(\mathbb{O})(m_{1}+1\,;\,i)_{/\Sigma}\,;\, \partial'\overline{\mathcal{I}b}(\mathbb{O})(m_{1}+1\,;\,i)_{/\Sigma}\,\big)\vspace{3pt}
\end{equation}
\end{pro}

\begin{proof}
The proof is similar to \cite[Section 7.1]{Ducoulombier18}. More precisely, for $i\in \{0,\ldots,m_{2}\}$, the fiber $F_{1}^{i}(g)$ is homeomorphic to to the space of section extensions (\ref{Z7}). The homeomorphism is obtained from the map (\ref{Q0}). Its inverse is given by the image of the points $[x\,;\,y^{1}\,;\,id]$ where $y^{1}\in \mathcal{B}^{\Lambda}(\mathbb{O}^{+})(m_{1}+1\,;\,i)$ and  $x=[(T\,;\,f_{1}\,,\,f_{2})\,;\,\theta_{r}\,;\,t_{r}]\in \mathcal{I}(1)$ with $[\theta_{r}\,;\,t_{r}]\in \Sigma \vec{O}(2)$, $T$ the pearled $2$-corolla, $f_{1}$ and $f_{2}$ the applications indexing the edges by $internal$. \vspace{5pt}

For $i=+$, the fiber $F_{1}^{i}(g)$ is weakly equivalent to to the space of section extensions (\ref{Z7}). Indeed, let us consider the following commutative diagram in which the right horizontal maps are induced by the weak equivalence $\iota_{1}:M(m_{1}+1\,;\,+)\rightarrow M(m_{1}+1\,;\,0)$:
\begin{equation}\label{Z8}
\xymatrix{
\partial\big(\,\Sigma O_{1}(2) \times \mathcal{B}^{\Lambda}(O_{1})(m_{1}+1)\,\big) \ar[r]^{\hspace{25pt}g_{\ast}}\ar[d] & M(m_{1}+1\,;\,+)\ar[d]\ar[r]^{\simeq} & M(m_{1}+1\,;\,0) \ar[d]\\
\Sigma O_{1}(2) \times \mathcal{B}^{\Lambda}(O_{1})(m_{1}+1) \ar@{-->}[rru]_{f''} \ar[r] & \mathcal{M}(M)(m_{1}+1\,;\,+) \ar[r]^{\simeq} & \mathcal{M}(M)(m_{1}+1\,;\,0)
}\vspace{5pt}
\end{equation}
Let $E'$ be the pullback of the right arrow of (\ref{Z8}) along its lower one. We get a fibration of $\Sigma$-spaces \vspace{5pt}
$$
\pi':E'\longrightarrow \Sigma O_{1}(2) \times \mathcal{B}^{\Lambda}(O_{1})(m_{1}+1).\vspace{5pt}
$$
Since $\Sigma O_{1}(2) \times \mathcal{B}^{\Lambda}(O_{1})(m_{1}+1)$ is $\Sigma$-cofibrant and the property of being a Serre fibration is local, the map \vspace{5pt}
$$
\pi'_{/\Sigma}:E'_{/\Sigma}\longrightarrow \Sigma O_{1}(2) \times \mathcal{B}^{\Lambda}(O_{1})(m_{1}+1)_{/\Sigma}\vspace{5pt}
$$
is still a Serre fibration. Finally, we denote by $F_{1}'(g_{\ast})$ is the space of section extensions\vspace{5pt}
$$
F_{1}'(g_{\ast}):= \Gamma_{\iota_{1}\circ g_{\ast}}\big( \, \pi'_{/\Sigma} \,;\, \Sigma O_{1}(2) \times \mathcal{B}^{\Lambda}(O_{1})(m_{1}+1)_{/\Sigma}\,;\, \partial\big(\Sigma O_{1}(2) \times \mathcal{B}^{\Lambda}(O_{1})(m_{1}+1)\big)_{/\Sigma}\,\big).\vspace{5pt}
$$

Since $M$ is Reedy fibrant and $\Sigma O_{1}(2) \times \mathcal{B}^{\Lambda}(O_{1})(m_{1}+1)$ is $\Sigma$-cofibrant, the weak equivalence $\iota_{1}$ induces a weak equivalence between the spaces of section extensions $F_{1}^{+}(g)\rightarrow F_{1}'(g)$. Furthermore, there is a homeomorphism between $F_{1}'(g)$ and $\tilde{F}_{1}^{+}(g)$. By using the identification \vspace{5pt}
$$
\overline{\mathcal{I}b}(\mathbb{O})(m_{1}+1\,;\,+)\cong \overline{\mathcal{I}b}(O_{1})(m_{1}+1),\vspace{5pt}
$$
the map $\phi: F_{1}'(g)\rightarrow \tilde{F}_{1}^{+}(g)$ sends an application \vspace{5pt}
$$
f'':\Sigma O_{1}(2) \times \mathcal{B}^{\Lambda}(O_{1})(m_{1}+1) \longrightarrow M(m_{1}+1\,;\,0),\vspace{5pt}
$$
making the diagram (\ref{Z8}) into a commutative diagram, to the point in $\tilde{F}_{1}^{+}(g)$  defined as follows:
$$
\begin{array}{rcl}\vspace{7pt}
\phi(f''):\overline{\mathcal{I}b}(O_{1})(m_{1}+1) & \longrightarrow & M(m_{1}+1\,;\,0); \\ 
\left[ x\,;\,y^{1},\ldots,y^{l}\,;\,\sigma \right] & \longmapsto & \left\{
\begin{array}{ll}\vspace{7pt}
f''(y^{1})\cdot \sigma & \text{if } l=1, \\ 
\iota_{1}\circ g_{\ast}\big(\left[ x\,;\,y^{1},\ldots,y^{l}\,;\,\sigma\right]\big) & \text{otherwise}.
\end{array} 
\right.
\end{array} 
$$
Finally, one has the commutative diagram
$$
\xymatrix{
F_{1}^{+}(g) \ar[r]^{\simeq} \ar[d] & F_{1}'(g) \ar[dl]^{\simeq} \\
\tilde{F}_{1}^{+}(g) & 
}
$$
where the left vertical map is induced by the map (\ref{Q0}). Thus finishes the proof of the proposition.
\end{proof}

\subsubsection*{Weak equivalence between the spaces of section extensions}

According to the identifications obtained in the previous subsections, one has the commutative diagram:
$$
\xymatrix{
F_{1}^{i}(g) \ar[d] & \Gamma_{g_{\ast}}\big( \, \pi_{/\Sigma} \,;\, \overline{\mathcal{I}b}(\mathbb{O})(m_{1}+1\,;\,i)_{/\Sigma}\,;\, \partial'\overline{\mathcal{I}b}(\mathbb{O})(m_{1}+1\,;\,i)_{/\Sigma}\,\big) \ar[l]^{\hspace{-80pt}\simeq} \ar[d] \\
F_{2}^{i}(g) & \Gamma_{g_{\ast}}\big( \, \pi_{/\Sigma} \,;\, \overline{\mathcal{I}b}(\mathbb{O})(m_{1}+1\,;\,i)_{/\Sigma}\,;\, \partial\overline{\mathcal{I}b}(\mathbb{O})(m_{1}+1\,;\,i)_{/\Sigma}\,\big) \ar[l]^{\hspace{-80pt}\cong}
}
$$  

\begin{lmm}{\cite[Lemma 7.2]{Ducoulombier18}}
Let $\pi:E\rightarrow Z$ be a Serre fibration and let  $X\subset Y \subset Z$ be cofibrant inclusion such that the inclusion $X\subset Y$ is a weak equivalence and $X$ is cofibrant. For any section $s:Y\rightarrow E$, the map 
$$\Gamma_{s}(\pi\,;\, Z\,;\,Y)\rightarrow \Gamma_{s}(\pi\,;\, Z\,;\,X)$$
is a weak equivalence.\vspace{5pt}
\end{lmm}

\begin{lmm}
The inclusion $\partial\overline{\mathcal{I}b}(\mathbb{O})(m_{1}+1\,;\,i)\rightarrow \partial'\overline{\mathcal{I}b}(\mathbb{O})(m_{1}+1\,;\,i)$ is a weak equivalence.
\end{lmm}

\begin{proof}
The homotopy retract introduced in Lemma \ref{L2} can be extended in order to get a homotopy retract 
$$
\xymatrix{
\overline{\mathcal{I}b}_{O_{1}}(O_{1})(m_{1}+1)\times  \overline{\mathcal{I}b}_{O_{2}}(O_{2})(i) \ar@<0.5ex>[r]^{\hspace{50pt}\Psi} & 
\overline{\mathcal{I}b}(\mathbb{O})(m_{1}+1\,;\,i)\ar@<0.5ex>[l]^{\hspace{50pt}\Upsilon}\hspace{-30pt} &  \ar@(dr,ur)_{h} 
}
$$
where the map $\Upsilon$ removes the $external$ edges while the map $\Psi$ consists in gluing together the elements along the path joining the pearls to the roots of the pearled trees. For instance, the map $\Psi$ sends the element in the product space $\overline{\mathcal{I}b}_{O_{1}}(O_{1})(4)\times \overline{\mathcal{I}b}_{O_{2}}(O_{2})(5)$\vspace{-5pt}
\begin{figure}[!h]
\begin{center}
\includegraphics[scale=0.18]{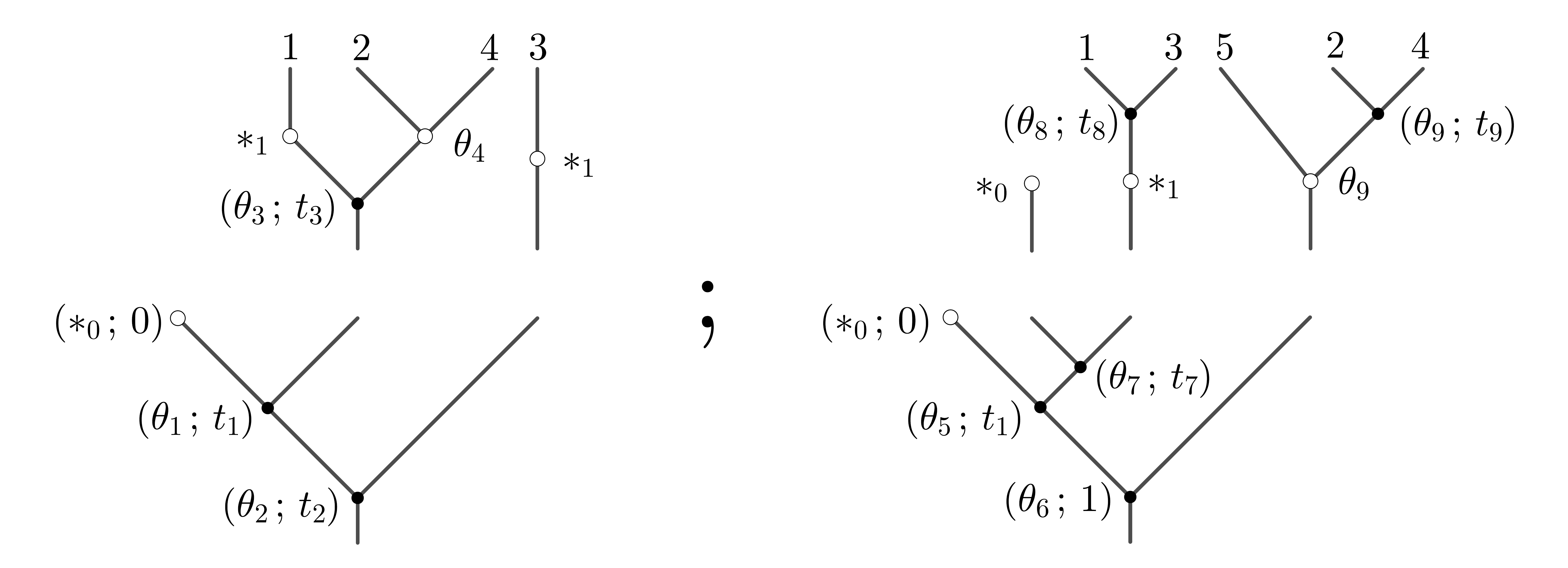}\vspace{-10pt}
\end{center}
\end{figure}

\noindent to the following point in $\overline{\mathcal{I}b}(\mathbb{O})(4\,;\,5)$:
\begin{figure}[!h]
\begin{center}
\includegraphics[scale=0.18]{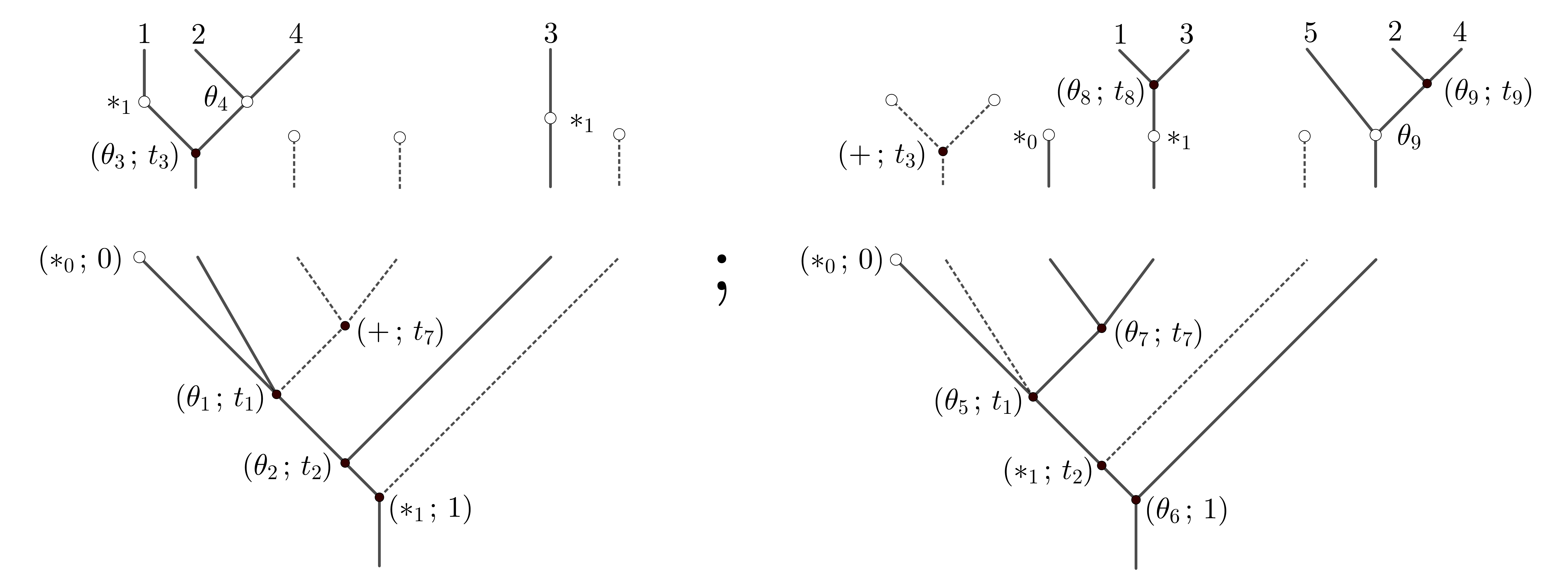}\vspace{-15pt}
\end{center}
\end{figure}

 Furthermore, the homotopy so obtained preserves the boundaries $\partial$ and $\partial'$ in the sense that the horizontal maps (again obtained by removing the $external$ edges) of the diagram below are homotopy equivalences \vspace{3pt}
$$
\xymatrix{
\partial\overline{\mathcal{I}b}(\mathbb{O})(m_{1}+1\,;\,i) \ar[r]_{\hspace{-130pt}\simeq} \ar[d] & \big(\partial\overline{\mathcal{I}b}_{O_{1}}(O_{1})(m_{1}+1)\times \overline{\mathcal{I}b}_{O_{2}}(O_{2})(i)\big) \hspace{-40pt} \underset{\substack{\vspace{3pt}\\ \partial\overline{\mathcal{I}b}_{O_{1}}(O_{1})(m_{1}+1)\times \partial\overline{\mathcal{I}b}_{O_{2}}(O_{2})(i)}}{\displaystyle \coprod} \hspace{-40pt}\big(\overline{\mathcal{I}b}_{O_{1}}(O_{1})(m_{1}+1)\times \partial\overline{\mathcal{I}b}_{O_{2}}(O_{2})(i)\big) \ar[d]^{\nu}\\
\partial'\overline{\mathcal{I}b}(\mathbb{O})(m_{1}+1\,;\,i) \ar[r]_{\hspace{-130pt}\simeq} \ar[d] & \big(\partial' \overline{\mathcal{I}b}_{O_{1}}(O_{1})(m_{1}+1)\times \overline{\mathcal{I}b}_{O_{2}}(O_{2})(i)\big) \hspace{-40pt} \underset{\partial'\overline{\mathcal{I}b}_{O_{1}}(O_{1})(m_{1}+1)\times \partial'\overline{\mathcal{I}b}_{O_{2}}(O_{2})(i)}{\displaystyle \coprod} \hspace{-40pt}\big(\overline{\mathcal{I}b}_{O_{1}}(O_{1})(m_{1}+1)\times \partial'\overline{\mathcal{I}b}_{O_{2}}(O_{2})(i)\big) \ar[d] \\
\overline{\mathcal{I}b}(\mathbb{O})(m_{1}+1\,;\,i) \ar[r]_{\hspace{-130pt}\simeq}  & \overline{\mathcal{I}b}_{O_{1}}(O_{1})(m_{1}+1)\times \overline{\mathcal{I}b}_{O_{2}}(O_{2})(i)  
}\vspace{5pt}
$$
In \cite{Ducoulombier18}, we show the connection between the boundary spaces above and the homotopy colimits (\ref{I2}). More precisely, one has the following diagrams in which the horizontal maps are homeomorphisms \cite[Theorem 4.14]{Ducoulombier18} and the vertical maps are cofibrations (can be proved by induction on the number of vertices): 
$$
\xymatrix@R=15pt{
\partial\overline{\mathcal{I}b}_{O_{1}}(O_{1})(m_{1}+1)\ar[d]\ar[r]^{\cong} & \underset {\partial\Psi_{m_{1}+1}}{\hocolim}\,\rho_{m_{1}+1}^{O_{1}}\ar[d]\\
\partial'\overline{\mathcal{I}b}_{O_{1}}(O_{1})(m_{1}+1)\ar[d]\ar[r]^{\cong} & \underset {\Psi'_{m_{1}+1}}{\hocolim}\,\rho_{m_{1}+1}^{O_{1}}\ar[d] \\
\overline{\mathcal{I}b}_{O_{1}}(O_{1})(m_{1}+1)\ar[r]^{\cong} & \underset {\Psi_{m_{1}+1}}{\hocolim}\,\rho_{m_{1}+1}^{O_{1}}
}
\hspace{40pt}
\xymatrix@R=15pt{
\partial\overline{\mathcal{I}b}_{O_{2}}(O_{2})(i)\ar[d]\ar[r]^{\cong} & \underset {\partial\Psi_{i}}{\hocolim}\,\rho_{i}^{O_{2}}\ar[d]\\
\partial'\overline{\mathcal{I}b}_{O_{2}}(O_{2})(i)\ar[d]\ar[r]^{\cong} & \underset {\Psi'_{i}}{\hocolim}\,\rho_{i}^{O_{2}}\ar[d] \\
\overline{\mathcal{I}b}_{O_{2}}(O_{2})(i)\ar[r]^{\cong} & \underset {\Psi_{i}}{\hocolim}\,\rho_{i}^{O_{2}}
}
$$
Since the operads $O_{1}$ are $O_{2}$ are assumed to be $j_{1}$-coherent and $j_{2}$-coherent,respectively, then the upper vertical maps in the above diagrams are weak equivalences.  Consequently, the map $\nu$ between the pushout product (which are cofibrant in the category of diagram in spaces) is a weak equivalence.Thus proves that the map from $\partial\overline{\mathcal{I}b}(\mathbb{O})(m_{1}+1\,;\,i)$ to $\partial'\overline{\mathcal{I}b}(\mathbb{O})(m_{1}+1\,;\,i)$ is a weak equivalence too.
\end{proof}

\section{Generalization and applications to the little cubes operads}\label{N4}

Unfortunately, we can not apply Theorem \ref{D3} to the family of operads $\mathcal{C}_{d_{1}},\ldots,\mathcal{C}_{d_{k}}$ relative to $\mathcal{C}_{n}$, since the little cubes operad is not $2$-reduced. We conjecture that there exists an analogue of Theorem \ref{D3} for families of operads $O_{1},\ldots,O_{k}$ relative to $O$ in which $O_{i}$ is a weakly $2$-reduced operad (i.e. $O_{i}(0)=\ast$ and $O_{i}(1)\simeq \ast$). More precisely, we expect the following statement:\vspace{7pt}

\begin{conj}\label{J3}
Let $O_{1},\ldots,O_{k}$ be a family of well-pointed $\Sigma$-cofibrant weakly $2$-reduced operads relative to another weakly  $2$-reduced operad $O$ and let $\eta:\mathbb{O}^{+}\rightarrow M$ be a map of $k$-fold bimodules such that $M(A_{1},\ldots,A_{k})=\ast$ for any elements $(A_{1},\ldots,A_{k})$ with $A_{i}\in \{+\,;\,\emptyset\}$. We assume that the applications \vspace{3pt}
$$
M(+,\ldots,+,A_{i},+,\ldots , +)\longrightarrow M(\emptyset,\ldots,\emptyset,A_{i},\emptyset,\ldots , \emptyset)\vspace{3pt}
$$
induced by the $k$-fold bimodule structure are weak equivalences. If $M$ is $\Sigma$-cofibrant and the operads $O_{1},\ldots,O_{k}$ are $j_{1}$-coherent, $\ldots$, $j_{k}$-coherent, respectively, then  one has\vspace{5pt}
\begin{equation*}
T_{\vec{r}}\,\mathbb{F}_{\vec{O}}(M)\simeq T_{\vec{r}}\,Ibimod_{\vec{O}}^{h}(\mathbb{O}\,;\,M^{-}),\hspace{15pt} \forall \vec{r}\leq \vec{j}=(j_{1},\ldots,j_{k}).\vspace{4pt}
\end{equation*}
In particular, if $M$ is $\Sigma$-cofibrant and the operads $O_{1},\ldots,O_{k}$ are coherent, then one has\vspace{2pt}
\begin{equation*}
\mathbb{F}_{\vec{O}}(M)\simeq Ibimod_{\vec{O}}^{h}(\mathbb{O}\,;\,M^{-}).\vspace{-15pt}
\end{equation*}
\end{conj}

\newpage

In this section, we prove the above conjecture for the family of weakly $2$-reduced operads $\mathcal{C}_{d_{1}},\ldots,\mathcal{C}_{d_{k}}$ relative to $\mathcal{C}_{n}$. As a consequence of this result together with Theorem \ref{K2}, we are able to give an explicit description of the iterated loop spaces associated to high-dimensional spaces of string links and their polynomial approximations. The last subsection provides another application of Theorem \ref{D3} to polynomial approximations of high-dimensional spaces of string links with singularities.

\subsection{Generalization to the family of little cubes operads}

In the rest of the paper, by $\vec{O}$, $\mathbb{O}$ and $\mathbb{O}^{+}$ we mean the objects introduced in Definition \ref{Z3}, Example \ref{E1} and Example \ref{M1}, respectively, associated to the family of operads $\mathcal{C}_{d_{1}},\ldots,\mathcal{C}_{d_{k}}$ relative to $\mathcal{C}_{n}$. This section is devoted to prove the following statement:\vspace{5pt}

\begin{thm}\label{R1}
Let $\eta:\mathbb{O}^{+}\rightarrow M$ be a $k$-fold bimodule map satisfying the conditions of Conjecture  \ref{J3}. Then, for any $\vec{r}\in \mathbb{N}^{k}$, one has the weak equivalences
$$
\begin{array}{rcl}\vspace{5pt}
\mathbb{F}_{\vec{O}}(M) & \simeq &  Ibimod_{\vec{O}}^{h}(\mathbb{O}\,;\,M^{-}), \\ 
T_{\vec{r}}\,\mathbb{F}_{\vec{O}}(M) & \simeq & T_{\vec{r}}\,Ibimod_{\vec{O}}^{h}(\mathbb{O}\,;\,M^{-}).
\end{array}
$$
In particular, if $d_{1}=\cdots = d_{k}=d$, then one has 
$$
\begin{array}{rcl}\vspace{5pt}
Map_{\ast}\big( \Sigma O(2)\,;\,  Bimod_{\vec{O}}^{h}(\mathbb{O}^{+}\,;\,M)\big) & \simeq &  Ibimod_{\vec{O}}^{h}(\mathbb{O}\,;\,M^{-}), \\ 
Map_{\ast}\big( \Sigma O(2)\,;\, T_{\vec{r}}\, Bimod_{\vec{O}}^{h}(\mathbb{O}^{+}\,;\,M)\big)& \simeq & T_{\vec{r}}\,Ibimod_{\vec{O}}^{h}(\mathbb{O}\,;\,M^{-}).
\end{array}
$$
\end{thm}

In order to prove the above statement, we use the Fulton-MacPherson operad which is a $2$-reduced operad weakly equivalent to the little cubes operad. We also need to recall some general facts about the homotopy theory of undercategories which are explained with more details in \cite[Section 7.2]{Ducoulombier18}. 

\subsubsection{Connection between the Fulton-MacPherson and the little cubes operads}

The Fulton-MacPherson operad $\mathcal{F}_d$ is a $2$-reduced operad weakly equivalent to the little cubes operad $\mathcal{C}_d$. It was introduced simultaneously by several people, in particular by Kontsevich and Getzler-Jones, since then it had a countless number of applications in mathematical physics and topology \cite{Getzler94,Kontsevich99,Salvatore99,Sinha04,Sinha09,Turchin13,
Volic14}. In this section we recall briefly the definition of the Fulton-MacPherson operad and its connection with the little cubes operads. Let $Conf(n\,;\,\mathbb{R}^{d})$ be the space of configurations of points in $\mathbb{R}^{d}$\vspace{5pt}
$$
Conf(n\,;\,\mathbb{R}^{d}):=\left\{
(x_{i})\in (\mathbb{R}^{d})^{n}\,\big|\, \forall i\neq j,\,\, x_{i}\neq x_{j}
\right\}.\vspace{5pt}
$$
Let $Conf_{0}(n\,;\,\mathbb{R}^{d})=Conf(n\,;\,\mathbb{R}^{d})/\sim$ be the quotient by the group of affine translations in $\mathbb{R}^{d}$ generated by translation and scalar multiplication. Then, we introduce the following map where $[0\,,\,+\infty]$ is the one point compactification of $[0\,,\,+\infty[$:
$$
\begin{array}{rcl}\vspace{7pt}
\iota_{n}:Conf_{0}(n\,;\,\mathbb{R}^{d}) & \longrightarrow & \left( S^{d-1}\right)^{\left(\hspace{-4pt} \begin{array}{c}\vspace{-2pt}
n \\ 
2
\end{array}\hspace{-4pt} \right)} \times \big( [0\,,\,+\infty]\big)^{\left(\hspace{-4pt} \begin{array}{c}\vspace{-2pt}
n \\ 
3
\end{array}\hspace{-4pt} \right)};\\ 
(x_{i})_{1\leq i\leq n} & \longmapsto & \left( \frac{x_{j}-x_{i}}{|x_{j}-x_{i}|} \right)_{(i\,;\,j)}\,\,;\,\, \left( \frac{|x_{i}-x_{j}|}{|x_{i}-x_{k}|} \right)_{(i\,;\,j\,;\,k)}.
\end{array} \vspace{5pt}
$$

\begin{defi}\textbf{The Fulton-MacPherson operad $\mathcal{F}_{d}$}

\noindent The $d$-dimensional Fulton-MacPherson operad $\mathcal{F}_{d}$ is defined by $\mathcal{F}_{d}(n)=\overline{Im(\iota_{n})}$. For instance, the point resulting from the operadic composition in Figure \ref{Q4} represents an infinitesimal configuration in $\mathcal{F}_{d}(6)$ in which the points $1$, $5$ and $6$ collided together, and so did $2$, $3$ and $4$. But the distance between the points $2$ and $4$ is infinitesimally small compared to the distance between $2$ and $3$. In the following, we recall the main properties of $\mathcal{F}_{d}$.\vspace{3pt}
\begin{itemize}
\item[$\blacktriangleright$] $\mathcal{F}_{d}(0)=\mathcal{F}_{d}(1)=\ast$ and $\mathcal{F}_{d}(n)$, with $n\geq 2$, is a manifold with corners whose interior is $Conf_{0}(n\,;\,\mathbb{R}^{d})$.\vspace{3pt}
\item[$\blacktriangleright$] $\mathcal{F}_{d}$ is an operad in compact semi-algebraic sets and the maps $\lambda_{i}^{\ast}:\mathcal{F}_{d}(n)\rightarrow \mathcal{F}_{d}(n-1)$, induced by the $\Lambda$-structure, are semi-algebraic fibrations. \vspace{3pt}
\item[$\blacktriangleright$] $\mathcal{F}_{d}$ is cofibrant in the Reedy model category of reduced operads. \vspace{-10pt}
\end{itemize}
\end{defi}

\begin{figure}[!h]
\begin{center}
\includegraphics[scale=0.18]{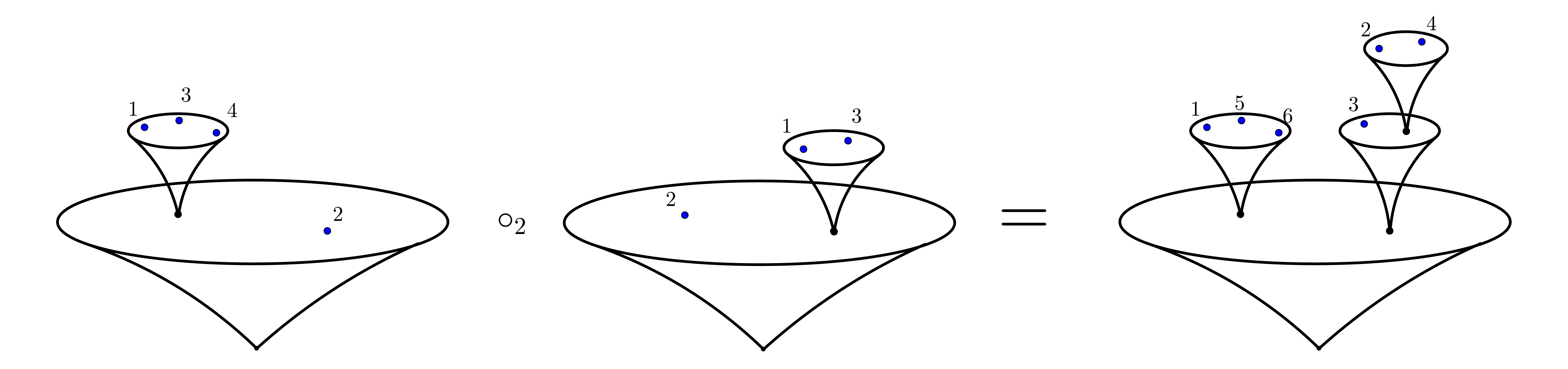}\vspace{-5pt}
\caption{Illustration of the operadic composition $\circ_{2}:\mathcal{F}_{d}(4)\times \mathcal{F}_{d}(3)\rightarrow \mathcal{F}_{d}(6)$.}\label{Q4}\vspace{-10pt}
\end{center}
\end{figure}

Unfortunately, there is no direct operadic map between the Fulton-MacPherson operad $\mathcal{F}_{d}$ and the little cubes operad $\mathcal{C}_{d}$. Nevertheless,  Salvatore builds in \cite[Proposition 3.9]{Salvatore99} a zig-zag of weak equivalences of reduced operads\vspace{3pt}
\begin{equation}\label{Q5}
\xymatrix{
\mathcal{C}_{d} & \ar[l]_{\mu''} \mathcal{W}(\mathcal{C}_{d}) \ar[r]^{g} & \mathcal{F}_{d}
}\vspace{5pt}
\end{equation}
where $\mathcal{W}(\mathcal{C}_{d})$ is the Boardman-Vogt resolution of the little cubes operad in the Reedy category of reduced operads. More precisely, if $O$ is a reduced operad, then \vspace{5pt}
$$
\begin{array}{rl}\vspace{7pt}
\mathcal{W}(O)(0)= & \ast, \\ 
\mathcal{W}(O)(n)= & \left.\left(\underset{T\in Tree_{n}}{\displaystyle\coprod} \,\,\underset{v\in V(T)}{\displaystyle\prod} O(|v|) \times \underset{e\in E^{int}(T)}{\displaystyle\prod} [0\,,\,1]\right)\,\, \right/ \sim,
\end{array} \vspace{5pt}
$$
where $Tree_{n}$ is the set of trees having $n$ leaves and without univalent vertices. The equivalence relation is generated by the compatibility with the symmetric group action and the unit axiom. Furthermore, if an inner edge is indexed by $0$, then we contract it using the operadic structure of $O$. There exists a map of operads\vspace{5pt}
$$
\mu'': \mathcal{W}(O)\longrightarrow O,\vspace{5pt}
$$
sending the real numbers indexing the inner edges to $0$. This map is a weak equivalence in the category of operads and it is a homotopy equivalence in the category of sequences. In the latter case, the homotopy consists in bringing the real numbers to $0$.\vspace{7pt}

On the other hand, Salvatore \cite{Salvatore99} gives an explicit description description of the operadic map $g$ in (\ref{Q5}). For the present work, we recall that the argument making $g$ into a weak equivalence is the commutativity of the diagram\vspace{5pt}
\begin{equation}\label{Q6}
\xymatrix{
\mathcal{W}(\mathcal{C}_{d})(n) \ar[rr]_{g} & & \mathcal{F}_{d}(n)  \\
\mathcal{C}_{d}(n)\ar[u]^{\simeq} \ar[r]^{\hspace{-35pt}\simeq}_{\hspace{-35pt}\gamma_{1}} & Conf_{0}(n\,;\,]0\,,\,1[^{d}) \ar[r]^{\simeq}_{\gamma_{2}} &  Conf_{0}(n\,;\,\mathbb{R}^{d}) \ar[u]^{\simeq}_{i}
}\vspace{3pt}
\end{equation}
The left horizontal map is the homotopy inverse of $\mu''$, the map $\gamma_{1}$ sends a configuration of rectangles to their centers, the map $\gamma_{2}$ is induced by the homeomorphism $]0\,,\,1[^{d}\hookrightarrow \mathbb{R}^{d}$ and the map $i$ is the h into the interior of $\mathcal{F}_{d}(n)$.\vspace{7pt}

\begin{notat}\label{Q7} 
For $d_{1}\leq \cdots\leq d_{k}<n$, we use the following notation:\vspace{3pt}
\begin{itemize}
\item[$\blacktriangleright$] $\vec{\mathcal{C}}$ is associated with the family of operads $\mathcal{C}_{d_{1}},\ldots,\mathcal{C}_{d_{k}}$ relative to $\mathcal{C}_{n}$;\vspace{3pt}
\item[$\blacktriangleright$] $\vec{\mathcal{W}(\mathcal{C})}$ is associated with the family of operads $\mathcal{W}(\mathcal{C}_{d_{1}}),\ldots,\mathcal{W}(\mathcal{C}_{d_{k}})$ relative to $\mathcal{W}(\mathcal{C}_{n})$;\vspace{3pt}
\item[$\blacktriangleright$] $\vec{\mathcal{F}}$ is associated with the family of operads $\mathcal{F}_{d_{1}},\ldots,\mathcal{F}_{d_{k}}$ relative to $\mathcal{F}_{n}$;
\end{itemize}\vspace{7pt}
\end{notat}

\begin{pro}\label{Q8}
The weak equivalences (\ref{Q5}) extends to weak equivalences 
$$
\xymatrix{
\vec{\mathcal{C}} & \ar[l]_{\vec{\mu''}} \vec{\mathcal{W}(\mathcal{C})} \ar[r]^{\vec{g}} & \vec{\mathcal{F}}.
}\vspace{5pt}
$$
\end{pro}

\begin{proof}
Without loss of generality, we prove the proposition for $k=2$. Let $S=(S_{1}\,;\,S_{2})$ be an element in $\mathcal{P}_{2}(\{1,\ldots,n\})$. If $S_{1}=+$ (the same arguments work in the case where $S_{2}=+$), then the following two continuous maps: \vspace{5pt}
$$
\mu''_{S}:\vec{\mathcal{W}(\mathcal{C})}_{S}(S_{1}\,;\, S_{2})\longrightarrow \vec{\mathcal{C}}_{S}(S_{1}\,;\,S_{2})\hspace{15pt} \text{and}\hspace{15pt} g_{S}:\vec{\mathcal{W}(\mathcal{C})}_{S}(S_{1}\,;\, S_{2})\longrightarrow \vec{\mathcal{F}}_{S}(S_{1}\,;\,S_{2})\vspace{5pt}$$ 
coincide by definition with the continuous maps
$$
\mu'':\mathcal{W}(\mathcal{C}_{d_{2}})(S_{2})\longrightarrow \mathcal{C}_{d_{2}}(S_{2})\hspace{15pt} \text{and}\hspace{15pt} g:\mathcal{W}(\mathcal{C}_{d_{2}})( S_{2})\longrightarrow \mathcal{F}_{d_{2}}(S_{2}),
$$ 
respectively, which are already weak equivalences. \vspace{7pt}

From now on, we assume that $S_{1}\neq +$ and $S_{2}\neq +$. Since the inclusion from $\mathcal{C}_{d_{1}}$ to $\mathcal{C}_{n}$ can be factorized through $\mathcal{C}_{d_{2}}$, one has\vspace{5pt}
$$
\mu''_{S}:\mathcal{W}(\mathcal{C}_{d_{1}})(S_{1})\underset{\mathcal{W}(\mathcal{C}_{d_{2}})(S_{1}\cap S_{2})}{\displaystyle \prod} \mathcal{W}(\mathcal{C}_{d_{2}})(S_{2})\longrightarrow \mathcal{C}_{d_{1}}(S_{1})\underset{\mathcal{C}_{d_{2}}(S_{1}\cap S_{2})}{\displaystyle \prod} \mathcal{C}_{d_{2}}(S_{2}).\vspace{5pt}
$$ 
The homotopy bringing the real numbers to $0$ is still well defined in the pullback product and the map $\mu''_{S}$ is a weak equivalence. \vspace{5pt}

In order to prove that the map $g_{S}$ is also a weak equivalence, we adapt the diagram (\ref{Q6}) introduced by Salvatore. First, the homotopy inverse of $\mu''_{S}$\vspace{5pt}
$$
\mathcal{C}_{d_{1}}(S_{1})\underset{\mathcal{C}_{d_{2}}(S_{1}\cap S_{2})}{\displaystyle \prod} \mathcal{C}_{d_{2}}(S_{2}) \longrightarrow \mathcal{W}(\mathcal{C}_{d_{1}})(S_{1})\underset{\mathcal{W}(\mathcal{C}_{d_{2}})(S_{1}\cap S_{2})}{\displaystyle \prod} \mathcal{W}(\mathcal{C}_{d_{2}})(S_{2})\vspace{5pt}
$$
is a weak equivalence. Then, the map sending a rectangle to its center provides a weak equivalence\vspace{5pt}
$$
\mathcal{C}_{d_{1}}(S_{1})\underset{\mathcal{C}_{d_{2}}(S_{1}\cap S_{2})}{\displaystyle \prod} \mathcal{C}_{d_{2}}(S_{2}) \longrightarrow Conf_{0}(S_{1}\,;\, ]0\,,\,1[^{d_{1}})\hspace{-20pt} \underset{Conf_{0}(S_{1}\cap S_{2}\,;\, ]0\,,\,1[^{d_{2}})}{\displaystyle \prod}\hspace{-20pt} Conf_{0}(S_{2}\,;\, ]0\,,\,1[^{d_{2}}),\vspace{5pt}
$$
where the map \vspace{5pt}
$$Conf_{0}(S_{1}\,;\, ]0\,,\,1[^{d_{1}})\longrightarrow Conf_{0}(S_{1}\cap S_{2}\,;\, ]0\,,\,1[^{d_{2}})\vspace{5pt}$$
is induced by the map forgetting the points in $S_{1}\setminus (S_{1}\cap S_{2})$ followed by the inclusion $]0\,,\,1[^{d_{1}}\hookrightarrow ]0\,,\,1[^{d_{2}}$. The homeomorphisms $]0\,,\,1[^{d_{1}}\rightarrow \mathbb{R}^{d_{1}}$ and $]0\,,\,1[^{d_{2}}\rightarrow \mathbb{R}^{d_{2}}$ give rise to a weak equivalence\vspace{5pt}
$$
Conf_{0}(S_{1}\,;\, ]0\,,\,1[^{d_{1}}) \hspace{-20pt}\underset{Conf_{0}(S_{1}\cap S_{2}\,;\, ]0\,,\,1[^{d_{2}})}{\displaystyle \prod} \hspace{-20pt} Conf_{0}(S_{2}\,;\, ]0\,,\,1[^{d_{2}})  \longrightarrow 
Conf_{0}(S_{1}\,;\, \mathbb{R}^{d_{1}}) \hspace{-20pt} \underset{Conf_{0}(S_{1}\cap S_{2}\,;\, \mathbb{R}^{d_{2}})}{\displaystyle \prod} \hspace{-20pt} Conf_{0}(S_{2}\,;\, \mathbb{R}^{d_{2}}).\vspace{5pt}
$$ 
Finally, the map between the limits\vspace{5pt}
$$
\xymatrix@C=13pt{
Conf_{0}(S_{1}\,;\, \mathbb{R}^{d_{1}}) \hspace{-20pt} \underset{Conf_{0}(S_{1}\cap S_{2}\,;\, \mathbb{R}^{d_{2}})}{\displaystyle \prod} \hspace{-20pt} Conf_{0}(S_{2}\,;\, \mathbb{R}^{d_{2}})
\ar[d] \ar@{=}[r] & lim\big(\hspace{-20pt} & Conf_{0}(S_{1}\,;\, \mathbb{R}^{d_{1}})\ar[d]_{t_{1}}\ar[r] &\ar[d]_{t_{2}} Conf_{0}(S_{1}\cap S_{2}\,;\, \mathbb{R}^{d_{2}}) & \ar[l]\ar[d]_{t_{3}} Conf_{0}(S_{2}\,;\, \mathbb{R}^{d_{2}})\big)\\
\mathcal{F}_{d_{1}}(S_{1})\underset{\mathcal{F}_{d_{2}}(S_{1}\cap S_{2})}{\displaystyle \prod} \mathcal{F}_{d_{2}}(S_{2}) \ar@{=}[r] &  lim\big(\hspace{-31pt} & \mathcal{F}_{d_{1}}(S_{1})\ar[r] & \mathcal{F}_{d_{2}}(S_{1}\cap S_{2}) &\ar[l] \mathcal{F}_{d_{2}}(S_{2})\big)
}
$$
is a weak equivalence. Indeed, the right horizontal maps in the limits are fibrations and all the spaces are fibrant. Consequently the limits are homotopically invariant. Furthermore, the vertical maps $t_{1}$, $t_{2}$ and $t_{3}$ are weak equivalences and they provide a weak equivalence between the limit. Thus finishes the proof of the proposition for $k=2$. The general case can be proved by induction.
\end{proof}

\subsubsection{General facts about homotopy theory}

Given a model category structure $\mathcal{C}$ and an object $c\in\mathcal{C}$ , the undercategory $(c\downarrow \mathcal{C})$ inherits a model category structure in which a map is a weak equivalence, fibration or a cofibration if the corresponding map in $\mathcal{C}$ is a weak equivalence, fibration or a cofibration, respectively. From a Quillen adjunction between two model categories\vspace{5pt}
\begin{equation}\label{I6}
\mathcal{L}:\mathcal{C}\leftrightarrows \mathcal{D}:\mathcal{R}\vspace{5pt}
\end{equation} 
as well as two objects $c\in \mathcal{C}$ and $d\in \mathcal{D}$ together with a map\vspace{5pt}
\begin{equation}\label{I7}
\phi:c\rightarrow \mathcal{R}(d),\vspace{5pt}
\end{equation}
then there is an adjunction \vspace{5pt}
\begin{equation}\label{I8}
\hat{\mathcal{L}}:(c\downarrow \mathcal{C})\leftrightarrows (d\downarrow \mathcal{D}):\hat{\mathcal{R}},\vspace{5pt}
\end{equation}
where $\hat{\mathcal{R}}$  sends $d\rightarrow d_{1}$ to the composition $c\rightarrow \mathcal{R}(d)\rightarrow \mathcal{R}(d_{1})$ and $\hat{\mathcal{L}}$ sends $c\rightarrow c_{1}$ to the pushout map $d\rightarrow d \sqcup _{\mathcal{L}(c)}\mathcal{L}(c_{1})$ in which the map $\mathcal{L}(c)\rightarrow d$ is the left adjunct of $\phi$.\vspace{5pt}


\begin{lmm}[\cite{Ducoulombier18}]\label{J1} 
In case (\ref{I6}) is a Quillen equivalence between left proper model categories, the functor $\mathcal{C}$ preserves weak equivalences and (\ref{I7}) is a weak equivalence, then (\ref{I8}) is a Quillen equivalence.  \vspace{5pt}
\end{lmm}

\begin{pro}\label{J4}
Let $\vec{\alpha}:\vec{O}\rightarrow \vec{O}'$ be a weak equivalence between two families of Reedy cofibrant operads and let $\eta:\mathbb{O}\rightarrow M$ be a cofibration in the Reedy model category of $k$-fold bimodules over $\vec{O}$. Then, there exists a $k$-fold bimodule $M_{1}$ over $\vec{O}'$ endowed with a map $\eta_{1}:\mathbb{O}'\rightarrow M_{1}$ of $k$-fold bimodules over $\vec{O'}$ and a weak equivalence $M\rightarrow M_{1}$ of $k$-fold bimodules over $\vec{O}$ such that the following square commutes:\vspace{5pt}
\begin{equation}\label{J2}
\xymatrix{
\mathbb{O} \ar[r]^{\eta} \ar[d]_{\simeq}^{\alpha_{\ast}} & M \ar[d]^{\simeq} \\
\mathbb{O}'\ar[r]^{\eta_{1}} & M_{1}
}
\end{equation}
\end{pro}\vspace{2pt}

\begin{proof}
The proof is similar to \cite[Proposition 7.8]{Ducoulombier18}. By Theorem \ref{J0}, the categories $\Lambda Bimod_{\vec{O}}$ and $\Lambda Bimod_{\vec{O}'}$ are left proper. As a consequence of Lemma \ref{J1} applied to the adjunction\vspace{5pt}
$$
\alpha^{!}_{B}:\Lambda Bimod_{\vec{O}} \leftrightarrows \Lambda Bimod_{\vec{O}'}: \alpha^{\ast}_{B}\vspace{5pt}
$$
and the weak equivalence of $k$-fold bimodules over $\vec{O}$\vspace{5pt}
$$
\alpha_{\ast}:\mathbb{O}\longrightarrow \mathbb{O}'\vspace{-9pt}
$$

\newpage

\noindent we obtain a Quillen equivalence \vspace{5pt}
$$
\hat{\alpha}^{!}_{B}:\big( \mathbb{O}\downarrow\Lambda Bimod_{\vec{O}}\big) \leftrightarrows \big( \mathbb{O}'\downarrow \Lambda Bimod_{\vec{O}'}\big): \hat{\alpha}^{\ast}_{B}.\vspace{3pt}
$$
Since $\eta:\mathbb{O}\rightarrow M$ is a cofibrant object in the category under $\mathbb{O}$, one can take $M_{1}\coloneqq\hat{\alpha}^{!}_{B}(M)$. The square (\ref{J2}) commutes because the natural map $M\rightarrow M_{1}$ is a morphism in the undercategory  $\mathbb{O}\downarrow\Lambda Bimod_{\vec{O}}$.
\end{proof}

\subsubsection{Proof of Theorem \ref{R1}}

According to Notation \ref{Q7}, we denote by $\mathbb{W}(\mathcal{C})$ and $\mathbb{F}$ the $k$-fold bimodules over $\vec{\mathcal{W}(\mathcal{C})}$ and $\vec{\mathcal{F}}$, respectively, defined as follows:\vspace{5pt}
$$
\begin{array}{rcll}\vspace{9pt}
\mathbb{W}(\mathcal{C})(n_{1},\ldots,n_{k}) & \coloneqq & \underset{\substack{1\leq i \leq k \\ n_{i}\neq +}}{\displaystyle \prod}\mathcal{W}(\mathcal{C}_{d_{i}})(n_{i}), & \text{for } n_{i}\in \mathbb{N}\sqcup\{+\} \text{ and } (n_{1},\ldots, n_{k})\neq (+,\ldots,+), \\ 
\mathbb{F}(n_{1},\ldots,n_{k}) & \coloneqq & \underset{\substack{1\leq i \leq k \\ n_{i}\neq +}}{\displaystyle \prod}\mathcal{F}_{d_{i}}(n_{i}), & \text{for } n_{i}\in \mathbb{N}\sqcup\{+\} \text{ and } (n_{1},\ldots, n_{k})\neq (+,\ldots,+).
\end{array}\vspace{5pt}
$$
Since $\vec{\mu''}:\vec{\mathcal{W}(\mathcal{C})}\rightarrow \vec{\mathcal{C}}$ and $\mu'':\mathbb{W}(\mathcal{C})\rightarrow \mathbb{O}^{+}$ are weak equivalences, one has the identifications\vspace{5pt}
$$
\begin{array}{lllll}
Bimod^{h}_{\vec{\mathcal{C}}}(\mathbb{O}^{+}\,;\,M) & \cong & Bimod_{\vec{\mathcal{W}(\mathcal{C})}}^{h}(\mathbb{O}^{+}\,;\,M) 
& \cong & Bimod_{\vec{\mathcal{W}(\mathcal{C})}}^{h}(\mathbb{W}(\mathcal{C})\,;\,M).
\end{array} \vspace{3pt}
$$ 
Let $M^{c}$ be the $k$-fold sequence obtained by taking a factorization $\mathbb{W}(\mathcal{C})\rightarrow M^{c}\rightarrow M$ into a cofibration and a trivial fibration of $\mathbb{W}(\mathcal{C})\rightarrow M$ in the category of $k$-fold bimodules over $\vec{\mathcal{W}(\mathcal{C})}$. Then, we denote by $M^{c}_{1}$ the $k$-fold bimodule over $\vec{\mathcal{F}}$ obtained from Proposition \ref{J4} applied to the cofibration $\mathbb{W}(\mathcal{C})\rightarrow M^{c}$ and the weak equivalence $\vec{g}:\vec{\mathcal{W}(\mathcal{C})} \rightarrow \vec{\mathcal{F}}$. Consequently, one has\vspace{5pt}
$$
\begin{array}{lllll}
Bimod^{h}_{\vec{\mathcal{C}}}(\mathbb{O}^{+}\,;\,M)  & \cong &   Bimod_{\vec{\mathcal{W}(\mathcal{C})}}^{h}(\mathbb{W}(\mathcal{C})\,;\,M^{c})  & \cong &  Bimod_{\vec{\mathcal{W}(\mathcal{C})}}^{h}(\mathbb{F}\,;\,M^{c}_{1}).
\end{array} \vspace{3pt}
$$  
Finally, the Quillen equivalence induced by the weak equivalence $\vec{g}:\vec{\mathcal{W}(\mathcal{C})} \rightarrow \vec{\mathcal{F}}$ gives rise to \vspace{5pt}
$$
\begin{array}{llll}
Bimod^{h}_{\vec{\mathcal{C}}}(\mathbb{O}^{+}\,;\,M) & \cong &   Bimod_{\vec{\mathcal{F}}}^{h}(\mathbb{F}\,;\,M^{c}_{1}).
\end{array}\vspace{3pt}
$$
Similarly, one has the same identifications in the context of $k$-fold infinitesimal bimodules: \vspace{5pt}
$$
\begin{array}{llll}\vspace{5pt}
Ibimod^{h}_{\vec{\mathcal{C}}}(\mathbb{O}\,;\,M^{-}) & \simeq  Ibimod_{\vec{\mathcal{W}(\mathcal{C})}}^{h}(\mathbb{O}^{+}\,;\,M^{-}) & \\ \vspace{5pt} 
& \simeq  Ibimod_{\vec{\mathcal{W}(\mathcal{C})}}^{h}(\mathbb{W}(\mathcal{C})^{-}\,;\,M^{-}) &  \\ \vspace{5pt}
 & \simeq  Ibimod_{\vec{\mathcal{W}(\mathcal{C})}}^{h}(\mathbb{W}(\mathcal{C})^{-}\,;\,(M^{c})^{-}) & \\ \vspace{5pt}
 & \simeq  Ibimod_{\vec{\mathcal{W}(\mathcal{C})}}^{h}(\mathbb{F}^{-}\,;\,(M^{c}_{1})^{-})  & \simeq  Ibimod_{\vec{\mathcal{F}}}^{h}(\mathbb{F}^{-}\,;\,(M^{c}_{1})^{-}),
\end{array} \vspace{5pt}
$$  
Finally, the theorem follows from the identifications\vspace{5pt}
$$
\mathbb{F}_{\vec{\mathcal{C}}}(M)\hspace{3pt}\cong \hspace{3pt}\mathbb{F}_{\vec{\mathcal{F}}}(M_{1}^{c}) \hspace{3pt}\simeq \hspace{3pt} Ibimod_{\vec{\mathcal{F}}}^{h}(\mathbb{F}^{-}\,;\,(M^{c}_{1})^{-})\hspace{3pt} \cong \hspace{3pt} Ibimod^{h}_{\vec{\mathcal{C}}}(\mathbb{O}\,;\,M^{-}),\vspace{5pt}
$$
where the weak equivalence is a consequence of Theorem \ref{D3}. The reader can easily check that the same arguments work for the truncated case. Thus finishes the proof of the theorem.

\newpage

\subsection{Iterated loop spaces associated to embedding spaces}\label{M8}

In this section, $\vec{O}$, $\mathbb{O}$ and $\mathbb{O}^{+}$ are the objects introduced in Definition \ref{Z3}, Example \ref{E1} and Example \ref{M1}, respectively, associated to the family of weakly $2$-reduced operads $\mathcal{C}_{d_{1}},\ldots,\mathcal{C}_{d_{k}}$ relative to $\mathcal{C}_{n}$. Without loss of generality, we assume that $d_{1}\leq \cdots \leq d_{k} < n$. Let $\eta:\mathbb{O}^{+}\rightarrow M$ be a map of $k$-fold bimodules satisfying the conditions of Conjecture \ref{J3}. According to Theorem \ref{R1}, one has\vspace{5pt}
$$
\begin{array}{rcl}\vspace{7pt}
\mathbb{F}_{\vec{O}}(M) & \simeq &  Ibimod_{\vec{O}}^{h}(\mathbb{O}\,;\,M^{-}), \\ 
T_{\vec{r}}\,\mathbb{F}_{\vec{O}}(M) & \simeq & T_{\vec{r}}\,Ibimod_{\vec{O}}^{h}(\mathbb{O}\,;\,M^{-}).
\end{array}\vspace{3pt}
$$

By construction, the spaces $\Sigma \vec{O}(2)$ and $\Sigma \mathcal{C}_{d_{i}}(2)$ are homotopically equivalent to the spheres $\Sigma S^{d_{1}-1}\cong S^{d_{1}}$ and $\Sigma S^{d_{i}-1}\cong S^{d_{i}}$, respectively. Then, according to the notation introduced in Section \ref{D5}, we introduce the covariant functor $F'_{\vec{O}\,;\,M}$ from the category $C_{k}$  to spaces given by \vspace{5pt}
\begin{equation}\label{L4}
\begin{array}{rl}\vspace{9pt}
F'_{\vec{O}\,;\,M}(0)= & Map_{\ast}\left( 
S^{0}\,;\, Bimod_{\vec{O}}^{h}(\mathbb{O}^{+}\,;\, M)\right), \\ \vspace{9pt}
F'_{\vec{O}\,;\,M}(i)= & Map_{\ast}\left( 
S^{d_{i}-d_{1}}\,;\, Bimod_{\mathcal{C}_{d_{i}}}^{h}(\mathcal{C}_{d_{i}}\,;\, M_{i})\right),\hspace{15pt} \text{for } i\in \{1,\ldots, k\}, \\ 
F'_{\vec{O}\,;\,M}(0\,,\,i)= & Map_{\ast}\left( 
S^{0}\,;\, Bimod_{\mathcal{C}_{d_{i}}}^{h}(\mathcal{C}_{d_{i}}\,;\, M_{i})\right),
\end{array}\vspace{3pt}
\end{equation}
On morphisms, there are two cases to consider. If the morphism corresponds to an element of the form $C_{k}(i\,;\,(0\,,\,i))$, then the functor is defined using the inclusion from the sphere of dimension $0$ to the sphere of dimension $d_{i}-d_{1}$. On the other hand, if the morphism corresponds to an element of the form $C_{k}(0\,;\,(0\,,\,i))$, then the functor $F'_{\vec{O}\,;\,M}$ is defined using the restriction map\vspace{1pt}
$$
rest_{i}: Bimod_{\vec{O}}^{h}(\mathbb{O}^{+}\,;\, M) \longrightarrow Bimod_{\mathcal{C}_{d_{i}}}^{h}(\mathcal{C}_{d_{i}}\,;\, M_{i}).\vspace{3pt}
$$
By definition, one has the following relation between the functors $F'_{\vec{O}\,;\,M}$ and $F_{\vec{O}\,;\,M}$:\vspace{1pt}
$$
\begin{array}{rll}\vspace{10pt}
F_{\vec{O}\,;\,M}(i)\simeq &  Map_{\ast}\big(\,S^{d_{1}}\,;\,F'_{\vec{O}\,;\,M}(i)\,\big) &  = \Omega^{d_{1}}F'_{\vec{O}\,;\,M}(i),\\ 
F_{\vec{O}\,;\,M}(0\,,\,i)\simeq &  Map_{\ast}\big(\,S^{d_{1}}\,;\,F'_{\vec{O}\,;\,M}(0\,,\,i)\,\big) &  = \Omega^{d_{1}}F'_{\vec{O}\,;\,M}(0\,,\,i).
\end{array} \vspace{3pt}
$$
So it defines a natural homotopy equivalence between the two functors.  Consequently, there is a isomorphism between the limits in which the last identification is due to the fact the loop space functor (seen as a limit) commute with the limit indexed by the category $C_{k}$:\vspace{2pt}
$$
\mathbb{F}_{\vec{O}}(M)=\hspace{-7pt}\underset{\hspace{20pt}C_{k}}{lim}\,F_{\vec{O}\,;\,M} \simeq \hspace{-7pt}\underset{\hspace{20pt}C_{k}}{lim}\,\Omega^{d_{1}}F'_{\vec{O}\,;\,M}\cong \Omega^{d_{1}}\big(\hspace{-7pt}\underset{\hspace{20pt}C_{k}}{lim}\,F'_{\vec{O}\,;\,M}\big)=:\Omega^{d_{1}}\mathbb{F}'_{\vec{O}}(M).\vspace{3pt}
$$
Similarly, one has the same description in the context of truncated bimodules. We only need to consider in the definition (\ref{L4}) the truncated spaces $T_{\vec{r}}\,Bimod_{\vec{O}}^{h}(\mathbb{O}\,;\,M)$ and $T_{\vec{r}}\,Bimod_{O_{i}}^{h}(O_{i}\,;\,M_{i})$ instead of $Bimod_{\vec{O}}^{h}(\mathbb{O}\,;\,M)$\vspace{-5pt} and $Bimod_{O_{i}}^{h}(O_{i}\,;\,M_{i})$, respectively. \vspace{5pt}

\begin{thm}\label{L6}
For any $\vec{r}\in \mathbb{N}^{k}$, one has \vspace{5pt}
$$
Ibimod_{\vec{O}}^{h}(\mathbb{O}\,;\,M^{-})\simeq \Omega^{d_{1}}\mathbb{F}'_{\vec{O}}(M)\hspace{15pt}\text{and}\hspace{15pt} Ibimod_{\vec{O}}^{h}(\mathbb{O}\,;\,M^{-})\simeq \Omega^{d_{1}}\big(\,T_{\vec{r}}\,\mathbb{F}'_{\vec{O}}(M)\,\big).\vspace{2pt}
$$
In particular, if $d_{1}=\cdots = d_{k}=d$, then one has\vspace{5pt}
$$
Ibimod_{\vec{O}}^{h}(\mathbb{O}\,;\,M^{-})\simeq\Omega^{d} Bimod_{\vec{O}}^{h}(\mathbb{O}^{+}\,;\,M) \hspace{15pt}\text{and}\hspace{15pt}
T_{\vec{r}}\,Ibimod_{\vec{O}}^{h}(\mathbb{O}\,;\,M^{-})\simeq
\Omega^{d}\big( T_{\vec{r}}\, Bimod_{\vec{O}}^{h}(\mathbb{O}^{+}\,;\,M)\big).\vspace{-13pt}
$$
\end{thm}

\newpage

As a consequence of Theorem \ref{K2} together with Theorem \ref{L6}, applied to the map of $k$-fold bimodules $\eta:\vec{O}\rightarrow \mathcal{R}_{n}^{k}$, one has the following statement:\vspace{5pt}

\begin{thm}\label{R3}
We recall that $d_{1}\leq \cdots \leq d_{k}<n$ is a family of integers. One has the following description of the loop spaces associated to high-dimensional spaces of string links: \vspace{5pt}
$$
\begin{array}{rcll}\vspace{9pt}
\overline{Emb}_{c}\big(\, \underset{1\leq i\leq k}{\displaystyle\coprod} \mathbb{R}^{d_{i}}\,;\,\mathbb{R}^{n}\,\big) & \simeq & \Omega^{d_{1}}\mathbb{F}'_{\vec{O}}(\mathcal{R}_{n}^{k}), & \text{if } d_{k}+3\leq n,\\
T_{\vec{r}}\,\overline{Emb}_{c}\big(\, \underset{1\leq i\leq k}{\displaystyle\coprod} \mathbb{R}^{d_{i}}\,;\,\mathbb{R}^{n}\,\big) & \simeq & \Omega^{d_{1}}\big(\, T_{\vec{r}}\,\mathbb{F}'_{\vec{O}}(\mathcal{R}_{n}^{k})\,\big),& \text{for any } \vec{r}\in (\mathbb{N}\sqcup \{\infty\})^{k}. 
\end{array}\vspace{2pt}
$$ 
In particular, if $d_{1}=\cdots =d_{k}=d$, then one has\vspace{5pt}
$$
\begin{array}{rcll}\vspace{9pt}
\overline{Emb}_{c}\big(\, \underset{1\leq i\leq k}{\displaystyle\coprod} \mathbb{R}^{d_{i}}\,;\,\mathbb{R}^{n}\,\big) & \simeq & \Omega^{d} Bimod_{\vec{O}}^{h}(\mathbb{O}^{+}\,;\,\mathcal{R}_{n}^{k}), & \text{if } d_{k}+3\leq n,\\
T_{\vec{r}}\,\overline{Emb}_{c}\big(\, \underset{1\leq i\leq k}{\displaystyle\coprod} \mathbb{R}^{d_{i}}\,;\,\mathbb{R}^{n}\,\big) & \simeq & \Omega^{d}\big( T_{\vec{r}}\, Bimod_{\vec{O}}^{h}(\mathbb{O}^{+}\,;\,\mathcal{R}_{n}^{k})\big) & \text{for any } \vec{r}\in (\mathbb{N}\sqcup \{\infty\})^{k}.

\end{array} 
$$ 
\end{thm}

\subsection{Polynomial approximation of embedding spaces with singularities}\label{N5}

Similarly to the previous section, $\vec{O}$ and $\mathbb{O}$ are the $k$-fold sequences associated with the family of weakly $2$-reduced operads $\mathcal{C}_{d_{1}},\ldots,\mathcal{C}_{d_{k}}$ relative to the contractible space $\mathcal{C}_{n}(1)$, with $d_{1}\leq \cdots \leq d_{k} < n$. In the following, we introduce another application of Theorems \ref{J3} and \ref{L6}. For a family $\vec{u}=\{u_{pq}\}_{1\leq p\leq q\leq k}$, with $u_{pq}\in \mathbb{N}_{>0}\sqcup\{\infty\}$, we consider the space of $\vec{u}$-immersions compactly supported\vspace{5pt}
$$
Imm_{c}^{\vec{u}}\big(\, \underset{1\leq i\leq k}{\bigsqcup} \mathbb{R}^{d_{i}}\,;\,\mathbb{R}^{n}\,\big)\subset Imm_{c}\big(\, \underset{1\leq i\leq k}{\bigsqcup} \mathbb{R}^{d_{i}}\,;\,\mathbb{R}^{n}\,\big),\vspace{2pt}
$$
formed by immersions $f=(f_{1},\ldots,f_{k}):\sqcup_{i}\mathbb{R}^{d_{i}}\rightarrow \mathbb{R}^{n}$ such that the cardinality of the preimage $f_{q}^{-1}(x)$ of a point $x\in Im(f_{p})$ is smaller than $u_{pq}$. In particular, if $u_{pq}=\infty$ for all $1\leq p\leq q\leq k$, then we get the usual space of immersions. Similarly, if $u_{pq}=1$ for all $1\leq p\leq q\leq k$, then we get the usual space of embeddings. The space of $\vec{r}$-immersions is endowed with a product, given by the concatenation, which commutative up to homotopy only if $d_{1}\geq 2$.
\begin{figure}[!h]
\begin{center}
\includegraphics[scale=1.8]{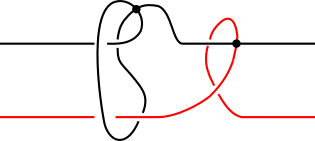}
\caption{An element in $Imm_{c}^{\vec{u}}\big(\,  \mathbb{R}\sqcup\mathbb{R}\,;\,\mathbb{R}^{3}\,\big)$ with $u_{11}=1$, $u_{12}=1$ and $u_{22}=0$.}\vspace{-14pt}
\end{center}
\end{figure}

Unfortunately, we have no information about the convergence of the Goodwillie-Weiss' tower associated with the space of $\vec{u}$-immersions. Nevertheless, using the method describe in Section \ref{L3}, we are able to identify the polynomial approximations of the homotopy fiber\vspace{5pt}
$$
\overline{Imm}_{c}^{\vec{u}}\big(\, \underset{1\leq i\leq k}{\bigsqcup} \mathbb{R}^{d_{i}}\,;\,\mathbb{R}^{n}\,\big)\coloneqq hofib\left(
Imm_{c}^{\vec{u}}\big(\, \underset{1\leq i\leq k}{\bigsqcup} \mathbb{R}^{d_{i}}\,;\,\mathbb{R}^{n}\,\big) \longrightarrow Imm_{c}\big(\, \underset{1\leq i\leq k}{\bigsqcup} \mathbb{R}^{d_{i}}\,;\,\mathbb{R}^{n}\,\big)
\right)\vspace{2pt}
$$
with derived mapping space of (possibly truncated) $k$-fold infinitesimal bimodule. Then, using this identification, we will prove that the polynomial approximations are weakly equivalent to explicit iterated loop spaces. For this purpose, we need to consider the following $k$-fold bimodule. \vspace{-15pt} 

\newpage

\begin{defi}\textbf{The $k$-fold bimodule $\mathcal{R}_{n}[\,\vec{u}\,]$}

\noindent Let $\vec{u}=\{u_{pq}\}_{1\leq p\leq q\leq k}$ be a family  with $u_{pq}\in \mathbb{N}_{>0}\sqcup\{\infty\}$. Then we consider the $k$-fold sequence
$$
\mathcal{R}_{n}[\,\vec{u}\,](A_{1},\ldots,A_{k})\subset \mathcal{R}_{n}^{\infty}(A_{1}\sqcup\ldots\sqcup A_{k}),\hspace{15pt}\forall (A_{1},\ldots,A_{k})\in \Sigma^{\times k},
$$
which is the subspace formed by configurations of little cubes $\{r_{a}\}_{1\leq i\leq k}^{a\in A_{i}}$ satisfying the relation
$$
\forall 1\leq p\leq q\leq k,\,\,\forall a\in A_{p},\,\,\forall (b_{1},\ldots,b_{u_{pq}})\subset A_{q},\,\,Im(r_{a})\cap\underset{1\leq i\leq u_{pq}}{\bigcap} Im(r_{b_{i}})=\emptyset.
$$
The $k$-fold sequence so obtained inherits a $k$-fold bimodule structure over the family of operads $\mathcal{C}_{d_{1}},\ldots,\mathcal{C}_{d_{k}}$ relative to $\mathcal{C}_{n}$ from $\mathcal{R}_{n}^{\infty}$. More precisely, the right $k$-fold bimodule operations are defined using the operadic structure of $\mathcal{R}_{n}^{\infty}$ and the maps $\mathcal{C}_{d_{i}}\rightarrow \mathcal{C}_{d_{k}}\hookrightarrow \mathcal{R}_{d_{k}} \rightarrow \mathcal{R}_{n}^{\infty}$. The left $k$-fold bimodule structure is obtained using the map $\varepsilon$ introduced in Example \ref{L7}.\vspace{3pt}  
\begin{figure}[!h]
\begin{center}
\includegraphics[scale=0.18]{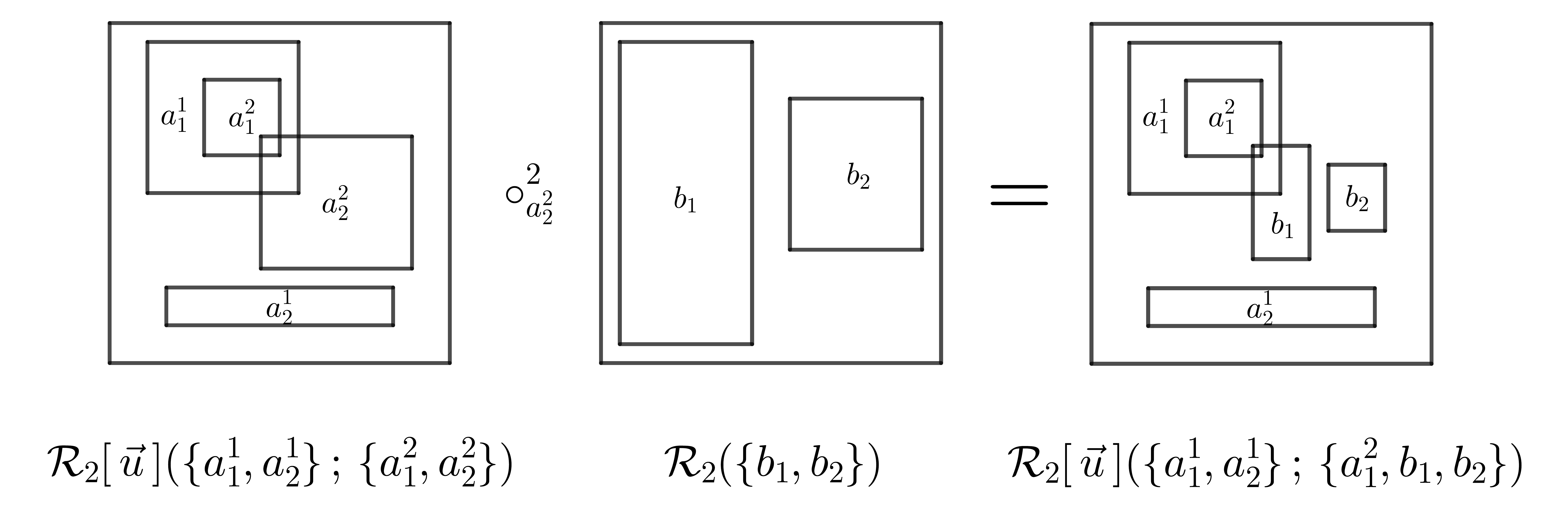}\vspace{-5pt}
\caption{Illustration of the operation $\circ_{a_{2}^{2}}^{2}$ with $u_{11}=1$, $u_{12}=3$ and $u_{22}=2$.}\vspace{-15pt}
\end{center}
\end{figure}
\end{defi}

Similarly to Theorem \ref{K2}, we can use Proposition \ref{K5} in order to express the polynomial approximation of the functor associated with the $\vec{u}$-immersions modulo immersions in terms of derived mapping space of $k$-fold infinitesimal bimodule. More precisely, we get the following identifications:\vspace{5pt}
$$
\begin{array}{rcl}\vspace{7pt}
T_{\vec{r}}\,\overline{Imm}^{\vec{u}}_{c}\big(\, \underset{1\leq i\leq k}{\displaystyle\coprod} \mathbb{R}^{d_{i}}\,;\,\mathbb{R}^{n}\,\big) & \simeq & T_{\vec{r}}\, Ibimod_{\vec{O}}^{h}\big( \mathbb{O}\,;\, \mathcal{R}_{n}[\,\vec{u}\,]\,\big),\\ 
T_{\vec{\infty}}\,\overline{Imm}^{\vec{u}}_{c}\big(\, \underset{1\leq i\leq k}{\displaystyle\coprod} \mathbb{R}^{d_{i}}\,;\,\mathbb{R}^{n}\,\big) & \simeq & Ibimod_{\vec{O}}^{h}\big( \mathbb{O}\,;\, \mathcal{R}_{n}[\,\vec{u}\,]\,\big).
\end{array} \vspace{2pt}
$$ 
As a consequence of Theorem \ref{L6}, one has the following delooping statement:\vspace{5pt}

\begin{thm}\label{R0}
For any integers $d_{1}\leq \cdots \leq d_{k}<n$, one has the identifications \vspace{5pt}
$$
\begin{array}{rcl}\vspace{7pt}
T_{\vec{r}}\,\overline{Imm}^{\vec{u}}_{c}\big(\, \underset{1\leq i\leq k}{\displaystyle\coprod} \mathbb{R}^{d_{i}}\,;\,\mathbb{R}^{n}\,\big) & \simeq &  \Omega^{d_{1}}\big(\,T_{\vec{r}}\,\mathbb{F}'_{\vec{O}}(\mathcal{R}_{n}[\,\vec{u}\,])\big),\\ 
T_{\vec{\infty}}\,\overline{Imm}^{\vec{u}}_{c}\big(\, \underset{1\leq i\leq k}{\displaystyle\coprod} \mathbb{R}^{d_{i}}\,;\,\mathbb{R}^{n}\,\big) & \simeq & \Omega^{d_{1}}\mathbb{F}'_{\vec{O}}(\mathcal{R}_{n}[\,\vec{u}\,]).
\end{array}\vspace{2pt}
$$ 
In particular, if $d_{1}=\cdots =d_{k}=d$, then one has \vspace{5pt}
$$
\begin{array}{rcl}\vspace{7pt}
T_{\vec{r}}\,\overline{Imm}^{\vec{u}}_{c}\big(\, \underset{1\leq i\leq k}{\displaystyle\coprod} \mathbb{R}\,;\,\mathbb{R}^{n}\,\big) & \simeq &  \Omega^{d}\big(\,T_{\vec{r}}\,Bimod_{\vec{O}}^{h}(\mathbb{O}^{+}\,;\,\mathcal{R}_{n}[\,\vec{u}\,])\big),\\ 
T_{\vec{\infty}}\,\overline{Imm}^{\vec{u}}_{c}\big(\, \underset{1\leq i\leq k}{\displaystyle\coprod} \mathbb{R}^{d}\,;\,\mathbb{R}^{n}\,\big) & \simeq & \Omega^{d}Bimod_{\vec{O}}^{h}(\mathbb{O}^{+}\,;\,\mathcal{R}_{n}[\,\vec{u}\,]).
\end{array}
$$ 
\end{thm} 

\newpage

\appendix 

\section{Table of coherence relations}\label{Z0}

\subsection{Axioms associated to $k$-fold infinitesimal bimodules}\label{W0}

Let $O_{1},\ldots,O_{k}$ be a family of reduced operads relative to a topological monoid $X$. A $k$-fold infinitesimal bimodule $N$ satisfies the following axioms: \vspace{5pt}
\begin{itemize}[leftmargin=12pt]
\item[$\blacktriangleright$] For $(A_{1},\ldots,A_{k})\in \Sigma^{\times k}$, finite sets $B$ and $C$ in $\Sigma$ and for all $i\leq j\in \{1,\ldots, k\}$, $a\in A_{i}$ and $b\in A_{j}$ (with $a\neq b$ if $i=j$), the following diagram commutes:
$$
\xymatrix{
N(A_{1},\ldots,A_{k})\times O_{i}(B)\times O_{j}(C) \ar[r]^{\hspace{-5pt}\circ_{i}^{a}\times id} \ar[d]_{\circ_{j}^{b}\times id} & N(A_{1},\ldots, A_{i}\cup_{a}B,\ldots,A_{k})\times O_{j}(C) \ar[d]^{\circ_{j}^{b}}\\
N(A_{1},\ldots, A_{j}\cup_{b}C,\ldots,A_{k})\times O_{i}(B) \ar[r]_{\hspace{-10pt}\circ_{i}^{a}} & N(A_{1},\ldots,  A_{i}\cup_{a}B,\ldots,A_{j}\cup_{b}C,\ldots,A_{k})
}
$$

\item[$\blacktriangleright$] For $(A_{1},\ldots,A_{k})\in \Sigma^{\times k}$, finite sets $B$ and $C$ in $\Sigma$ and for all $i\in \{1,\ldots, k\}$, $a\in A_{i}$ and $b\in B$, the following diagram commutes:
$$
\xymatrix{
N(A_{1},\ldots,A_{k})\times O_{i}(B)\times O_{i}(C) \ar[r]^{\hspace{-5pt}\circ_{i}^{a}\times id} \ar[d]_{id\times \circ_{b}} & N(A_{1},\ldots, A_{i}\cup_{a}B,\ldots,A_{k})\times O_{i}(C) \ar[d]^{\circ_{i}^{b}}\\
N(A_{1},\ldots,A_{k})\times O_{i}(B\cup_{b}C) \ar[r]_{\hspace{-5pt}\circ_{i}^{a}} & N(A_{1},\ldots,  A_{i}\cup_{a}B \cup_{b}C,\ldots,A_{k})
}
$$

\item[$\blacktriangleright$] For $(A_{1},\ldots,A_{k})$, $(B_{1},\ldots,B_{k})$ and $(C_{1},\ldots,C_{k})$ in $\Sigma^{\times k}$, the following diagram commutes:
$$
\xymatrix{
\vec{O}(A_{1},\ldots,A_{k}) \times \vec{O}(B_{1},\ldots,B_{k}) \times N(C_{1},\ldots,C_{k}) \ar[r]^{\hspace{2pt}\tilde{\mu} \times id} \ar[d]_{id\times \mu} & \vec{O}(A_{1}\sqcup B_{1},\ldots,A_{k}\sqcup B_{k}) \times N(C_{1},\ldots, C_{k}) \ar[d]^{\mu}\\
\vec{O}(A_{1},\ldots,A_{k}) \times N(B_{1}\sqcup C_{1},\ldots,B_{k}\sqcup C_{k}) \ar[r]_{\hspace{-1pt}\mu} & N(A_{1}\sqcup B_{1}\sqcup C_{1},\ldots,  A_{k}\sqcup B_{k}\sqcup C_{k})
}
$$

\item[$\blacktriangleright$] For $(A_{1},\ldots,A_{k})$ and $(B_{1},\ldots,B_{k})$ in $\Sigma^{\times k}$, a finite set $C\in \Sigma$ and for all $i\in \{1,\ldots,k\}$ and $b\in B_{i}$, the following diagram commutes: 
$$
\xymatrix{
\vec{O}(A_{1},\ldots,A_{k}) \times N(B_{1},\ldots,B_{k}) \times O_{i}(C) \ar[r]^{\hspace{-10pt}id\times \circ_{i}^{b}} \ar[d]_{\mu\times id} & \vec{O}(A_{1},\ldots,A_{k}) \times N(B_{1},\ldots, B_{i}\cup_{b}C,\ldots,B_{k}) \ar[d]^{\mu}\\
N(A_{1}\sqcup B_{1},\ldots,A_{k}\sqcup B_{k})\times O_{i}(C) \ar[r]_{\hspace{-20pt}\circ_{i}^{b}} & N(A_{1}\sqcup B_{1},\ldots,A_{i}\sqcup B_{i} \cup_{b}C,\ldots,  A_{k}\sqcup B_{k})
}
$$

\item[$\blacktriangleright$]For $(A_{1},\ldots,A_{k})$ and $(B_{1},\ldots,B_{k})$ in $\Sigma^{\times k}$, a finite set $C\in \Sigma$ and for all $i\in \{1,\ldots,k\}$ and $a\in A_{i}$, the following diagram commutes: 
$$
\xymatrix{
\vec{O}(A_{1},\ldots,A_{k}) \times N(B_{1},\ldots,B_{k}) \times O_{i}(C) \ar[r]^{\hspace{-10pt}\circ_{i}^{a}\times id} \ar[d]_{\mu\times id} & \vec{O}(A_{1},\ldots, A_{i}\cup_{a}C,\ldots,A_{k}) \times N(B_{1},\ldots,B_{k}) \ar[d]^{\mu}\\
N(A_{1}\sqcup B_{1},\ldots,A_{k}\sqcup B_{k})\times O_{i}(C) \ar[r]_{\hspace{-20pt}\circ_{i}^{a}} & N(A_{1}\sqcup B_{1},\ldots,A_{i}\sqcup B_{i} \cup_{a}C,\ldots,  A_{k}\sqcup B_{k})
}
$$

\item[$\blacktriangleright$] For $i\in \{1,\ldots,k\}$ and $(A_{1},\ldots,A_{k})\in \Sigma^{\times k}$, the following diagrams commute in which the horizontal maps are obtained by taking the unit $\ast_{1}\in O_{i}(1)$ and the point $(\ast_{1},\ldots,\ast_{1})\in \vec{O}(\emptyset,\ldots,\emptyset)$, respectively:
$$
\xymatrix{
N(A_{1},\ldots,A_{k}) \ar[r] \ar@{=}[d] & N(A_{1},\ldots,A_{k})\times O_{i}(1) \ar[dl]^{\circ_{i}^{1}}\\
N(A_{1},\ldots,A_{k}) &
}\hspace{20pt}
\xymatrix{
N(A_{1},\ldots,A_{k}) \ar[r] \ar@{=}[d] & \vec{O}(\emptyset,\ldots,\emptyset) \times N(A_{1},\ldots,A_{k}) \ar[dl]^{\mu}\\
N(A_{1},\ldots,A_{k}) &
}
$$
\end{itemize}

\subsection{Axioms associated to $k$-fold bimodules}\label{W1}

Let $O_{1},\ldots,O_{k}$ be a family of reduced operads relative to a reduced operad $O$. A $k$-fold bimodule $M$ satisfies the following axioms: \vspace{5pt}
\begin{itemize}[leftmargin=12pt]
\item[$\blacktriangleright$] For $i\leq j\in \{1,\ldots,k\}$, $(A_{1},\ldots,A_{k})\in \Sigma^{\times k}_{+}$ with $A_{i}\neq +\neq A_{j}$ and finite sets $B$ and $C$ in $\Sigma$, and for all $a\in A_{i}$ and $b\in A_{j}$ (with $a\neq b$ if $i=j$), the following diagram commutes:
$$
\xymatrix{
M(A_{1},\ldots,A_{k})\times O_{i}(B)\times O_{j}(C) \ar[r]^{\hspace{-5pt}\circ_{i}^{a}\times id} \ar[d]_{\circ_{j}^{b}\times id} & M(A_{1},\ldots, A_{i}\cup_{a}B,\ldots,A_{k})\times O_{j}(C) \ar[d]^{\circ_{j}^{b}}\\
M(A_{1},\ldots, A_{j}\cup_{b}C,\ldots,A_{k})\times O_{i}(B) \ar[r]_{\hspace{-10pt}\circ_{i}^{a}} & M(A_{1},\ldots,  A_{i}\cup_{a}B,\ldots,A_{j}\cup_{b}C,\ldots,A_{k})
}
$$

\item[$\blacktriangleright$] For $i\in \{1,\ldots,k\}$, $(A_{1},\ldots,A_{k})\in \Sigma^{\times k}_{+}$ with $A_{i}\neq +$ and finite sets $B$ and $C$ in $\Sigma$, and for all $a\in A_{i}$ and $b\in B$, the following diagram commutes:
$$
\xymatrix{
M(A_{1},\ldots,A_{k})\times O_{i}(B)\times O_{i}(C) \ar[r]^{\hspace{-5pt}\circ_{i}^{a}\times id} \ar[d]_{id\times \circ_{b}} & M(A_{1},\ldots, A_{i}\cup_{a}B,\ldots,A_{k})\times O_{i}(C) \ar[d]^{\circ_{i}^{b}}\\
M(A_{1},\ldots,A_{k})\times O_{i}(B\cup_{b}C) \ar[r]_{\hspace{-5pt}\circ_{i}^{a}} & M(A_{1},\ldots,  A_{i}\cup_{a}B \cup_{b}C,\ldots,A_{k})
}
$$

\item[$\blacktriangleright$] For $i\in \{1,\ldots,k\}$ and $(A_{1},\ldots,A_{k})\in \Sigma^{\times k}_{+}$, the following diagrams commute in which the horizontal maps are obtained by taking the unit $\ast_{1}\in O_{i}(1)$ and the point $(\ast_{1},\ldots,\ast_{1})\in \vec{O}_{S}(\emptyset,\ldots,\emptyset)$, respectively:
$$
\xymatrix{
M(A_{1},\ldots,A_{k}) \ar[r] \ar@{=}[d] & M(A_{1},\ldots,A_{k})\times O_{i}(1) \ar[dl]^{\circ_{i}^{1}}\\
M(A_{1},\ldots,A_{k}) &
}\hspace{20pt}
\xymatrix{
M(A_{1},\ldots,A_{k}) \ar[r] \ar@{=}[d] & \vec{O}_{S}(\emptyset,\ldots,\emptyset) \times M(A_{1},\ldots,A_{k}) \ar[dl]^{\mu_{S}}\\
M(A_{1},\ldots,A_{k}) &
}
$$

\item[$\blacktriangleright$] For finite sets $A$ and $B$, any element $a\in A$, $S=(S_{1},\ldots,S_{k})\in \mathcal{P}_{k}(A)$ and  $S'=(S_{1}',\ldots,S_{k}')\in \mathcal{P}_{k}(B)$ with the condition $S_{i}'=+$ iff $a\notin S_{i}$, and $(A_{1}^{c},\ldots,A_{k}^{c})\in \Sigma^{\times k}_{+}$, with $c\in A\cup_{a}B$ and $B_{i}^{c}=+$ iff $c\in S_{i}\cup_{a}S_{i}'$, the following diagram commutes (by abuse of notation, we denote by $\vec{O}_{S}$ the space $\vec{O}_{S}(S_{1},\ldots,S_{k})$): 
$$
\xymatrix{
\vec{O}_{S}\times \vec{O}_{S'}\times \underset{c\in A\cup_{a}B}{\displaystyle\prod} M(A_{1}^{c},\ldots,A_{k}^{c}) \ar[r]^{\mu_{a}\times id} \ar[d]_{id\times \mu_{S'}} & \vec{O}_{S\cup_{a}S'}\times \underset{c\in A\cup_{a}B}{\displaystyle\prod} M(A_{1}^{c},\ldots,A_{k}^{c})\ar[d]^{\mu_{A\cup_{a}B}}\\
\vec{O}_{S} \times M\left( \underset{\substack{c\in B\\ A_{1}^{c}\neq +}}{\displaystyle\coprod} A_{1}^{c} ,\ldots, \underset{\substack{c\in B\\ A_{k}^{c}\neq +}}{\displaystyle\coprod} A_{k}^{c}  \right) \times \underset{c\in A\setminus \{a\}}{\displaystyle\prod} M(A_{1}^{c},\ldots,A_{k}^{c})\ar[r]_{\hspace{40pt}\mu_{S}}
& M\left( \underset{\substack{c\in A\cup_{a}B\\ A_{1}^{c}\neq +}}{\displaystyle\coprod} A_{1}^{c} ,\ldots, \underset{\substack{c\in A\cup_{a}B\\ A_{k}^{c}\neq +}}{\displaystyle\coprod} A_{k}^{c}  \right)
}
$$

\item[$\blacktriangleright$] For finite sets $A$ and $C$, $S=(S_{1},\ldots,S_{k})\in \mathcal{P}_{k}(A)$, $(B_{1}^{a},\ldots,B_{k}^{a})\in \Sigma^{\times k}_{+}$, with $a\in A$ and $B_{i}^{a}=+$ iff $a\in S_{i}$, and for all $i\in \{1,\ldots,k\}$, $a'\in A$ and $b\in B_{i}^{a'}\neq +$, the following diagram commutes: 
$$
\xymatrix{
\vec{O}_{S}\times \underset{a\in A}{\displaystyle\prod} M(B_{1}^{a},\ldots,B_{k}^{a})\times O_{i}(C)\ar[r]^{\mu_{S}\times id} \ar[d]_{id\times \circ_{i}^{b}} & M\left( \underset{\substack{a\in A\\ B_{1}^{a}\neq +}}{\displaystyle\coprod} B_{1}^{a} ,\ldots, \underset{\substack{a\in A\\ B_{k}^{a}\neq +}}{\displaystyle\coprod} B_{k}^{a}  \right) \times O_{i}(C)\ar[d]^{\circ_{i}^{b}}\\
\vec{O}_{S} \times M(B_{1}^{a'},\ldots,B_{1}^{a'}\cup_{b}C,\ldots,B_{k}^{a'}) \underset{a\in A\setminus a'}{\displaystyle\prod} M(B_{1}^{a},\ldots,B_{k}^{a})\ar[r]_{\mu_{S}} & M\left( \underset{\substack{a\in A\\ B_{1}^{a}\neq +}}{\displaystyle\coprod} B_{1}^{a} ,\ldots, \underset{\substack{a\in A\\ B_{i}^{a}\neq +}}{\displaystyle\coprod} B_{i}^{a} \cup_{b}C, \ldots, \underset{\substack{a\in A\\ B_{k}^{a}\neq +}}{\displaystyle\coprod} B_{k}^{a}  \right)
}
$$

\end{itemize}

\begin{merci}
I would like to thank Thomas Willwacher for his help in preparing this paper. I also wish to express my gratitude to Pedro Boavida, Najib Idrissi, Muriel Livernet and Victor Turchin for many helpful comments. The author is also indebted to Benoit Fresse for answering numerous questions. 
\end{merci}

\bibliographystyle{amsplain}
\bibliography{bibliography}

\vspace{20pt}

\noindent Julien Ducoulombier: Department of Mathematics, ETH Zurich, Ramistrasse 101, Zurich, Switzerland\\
\textit{E-mail address: } \href{mailto:julien.ducoulombier@math.ethz.ch}{julien.ducoulombier@math.ethz.ch}

\end{document}